\documentclass[11pt,twoside]{amsart}

\usepackage[T1]{fontenc}
\usepackage[latin1]{inputenc}
\usepackage[english]{babel}

\usepackage{amsmath,amsthm,amssymb,amsfonts,amscd}
\usepackage{url}
\usepackage{color}
\usepackage{setspace}
\usepackage{enumerate}
\usepackage{enumitem}
\usepackage{xcolor}

\usepackage[all]{xy}

\usepackage{epsfig}
\usepackage{graphicx}
\usepackage{mathrsfs}
\usepackage{pinlabel}
\usepackage[outercaption]{sidecap}

\definecolor{light-gray}{gray}{0.60}
\definecolor{lightblue}{rgb}{0.5,0.5,.9}

\theoremstyle{plain}
\newtheorem{theorem}{Theorem}
\newtheorem{proposition}[theorem]{Proposition}
\newtheorem{corollary}[theorem]{Corollary}
\newtheorem{lemma}[theorem]{Lemma}
\newtheorem{fact}[theorem]{Fact}

\theoremstyle{definition}
\newtheorem{definition}[theorem]{Definition}
\newtheorem{example}[theorem]{Example}
\newtheorem{examples}[theorem]{Examples}
\newtheorem{remark}[theorem]{Remark}
\newtheorem{remarks}[theorem]{Remarks}
\newtheorem{question}[theorem]{Question}

\numberwithin{theorem}{section}
\numberwithin{equation}{section}

\newcommand{\C}{\mathcal C}

\newcommand{\NN}{\mathbb{N}}
\newcommand{\ZZ}{\mathbb{Z}}

\newcommand{\RR}{\mathbb{R}}

\newcommand{\HH}{\mathbb{H}}
\newcommand{\PP}{\mathbb{P}}

\renewcommand{\SS}{\mathbb{S}}

\newcommand{\SL}{\mathrm{SL}}
\newcommand{\GL}{\mathrm{GL}}
\newcommand{\SO}{\mathrm{SO}}
\newcommand{\OO}{\mathrm{O}}
\newcommand{\PO}{\mathrm{PO}}
\newcommand{\PSO}{\mathrm{PSO}}
\newcommand{\PSL}{\mathrm{PSL}}
\newcommand{\PGL}{\mathrm{PGL}}
\newcommand{\Sp}{\mathrm{Sp}}

\newcommand{\Hom}{\mathrm{Hom}}
\newcommand{\ie}{i.e.\ }
\newcommand{\eg}{e.g.\ }
\newcommand{\resp}{resp.\ }

\newcommand{\Fr}{\mathrm{Fr}}
\newcommand{\Int}{\mathrm{Int}}

\newcommand{\partiali}{\partial_{\mathrm{i}}}
\newcommand{\partialn}{\partial_{\mathrm{n}}}
\newcommand{\Lambdao}{\Lambda^{\mathsf{orb}}}
\newcommand{\Lambdacon}{\Lambda^{\mathsf{con}}}
\newcommand{\Ccore}{\C^{\mathsf{cor}}}

\newcommand{\cro}[4]{[#1 \!:\! #2 \!:\! #3 \!:\! #4]}

\title{Convex cocompact actions in real projective geometry}

\author{Jeffrey Danciger}
\address{Department of Mathematics, The University of Texas at Austin, 1 University Station C1200, Austin, TX 78712, USA}
\email{jdanciger@math.utexas.edu}

\author{Fran\c{c}ois Gu\'eritaud}
\address{CNRS and IRMA, Universit\'e de Strasbourg, 7 rue Ren\'e Descartes, 67084 Strasbourg Cedex, France}
\email{francois.gueritaud@unistra.fr}

\author{Fanny Kassel}
\address{CNRS and Laboratoire Alexander Grothendieck, Institut des Hautes \'Etudes Scientifiques, Universit\'e Paris-Saclay, 35 route de Chartres, 91440 Bures-sur-Yvette, France}
\email{kassel@ihes.fr}

\thanks{This project received funding from the European Research Council (ERC) under the European Union's Horizon 2020 research and innovation programme (ERC starting grant DiGGeS, grant agreement No 715982).
The authors also acknowledge support from the GEAR Network, funded by the National Science Foundation under grants DMS 1107452, 1107263, and 1107367 (``RNMS: GEometric structures And Representation varieties'').
J.D. was partially supported by an Alfred P. Sloan Foundation fellowship, and by the National Science Foundation under grant DMS 1510254.
F.G. and F.K. were partially supported by the Agence Nationale de la Recherche through the Labex CEMPI (ANR-11-LABX-0007-01) and the grant DynGeo (ANR-16-CE40-0025-01).
Part of this work was completed while F.K. was in residence at the MSRI in Berkeley, California, for the program \emph{Geometric Group Theory} (Fall 2016) supported by NSF grant DMS 1440140, and at the INI in Cambridge, UK, for the program \emph{Nonpositive curvature group actions and cohomology} (Spring~2017) supported by EPSRC grant EP/K032208/1.}

\begin{document}

\begin{abstract}
We study a notion of convex cocompactness for discrete subgroups of the projective general linear group acting (not necessarily irreducibly) on real projective space, and give various characterizations.
A convex cocompact group in this sense need not be word hyperbolic, but we show that it still has some of the good properties of classical convex cocompact subgroups in rank-one Lie groups.
Extending our earlier work \cite{dgk-ccHpq} from the context of projective orthogonal groups, we show that for word hyperbolic groups preserving a properly convex open set in projective space, the above general notion of convex cocompactness is equivalent to a stronger convex cocompactness condition studied by Crampon--Marquis, and also to the condition that the natural inclusion be a projective Anosov representation.
We investigate examples.
\end{abstract}

\maketitle
\tableofcontents

\section{Introduction}

In the classical setting of semisimple Lie groups $G$ of real rank one, a discrete subgroup of~$G$ is said to be convex cocompact if it acts cocompactly on some nonempty closed convex subset of the Riemannian symmetric space $G/K$ of~$G$.
Such subgroups have been abundantly studied, in particular in the context of Kleinian groups and real hyperbolic geometry, where there is a rich world of examples.
They are known to display good geometric and dynamical behavior.

On the other hand, in higher-rank semisimple Lie groups~$G$, the condition that a discrete subgroup $\Gamma$ act cocompactly on some nonempty convex subset of the Riemannian symmetric space $G/K$ turns out to be quite restrictive: Kleiner--Leeb \cite{kl06} and Quint \cite{qui05} proved, for example, that if $G$ is simple and such a subgroup $\Gamma$ is Zariski-dense in~$G$, then it is in fact a uniform lattice of~$G$.

The notion of an \emph{Anosov representation} of a word hyperbolic group in a higher-rank semisimple Lie group $G$, introduced by Labourie \cite{lab06} and generalized by Guichard--Wienhard \cite{gw12}, is a much more flexible notion, which has earned a central role in higher Teichm\"uller--Thurston theory, see \eg \cite{biw14,bcls15,klp-survey,ggkw17,bps}.
Anosov representations are defined, not in terms of convex subsets of the Riemannian symmetric space $G/K$, but rather in terms of a dynamical condition for the action on a certain flag variety, \ie on a compact homogeneous space $G/P$.
This dynamical condition guarantees many desirable analogies with convex cocompact subgroups in rank one: see \eg \cite{lab06, gw12, klp14, klp18, klp-survey}.
It also allows for the definition of certain interesting geometric structures associated to Anosov representations: see \eg \cite{gw08, gw12, klp18, ggkw17, ctt19}.
However, natural \emph{convex} geometric structures associated to Anosov representations have been lacking in general.
Such structures could allow geometric intuition to bear more fully on Anosov representations, making them more accessible through familiar geometric constructions such as convex fundamental domains, and potentially unlocking new sources of examples.
While there is a rich supply of examples of Anosov representations into higher-rank Lie groups in the case of surface groups or free groups, it has proven difficult to construct examples for more complicated word hyperbolic groups.

One of the goals of this paper is to show that, when $G = \PGL(\RR^n)$ is a projective linear group, there are, in many cases, natural convex cocompact geometric structures modeled on $\PP(\RR^n)$ associated to Anosov representations into~$G$.
The idea is the following: to any discrete subgroup $\Gamma$ of $G = \PGL(\RR^n)$ are associated two \emph{limit sets} $\Lambda_{\Gamma}$ and $\Lambda_{\Gamma}^*$.
Recall that an element $g\in\PGL(\RR^n)$ is said to be \emph{proximal} in the real projective space $\PP(\RR^n)$ if it admits a unique attracting fixed point in $\PP(\RR^n)$ (see Section~\ref{subsec:prox}).
The \emph{proximal limit set} $\Lambda_{\Gamma}$ of $\Gamma$ in $\PP(\RR^n)$ is defined as the closure of the set of attracting fixed points of proximal elements of $\Gamma$ in $\PP(\RR^n)$.
Similarly, we consider the proximal limit set $\Lambda_{\Gamma}^*$ of $\Gamma$ in the dual projective space $\PP((\RR^n)^*)$ for the dual action; we can view it as a set of projective hyperplanes in $\PP(\RR^n)$.
Suppose the complement $\PP(\RR^n) \smallsetminus \bigcup_{H\in\Lambda_{\Gamma}^*} H$ is nonempty.
Its connected components are open sets which (as soon as $\Lambda_{\Gamma}^*\neq\nolinebreak\emptyset$) are \emph{convex}, in the sense that they are contained and convex in some affine chart of $\PP(\RR^n)$; when $\Gamma$ acts irreducibly on $\PP(\RR^n)$, these components are even \emph{properly convex}, in the sense that they are convex and bounded in some affine chart.
If one of these convex open sets, call it $\Omega_{\max}$, is properly convex and \emph{invariant} under~$\Gamma$, then the action of $\Gamma$ on $\Omega_{\max}$ is necessarily properly discontinuous (see Section~\ref{subsec:prop-conv-proj}), the set $\Lambda_{\Gamma}$ is contained in the boundary $\partial\Omega_{\max}$, and it makes sense to consider the convex hull of $\Lambda_{\Gamma}$ in $\Omega_{\max}$: this is a closed convex subset $\C$ of $\Omega_{\max}$.
When $\Gamma\hookrightarrow G$ is (projective) Anosov, we prove that the action of $\Gamma$ on the convex set~$\C$ is cocompact.
Further, we prove that $\Gamma$ satisfies a \emph{stronger} notion of projective convex cocompactness introduced by Crampon--Marquis~\cite{cm14}. 
Conversely, we show that convex cocompact subgroups of $\PP(\RR^n)$ in the sense of \cite{cm14} always give rise to Anosov representations, which enables us to give new examples of Anosov representations and study their deformation spaces by constructing these geometric structures directly.
In~\cite{dgk-ccHpq} we had previously established this close connection between convex cocompactness in projective space and Anosov representations in the case of irreducible representations valued in a projective orthogonal group $\PO(p,q)$.

One context where such a connection between Anosov representations and convex projective structures has been known for some time is the deformation theory of real projective surfaces, for $G = \PGL(\RR^3)$ \cite{gol90,cg05}.
More generally, it follows from work of Benoist \cite{ben04} that if a discrete subgroup $\Gamma$ of $G=\PGL(\RR^n)$ \emph{divides} (\ie acts cocompactly on) a strictly convex open subset $\Omega$ of $\PP(\RR^n)$, then~$\Gamma$ is word hyperbolic and the natural inclusion $\Gamma\hookrightarrow G$ is Anosov.
In this particular case $\Lambda_{\Gamma}$ is the boundary of~$\Omega$ and $\Lambda_{\Gamma}^*$ is the collection of supporting hyperplanes to~$\Omega$.

Benoist \cite{ben06} also found examples of discrete subgroups of $\PGL(\RR^n)$, acting irreducibly on $\PP(\RR^n)$, which divide properly convex open sets that are not strictly convex, for $4\leq n\leq 7$; these subgroups are not word hyperbolic.
In this paper we study a broad notion of convex cocompactness for discrete subgroups $\Gamma$ of $\PGL(\RR^n)$ acting on $\PP(\RR^n)$ which simultaneously generalizes Crampon--Marquis's notion and Benoist's divisible convex sets \cite{ben04,ben03,ben05,ben06}.
While we mainly take the point of view of examining limit sets and convex hulls in projective space, we also show that this notion of convex cocompactness is characterized by the property that $\Gamma$ is (up to finite index) the holonomy group of a compact convex projective manifold with strictly convex boundary.
Cooper--Long--Tillmann \cite{clt18} have studied the deformation theory of such manifolds and their work implies that this notion is \emph{stable} under small deformations of $\Gamma$ in $\PGL(\RR^n)$.
We show further that it is stable under deformation into larger projective general linear groups $\PGL(\RR^{n+n'})$, after the model of quasi-Fuchsian deformations of Fuchsian groups.
When nontrivial deformations exist, this yields examples of nonhyperbolic discrete subgroups which satisfy our convex cocompactness property but do not divide a properly convex open set.

We now describe our results in more detail.

\subsection{Strong convex cocompactness in $\PP(V)$ and Anosov representations} \label{subsec:intro-strong-cc}

In the whole paper, we fix an integer $n\geq 2$ and set $V := \RR^n$.
Recall that an open subset $\Omega$ of the projective space $\PP(V)$ is called \emph{convex} if it is contained and convex in some affine chart, \emph{properly convex} if its closure is convex, and \emph{strictly convex} if in addition its boundary does not contain any nontrivial projective line segment.
It is said to have \emph{boundary of class } $C^1$ (or just \emph{$C^1$ boundary}) if every point of the boundary of~$\Omega$ has a unique supporting hyperplane.

In \cite{cm14}, Crampon--Marquis introduced a notion of \emph{geometrically finite action} of a discrete subgroup $\Gamma$ of $\PGL(V)$ on a strictly convex open domain of $\PP(V)$ with boundary of class $C^1$.
If cusps are not allowed (or equivalently, if we request all infinite-order elements to be proximal), this notion reduces to a natural notion of convex cocompact action on such domains.
We will call discrete groups $\Gamma$ with such actions \emph{strongly convex cocompact}. 

\begin{definition} \label{def:strong-cc}
Let $\Gamma \subset \PGL(V)$ be an infinite discrete subgroup.
\begin{itemize}
 \item Let $\Omega$ be a $\Gamma$-invariant properly convex open subset of $\PP(V)$.
  The action of $\Gamma$ on~$\Omega$ is \emph{strongly convex cocompact} if $\Omega$ is strictly convex with boundary of class~$C^1$ and for some $x \in \Omega$, the convex hull in~$\Omega$ of the \emph{orbital limit set} $\overline{\Gamma\cdot x} \cap \partial \Omega$ is nonempty and has compact quotient by~$\Gamma$. 
  \item The group $\Gamma$ is \emph{strongly convex cocompact in $\PP(V)$} if it admits a strongly convex cocompact action on some properly convex open subset $\Omega$ of $\PP(V)$.
\end{itemize}
\end{definition}

In the definition above, the convex hull in $\Omega$ of the orbital limit set $\overline{\Gamma\cdot x} \cap \partial \Omega$ is defined as the intersection of $\Omega$ with the convex hull of $\overline{\Gamma\cdot x} \cap \partial \Omega$ in the convex set $\overline{\Omega}$.
Moreover, $\overline{\Gamma\cdot x} \cap \partial \Omega$ is independent of the choice of $x \in \Omega$ since $\Omega$ is strictly convex, see Lemma~\ref{lem:strict-convex-Lambda-orb}.
The following observation is an easy consequence.

\begin{remark} \label{rem:def-cc-simpler}
The action of $\Gamma$ on~$\Omega$ is strongly convex cocompact in the sense of Definition~\ref{def:strong-cc} if and only if~$\Omega$ is strictly convex with boundary of class~$C^1$ and \emph{some} nonempty closed convex subset of~$\Omega$ has compact quotient by~$\Gamma$.
\end{remark}

\begin{example} \label{ex:H-p}
For $V=\RR^{p,1}$ with $p\geq 2$, any discrete subgroup of $\mathrm{Isom}(\HH^p)=\PO(p,1)\subset\PGL(V)$ which is convex cocompact in the usual sense is strongly convex cocompact in $\PP(V)$, taking $\Omega = \{ [v] \in \PP(\RR^{p+1}) ~|~ \langle v,v\rangle_{p,1} < 0\}$ to be the projective model of~$\HH^p$ (where $\langle\cdot,\cdot\rangle_{p,1}$ is a symmetric bilinear form of signature $(p,1)$ on~$\RR^{p+1}$).
\end{example}

The first main result of this paper is a close connection between strong convex cocompactness in $\PP(V)$ and Anosov representations into $\PGL(V)$.
Let $P_1$ (\resp $P_{n-1}$) be the stabilizer in $G=\PGL(V)$ of a line (\resp hyperplane) of $V=\RR^n$; it is a maximal proper parabolic subgroup of~$G$, and $G/P_1$ (\resp $G/P_{n-1}$) identifies with $\PP(V)$ (\resp with the dual projective space $\PP(V^*)$).
We shall think of $\PP(V^*)$ as the space of projective hyperplanes in $\PP(V)$.
Let $\Gamma$ be a word hyperbolic group, with Gromov boundary $\partial_{\infty} \Gamma$.
A \emph{$P_1$-Anosov representation} (sometimes also called a \emph{projective Anosov representation}) of $\Gamma$ into~$G$ is a representation $\rho : \Gamma\to G$ for which there exist two continuous, $\rho$-equivariant boundary maps $\xi : \partial_{\infty}\Gamma\to\PP(V)$ and $\xi^* : \partial_{\infty}\Gamma\to\PP(V^*)$ which
\begin{enumerate}[label=(A\arabic*)]
  \item\label{item:compatible} are compatible, \ie $\xi(\eta)\in\xi^*(\eta)$ for all $\eta\in\partial_{\infty}\Gamma$,
  \item\label{item:transverse} are transverse, \ie $\xi(\eta)\notin\xi^*(\eta')$ for all $\eta\neq\eta'$ in $\partial_{\infty}\Gamma$,
  \item\label{item:flow} have an associated flow with some uniform contraction/expansion property described in \cite{lab06,gw12}.
\end{enumerate}
We do not state condition~\ref{item:flow} precisely, since we will use in place of it a simple condition on eigenvalues or singular values described in Definition~\ref{def:P1-Ano} and Fact~\ref{fact:charact-Ano} below, taken from \cite{ggkw17}.
A consequence of~\ref{item:flow} is that every infinite-order element of $\rho(\Gamma)$ is proximal in $\PP(V)$ and in $\PP(V^*)$, and that the image $\xi(\partial_{\infty} \Gamma)$ (\resp $\xi^*(\partial_{\infty} \Gamma)$) of the boundary map is the proximal limit set $\Lambda_\Gamma$ (\resp $\Lambda_\Gamma^*$) of $\rho(\Gamma)$ in $\PP(V)$ (\resp $\PP(V^*)$).
By \cite[Prop.\,4.10]{gw12}, if $\rho$ is irreducible then condition~\ref{item:flow} is automatically satisfied as soon as \ref{item:compatible} and \ref{item:transverse} are, but this is not true in general: see \cite[Ex.\,7.15]{ggkw17}.

It is well known (see \cite[\S\,6.1 \& Th.\,4.3]{gw12}) that a discrete subgroup of $\PO(p,1)$ is convex cocompact in the classical sense if and only if it is word hyperbolic and the natural inclusion $\Gamma \hookrightarrow \PO(p,1) \hookrightarrow \PGL(\RR^{p+1})$ is $P_1$-Anosov.
In this paper, we prove the following higher-rank generalization, where $\Lambda_{\Gamma}^*$ denotes the proximal limit set of $\Gamma$ in $\PP(V^*)$, viewed as a set of projective hyperplanes in $\PP(V)$.

\begin{theorem} \label{thm:Ano-PGL}
Let $\Gamma$ be an infinite discrete subgroup of $G=\PGL(V)$.
Suppose the set $\PP(V) \smallsetminus \bigcup_{H\in\Lambda_{\Gamma}^*} H$ admits a $\Gamma$-invariant connected component (this is always the case if $\Gamma$ preserves a nonempty properly convex open subset of $\PP(V)$, see Proposition~\ref{prop:max-inv-conv}).
Then the following are equivalent:
\begin{enumerate}
  \item\label{item:strong-proj-cc} $\Gamma$ is strongly convex cocompact in $\PP(V)$;
  \item\label{item:Ano-PGL} $\Gamma$ is word hyperbolic and the natural inclusion $\Gamma\hookrightarrow G$ is $P_1$-Anosov.
\end{enumerate}
\end{theorem}

As mentioned above, for $\Gamma$ acting cocompactly on a strictly convex open set (which is then a \emph{divisible strictly convex set}), the implication \eqref{item:strong-proj-cc}~$\Rightarrow$~\eqref{item:Ano-PGL} follows from work of Benoist \cite{ben04}.

\begin{remarks} \label{rem:Ano-PGL}
\begin{enumerate}[label=(\alph*)]
  \item The fact that a strongly convex cocompact group is word hyperbolic is due to Crampon--Marquis \cite[Th.\,1.8]{cm14}.
  \item\label{item:POpq-irred} In the case where $\Gamma$ acts irreducibly on $\PP(V)$ (\ie does not preserve any nontrivial projective subspace of $\PP(V)$) and is contained in $\PO(p,q) \subset \PGL(V)$ for some $p,q\in\NN^*$ with $p+q=n=\dim(V)$, Theorem~\ref{thm:Ano-PGL} was first proved in our earlier work \cite[Th.\,1.11 \& Prop.\,1.17 \& Prop.\,3.7]{dgk-ccHpq}.
In that case we actually gave a more precise version of Theorem~\ref{thm:Ano-PGL} involving the notion of negative/positive proximal limit set: see Section~\ref{subsec:intro-cc-POpq} below.
Our proof of Theorem~\ref{thm:Ano-PGL} in the present paper uses many of the ideas of \cite{dgk-ccHpq}.
One main improvement here is the treatment of duality in the general case where there is no nonzero $\Gamma$-invariant quadratic form.
  \item\label{item:Zimmer} In independent and simultaneous work, Zimmer \cite{zim} also extends \cite{dgk-ccHpq} by studying a slightly different notion for actions of discrete subgroups $\Gamma$ of $\PGL(V)$ on properly convex open subsets $\Omega$ of $\PP(V)$: by definition \cite[Def.\,1.8]{zim}, a subgroup $\Gamma$ of $\mathrm{Aut}(\Omega)$ is regular convex cocompact if it acts cocompactly on some nonempty, $\Gamma$-invariant, closed, properly convex subset $\C$ of~$\Omega$ such that every extreme point of $\overline{\C}$ in $\partial\Omega$ is a $C^1$ extreme point of~$\Omega$.
By \cite[Th.\,1.22]{zim}, if $\Gamma \subset \mathrm{Aut}(\Omega)$ is regular convex cocompact and acts irreducibly on $\PP(V)$, then $\Gamma$ is word hyperbolic and the natural inclusion $\Gamma\hookrightarrow\PGL(V)$ is $P_1$-Anosov.
Conversely, by \cite[Th.\,1.10]{zim}, any irreducible $P_1$-Anosov representation $\Gamma\hookrightarrow\SL(V)$ can be composed with an irreducible representation $\SL(V)\to\SL(V')$, for some larger vector space~$V'$, so that $\Gamma$ becomes regular convex cocompact in $\mathrm{Aut}(\Omega')$ for some $\Omega'\subset\PP(V')$.
It follows from Theorem~\ref{thm:main-noPETs} below that Zimmer's notion of regular convex cocompactness is equivalent to our notion of strong convex cocompactness even in the case that $\Gamma$ does not act irreducibly on $\PP(V)$ (condition~\ref{item:ccc-CM} of Theorem~\ref{thm:main-noPETs} implies regular convex cocompactness, which implies condition~\ref{item:ccc-noPETs-some} of Theorem~\ref{thm:main-noPETs}).
  \item In this paper, unlike in \cite{dgk-ccHpq} or \cite{zim}, we do not assume $\Gamma$ to act irreducibly on $\PP(V)$.
  This makes the notion of Anosov representation slightly more involved (condition \ref{item:flow} above is not automatic), and also adds some subtleties to the notion of convex cocompactness (see \eg Remarks \ref{rem:equality-sets} and~\ref{rem:not-irred}).
  We note that there exist strongly convex cocompact groups in $\PP(V)$ which do not act irreducibly on $\PP(V)$, and whose Zariski closure is not even reductive (which means that there is an invariant projective subspace $\PP(W)$ of $\PP(V)$, but no complementary subspace $W'$ of~$W$ such that $\PP(W')$ is invariant): see Section~\ref{subsec:cocycles}.
  This contrasts with the case of divisible convex sets described in \cite{vey70}.
\end{enumerate}
\end{remarks}

\begin{remark} \label{rem:Ano-no-conv}
For $n=\dim(V)\geq 3$, there exist $P_1$-Anosov representations $\rho : \Gamma\to G=\PGL(V)$ which do not preserve any nonempty properly convex subset of $\PP(V)$: see \cite[Ex.\,5.2 \& 5.3]{dgk-ccHpq}.
However, by \cite[Th.\,1.7]{dgk-ccHpq}, if $\partial_{\infty}\Gamma$ is \emph{connected}, then any $P_1$-Anosov representation $\rho$ valued in $\PO(p,q)\subset\PGL(\RR^{p+q})$ preserves a nonempty properly convex open subset of $\PP(\RR^{p+q})$, hence $\rho(\Gamma)$ is strongly convex cocompact in $\PP(\RR^{p+q})$ by Theorem~\ref{thm:Ano-PGL}.
Extending this, \cite[Th.\,1.24]{zim} gives sufficient group-theoretic conditions on~$\Gamma$ for $P_1$-Anosov representations $\Gamma\hookrightarrow\PGL(V)$ to be regular convex cocompact in $\mathrm{Aut}(\Omega)$ (in the sense of Remark~\ref{rem:Ano-PGL}.\ref{item:Zimmer}) for some $\Omega\subset\PP(V)$.
\end{remark}

\subsection{Convex projective structures for Anosov representations}

We can apply Theorem~\ref{thm:Ano-PGL} to show that some well-known families of Anosov representations, such as Hitchin representations in odd dimension, naturally give rise to convex cocompact real projective manifolds.

\begin{proposition} \label{prop:Hitchin}
Let $\Gamma$ be a closed surface group of genus $\geq 2$ and $\rho : \Gamma\to\PSL(\RR^n)$ a Hitchin representation.
\begin{enumerate}
  \item\label{item:Hit-odd} If $n$ is odd, then $\rho(\Gamma)$ is strongly convex cocompact in $\PP(\RR^n)$.
  \item\label{item:Hit-even} If $n$ is even, then $\rho(\Gamma)$ is not strongly convex cocompact in $\PP(\RR^n)$; in fact it does not even preserve any nonempty properly convex subset of $\PP(\RR^n)$.
\end{enumerate}
\end{proposition}

For statement~\eqref{item:Hit-odd}, see also \cite[Cor.\,1.31]{zim}.
This extends \cite[Prop.\,1.19]{dgk-ccHpq}, about Hitchin representations valued in $\SO(k+1,k)\subset\PSL(\RR^{2k+1})$.

\begin{remarks}
\begin{enumerate}[label=(\alph*)]
  \item The case $n=3$ of Proposition~\ref{prop:Hitchin}.\eqref{item:Hit-odd} is due to Choi--Goldman \cite{gol90,cg05}, who proved that the holonomy representations of the convex projective structures on a given closed hyperbolic surface~$S$ are exactly the elements of the Hitchin component of $\Hom(\pi_1(S),\PSL(\RR^3))$.
  \item Guichard--Wienhard \cite[Th.\,11.3 \& 11.5]{gw12} associated different geometric structures to Hitchin representations into $\PSL(\RR^n)$.
  For even~$n$, their geometric structures are modeled on $\PP(\RR^n)$ but can never be convex (see Proposition~\ref{prop:Hitchin}.\eqref{item:Hit-even}).
  For odd~$n$, their geometric structures are modeled on the space $\mathcal{F}_{1,n-1}\subset\PP(\RR^n)\times\PP((\RR^n)^*)$ of pairs $(\ell,H)$ where $\ell$ is a line of~$\RR^n$ and $H$ a hyperplane containing~$\ell$; these geometric structures lack a notion of convexity but live on manifolds that are \emph{closed}, unlike the manifolds $\rho(\Gamma)\backslash\Omega$ of Proposition~\ref{prop:Hitchin}.\eqref{item:Hit-odd} for $n>3$.
\end{enumerate}
\end{remarks}

We refer to Proposition~\ref{prop:conn-comp-Ano} for a more general statement on convex structures for connected open sets of Anosov representations.

\subsection{New examples of Anosov representations}

We can also use the implication \eqref{item:strong-proj-cc}~$\Rightarrow$~\eqref{item:Ano-PGL} of Theorem~\ref{thm:Ano-PGL} to obtain new examples of Anosov representations by constructing
explicit strongly convex cocompact groups in $\PP(V)$.
Following this strategy, in \cite[\S\,8]{dgk-ccHpq} we showed that every word hyperbolic right-angled Coxeter group $W$ admits reflection-group representations into some $\PGL(V)$ which are $P_1$-Anosov.
Extending this approach, in~\cite{dgklm} we give an explicit description of the deformation spaces of such representations, for Coxeter groups that are not necessarily right-angled.

\subsection{General convex cocompactness in $\PP(V)$} \label{subsec:intro-general-cc}

We now discuss generalizations of Definition~\ref{def:strong-cc} where the properly convex open set $\Omega$ is not assumed to have any regularity at the boundary.
These cover a larger class of groups, not necessarily word hyperbolic.

Given Remark~\ref{rem:def-cc-simpler}, a naive generalization of Definition~\ref{def:strong-cc} that immediately comes to mind is the following.

\begin{definition} \label{def:cc-naive}
An infinite discrete subgroup $\Gamma$ of $\PGL(V)$ is \emph{naively convex cocompact in $\PP(V)$} if it preserves a properly convex open subset $\Omega$ of $\PP(V)$ and acts cocompactly on some nonempty closed convex subset $\C$ of~$\Omega$.
\end{definition}

However, the class of naively convex cocompact subgroups of $\PGL(V)$ is not stable under small deformations: see Remark~\ref{rem:stable-precise}.\ref{item:weak-cond-not-open}.
This is linked to the fact that if $\Gamma$ and~$\Omega$ are as in Definition~\ref{def:cc-naive} with $\Omega$ not strictly convex, then the set of accumulation points of a $\Gamma$-orbit of~$\Omega$ may depend on the orbit (see Example~\ref{ex:Zn}).

To address this issue, we introduce a notion of limit set that does not depend on a choice of orbit.

\begin{definition}\label{def:lambdaorb}
Let $\Gamma \subset \PGL(V)$ be an infinite discrete subgroup and let $\Omega$ be a properly convex open subset of $\PP(V)$ invariant under~$\Gamma$.
The \emph{full orbital limit set} $\Lambdao_{\Omega}(\Gamma)$ of $\Gamma$ in~$\Omega$ is the union of all accumulation points of all $\Gamma$-orbits in~$\Omega$.
\end{definition}

The full orbital limit set $\Lambdao_{\Omega}(\Gamma)$ always contains the proximal limit set $\Lambda_{\Gamma}$ (Lemma~\ref{lem:Lambda-prox-Lambda-orb}), but may be larger.
Using this new limit set, we can introduce another generalization of Definition~\ref{def:strong-cc} which is slightly stronger and has better properties than Definition~\ref{def:cc-naive}: here is the main definition of the paper.

\begin{definition} \label{def:cc-general}
Let $\Gamma$ be an infinite discrete subgroup of $\PGL(V)$.
\begin{itemize}
  \item Let $\Omega$ be a nonempty $\Gamma$-invariant properly convex open subset of $\PP(V)$.
  The action of $\Gamma$ on~$\Omega$ is \emph{convex cocompact} if the convex hull $\Ccore_\Omega(\Gamma)$ of the full orbital limit set $\Lambdao_{\Omega}(\Gamma)$ in~$\Omega$ is nonempty and has compact quotient by~$\Gamma$. 
  \item $\Gamma$ is \emph{convex cocompact in $\PP(V)$} if it admits a convex cocompact action on some nonempty properly convex open subset $\Omega$ of $\PP(V)$.
\end{itemize}
\end{definition}

Note that $\Lambdao_{\Omega}(\Gamma)$ need not be closed in general, hence its convex hull in $\Omega$ need not be either. However, if the action of $\Gamma$ on $\Omega$ is convex cocompact in the above sense, then $\Lambdao_{\Omega}(\Gamma)$ is closed (see Corollary~\ref{cor:ideal-bound-naive-cc}.\eqref{item:ideal-bound-cc}).

In the setting of Definition~\ref{def:cc-general}, the set $\Gamma\backslash\Ccore_\Omega(\Gamma)$ is a compact convex subset of the projective manifold (or orbifold) $\Gamma\backslash\Omega$ which contains all the topology and which is minimal (Lemma~\ref{lem:Ccore-min}); we shall call it the \emph{convex core} of $\Gamma\backslash\Omega$.
By analogy, we shall also call $\Ccore_\Omega(\Gamma)$ the \emph{convex core} of $\Omega$ for~$\Gamma$.

We shall see (Corollary~\ref{cor:ideal-bound-naive-cc}.\eqref{item:ideal-bound-cc}) that if $\Gamma$ acts cocompactly on some closed convex subset $\C$ of~$\Omega$ \emph{containing} $\Ccore_{\Omega}(\Gamma)$, then the action of $\Gamma$ on~$\Omega$ is convex cocompact in the sense of Definition~\ref{def:cc-general}.

\begin{example} \label{ex:strong-cc-implies-cc}
If $\Gamma$ is strongly convex cocompact in $\PP(V)$, then it is convex cocompact in $\PP(V)$.
This is immediate from the definitions since when $\Omega$ is strictly convex, the full orbital limit set $\Lambdao_\Omega(\Gamma)$ coincides with the accumulation set of a single $\Gamma$-orbit of~$\Omega$.
See Theorem~\ref{thm:main-noPETs} below for a refinement.
\end{example}

\begin{example} \label{ex:div-implies-cc}
If $\Gamma$ divides (\ie acts cocompactly on) a nonempty properly convex open subset $\Omega$ of $\PP(V)$, then $\Ccore_\Omega(\Gamma)=\Omega$ and $\Gamma$ is convex cocompact in $\PP(V)$ (see Corollary~\ref{cor:ideal-bound-naive-cc}.\eqref{item:ideal-bound-cc}).
There exist divisible convex sets which are not strictly convex, yielding convex cocompact groups that are not strongly convex cocompact in $\PP(V)$: see Example~\ref{ex:Zn} for a basic case.
As discussed in Section~\ref{subsubsec:ex-div}, an important class of examples is the so-called symmetric ones where $\Omega$ is a higher-rank symmetric space.
The first other indecomposable examples were constructed by Benoist \cite{ben06} for $4\leq\dim(V)\leq 7$, and further examples were recently constructed for $\dim(V)=4$ in \cite{bdl18} and for $5\leq\dim(V)\leq 7$~in~\cite{clm20}.
\end{example}

In Theorem~\ref{thm:main-general} we shall give several characterizations of convex cocompactness in $\PP(V)$.
In particular, we shall prove that the class of subgroups of $\PGL(V)$ that are convex cocompact in $\PP(V)$ is precisely the class of holonomy groups of compact properly convex projective orbifolds with strictly convex boundary; this ensures that this class of subgroups is stable under small deformations, using \cite{clt15}.
Before stating this and other results, let us make the connection with the context of Theorem~\ref{thm:Ano-PGL} (strong convex cocompactness).

\subsection{Word hyperbolic convex cocompact groups in $\PP(V)$} \label{subsec:intro-hyp-cc}

Let us recall the following terminology of Benoist \cite{ben06}.

\begin{definition} \label{def:PET}
Let $\C$ be a properly convex subset of $\PP(V)$.
A \emph{properly embedded triangle} (or \emph{PET} for short) in~$\C$ is a nondegenerate planar triangle whose interior is contained in~$\C$, but whose edges and vertices are contained in $\partiali\C := \overline{\C}\smallsetminus\C$.
\end{definition}

Let $\Omega$ be a properly convex open subset of $\PP(V)$ and $\Gamma$ a discrete subgroup of $\PGL(V)$ preserving $\Omega$.
A PET in $\Omega$ plays an analogous role to a flat in a Riemannian manifold.
In particular, the presence of a PET in $\Omega$ obstructs $\delta$-hyperbolicity of the Hilbert metric on $\Omega$ (see Section~\ref{subsec:veryshort}).
Hence by the \v{S}varc--Milnor lemma, a PET in $\Omega$ obstructs word hyperbolicity of $\Gamma$ in the case that $\Gamma$ divides $\Omega$.
In fact, even the presence of a segment in $\partial\Omega$ obstructs word hyperbolicity: Benoist \cite[Th.\,1.1]{ben04} proved that if $\Gamma$ divides~$\Omega$, then $\Gamma$ is word hyperbolic if and only if $\Omega$ is strictly convex.

The direct analogue of Benoist's theorem is not true for convex cocompact actions: for instance, Schottky subgroups of $\PO(2,1)$ act convex cocompactly on properly convex domains $\Omega \supset \HH^2$ that are not strictly convex (see Figure~\ref{fig:Omega-min-max}).
However, we show that for $\Gamma$ acting convex cocompactly on $\Omega$, the hyperbolicity of~$\Gamma$ is determined by the convexity behavior of $\partial \Omega$ \emph{at the full orbital limit set} or, relatedly, by the presence of PETs in the \emph{convex core}.
Here is an expanded version of Theorem~\ref{thm:Ano-PGL}.

\begin{theorem} \label{thm:main-noPETs}
Let $\Gamma$ be an infinite discrete subgroup of $\PGL(V)$.
Then the following are equivalent:
\begin{enumerate}[label=(\roman*)]
  \item\label{item:ccc-CM} $\Gamma$ is strongly convex cocompact in $\PP(V)$ (Definition~\ref{def:strong-cc});
  \item \label{item:ccc-hyp} $\Gamma$ is convex cocompact in $\PP(V)$ (Definition~\ref{def:cc-general}) and word hyperbolic;
  \item \label{item:ccc-noseg-any} $\Gamma$ is convex cocompact in $\PP(V)$ and for \emph{any} properly convex open set $\Omega \subset \PP(V)$ on which $\Gamma$ acts convex cocompactly, the full orbital limit set $\Lambdao_\Omega(\Gamma)$ does not contain any nontrivial projective line segment;
\item \label{item:ccc-noPETs-some} $\Gamma$ is convex cocompact in $\PP(V)$ and for \emph{some} nonempty properly convex open set $\Omega \subset \PP(V)$ on which $\Gamma$ acts convex cocompactly, the convex core $\Ccore_\Omega(\Gamma)$ does not contain a PET;
  \item\label{item:ccc-ugly} $\Gamma$ preserves a properly convex open subset $\Omega$ of $\PP(V)$ and acts cocompactly on some closed convex subset $\C$ of~$\Omega$ with nonempty interior such that $\partiali\C:=\overline{\C}\smallsetminus\C=\overline{\C}\cap\partial\Omega$ does not contain any nontrivial projective line segment;
  \item\label{item:P1Anosov} $\Gamma$ is word hyperbolic, the inclusion $\Gamma\hookrightarrow\PGL(V)$ is $P_1$-Anosov, and $\Gamma$ preserves some nonempty properly convex open subset of $\PP(V)$;
  \item\label{item:P1Anosov-bis} $\Gamma$ is word hyperbolic, the inclusion $\Gamma\hookrightarrow\PGL(V)$ is $P_1$-Anosov, and the set\linebreak $\PP(V) \smallsetminus \bigcup_{H\in\Lambda_{\Gamma}^*} H$ admits a $\Gamma$-invariant connected component.
\end{enumerate}
When these conditions hold, there is equality between the following four sets:
\begin{itemize}
  \item the orbital limit set $\Lambdao_{\Omega}(\Gamma)$ of $\Gamma$ in any $\Gamma$-invariant properly convex open subset $\Omega$ of $\PP(V)$ which is strictly convex with boundary of class $C^1$ as in Definition~\ref{def:strong-cc} (condition~\ref{item:ccc-CM});
  \item the full orbital limit set $\Lambdao_\Omega(\Gamma)$ for any properly convex open set $\Omega$ on which $\Gamma$ acts convex cocompactly as in Definition~\ref{def:cc-general} (conditions \ref{item:ccc-hyp}, \ref{item:ccc-noseg-any}, \ref{item:ccc-noPETs-some});
  \item the segment-free set $\partiali\C$ for any convex subset $\C$ on which $\Gamma$ acts cocompactly in a $\Gamma$-invariant properly convex open set~$\Omega$ (condition~\ref{item:ccc-ugly});
  \item the image of the boundary map $\xi : \partial_{\infty}\Gamma\to\PP(V)$ of the Anosov representation $\Gamma\hookrightarrow\PGL(V)$ (conditions \ref{item:P1Anosov} and~\ref{item:P1Anosov-bis}), which is also the proximal limit set $\Lambda_{\Gamma}$ of $\Gamma$ in $\PP(V)$ (Definition~\ref{def:prox-lim-set}).
\end{itemize}
\end{theorem}

\subsection{Properties of convex cocompact groups in $\PP(V)$}

We show that, even in the case of nonhyperbolic discrete groups, the notion of convex cocompactness in $\PP(V)$ (Definition~\ref{def:cc-general}) still has some of the nice properties enjoyed by Anosov representations and convex cocompact subgroups of rank-one Lie groups.
In particular, we prove the following.

\begin{theorem} \label{thm:properties}
Let $\Gamma$ be an infinite discrete subgroup of $G=\PGL(V)$.
\begin{enumerate}[label=(\Alph*)]
  \item\label{item:dual} The group $\Gamma$ is convex cocompact in $\PP(V)$ if and only if it is convex cocompact in $\PP(V^*)$ (for the dual action).
  \item\label{item:cc-QI} If $\Gamma$ is convex cocompact in $\PP(V)$, then it is finitely generated and quasi-isometrically embedded in~$G$.
  \item\label{item:cc-no-unipotent} If $\Gamma$ is convex cocompact in $\PP(V)$, then $\Gamma$ does not contain any unipotent element.
    \item\label{item:stable} If $\Gamma$ is convex cocompact in $\PP(V)$, then there is a neighborhood $\mathcal{U}\subset\Hom(\Gamma,G)$ of the natural inclusion such that any $\rho \in \mathcal{U}$ is injec\-tive and discrete with image $\rho(\Gamma)$ convex cocompact in~$\PP(V)$.
  \item\label{item:include} Let $V'=\RR^{n'}$ and let $i :\nolinebreak \SL^{\pm}(V) \hookrightarrow \SL^{\pm}(V \oplus V')$ be the natural inclusion acting trivially on the second factor.
  If $\Gamma$ is convex cocompact in $\PP(V)$, then $i(\hat \Gamma)$ is convex cocompact in $\PP(V \oplus V')$, where $\hat \Gamma$ is the lift of $\Gamma$ to $\SL^{\pm}(V)$ that preserves a properly convex cone of $V$ lifting~$\Omega$ (see Remark~\ref{rem:lift-Gamma}).
  \item\label{item:cc-ss} If the semisimplification (Definition~\ref{def:reduc-part}) of the natural inclusion $\Gamma\hookrightarrow\PGL(V)$ is injective and its image is convex cocompact in $\PP(V)$, then $\Gamma$ is convex cocompact in $\PP(V)$.
\end{enumerate}
\end{theorem}

\begin{remark} \label{rem:properties-strong-cc}
The equivalence \ref{item:ccc-CM} $\Leftrightarrow$ \ref{item:ccc-hyp} of Theorem~\ref{thm:main-noPETs} shows that Theorem~\ref{thm:properties} still holds if all the occurrences of ``convex cocompact'' are replaced by ``strongly convex cocompact''.
\end{remark}

While some of the properties of Theorem~\ref{thm:properties} are proved directly from Definition~\ref{def:cc-general}, others will be most naturally established using alternative characterizations of convex cocompactness in $\PP(V)$ (Theorem~\ref{thm:main-general}).
We refer to Propositions~\ref{prop:cc-block} and~\ref{prop:cc-quotient} for further operations that preserve convex cocompactness, strengthening properties~\ref{item:include} and~\ref{item:cc-ss}. 

Properties \ref{item:include} and~\ref{item:stable} give a source for many new examples of convex cocompact groups by starting with known examples in $\PP(V)$, including them into $\PP(V \oplus V')$, and then deforming.
This generalizes the picture of Fuchsian groups in $\PO(2,1)$ being deformed into quasi-Fuchsian groups in $\PO(3,1)$.
For instance, we can use Example~\ref{ex:div-implies-cc} to obtain examples of nonhyperbolic discrete subgroups of $\PGL(\RR^{m})$ (where $m=\mathrm{dim}(V\oplus V')$) which are convex cocompact in $\PP(\RR^{m})$ and act irreducibly on $\PP(\RR^{m})$, but do not divide any properly convex open subset of~$\PP(\RR^{m})$.
These groups $\Gamma$ are quasi-isometrically embedded in $\PGL(\RR^{m})$ and structurally stable (\ie there is a neighborhood of the natural inclusion in $\Hom(\Gamma,\PGL(\RR^{m}))$ which consists entirely of injective representations) without being word hyperbolic; compare with Sullivan \cite[Th.\,A]{sul85}.
We refer to Section~\ref{subsec:QF-div} for more details.

\subsection{Holonomy groups of convex projective orbifolds}

We now give alternative characterizations of convex cocompact subgroups in $\PP(V)$.
These characterizations are motivated by a familiar picture in rank one: namely, if $\Omega = \HH^p$ is the $p$-dimensional real hyperbolic space and $\Gamma$ is a convex cocompact torsion-free subgroup of $\PO(p,1) = \mathrm{Isom}(\HH^p)$, then any closed uniform neighborhood $\C_{\mathsf{unif}}$ of the convex core $\Ccore_\Omega(\Gamma)$ has strictly convex boundary and the quotient $\Gamma \backslash \C_{\mathsf{unif}}$ is a compact hyperbolic manifold with strictly convex boundary.
Let us fix some terminology and notation in order to discuss the appropriate generalization of this picture to real projective geometry.

\begin{definition} \label{def:boundaries}
Let $\C$ be a nonempty convex subset of $\PP(V)$ (not necessarily open nor closed, possibly with empty interior).
\begin{itemize}
  \item The \emph{frontier} of~$\C$ is $\Fr(\C):=\overline{\C}\smallsetminus\Int(\C)$.
  \item A \emph{supporting hyperplane} of~$\C$ at a point $x\in\Fr(\C)$ is a projective hyperplane $H$ such that 
  $x \in H \cap \overline \C$ and $\overline \C \smallsetminus H$ is connected (possibly empty).
  \item The \emph{ideal boundary} of~$\C$ is $\partiali\C:=\overline{\C}\smallsetminus\C$.
  \item The \emph{nonideal boundary} of~$\C$ is $\partialn\C:= \C \smallsetminus \Int(\C) = \Fr(\C)\smallsetminus\partiali\C$.
  Note that if $\C$ is open, then $\partiali \C = \Fr(\C)$ and $\partialn \C = \emptyset$; in this case, it is common in the literature to denote $\partiali \C$ simply by $\partial \C$.
  \item The convex set~$\C$ has \emph{strictly convex nonideal boundary} if every point $x \in \partialn \C$ is an extreme point of $\overline \C$.
   \item The convex set~$\C$ has \emph{$C^1$ nonideal boundary} if it has a unique supporting hyperplane at each point $x \in \partialn \C$.
  \item \label{item:bisatdef} The convex set~$\C$ has \emph{bisaturated boundary} if for any supporting hyperplane $H$ of~$\C$, the set $H \cap \overline{\C} \subset \Fr(\C)$ is either fully contained in $\partiali \C$ or fully contained in $\partialn \C$.
  \end{itemize}
\end{definition}

Let $\C$ be a properly convex subset of $\PP(V)$ on which the discrete subgroup $\Gamma$ acts properly discontinuously and cocompactly.
To simplify this intuitive discussion, assume $\Gamma$ torsion-free, so that the quotient $M = \Gamma \backslash \C$ is a compact properly convex projective manifold, possibly with boundary.
The group $\Gamma$ is called the \emph{holonomy group} of $M$.
We show that if the boundary $\partial M = \Gamma \backslash \partialn \C$ of $M$ is assumed to have some regularity, then $\Gamma$ is convex cocompact in $\PP(V)$ and, conversely, convex cocompact subgroups in $\PP(V)$ are holonomy groups of properly convex projective manifolds (or more generally orbifolds, if $\Gamma$ has torsion)  whose boundaries are well-behaved.

\begin{theorem}\label{thm:main-general}
Let $\Gamma$ be an infinite discrete subgroup of $\PGL(V)$.
Then the following are equivalent:
\begin{enumerate}
  \item \label{item:ccc-limit-set} $\Gamma$ is convex cocompact in $\PP(V)$ (Definition~\ref{def:cc-general}): it acts convex cocompactly on a nonempty properly convex open subset $\Omega$ of $\PP(V)$;
  \item \label{item:ccc-bisat} $\Gamma$ acts properly discontinuously and cocompactly on a nonempty properly convex set $\C_{\mathsf{bisat}} \subset \PP(V)$ with bisaturated boundary;
  \item \label{item:ccc-strict} $\Gamma$ acts properly discontinuously and cocompactly on a nonempty properly convex set $\C_{\mathsf{strict}} \subset \PP(V)$ with strictly convex nonideal boundary;
  \item \label{item:ccc-strict-C1} $\Gamma$ acts properly discontinuously and cocompactly on a nonempty properly convex set $\C_{\mathsf{smooth}} \subset \PP(V)$ with strictly convex $C^1$ nonideal boundary.
\end{enumerate}
When these conditions hold, $\C_{\mathsf{bisat}}$, $\C_{\mathsf{strict}}$, $\C_{\mathsf{smooth}}$ can be chosen equal, with $\Omega$ satisfying $\Lambdao_\Omega(\Gamma) = \partiali \C_{\mathsf{smooth}}$.
\end{theorem}

\begin{remark}
Given a properly discontinuous and cocompact action of a group~$\Gamma$ on a properly convex set~$\C$, Theorem~\ref{thm:main-general} interprets the convex cocompactness of~$\Gamma$ in terms of the regularity of~$\C$ at the nonideal boundary $\partialn\C$, whereas Theorem~\ref{thm:main-noPETs} (equivalence \ref{item:ccc-CM}~$\Leftrightarrow$~\ref{item:ccc-ugly}) and Theorem~\ref{thm:main-general} together interpret the \emph{strong} convex cocompactness of~$\Gamma$ in terms of the regularity of~$\C$ at both $\partiali\C$ and~$\partialn\C$.
\end{remark}

The equivalence \eqref{item:ccc-limit-set} $\Leftrightarrow$ \eqref{item:ccc-strict} of Theorem~\ref{thm:main-general} states that convex cocompact torsion-free subgroups in $\PP(V)$ are precisely the holonomy groups of compact properly convex projective manifolds with strictly convex boundary. Cooper--Long--Tillmann~\cite{clt18} studied the deformation theory of such manifolds (allowing in addition certain types of cusps).
They established a holonomy principle, which reduces to the following statement in the absence of cusps: the holonomy groups of compact properly convex projective manifolds with strictly convex boundary form an open subset of the representation space of~$\Gamma$.
This result, together with Theorem~\ref{thm:main-general}, gives the stability property~\ref{item:stable} of Theorem~\ref{thm:properties}.
For the case of divisible convex sets, see \cite{kos68}.

The requirement that $\C_{\mathsf{bisat}}$ have bisaturated boundary in condition~\eqref{item:ccc-bisat} can be seen as a ``coarse'' version of the strict convexity of the nonideal boundary in~\eqref{item:ccc-strict-C1}: the prototype situation is that of a convex cocompact real hyperbolic manifold, with $\C_{\mathsf{bisat}}$ a closed polyhedral neighborhood of the convex core, see \eg the top right panel of Figure~\ref{fig:Coxeter}.
Both of these conditions behave well under a natural duality operation generalizing that of open properly convex sets (see Section~\ref{sec:bisat-dual}).
The equivalence \eqref{item:ccc-limit-set}~$\Leftrightarrow$~\eqref{item:ccc-bisat} will be used to prove the duality property~\ref{item:dual} of Theorem~\ref{thm:properties}.

We note that without some form of strengthened convexity requirement on the boundary, properly convex projective manifolds have a poorly-behaved deformation theory: see Remark~\ref{rem:stable-precise}.\ref{item:weak-cond-not-open}.

\begin{figure}
\includegraphics[width = 4.1cm]{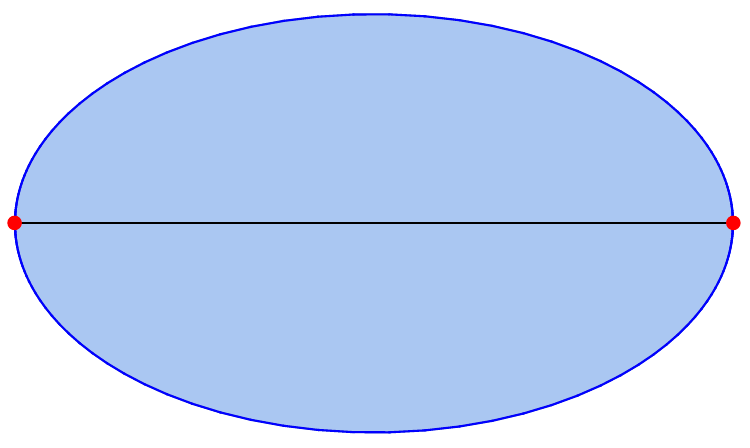}
\hspace{1.5cm}
\includegraphics[width = 4.0cm]{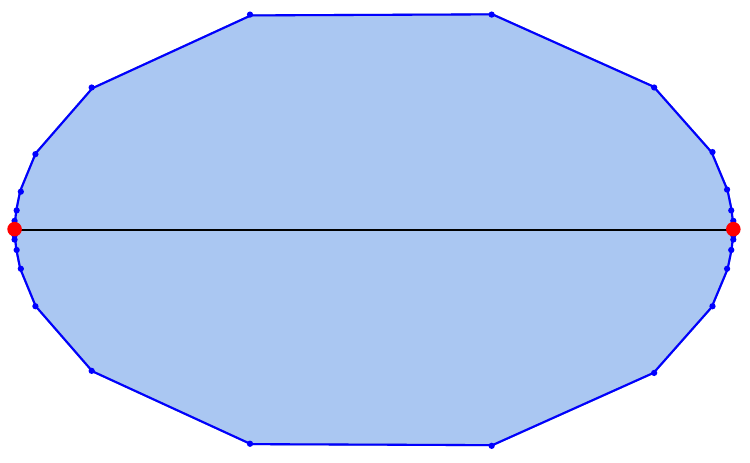} 
\\
\includegraphics[width = 4.0cm]{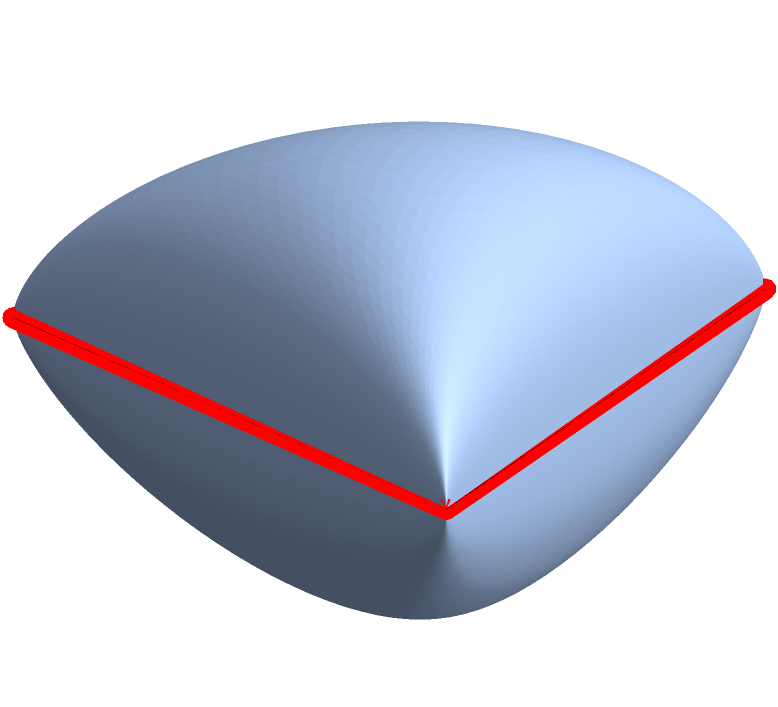}
\hspace{1.0cm}
\includegraphics[width = 4.0cm]{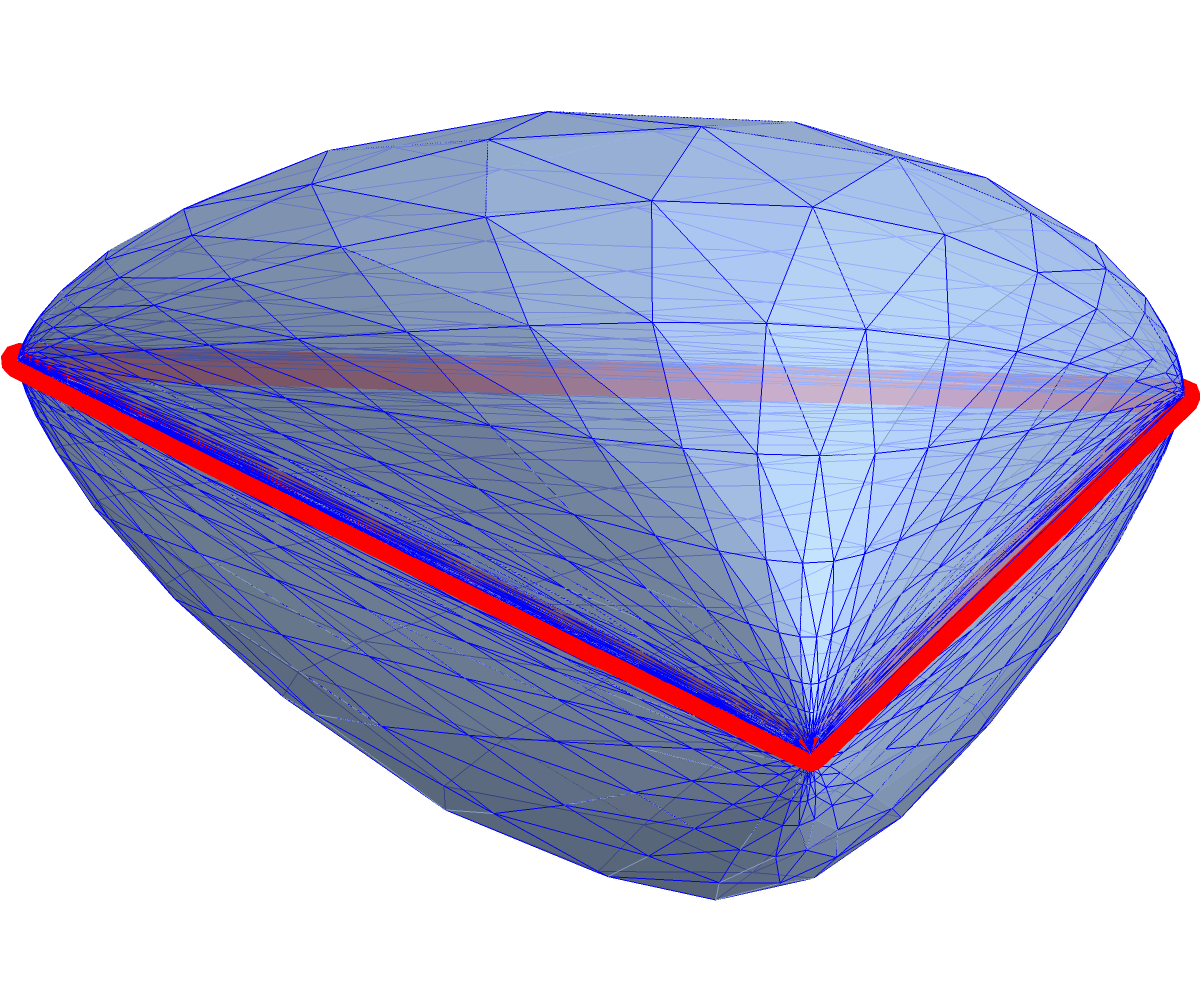}
\caption{The group $\ZZ^{n-1}$ acts diagonally on $\PP(\RR^n)\subset \PP(\RR^{n+1})$, where $n=2$ (top row) or $3$ (bottom row). These actions are convex cocompact in $\PP(\RR^{n+1})$: each row shows two properly convex invariant subsets $\C \subset \PP(\RR^{n+1})$ with bisaturated boundary, the one on the left having also strictly convex $C^1$ nonideal boundary $\partiali\C = \overline{\C} \smallsetminus \C$. We represent $\partiali\C$ in red.}
\label{fig:Coxeter}
\end{figure}

\begin{remark} \label{rem:equality-sets}
If $\Gamma$ acts strongly irreducibly on $\PP(V)$ (\ie all finite-index subgroups of~$\Gamma$ act irreducibly) and if the equivalent conditions of Theorem~\ref{thm:main-general} hold, then $\Lambdao_\Omega(\Gamma) = \partiali \C_{\mathsf{bisat}} = \partiali \C_{\mathsf{strict}} = \partiali \C_{\mathsf{smooth}}$ for \emph{any} $\Omega, \C_{\mathsf{bisat}}, \C_{\mathsf{strict}}$, $\C_{\mathsf{smooth}}$ as in conditions~\eqref{item:ccc-limit-set}, \eqref{item:ccc-bisat}, \eqref{item:ccc-strict}, and~\eqref{item:ccc-strict-C1} respectively: see Section~\ref{sec:unique-limit-set}.
In particular, this set then depends only on~$\Gamma$.
This is not necessarily the case if the action of $\Gamma$ is not strongly irreducible: see Examples \ref{ex:Zn} and~\ref{ex:An}.
\end{remark}

\subsection{Convex cocompactness for subgroups of $\PO(p,q)$} \label{subsec:intro-cc-POpq}

As mentioned in Remark \ref{rem:Ano-PGL}.\ref{item:POpq-irred}, when $\Gamma\subset\PGL(V)$ acts irreducibly on $\PP(V)$ and is contained in the subgroup $\PO(p,q) \subset \PGL(V)$ of projective linear transformations that preserve a nondegenerate symmetric bilinear form $\langle \cdot, \cdot \rangle_{p,q}$ of some signature $(p,q)$ on $V=\RR^n$, a more precise version of Theorem~\ref{thm:Ano-PGL} was established in \cite[Th.\,1.11 \& Prop.\,1.17]{dgk-ccHpq}.
Here we remove the irreducibility assumption on~$\Gamma$.

For $p,q\in\NN^*$, let $\RR^{p,q}$ be $\RR^{p+q}$ endowed with a symmetric bilinear form $\langle \cdot, \cdot \rangle_{p,q}$ of signature $(p,q)$.
The spaces
\begin{align*}
\HH^{p,q-1} &= \left\{ [v] \in \PP(\RR^{p,q}) ~|~ \langle v,v\rangle_{p,q} < 0 \right\}\\
\mathrm{and}\quad \SS^{p-1,q} &= \left\{ [v] \in \PP(\RR^{p,q}) ~|~ \langle v,v\rangle_{p,q} >0 \right\}
\end{align*}
are the projective models for pseudo-Riemannian hyperbolic space of signature $(p,q-1)$, and pseudo-Riemannian spherical space of signature $(p-1,q)$.
The geodesics of these two spaces are the (nonempty) intersections of the spaces with the projective lines of $\PP(\RR^{p+q})$.
The group $\PO(p,q)$ acts transitively on $\HH^{p,q-1}$ and on $\SS^{p-1,q}$.
Multiplying the form $\langle \cdot, \cdot \rangle_{p,q}$ by $-1$ produces a form of signature $(q,p)$ and turns~$\PO(p,q)$ into $\PO(q,p)$ and $\SS^{p-1,q}$ into a copy of $\HH^{q,p-1}$, so we consider only the pseudo-Riemannian hyperbolic spaces $\HH^{p,q-1}$.
We denote by $\partial\HH^{p,q-1}$ the boundary of $\HH^{p,q-1}$, namely
$$\partial\HH^{p,q-1} = \left\{ [v] \in \PP(\RR^{p,q}) ~|~ \langle v,v \rangle_{p,q} = 0 \right\}.$$
We call a subset of $\HH^{p,q-1}$ \emph{convex} (\resp \emph{properly convex}) if it is convex (\resp properly convex) as a subset of $\PP(\RR^{p+q})$. 
Since the straight lines of $\PP(\RR^{p+q})$ are the geodesics of the pseudo-Riemannian metric on $\HH^{p,q-1}$, convexity in $\HH^{p,q-1}$ is an intrinsic notion.

\begin{remarks} \label{rem:convex-in-Hpq}
Let $\C$ be a closed convex subset of $\HH^{p,q-1}$.
\begin{enumerate}[label=(\alph*)]
  \item If $\C$ has nonempty interior, then $\C$ is properly convex.
  Indeed, if $\C$ is not properly convex, then it contains a line $\ell$ of an affine chart $\RR^{p+q-1}\supset\nolinebreak\C$, and $\C$ is a union of lines parallel to $\ell$ in that chart.
  All these lines must be tangent to $\partial \HH^{p,q-1}$ at their common endpoint $z \in \partial \HH^{p,q-1}$, which implies that $\C$ is contained in the hyperplane $z^\perp$ and has empty interior.
  \item The ideal boundary $\partiali\C$ is the set of accumulation points of $\C$ in $\partial\HH^{p,q-1}$.
  This set contains no projective line segment if and only if it is \emph{transverse}, \ie $y\notin z^{\perp}$ for all $y\neq z$ in $\partiali \C$.
  \item The nonideal boundary $\partialn\C = \Fr(\C) \cap \HH^{p,q-1}$ is the boundary of $\C$ in~$\HH^{p,q-1}$ in the usual sense.
 \item\label{item:bisat-Hpq-intro} It is easy to see that $\C$ has bisaturated boundary if and only if a segment contained in $\partialn\C\subset\HH^{p,q-1}$ never extends to a full geodesic ray of~$\HH^{p,q-1}$ --- a form of coarse strict convexity for $\partialn\C$.
\end{enumerate}
\end{remarks}

The following definition was introduced in \cite{dgk-ccHpq}, where it was studied for discrete subgroups~$\Gamma$ acting irreducibly on $\PP(\RR^{p+q})$. 

\begin{definition}[{\cite[Def.\,1.2]{dgk-ccHpq}}] \label{def:Hpq-cc}
A discrete subgroup $\Gamma$ of $\PO(p,q)$ is \emph{$\HH^{p,q-1}$-convex cocompact} if it acts properly discontinuously with compact quotient on some closed properly convex subset $\C$ of $\HH^{p,q-1}$ with nonempty interior whose ideal boundary $\partiali \C \subset \partial \HH^{p,q-1}$ does not contain any nontrivial projective line segment.
\end{definition}

If $\Gamma$ acts irreducibly on $\PP(\RR^{p+q})$, then any nonempty $\Gamma$-invariant properly convex subset of $\HH^{p,q-1}$ has nonempty interior, and so the nonempty interior requirement in Definition~\ref{def:Hpq-cc} may simply be replaced by the requirement that $\C$ be nonempty.
See Section~\ref{subsec:ex-Hpq-cc} for examples.

For a word hyperbolic group~$\Gamma$, we shall say that a representation $\rho : \Gamma\to\PO(p,q)$ is \emph{$P_1^{p,q}$-Anosov} if it is $P_1$-Anosov as a representation into $\PGL(\RR^{p+q})$.
We refer to \cite{dgk-ccHpq} for further discussion of this notion.

If the natural inclusion $\Gamma\hookrightarrow\PO(p,q)$ is $P_1^{p+q}$-Anosov, then the boundary map $\xi$ takes values in $\partial \HH^{p,q-1}$ and the hyperplane-valued boundary map $\xi^*$ satisfies $\xi^*(\cdot) = \xi(\cdot)^\perp$ where $z^\perp$ denotes the orthogonal of~$z$ with respect to $\langle \cdot, \cdot \rangle_{p,q}$; the image of~$\xi$ is the proximal limit set $\Lambda_{\Gamma}$ of $\Gamma$ in $\partial\HH^{p,q-1}$ (Definition~\ref{def:prox-lim-set} and Remark~\ref{rem:prox-lim-set-POpq}).
Following \cite[Def.\,1.9]{dgk-ccHpq}, we shall say that $\Lambda_{\Gamma}$ is \emph{negative} (\resp \emph{positive}) if it lifts to a cone of $\RR^{p,q}\smallsetminus\{0\}$ on which all inner products $\langle\cdot,\cdot\rangle_{p,q}$ of noncollinear points are negative (\resp positive); equivalently (see \cite[Lem.\,3.2]{dgk-ccHpq}), any three distinct points of $\Lambda_{\Gamma}$ span a linear subspace of $\RR^{p,q}$ of signature $(2,1)$ (\resp $(1,2)$).
Note that $\Lambda_\Gamma$ can be both negative and positive only if $|\Lambda_\Gamma| = 2$, \ie $\Gamma$ is virtually~$\ZZ$.
With this notation, we prove the following, where $\Gamma$ is \emph{not} assumed to act irreducibly on $\PP(\RR^{p+q})$; for \eqref{item:ccc-CM-Hpq}~$\Leftrightarrow$~\eqref{item:ccc-Hpq-Hqp} and ($2^{\pm}$)~$\Leftrightarrow$~($3^{\pm}$)~$\Leftrightarrow$~($5^{\pm}$)~$\Rightarrow$~($4^{\pm}$) in the irreducible case, see \cite[Th.\,1.11 \& Prop.\,1.17]{dgk-ccHpq}.

\begin{theorem} \label{thm:main-POpq-reducible}
Let $p,q\in\NN^*$ and let $\Gamma$ be an infinite discrete subgroup of $\PO(p,q)$.
Then the following are equivalent:
\begin{enumerate}
  \item\label{item:ccc-POpq} $\Gamma$ is convex cocompact in $\PP(\RR^{p+q})$ (Definition~\ref{def:strong-cc});
  \item\label{item:ccc-CM-Hpq} $\Gamma$ is strongly convex cocompact in $\PP(\RR^{p+q})$ (Definition~\ref{def:strong-cc});
  \item\label{item:ccc-Hpq-Hqp} $\Gamma$ is $\HH^{p,q-1}$-convex cocompact or $\HH^{q,p-1}$-convex cocompact (after identifying $\PO(p,q)$ with $\PO(q,p)$ as above).
\end{enumerate}
The following are also equivalent:
\begin{enumerate}[label=(\arabic*$^-$)]
  \item\label{item:ccc-POpq-neg} $\Gamma$ is convex cocompact in $\PP(\RR^{p+q})$ and $\Lambda_{\Gamma}\subset\partial\HH^{p,q-1}$ is negative;
  \item\label{item:ccc-CM-Hpq-neg} $\Gamma$ is strongly convex cocompact in $\PP(\RR^{p+q})$ and $\Lambda_{\Gamma}\subset\partial\HH^{p,q-1}$ is negative;
  \item\label{item:Hpq-ccc} $\Gamma$ is $\HH^{p,q-1}$-convex cocompact;
  \item\label{item:ccc-CM-Hpq-in-Hpq} $\Gamma$ acts convex cocompactly on some nonempty properly convex open subset $\Omega$ of $\HH^{p,q-1}$;
  \item\label{item:Anosov-neg} $\Gamma$ is word hyperbolic, the natural inclusion $\Gamma\hookrightarrow\PO(p,q)$ is $P_1^{p,q}$-Anosov, and $\Lambda_{\Gamma}\subset\partial\HH^{p,q-1}$ is negative.
\end{enumerate}
Similarly, the following are equivalent:
\begin{enumerate}[label=(\arabic*$^+$)]
  \item\label{item:ccc-POpq-pos} $\Gamma$ is convex cocompact in $\PP(\RR^{p+q})$ and $\Lambda_{\Gamma}\subset\partial\HH^{p,q-1}$ is positive;
  \item\label{item:ccc-CM-Hpq-pos} $\Gamma$ is strongly convex cocompact in $\PP(\RR^{p+q})$ and $\Lambda_{\Gamma}\subset\partial\HH^{p,q-1}$ is positive;
  \item\label{item:Hqp-ccc} $\Gamma$ is $\HH^{q,p-1}$-convex cocompact (after identifying $\PO(p,q)$ with $\PO(q,p)$);
  \item\label{item:ccc-CM-Hpq-in-Hqp} $\Gamma$ acts convex cocompactly on some nonempty properly convex open subset $\Omega$ of $\HH^{q,p-1}$ (after identifying $\PO(p,q)$ with $\PO(q,p)$);
  \item\label{item:Anosov-pos} $\Gamma$ is word hyperbolic, the natural inclusion $\Gamma\hookrightarrow\PO(p,q)$ is $P_1^{p,q}$-Anosov, and $\Lambda_{\Gamma}\subset\partial\HH^{p,q-1}$ is positive.
\end{enumerate}
\end{theorem}

It is not difficult to see that if $\mathcal{T}$ is an open subset of $\Hom(\Gamma,\PO(p,q))$ consisting entirely of $P_1^{p,q}$-Anosov representations, then the condition that $\Lambda_{\rho(\Gamma)}\subset\partial\HH^{p,q-1}$ be negative is both an open and closed condition for $\rho\in\mathcal{T}$ \cite[Prop.\,3.5]{dgk-ccHpq}.
Therefore Theorem~\ref{thm:main-POpq-reducible} and the openness of the set of $P_1^{p,q}$-Anosov representations~\cite{lab06, gw12}  implies the following.

\begin{corollary} \label{cor:Hpq-cc-for-all}
Let $p,q\in\NN^*$ and let $\Gamma$ be a discrete group.
\begin{enumerate}
\item The set of representations with finite kernel and $\HH^{p,q-1}$-convex cocompact image is open in $\Hom(\Gamma,\PO(p,q))$.
\item Let $\mathcal{T}$ be a connected open subset of $\Hom(\Gamma,\PO(p,q))$ consisting entirely of $P_1^{p,q}$-Anosov representations.
If $\rho(\Gamma)$ is $\HH^{p,q-1}$-convex cocompact for some $\rho\in\mathcal{T}$, then $\rho(\Gamma)$ is $\HH^{p,q-1}$-convex cocompact for all $\rho\in\mathcal{T}$.
\end{enumerate}
Both statements also hold if ``$\HH^{p,q-1}$-convex cocompact'' is replaced with ``$\HH^{q,p-1}$-convex cocompact.''
\end{corollary}

It is also not difficult to see that if a closed subset $\Lambda$ of $\partial\HH^{p,q-1}$ is transverse (\ie $z^\perp\cap \Lambda=\{z\}$ for all $z\in \Lambda$) and connected, then it is negative or positive \cite[Prop.\,1.10]{dgk-ccHpq}.
Therefore Theorem~\ref{thm:main-POpq-reducible} implies the following.

\begin{corollary} \label{cor:pqqp-intro}
Let $\Gamma$ be a word hyperbolic group with \emph{connected} boundary $\partial_{\infty}\Gamma$, and let $p,q\in\NN^*$.
For any $P_1^{p,q}$-Anosov representation $\rho : \Gamma\to\PO(p,q)$, the group $\rho(\Gamma)$ is $\HH^{p,q-1}$-convex cocompact or $\HH^{q,p-1}$-convex cocompact (after identifying $\PO(p,q)$ with $\PO(q,p)$).
 \end{corollary}

\begin{remark}
In the special case when $q=2$ (\ie $\HH^{p,q-1}=\mathrm{AdS}^{p+1}$ is the Lorentzian \emph{anti-de Sitter space}) and $\Gamma$ is the fundamental group of a closed hyperbolic $p$-manifold, the equivalence \ref{item:Hpq-ccc}~$\Leftrightarrow$~\ref{item:Anosov-neg} of Theorem~\ref{thm:main-POpq-reducible} follows from work of Mess \cite{mes90} for $p=2$ and Barbot--M\'erigot \cite{bm12} for $p\geq 3$.
\end{remark}

As a consequence of Theorem~\ref{thm:main-general}, we obtain characterizations of $\HH^{p,q-1}$-convex cocompactness where the assumption on the ideal boundary $\partiali \C \subset \partial\HH^{p,q-1}$ in Definition~\ref{def:Hpq-cc} is replaced by various regularity conditions on the nonideal boundary $\partialn \C \subset \HH^{p,q-1}$.

\begin{theorem}\label{thm:Hpq-general}
$p,q\in\NN^*$ and let $\Gamma$ be an infinite discrete subgroup of $\PO(p,q)$.
Then the following are equivalent:
\begin{enumerate}
  \item \label{item:Hpq-cc} $\Gamma$ is $\HH^{p,q-1}$-convex cocompact: it acts properly discontinuously and cocompactly on a closed convex subset $\C$ of $\HH^{p,q-1}$ with nonempty interior whose ideal boundary $\partiali \C$ does not contain any nontrivial projective line segment;
  \item \label{item:Hpq-bisat} $\Gamma$ acts properly discontinuously and cocompactly on a nonempty closed convex subset $\C_{\mathsf{bisat}}$ of $\HH^{p,q-1}$ whose boundary $\partialn \C_{\mathsf{bisat}}$ in $\HH^{p,q-1}$ does not contain any infinite geodesic line of $\HH^{p,q-1}$;
  \item \label{item:Hpq-strict} $\Gamma$ acts properly discontinuously and cocompactly on a nonempty closed convex subset $\C_{\mathsf{strict}}$ of $\HH^{p,q-1}$ whose boundary $\partialn \C_{\mathsf{strict}}$ in $\HH^{p,q-1}$ is strictly convex;
  \item \label{item:Hpq-strict-C1} $\Gamma$ acts properly discontinuously and cocompactly on a nonempty closed convex set $\C_{\mathsf{smooth}}$ of $\HH^{p,q-1}$ whose boundary $\partialn \C_{\mathsf{smooth}}$ in $\HH^{p,q-1}$ is strictly convex and of class~$C^1$.
\end{enumerate}
When these conditions hold, $\C$, $\C_{\mathsf{bisat}}$, $\C_{\mathsf{strict}}$, and $\C_{\mathsf{smooth}}$ may be taken equal.
\end{theorem}

\subsection{Organization of the paper}\label{subsec:organization}

Section~\ref{sec:reminders} contains reminders about properly convex domains in projective space, the Cartan decomposition, and Anosov representations.
In Section~\ref{sec:basic-examples} we establish some general facts about discrete group actions on convex subsets of $\PP(V)$.
Section~\ref{sec:naive-cc} contains some basic examples and develops the basic theory of naively convex cocompact and convex cocompact subgroups of $\PGL(V)$, which will be used throughout the paper.

Sections \ref{sec:bisat-dual} to~\ref{sec:smooth} are devoted to the proofs of the main Theorems \ref{thm:main-noPETs} and~\ref{thm:main-general}, which contain Theorem~\ref{thm:Ano-PGL}.
More precisely, in Section~\ref{sec:bisat-dual} we prove the equivalence \eqref{item:ccc-limit-set}~$\Leftrightarrow$~\eqref{item:ccc-bisat} of Theorem~\ref{thm:main-general}, study a notion of duality, and establish property~\ref{item:dual} of Theorem~\ref{thm:properties}.
In Section~\ref{sec:regularity-limit-set} we prove the equivalences \ref{item:ccc-hyp}~$\Leftrightarrow$~\ref{item:ccc-noseg-any}~$\Leftrightarrow$~\ref{item:ccc-noPETs-some}~$\Leftrightarrow$~\ref{item:ccc-ugly} of Theorem~\ref{thm:main-noPETs} by studying segments in the limit set. 
The connection with the Anosov property is made in Sections~\ref{sec:ccc-implies-Anosov} (\ref{item:ccc-hyp}~$\Rightarrow$~\ref{item:P1Anosov}) and~\ref{sec:Anosov-implies-ccc-noPETs} (\ref{item:P1Anosov}~$\Leftrightarrow$~\ref{item:P1Anosov-bis}~$\Rightarrow$~\ref{item:ccc-hyp}).  
In Section~\ref{sec:smooth} we give a smoothing construction to address the remaining implications of Theorems~\ref{thm:main-noPETs} and~\ref{thm:main-general}.

In Section~\ref{sec:other-properties} we establish properties \ref{item:cc-QI}--\ref{item:cc-ss} of Theorem~\ref{thm:properties}.
In Section~\ref{sec:Hpq-cc} we prove Theorems~\ref{thm:main-POpq-reducible} and~\ref{thm:Hpq-general} on $\HH^{p,q-1}$-convex cocompactness.
Section~\ref{sec:examples} is devoted to more sophisticated examples, including Proposition~\ref{prop:Hitchin}.
Finally, in Appendix~\ref{app:open-q} we collect a few open questions on convex cocompact groups~in~$\PP(V)$; in Appendix~\ref{app:hilbert} we give a sharp statement about Hausdorff limits of balls in Hilbert geometry.

\subsection*{Acknowledgements}

We are grateful to Yves Benoist, Sam Ballas, Gye-Seon Lee, Arielle Leitner, and Ludovic Marquis for helpful and inspiring discussions, and to Thierry Barbot for his encouragement.
We thank Daryl Cooper for a useful reference.
The first-named author thanks Steve Kerckhoff for discussions related to Example~\ref{exa:aff-hp} from several years ago, which proved valuable for understanding the general theory.
We warmly thank Pierre-Louis Blayac for useful questions and comments during the revisions of this paper, which contributed in particular to clarifying the discussion on conical limit points.
We are grateful to the referee for many valuable comments and suggestions which helped improve the paper.

\section{Reminders} \label{sec:reminders}

\subsection{Properly convex domains in projective space} \label{subsec:prop-conv-proj}

Recall that a subset of  projective space is called \emph{convex} if it is contained in and convex in some affine chart, and \emph{properly convex} if its closure is convex.

Let $\Omega$ be a properly convex open subset of $\PP(V)$, with boundary $\partial\Omega$.
Recall the \emph{Hilbert metric} $d_{\Omega}$ on~$\Omega$:
\begin{equation} \label{eqn:d-Omega}
d_{\Omega}(x,y) := \frac{1}{2} \log \, \cro{a}{x}{y}{b}
\end{equation}
for all distinct $x,y\in\Omega$, where $\cro{\,\underline{~~}}{\underline{~~}}{\underline{~~}}{\underline{~~}\,}$ is the cross-ratio on $\PP^1(\RR)$, normalized so that $\cro{0}{1}{t}{\infty}=t$, and where $a,b$ are the intersection points of $\partial\Omega$ with the projective line through $x$ and~$y$, with $a,x,y,b$ in this order.
The metric space $(\Omega,d_{\Omega})$ is proper (closed balls are compact) and complete, and the group
$$\mathrm{Aut}(\Omega) := \{g\in\PGL(V) ~|~ g\cdot\Omega=\Omega\}$$
acts on~$\Omega$ by isometries for~$d_{\Omega}$.
As a consequence, any discrete subgroup of $\mathrm{Aut}(\Omega)$ acts properly discontinuously on~$\Omega$.

\begin{remark} \label{rem:Hilb-metric-include}
It follows from the definition that if $\Omega_1 \subset \Omega_2$ are nonempty properly convex open subsets of $\PP(V)$, then the corresponding Hilbert metrics satisfy $d_{\Omega_1}(x,y) \geq d_{\Omega_2}(x,y)$ for all $x,y\in\Omega_1$.
\end{remark}

Let $V^*$ be the dual vector space of~$V$.
By definition, the \emph{dual convex set} of~$\Omega$ is
\begin{equation} \label{eqn:dual-open-convex}
\Omega^* := \PP\big(\big\{ \varphi\in V^* ~|~ \varphi(v)<0\quad \forall v\in\overline{\widetilde{\Omega}}\big\}\big),
\end{equation}
where $\overline{\widetilde{\Omega}}$ is the closure in $V\smallsetminus \{0\}$ of an open convex cone of $V$ lifting~$\Omega$.
The set $\Omega^*$ is a properly convex open subset of $\PP(V^*)$, which is preserved by the dual action of $\mathrm{Aut}(\Omega)$ on $\PP(V^*)$.

Straight lines (contained in projective lines) are always geodesics for the Hilbert metric~$d_{\Omega}$.
When $\Omega$ is not strictly convex, there may be other geodesics as well.
However, a biinfinite geodesic of $(\Omega,d_{\Omega})$ always has well-defined, distinct endpoints in $\partial \Omega$, see \cite[Lem.\,2.6]{dgk-ccHpq}.

\subsection{Cartan decomposition} \label{subsec:Cartan}

The group $\hat{G}=\GL(V)$ admits the \emph{Cartan decomposition} $\hat{G}=\hat{K}\exp(\hat{\mathfrak{a}}^+)\hat{K}$ where $\hat{K}=\OO(n)$ and
$$\hat{\mathfrak{a}}^+ := \{ \mathrm{diag}(t_1,\dots,t_n) ~|~ t_1\geq\dots\geq t_n\}.$$
This means that any $\hat{g}\in\hat{G}$ may be written $\hat{g} = \hat{k}_1\exp(\hat{a})\hat{k}_2$ for some $\hat{k}_1,\hat{k}_2\in\hat{K}$ and a unique $\hat{a}=\mathrm{diag}(t_1,\dots,t_n)\in\hat{\mathfrak{a}}^+$; for any $1\leq i\leq n$, the real number $t_i$ is the logarithm $\mu_i(\hat{g})$ of the $i$-th largest \emph{singular value} of~$\hat{g}$.
This induces a Cartan decomposition $G=K\exp(\mathfrak{a}^+)K$ with $K=\PO(n)$ and $\mathfrak{a}^+=\hat{\mathfrak{a}}^+/\RR$, and for any $1\leq i<j\leq n$ a map
\begin{equation} \label{eqn:mu-i-j}
\mu_i - \mu_j : G=\PGL(V) \longrightarrow \RR^+.
\end{equation}
If $\Vert\cdot\Vert_{\scriptscriptstyle V}$ denotes the operator norm associated with the standard Euclidean norm on $V=\RR^n$ invariant under $\hat{K}=\OO(n)$, then for any $g\in G$ with lift $\hat{g}\in\hat{G}$ we have
\begin{equation} \label{eqn:mu-1-n}
(\mu_1 - \mu_n)(g) = \log \big(\Vert\hat{g}\Vert_{\scriptscriptstyle V} \, \Vert{\hat{g}}^{-1}\Vert_{\scriptscriptstyle V}\big).
\end{equation}

\subsection{Proximality in projective space} \label{subsec:prox}

We shall use the following classical terminology.

\begin{definition} \label{def:prox}
An element $g\in\PGL(V)$ is \emph{proximal in $\PP(V)$} (\resp $\PP(V^*)$) if it admits a unique attracting fixed point in $\PP(V)$ (\resp $\PP(V^*)$).
Equivalently, any lift $\hat{g}\in\GL(V)$ of~$g$ has a unique complex eigenvalue of maximal (\resp minimal) modulus, with multiplicity~$1$.
This eigenvalue is necessarily real.
\end{definition}

For any $\hat{g}\in\GL(V)$, we denote by $\lambda_1(\hat{g})\geq\lambda_2(\hat{g})\geq\dots\geq\lambda_n(\hat{g})$ the logarithms of the moduli of the complex eigenvalues of~$\hat{g}$.
For any $1\leq i<j\leq n$, this induces a function
\begin{equation} \label{eqn:lambda-i-j}
\lambda_i - \lambda_j : \PGL(V) \longrightarrow \RR^+.
\end{equation}
Thus, an element $g\in\PGL(V)$ is proximal in $\PP(V)$ (\resp $\PP(V^*)$) if and only if\linebreak $(\lambda_1-\nolinebreak\lambda_2)(g)>\nolinebreak 0$ (\resp $(\lambda_{n-1}-\lambda_n)(g)>0$).
We shall use the following terminology.

\begin{definition} \label{def:prox-lim-set}
Let $\Gamma$ be a discrete subgroup of $\PGL(V)$.
The \emph{proximal limit set} of $\Gamma$ in $\PP(V)$ is the closure $\Lambda_{\Gamma}$ of the set of attracting fixed points of elements of~$\Gamma$ which are proximal in $\PP(V)$.
\end{definition}

\begin{remark} \label{rem:irred-prox-lim-set}
When $\Gamma$ is a discrete subgroup of $\PGL(V)$ acting \emph{irreducibly} on $\PP(V)$ and containing at least one proximal element, the proximal limit set $\Lambda_{\Gamma}$ was first studied in \cite{gui90,ben97,ben00}.
In that setting, the action of $\Gamma$ on $\Lambda_{\Gamma}$ is minimal (\ie any orbit is dense), and $\Lambda_{\Gamma}$ is contained in any nonempty, closed, $\Gamma$-invariant subset of $\PP(V)$, by \cite[Lem.\,2.5]{ben00}.
\end{remark}

\subsection{Anosov representations} \label{subsec:Anosov}

Let $P_1$ (\resp $P_{n-1}$) be the stabilizer in $G=\PGL(V)$ of a line (\resp hyperplane) of~$V$; it is a maximal proper parabolic subgroup of~$G$, and $G/P_1$ (\resp $G/P_{n-1}$) identifies with $\PP(V)$ (\resp with the dual projective space $\PP(V^*)$).
As in the introduction, we shall think of $\PP(V^*)$ as the space of projective hyperplanes in $\PP(V)$.
The following is not the original definition from \cite{lab06,gw12}, but an equivalent characterization taken from \cite[Th.\,1.7 \& Rem.\,4.3.(c)]{ggkw17}.

\begin{definition} \label{def:P1-Ano}
Let $\Gamma$ be a word hyperbolic group.
A representation $\rho : \Gamma\to G=\PGL(V)$ is \emph{$P_1$-Anosov} if there exist two continuous, $\rho$-equivariant boundary maps $\xi : \partial_{\infty}\Gamma\to\PP(V)$ and $\xi^* : \partial_{\infty}\Gamma\to\PP(V^*)$ such that
\begin{enumerate}[label=(A\arabic*)]
  \item \label{item:ano-comp} $\xi$ and~$\xi^*$ are compatible, \ie $\xi(\eta)\in\xi^*(\eta)$ for all $\eta\in\partial_{\infty}\Gamma$;
  \item \label{item:ano-trans} $\xi$ and~$\xi^*$ are transverse, \ie $\xi(\eta)\notin\xi^*(\eta')$ for all $\eta\neq\eta'$ in $\partial_{\infty}\Gamma$;
\end{enumerate}
\begin{enumerate}[label=(A\arabic*)']\setcounter{enumi}{2}
  \item \label{item:ano-lambda} $\xi$ and~$\xi^*$ are dynamics-preserving and there exist $c,C>0$ such that for any $\gamma\in\Gamma$,
  $$(\lambda_1-\lambda_2)(\rho(\gamma)) \geq c\,\ell_{\Gamma}(\gamma) - C,$$
  where $\ell_{\Gamma} : \Gamma\to\NN$ is the translation length function of $\Gamma$ in its Cayley graph (for some fixed choice of finite generating subset).
\end{enumerate}
\end{definition}

In condition~\ref{item:ano-lambda} we use the notation $\lambda_1-\lambda_2$ from \eqref{eqn:lambda-i-j}.
By \emph{dynamics-preserving} we mean that for any $\gamma\in\Gamma$ of infinite order, the element $\rho(\gamma)\in G$ is proximal and $\xi$ (\resp $\xi^*$) sends the attracting fixed point of $\gamma$ in $\partial_{\infty}\Gamma$ to the attracting fixed point of $\rho(\gamma)$ in $\PP(V)$ (\resp $\PP(V^*)$).
In particular, the set $\xi(\partial_{\infty} \Gamma)$ (\resp $\xi^*(\partial_{\infty} \Gamma)$) is the proximal limit set (Definition~\ref{def:prox-lim-set}) of $\rho(\Gamma)$ in $\PP(V)$ (\resp $\PP(V^*)$).
By \cite[Prop.\,4.10]{gw12}, if $\rho$ is irreducible then condition~\ref{item:ano-lambda} is automatically satisfied as soon as \ref{item:ano-comp} and \ref{item:ano-trans} are, but this is not true in general (see \cite[Ex.\,7.15]{ggkw17}).

If $\Gamma$ is not elementary (\ie not cyclic up to finite index), then the action of $\Gamma$ on $\partial_{\infty}\Gamma$ is minimal, \ie every orbit is dense; therefore the action of $\Gamma$ on $\xi(\partial_{\infty} \Gamma)$ and $\xi^*(\partial_{\infty} \Gamma)$ is also minimal.

We shall use a related characterization of Anosov representations, which we take from \cite[Th.\,1.3]{ggkw17}; it also follows from \cite{klp18}.

\begin{fact} \label{fact:charact-Ano}
Let $\Gamma$ be a word hyperbolic group.
A representation $\rho : \Gamma\to G=\PGL(V)$ is \emph{$P_1$-Anosov} if there exist two continuous, $\rho$-equivariant boundary maps $\xi : \partial_{\infty}\Gamma\to\PP(V)$ and $\xi^* : \partial_{\infty}\Gamma\to\PP(V^*)$ 
satisfying conditions \ref{item:ano-comp}--\ref{item:ano-trans} of Definition~\ref{def:P1-Ano}, as well as
\begin{enumerate}[label=(A\arabic*)'']\setcounter{enumi}{2}
  \item \label{item:ano-mu} $\xi$ and~$\xi^*$ are dynamics-preserving and
  $$(\mu_1-\mu_2)(\rho(\gamma)) \underset{|\gamma|_\Gamma\to +\infty}{\longrightarrow} +\infty,$$
  where $|\cdot|_{\Gamma} : \Gamma\to\NN$ is the word length function of~$\Gamma$ (for some fixed choice of finite generating subset).
\end{enumerate}
\end{fact}

In condition~\ref{item:ano-mu} we use the notation $\mu_1-\mu_2$ from \eqref{eqn:mu-i-j}.

By construction, the image of the boundary map $\xi : \partial_{\infty}\Gamma\to\PP(V)$ (\resp $\xi^* : \partial_{\infty}\Gamma\to\PP(V^*)$) of a $P_1$-Anosov representation $\rho : \Gamma\to\PGL(V)$ is the closure of the set of attracting fixed points of proximal elements of $\rho(\Gamma)$ in $\PP(V)$ (\resp $\PP(V^*)$.
Here is a useful alternative description.
We denote by $(e_1,\dots,e_n)$ the standard basis of $V = \RR^n$, orthonormal for the inner product preserved by $\hat{K} = \OO(n)$.

\begin{fact}[{\cite[Th.\,1.3 \& 5.3]{ggkw17}}] \label{fact:Im-xi-Cartan}
Let $\Gamma$ be a word hyperbolic group and $\rho : \Gamma\to\PGL(V)$ a $P_1$-Anosov representation with boundary maps $\xi : \partial_{\infty}\Gamma\to\PP(V)$ and $\xi^* : \partial_{\infty}\Gamma\to\PP(V^*)$.
Let $(\gamma_m)_{m\in\NN}$ be a sequence of elements of~$\Gamma$ converging to some $\eta\in\partial_{\infty}\Gamma$.
For any~$m$, choose $k_m\in K$ such that $\rho(\gamma_m)\in k_m\exp(\mathfrak{a}^+)K$ (see Section~\ref{subsec:Cartan}).
Then, writing $[e_n^*]:=\PP(\mathrm{span}(e_1,\dots,e_{n-1}))$,
$$\left \{\begin{array}{rcl}
\xi(\eta) & = & \lim_{m\to +\infty} k_m\cdot [e_1],\\
\xi^*(\eta) & = & \lim_{m\to +\infty} k_m\cdot [e_n^*].
\end{array} \right.$$
\end{fact}

In particular, the image of~$\xi$ is the set of accumulation points in $\PP(V)$ of the set\linebreak $\{ k_{\rho(\gamma)}\cdot [e_1] \,|\, \gamma\in\Gamma\}$ where $\gamma\in k_{\rho(\gamma)}\exp(\mathfrak{a}^+)K$; the image of~$\xi^*$ is the set of accumulation points in $\PP(V^*)$ of $\{ k_{\rho(\gamma)}\cdot [e_n^*] \,|\, \gamma\in\nolinebreak\Gamma\}$.

Here is an easy consequence of Fact~\ref{fact:charact-Ano}.

\begin{remark} \label{rem:Ano-restrict}
Let $\Gamma$ be a word hyperbolic group, $\Gamma'$ a subgroup of~$\Gamma$, and $\rho : \Gamma\to\PGL(V)$ a representation.
\begin{itemize}
  \item If $\Gamma'$ has finite index in~$\Gamma$, then $\rho$ is $P_1$-Anosov if and only if its restriction to~$\Gamma'$ is $P_1$-Anosov.
  \item If $\Gamma'$ is quasi-isometrically embedded in~$\Gamma$ and if $\rho$ is $P_1$-Anosov, then the restriction of $\rho$ to~$\Gamma'$ is $P_1$-Anosov.
\end{itemize}
\end{remark}

\section{Basic facts: actions on convex subsets of $\PP(V)$} \label{sec:basic-examples}

In this section we collect a few useful facts on actions of discrete subgroups of $\PGL(V)$ on properly convex open subsets of $\PP(V)$.

\subsection{The full orbital limit set in the strictly convex case}

The following elementary observation was mentioned in Section~\ref{subsec:intro-strong-cc}.

\begin{lemma} \label{lem:strict-convex-Lambda-orb}
Let $\Gamma$ be a discrete subgroup of $\PGL(V)$ preserving a nonempty properly convex open subset $\Omega$ of $\PP(V)$.
If $\Omega$ is strictly convex, then all $\Gamma$-orbits of~$\Omega$ have the same accumulation points in $\partial \Omega$.
\end{lemma}

\begin{proof}
It is enough to prove that for any points $x,y\in\Omega$, any accumulation point of $\Gamma\cdot x$ is also an accumulation point of $\Gamma\cdot y$.
Consider $(\gamma_m) \in \Gamma^{\NN}$ such that $(\gamma_m\cdot x)$ converges to some $x_{\infty}\in \partial\Omega$. 
After possibly passing to a subsequence, $(\gamma_m \cdot y)$ converges to some $y_{\infty} \in \partial \Omega$.
By properness of the action of $\Gamma$ on~$\Omega$, the limit $[x_{\infty}, y_{\infty}]$ of the sequence of compact intervals $(\gamma_m\cdot [x,y])$ is contained in $\partial \Omega$.
Strict convexity then implies that $x_{\infty} = y_{\infty}$, and so $x_{\infty}$ is also an accumulation point of $\Gamma\cdot y$.
\end{proof}

\subsection{Divergence for actions on properly convex cones}

We shall often use the following observation.

\begin{remark} \label{rem:lift-Gamma}
Let $\Gamma$ be a discrete subgroup of $\PGL(V)$ preserving a nonempty properly convex open subset $\Omega$ of $\PP(V)$.
Then there is a unique lift $\hat\Gamma$ of $\Gamma$ to $\SL^{\pm}(V)$ that preserves a properly convex cone $\widetilde \Omega$ of~$V$ lifting~$\Omega$.
\end{remark}
Here we say that a convex open cone of~$V$ is \emph{properly convex} if its projection to $\PP(V)$ is properly convex in the sense of Section~\ref{subsec:intro-strong-cc}.

The following observation will be useful in Sections \ref{sec:ccc-implies-Anosov} and~\ref{sec:other-properties}.

\begin{lemma} \label{lem:vector-growth}
Let $\hat\Gamma$ be a discrete subgroup of $\SL^{\pm}(V)$ preserving a properly convex open cone $\widetilde \Omega$ in $V$.
For any sequence $(\gamma_m)_{m\in\NN}$ of pairwise distinct elements of~$\hat\Gamma$ and any nonzero vector $v \in \widetilde \Omega$, the sequence $(\gamma_m\cdot v)_{m\in\NN}$ goes to infinity in $V$ as $m\to +\infty$. 
This divergence is uniform as $v$ varies in a compact set $\mathcal{K} \subset \widetilde \Omega$.
\end{lemma}

\begin{proof}
Fix a compact subset $\mathcal{K}$ of $\widetilde{\Omega}$.
Let $\varphi\in V^*$ be a linear form which takes positive values on the closure of~$\widetilde{\Omega}$.
The set $\widetilde{\Omega}\cap\nolinebreak\{\varphi=\nolinebreak 1\}$ is bounded, with compact boundary~$\mathcal{B}$ in~$V$.
By compactness of $\mathcal{K}$ and~$\mathcal{B}$, we can find $0<\varepsilon<1$ such that for any $v\in\mathcal{K}$ and $w\in\mathcal{B}$ the line through $v/\varphi(v)$ and~$w$ intersects~$\mathcal B$ in a point $w'\neq w$ such that $v/\varphi(v) = tw + (1-t)w'$ for some $t\geq\varepsilon$.
For any $m\in\NN$, we then have $\varphi(\gamma_m\cdot v)/\varphi(v) \geq \varepsilon \, \varphi(\gamma_m\cdot w)$, and since this holds for any $w\in\mathcal{B}$ we obtain
$$\varphi(\gamma_m\cdot v) \geq \kappa \, \max_{\mathcal{B}}(\varphi\circ\gamma_m)$$
where $\kappa := \varepsilon \, \min_{\mathcal{K}}(\varphi) > 0$.
Thus it is sufficient to see that the maximum of $\varphi\circ\gamma_m$ over~$\mathcal{B}$ tends to infinity with~$m$.
By convexity, it is in fact sufficient to see that the maximum of $\varphi\circ\gamma_m$ over $\widetilde{\Omega}\cap\{\varphi<1\}$ tends to infinity with~$m$.
This follows from the fact that the set $\widetilde{\Omega} \cap \{\varphi<1\}$ is open, that the operator norm of $\gamma_m\in\hat\Gamma\subset\mathrm{End}(V)$ goes to $+\infty$, and that $\varphi |_{\widetilde{\Omega}}$ is bounded above and below by positive multiples of any given norm of~$V$.
\end{proof}

In the case that $\hat \Gamma$ preserves a quadratic form on~$V$, the fact that $(\gamma_m\cdot v)_{m\in\NN}$ goes to infinity as $m\to +\infty$ implies that any accumulation point of $([\gamma_m \cdot v])_{m\in\NN}$ in $\PP(V)$ is isotropic.
We thus get the following corollary, which will be used in Section~\ref{sec:Hpq-cc}.

\begin{corollary} \label{cor:Lambdao-Hpq}
For $p,q\in\NN^*$, let $\Gamma$ be an infinite discrete subgroup of $\PO(p,q)$ preserving a properly convex open subset $\Omega$ of $\PP(\RR^{p+q})$.
Then the full orbital limit set $\Lambdao_{\Omega}(\Gamma)$ is contained in $\partial \HH^{p,q-1}$.
\end{corollary}

\subsection{Comparison between the Hilbert and Euclidean metrics}

The following will be used later in this section, and in the proofs of Lemmas \ref{lem:Lambda-con} and~\ref{lem:limit-hyperplanes}.

\begin{lemma} \label{lem:compare-Hilb-Eucl}
Let $\Omega\subset \PP(V)$ be open, properly convex, contained in a Euclidean affine chart $(\RR^{n-1},d_{\mathrm{Euc}})$ of~$\PP(V)$.
If $R>0$ is the diameter of $\Omega$ in $(\RR^{n-1},d_{\mathrm{Euc}})$, then the natural inclusion defines an $R/2$-Lipschitz map
$$(\Omega,d_{\Omega}) \longrightarrow (\RR^{n-1},d_{\mathrm{Euc}}).$$
\end{lemma}

\begin{proof}
We may assume $n=2$ up to restricting to an affine line intersecting~$\Omega$, and $R=2$ up to rescaling, so that $\Omega= \HH^1 = (-1,1) \subset \RR$.
Since $\frac{1}{2} \log \cro{-1}{0}{\tanh(t)}{1}=t$, the arclength parametrization of $\Omega$ is then given by the map $\tanh:\RR\to (-1,1)$, which is $1$-Lipschitz.\end{proof}

\begin{corollary} \label{cor:compare-Hilb-Eucl}
Let $\Omega$ be a properly convex open subset of $\PP(V)$.
Let $(x_m)_{m\in\NN}$ and $(y_m)_{m\in\NN}$ be two sequences of points of~$\Omega$, and $z\in\partial\Omega$.
If $x_m\to z$ and $d_{\Omega}(x_m,y_m)\to 0$, then $y_m\to z$.
\end{corollary}

\subsection{Closed ideal boundary}

The following elementary observation will be used in Sections \ref{sec:naive-cc}, \ref{sec:bisat-dual}, and~\ref{sec:regularity-limit-set}.

\begin{remark} \label{rem:closed-ideal-bound-C-Omega}
Let $\Gamma$ be an infinite discrete subgroup of $\PGL(V)$ preserving a properly convex open subset $\Omega$ of $\PP(V)$ and a subset $\C$ of~$\Omega$.
If $\Gamma\backslash\C$ is closed in $\Gamma\backslash\Omega$, then $\C$ is closed in~$\Omega$ and $\partiali\C = \overline{\C}\cap\partial\Omega$; this is the case in particular if the action of $\Gamma$ on~$\C$ is cocompact.
\end{remark}

In Sections \ref{sec:bisat-dual} and~\ref{sec:smooth} we shall consider convex subsets $\C$ of $\PP(V)$ that are not assumed to be subsets of a properly convex open set~$\Omega$.
An arbitrary convex subset $\C$ of $\PP(V)$ (not necessarily open nor closed) is locally compact for the induced topology if and only if its ideal boundary $\partiali\C = \overline{\C}\smallsetminus\C$ is closed in $\PP(V)$. In the non-locally compact setting, there are several possible definitions for proper discontinuity and cocompactness, of varying strengths, see~\cite[\S 4]{bourbaki} or~\cite{kap-prop-disc}.
We will use the following definitions: the action of a discrete group $\Gamma$ on a topological space $\C$ is \emph{properly discontinuous} if for any compact subset $K$ of~$\C$, the set of elements $\gamma\in\Gamma$ such that $K \cap \gamma\cdot K \neq \emptyset$ is finite.
The action is \emph{cocompact} if there exists a compact subset $\mathcal{D}$ of~$\C$ such that $\C = \bigcup_{\gamma\in\Gamma} \gamma\cdot\mathcal{D}$. 
Then the following holds.

\begin{lemma} \label{lem:closed-ideal-boundary}
Let $\Gamma$ be a discrete subgroup of $\PGL(V)$ and $\C$ a $\Gamma$-invariant convex subset of $\PP(V)$.
Suppose the action of $\Gamma$ on~$\C$ is properly discontinuous and cocompact.
Then $\partiali\C$ is closed in $\PP(V)$.
\end{lemma}

\begin{proof}
Let $(z_m)_{m\in\NN}$ be a sequence of points of $\partiali \C$ converging to some $z\in\Fr(\C)$.
Suppose for contradiction that $z \notin \partiali \C$, so that $z \in \partialn \C$.
For each $m$, let $(y_{m,k})_{k\in\NN}$ be a sequence of points of~$\C$ converging to $z_m$ as $k \to +\infty$.
Let $\mathcal{D}\subset\C$ be a compact subset such that $\C = \bigcup_{\gamma\in\Gamma} \gamma\cdot\mathcal{D}$.
For any $m,k\in\NN$, there exists $\gamma_{m,k} \in \Gamma$ such that $\gamma_{m,k} \cdot y_{m,k} \in \mathcal{D}$.
Note that for each $m$, the collection $\{\gamma_{m,k}\}_{k=1}^{\infty}$ is infinite.
We now choose a sequence $(y_{m,k_m})_{m\in\NN}$ inductively as follows.
Fix an auxiliary metric $d(\cdot, \cdot)$ on $\PP(V)$. Let $k_1$ be such that $d(y_{1,k_1}, z_1) < 1$.
For each $m > 1$, let $k_m$ be such that $d(y_{m,k_m}, z_m) < 1/m$ and $\gamma_{m, k_m}$ is distinct from all $\gamma_{1,k_1}, \ldots, \gamma_{m-1, k_{m-1}}$. 
Then $y_{m,k_m}$ converges to $z$ as $m\to +\infty$ and the sequence $(\gamma_{m,k_m})_{m=1}^\infty$ is injective.
The compact subset $\{z\}\cup\{y_{m,k_m}\}_{m\in \NN}$ of $\C$ then has the property that its translate by any of the infinitely many elements $\{\gamma_{m,k_m}\}_{m=1}^\infty $ intersects the compact set $\mathcal{D}$.
This contradicts properness.
\end{proof}

\subsection{Nonempty interior}

The following observation shows that in the setting of Definition~\ref{def:cc-naive} we may always assume $\C$ to have nonempty interior.

\begin{lemma} \label{lem:naive-cc-nonempty-int}
Let $\Gamma$ be an infinite discrete subgroup of $\PGL(V)$ preserving a properly convex open subset $\Omega$ of $\PP(V)$ and acting cocompactly on some nonempty closed convex subset $\C$ of~$\Omega$.
For $R>0$, let $\C_R$ be the closed uniform $R$-neighborhood of $\C$ in $(\Omega,d_{\Omega})$.
Then $\C_R$ is a closed convex subset of~$\Omega$ with nonempty interior on which $\Gamma$ acts cocompactly.
\end{lemma}

\begin{proof}
The set $\C_R$ is properly convex by \cite[(18.12)]{bus55}.
The group $\Gamma$ acts properly discontinuously on~$\C_R$ since it acts properly discontinuously on~$\Omega$, and cocompactly on~$\C_R$ since it acts cocompactly on~$\C$: the set $\C_R$ is the union of the $\Gamma$-translates of the closed uniform $R$-neighborhood of a compact fundamental domain of $\C$ in $(\Omega,d_{\Omega})$.
\end{proof}

\subsection{Maximal invariant convex sets} \label{sec:unique-limit-set}

The following was first observed by Benoist \cite[Prop.\,3.1]{ben00} for discrete subgroups of $\PGL(V)$ acting irreducibly on $\PP(V)$.
Here we do not make any irreducibility assumption.

\begin{proposition} \label{prop:max-inv-conv}
Let $\Gamma$ be a discrete subgroup of $\PGL(V)$ preserving a nonempty properly convex open subset $\Omega$ of $\PP(V)$ and containing a proximal element.
Let $\Lambda_{\Gamma}$ (\resp $\Lambda_{\Gamma}^*$) be the proximal limit set of $\Gamma$ in $\PP(V)$ (\resp $\PP(V^*)$) (Definition~\ref{def:prox-lim-set}).
Then
\begin{enumerate}
  \item\label{item:Lambda-prox-in-boundary} $\Lambda_{\Gamma}$ (\resp $\Lambda_{\Gamma}^*$) is contained in the boundary of $\Omega$ (\resp its dual~$\Omega^*$);
  \item\label{item:Lambda-prox-lift-Omega} more specifically, $\Omega$ and $\Lambda_{\Gamma}$ lift to cones $\widetilde{\Omega}$ and $\widetilde{\Lambda}_\Gamma$ of $V\smallsetminus\{0\}$ with $\widetilde{\Omega}$ properly convex containing $\widetilde{\Lambda}_{\Gamma}$ in its boundary, and $\Omega^*$ and $\Lambda_{\Gamma}^*$ lift to cones $\widetilde{\Omega}^*$ and $\widetilde{\Lambda}_\Gamma^*$ of $V^*\smallsetminus\{0\}$ with $\widetilde{\Omega}^*$ properly convex containing $\widetilde{\Lambda}_{\Gamma}^*$ in its boundary, such that $\varphi(v)\geq 0$ for all $v\in\widetilde{\Lambda}_{\Gamma}$ and $\varphi\in\widetilde{\Lambda}_{\Gamma}^*$;
  \item\label{item:Omega-max} for $\widetilde{\Lambda}_{\Gamma}^*$ as in (\ref{item:Lambda-prox-lift-Omega}), the set
  $$\Omega_{\max} = \PP(\{ v\in V ~|~ \varphi(v)>0\quad \forall\varphi\in\widetilde{\Lambda}_{\Gamma}^*\})$$
is the unique connected component of  $\PP(V) \smallsetminus \bigcup_{z^*\in\Lambda_{\Gamma}^*} z^*$ containing~$\Omega$; it is $\Gamma$-invariant, convex, and open in $\PP(V)$; any $\Gamma$-invariant properly convex open subset $\Omega'$ of $\PP(V)$ containing~$\Omega$ is contained in~$\Omega_{\max}$.
\end{enumerate}
\end{proposition}

\begin{proof}
\eqref{item:Lambda-prox-in-boundary} Let $\gamma\in\Gamma$ be proximal in $\PP(V)$, with attracting fixed point $z_{\gamma}^+$ and complementary $\gamma$-invariant hyperplane $H_{\gamma}^-$.
Since $\Omega$ is open, there exists $x\in\Omega\smallsetminus H_{\gamma}^-$.
We then have $\gamma^m\cdot x\to z_{\gamma}^+$, and so $z_{\gamma}^+\in\partial\Omega$ since the action of $\Gamma$ on~$\Omega$ is properly discontinuous.
Thus $\Lambda_{\Gamma}\subset \partial \Omega$.
Similarly, $\Lambda_{\Gamma}^*\subset \partial \Omega^*$.

\eqref{item:Lambda-prox-lift-Omega} The set $\Omega$ lifts to a properly convex cone $\widetilde{\Omega}$ of $V\smallsetminus\{0\}$, unique up to global sign.
This determines a cone $\widetilde{\Lambda}_{\Gamma}$ of $V\smallsetminus\{0\}$ lifting $\Lambda_{\Gamma}\subset\partial\Omega$ and contained in the boundary of~$\widetilde{\Omega}$.
By definition, $\Omega^*$ is the projection to $\PP(V^*)$ of the dual cone
$$\widetilde{\Omega}^* := \big\{ \varphi\in V^*\smallsetminus\{0\} ~|~ \varphi(v)>0\quad \forall v\in\overline{\widetilde{\Omega}}\smallsetminus\{0\}\big\}.$$
This cone determines a cone $\widetilde{\Lambda}^*_{\Gamma}$ of $V^*\smallsetminus\{0\}$ lifting $\Lambda_{\Gamma}^*\subset\partial\Omega^*$ and contained in the boundary of~$\widetilde{\Omega}^*$.
By construction, $\varphi(v)\geq 0$ for all $v\in\widetilde{\Lambda}_{\Gamma}$ and $\varphi\in\widetilde{\Lambda}_{\Gamma}^*$.

\eqref{item:Omega-max} The set $\Omega_{\max} := \PP(\{ v\in V \,|\, \varphi(v)>0\ \forall\varphi\in\widetilde{\Lambda}_{\Gamma}^*\})$ is a connected component of $\PP(V) \smallsetminus \bigcup_{z^*\in\Lambda_{\Gamma}^*} z^*$.
It is convex, open in $\PP(V)$ by compactness of~$\Lambda_{\Gamma}^*$, and it contains~$\Omega$.
The action of~$\Gamma$, which permutes the connected components of $\PP(V) \smallsetminus \bigcup_{z^*\in\Lambda_{\Gamma}^*} z^*$, preserves $\Omega_{\max}$ because it preserves $\Omega$.
Moreover, any $\Gamma$-invariant properly convex open subset $\Omega'$ of $\PP(V)$ containing~$\Omega$ is contained in~$\Omega_{\max}$: indeed, $\Omega'$ cannot meet $z^*$ for $z^*\in\Lambda_{\Gamma}^*$ since $\Lambda_{\Gamma}^*\subset\partial{\Omega'}^*$ by~\eqref{item:Lambda-prox-in-boundary}.
\end{proof}

\begin{remark} \label{rem:not-irred}
In the context of Proposition~\ref{prop:max-inv-conv}, when $\Gamma$ preserves a nonempty properly convex open subset $\Omega$ of $\PP(V)$ but does not act irreducibly on $\PP(V)$, the following may happen:
\begin{enumerate}[label=(\alph*)]
  \item $\Gamma$ may not contain any proximal element in $\PP(V)$: this is the case \eg if $V = V'\oplus V'$ for some vector space~$V'$ and $\Gamma\subset\PGL(V)$ is the image of a diagonal embedding of a discrete group $\hat\Gamma'\subset\SL^{\pm}(V')$ preserving a properly convex open set in $\PP(V')$;
  \item assuming that $\Gamma$ contains a proximal element in $\PP(V)$ (hence $\Lambda_{\Gamma},\Lambda_{\Gamma}^*\neq\nolinebreak\emptyset$), the set $\Omega_{\max}$ of Proposition~\ref{prop:max-inv-conv}.\eqref{item:Omega-max} may fail to be properly convex: this is the case \eg if $\Gamma$ is a convex cocompact subgroup of $\SO(2,1)_0$, embedded into $\PGL(\RR^4)$ where it preserves the properly convex open set $\HH^3 \subset \PP(\RR^4)$ of Example~\ref{ex:H-p};
  \item even if $\Omega_{\max}$ is properly convex, it may not be the unique maximal $\Gamma$-invariant properly convex open set in $\PP(V)$: indeed, there may be multiple components of $\PP(V) \smallsetminus \bigcup_{z^*\in\Lambda_{\Gamma}^*} z^*$ that are properly convex and $\Gamma$-invariant, as in Example~\ref{ex:Zn} below.
\end{enumerate}
\end{remark}

However, if $\Gamma$ acts irreducibly on $\PP(V)$, then the following holds.

\begin{fact}[{\cite[Prop.\,3.1]{ben00}}] \label{fact:benoist-irred-Gamma}
Let $\Gamma$ be a discrete subgroup of $\PGL(V)$ acting irreducibly on $\PP(V)$ and preserving a nonempty properly convex open subset $\Omega$ of $\PP(V)$.
Then
\begin{enumerate}
  \item\label{item:benoist-irred-1} $\Gamma$ always contains a proximal element and the set $\Omega_{\max}$ of Proposition~\ref{prop:max-inv-conv}.(\ref{item:Omega-max}) is always properly convex (see Figure~\ref{fig:Omega-min-max}); it is a maximal $\Gamma$-invariant properly convex open subset of $\PP(V)$ containing~$\Omega$;
  \item\label{item:benoist-irred-2} if moreover $\Gamma$ acts strongly irreducibly on $\PP(V)$ (\ie all finite-index subgroups of~$\Gamma$ act irreducibly), then $\Omega_{\max}$ is the unique maximal $\Gamma$-invariant properly convex open set in $\PP(V)$; it contains all other invariant properly convex open subsets;
  \item\label{item:benoist-irred-3} in general, there is a smallest nonempty $\Gamma$-invariant convex open subset $\Omega_{\min}$ of $\Omega_{\max}$, namely the interior of the convex hull of $\Lambda_{\Gamma}$ in~$\Omega_{\max}$.
\end{enumerate}
\end{fact}

\begin{figure}
\includegraphics[width = 5.0cm]{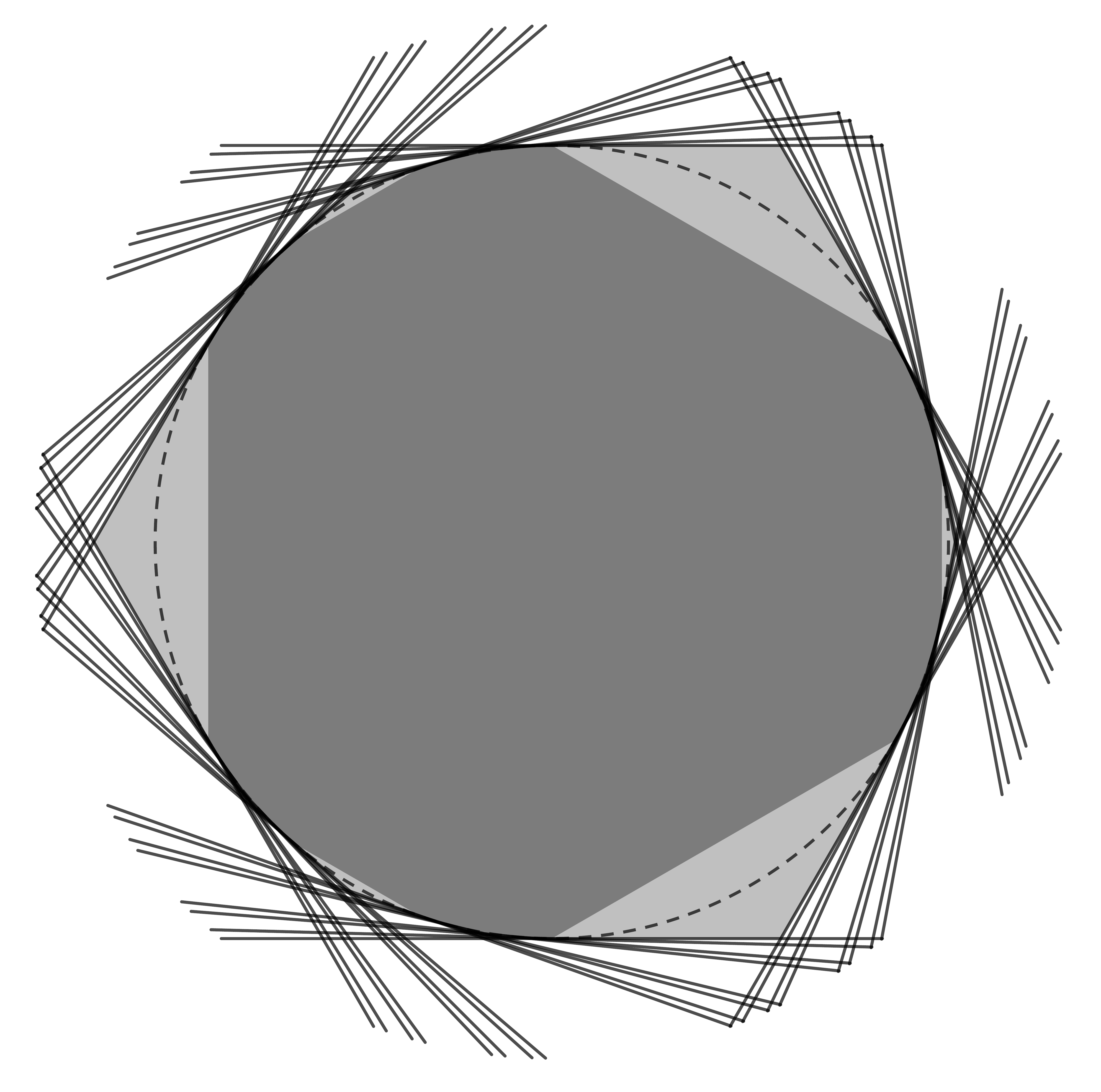}
\caption{The sets $\Omega_{\min}$ (dark gray) and $\Omega_{\max}$ (light gray) for a convex cocompact subgroup of $\PO(2,1)$. Here $\HH^2$ is the disk bounded by the dashed circle. The set $\Gamma \backslash \Ccore_{\Omega_{\max}}(\Gamma)$ is compact (equal to $\Gamma \backslash \Ccore_{\HH^2}(\Gamma)$) but $\Gamma \backslash \Ccore_{\Omega_{\min}}(\Gamma)$ is not.}
\label{fig:Omega-min-max}
\end{figure}

In general, the following strengthening of Proposition~\ref{prop:max-inv-conv}.\eqref{item:Lambda-prox-in-boundary} holds.

\begin{lemma} \label{lem:Lambda-prox-Lambda-orb}
Let $\Gamma$ be a discrete subgroup of $\PGL(V)$ preserving a nonempty properly convex open subset $\Omega$ of $\PP(V)$.
Then the proximal limit set $\Lambda_{\Gamma}$ is contained in the set of accumulation points of any $\Gamma$-orbit of~$\Omega$; in particular, $\Lambda_{\Gamma}$ is contained in the full orbital limit set $\Lambdao_{\Omega}(\Gamma)$ (Definition~\ref{def:lambdaorb}).
\end{lemma}

\begin{proof}
For any proximal element $\gamma\in \Gamma$, the repelling hyperplane of $\gamma$ is disjoint from (in fact, tangent to) $\Omega$, by Proposition~\ref{prop:max-inv-conv}.\eqref{item:Lambda-prox-in-boundary}. 
Therefore, for any point $x\in \Omega$, the sequence $(\gamma^m \cdot x)_{m\in \mathbb{N}}$ converges to the attracting fixed point of $\gamma$. A diagonal extraction argument shows that $\Gamma \cdot x$ contains the whole proximal limit set $\Lambda_\Gamma$. 
\end{proof}

\subsection{The case of a connected proximal limit set}

We make the following observation.

\begin{lemma} \label{lem:Omega-max-connected}
Let $\Gamma$ be an infinite discrete subgroup of $\PGL(V)$ preserving a nonempty properly convex open subset $\Omega$ of $\PP(V)$.
Suppose the proximal limit set $\Lambda_{\Gamma}^*$ of $\Gamma$ in $\PP(V^*)$ (Definition~\ref{def:prox-lim-set}) is \emph{connected}.
Then the set $\PP(V) \smallsetminus \bigcup_{z^*\in\Lambda_{\Gamma}^*} z^*$ is a nonempty $\Gamma$-invariant convex open subset of $\PP(V)$, not necessarily properly convex, but containing all $\Gamma$-invariant properly convex open subsets of $\PP(V)$; it is equal to the set $\Omega_{\max}$ of Proposition~\ref{prop:max-inv-conv}.
\end{lemma}

This lemma will follow from a basic fact about well-definedness of convex hulls in projective space.

\begin{lemma} \label{lem:general-convex-hulls}
\begin{enumerate}
  \item \label{ite:gch-one} Let $L$ be a closed, connected subset of $\PP(V)$, and let $H\neq H'$ be two hyperplanes in $\PP(V)$ disjoint from $L$.
  Then the convex hull of $L$ taken in the affine chart $\PP(V) \smallsetminus H$ agrees with the convex hull of $L$ taken in the affine chart $\PP(V) \smallsetminus H'$.
  \item \label{ite:gch-two} Let $L^*$ be a closed, connected subset of $\PP(V^*)$.
  Then, thinking of each $z^* \in \PP(V^*)$ as a hyperplane in $\PP(V)$, the open set $\PP(V) \smallsetminus \bigcup_{z^*\in L^*} z^*$ has at most one connected component, which is convex.
\end{enumerate}
\end{lemma}

\begin{proof}
\eqref{ite:gch-one} The set $\PP(V) \smallsetminus (H \cup H')$ has two connected components, each of which is convex. Since $L$ is connected, it must be contained in exactly one of these, which we denote by~$\mathcal{O}$.
The convex hull of $\Lambda$ in $\PP(V) \smallsetminus H$, or in $\PP(V) \smallsetminus H'$, agrees with the convex hull of $\Lambda$ in $\mathcal{O}$.

\eqref{ite:gch-two} Let $\Omega$ be a connected component of $\PP(V) \smallsetminus \bigcup_{z^*\in L^*} z^*$. The dual convex set $\Omega^*$ contains $L^*$ in its closure. In fact, if $x \in \Omega$, then $\overline{\Omega^*}$ is the convex hull of $L^*$, taken in the affine chart $\PP(V^*) \smallsetminus x$, where we think of $x \in \PP(V) = \PP((V^*)^*)$ as a hyperplane in $\PP(V^*)$. 
Since $L^*$ is connected, it now follows from~\eqref{ite:gch-one} that the open set $\PP(V) \smallsetminus \bigcup_{z^*\in L^*} z^*$ can have at most one connected component.
\end{proof}

\begin{proof}[Proof of Lemma~\ref{lem:Omega-max-connected}]
By Lemma~\ref{lem:general-convex-hulls}, since $\Lambda_{\Gamma}^*$ is connected, the set $\PP(V) \smallsetminus \bigcup_{z^*\in\Lambda_{\Gamma}^*} z^*$ has either one or zero connected components.
It contains $\Omega$, hence is nonempty, so it must have exactly one component, which is convex.
\end{proof}

\section{Convex cocompact and naively convex cocompact subgroups} \label{sec:naive-cc}

In this section we investigate the difference between naively convex cocompact groups (Definition~\ref{def:cc-naive}) and convex cocompact groups in $\PP(V)$ (Definition~\ref{def:cc-general}). We also study a notion of conical limit points and how it relates to faces of the domain $\Omega$, and study minimality properties of the convex core.

\subsection{Examples} \label{subsec:ex-counterex}

The following basic examples are designed to make the notions of convex cocompactness and naive convex cocompactness in $\PP(V)$ more concrete, and to point out some subtleties.

\begin{example} \label{ex:Zn}
Let $V=\RR^n$ with standard basis $(e_1,\dots,e_n)$, where $n\geq 2$.
Let $\Gamma\simeq \ZZ^n/\ZZ \simeq \ZZ^{n-1}$ be the discrete subgroup of $\PGL(V)$ of diagonal matrices whose entries are powers of some fixed~$t>\nolinebreak 1$; it is not word hyperbolic if $n\geq 3$.
The hyperplanes
$$H_k = \PP(\mathrm{span}(e_1,\dots,e_{k-1},e_{k+1},\dots,e_n))$$
for $1\leq k\leq n$ cut $\PP(V)$ into $2^{n-1}$ properly convex open connected components~$\Omega$, which are not strictly convex if $n\geq 3$ (see Figure~\ref{fig:triangle}).
The group $\Gamma$ acts properly discontinuously and cocompactly on each of them, hence is convex cocompact in $\PP(V)$ (see Example~\ref{ex:div-implies-cc}).
\end{example}

\begin{figure}
\includegraphics[width = 4.0cm]{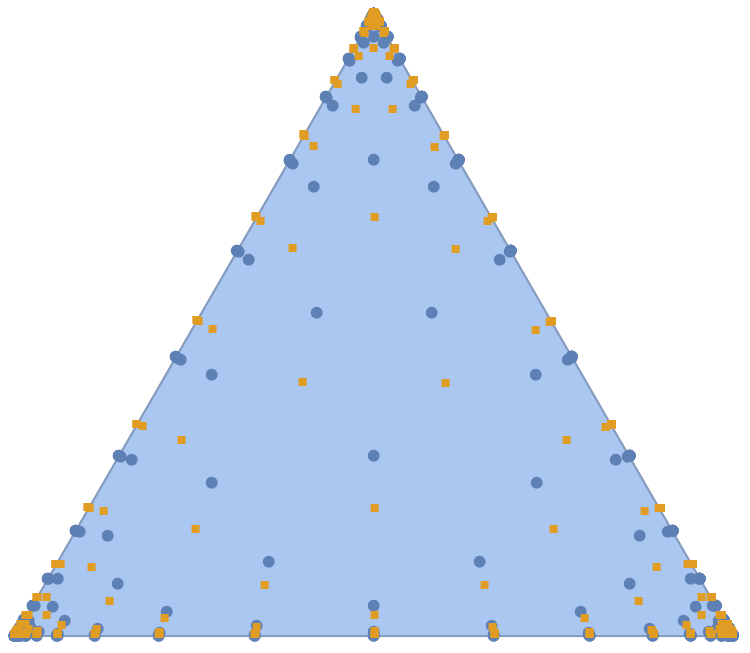}
\caption{The subgroup $\Gamma \subset \PGL(\RR^3)$ of all diagonal matrices with each entry a power of $t=2$ divides a triangle $\Omega$ in $\PP(\RR^3)$ with vertices the standard basis. Shown are two $\Gamma$-orbits with different accumulation sets in~$\partial\Omega$.}
\label{fig:triangle}
\end{figure}

In Example~\ref{ex:Zn} the set of accumulation points of one $\Gamma$-orbit of~$\Omega$ depends on the choice of orbit (see Figure~\ref{fig:triangle}), but the full orbital limit set $\Lambdao_{\Omega}(\Gamma)=\partial\Omega$ depends only on~$\Omega$.
The proximal limit set $\Lambda_{\Gamma}$ (Definition~\ref{def:prox-lim-set}) is the set of extreme points of 
$\overline{\Omega}$ and is the projectivized standard basis, independently of the choice of~$\Omega$.

Here is an irreducible variant of Example~\ref{ex:Zn}, containing it as a finite-index subgroup.
It shows that even when $\Gamma$ acts irreducibly on $\PP(V)$, it may preserve and divide several disjoint properly convex open subsets of $\PP(V)$.
(This does not happen under strong irreducibility, see Fact~\ref{fact:benoist-irred-Gamma}.\eqref{item:benoist-irred-2}.)

\begin{example} \label{ex:An}
Let $F$ be a finite group acting transitively on a set $I$ of cardinality $n\geq 2$. 
Let $V=\RR^I\simeq \RR^n$, and $\Gamma \simeq \ZZ^I/\ZZ$ be as in Example~\ref{ex:Zn}.
Then $\overline{\Gamma}:=F\ltimes \Gamma$ acts irreducibly on $\PP(V)$, preserving the positive orthant $\Delta_+:=\PP(\RR_{>0}^I)$.
For any index-two subgroup $F'$ of~$F$ whose restricted action has two orbits $I',I''$ partitioning $I$, the same group $\overline{\Gamma}$ also preserves and acts cocompactly on the orthant 
$\Delta_{F'}:=\PP(\RR_{>0}^{I'} \times \RR_{<0}^{I''})$.
There may be many such $F'\subset F$, \eg any index-two subgroup for $F=I=(\ZZ/2\ZZ)^{\nu}$ with $\nu\in\NN^*$.
\end{example}

Here is an example showing that a discrete subgroup which is convex cocompact in $\PP(V)$ need not act convex cocompactly on every invariant properly convex open subset of $\PP(V)$.

\begin{example}\label{ex:KleinianFuchsian}
Suppose $\Gamma$ is a convex cocompact subgroup (in the classical sense) of $\PO(p-1,1)\subset \PO(p,1) \subset \PGL(\RR^{p+1})$.
Then $\Gamma$ acts convex cocompactly (Definition~\ref{def:cc-general}), and even strongly convex cocompactly (Definition~\ref{def:strong-cc}), on the properly convex open set $\HH^p$, hence $\Gamma$ is strongly convex cocompact in $\PP(\RR^{p+1})$.
Note that the set $\Lambda_{\Gamma}$ is contained in the equatorial sphere $\partial\HH^{p-1}$ of $\partial\HH^p\subset\PP(\RR^{p+1})$.
Let $\Omega$ be one of the two $\Gamma$-invariant hyperbolic half-spaces of $\HH^p$ bounded by $\HH^{p-1}$. Then $\Ccore_{\Omega}(\Gamma) \subset \Ccore_{\HH^p}(\Gamma) \cap \Omega = \emptyset$.
Hence $\Gamma$ does not act convex cocompactly on~$\Omega$.
\end{example}

The following basic examples, where $\Gamma$ is a cyclic group acting on the projective plane $\PP(\RR^3)$, may be useful to keep in mind.

\begin{examples} \label{ex:triangle}
Let $V=\RR^3$, with standard basis $(e_1,e_2,e_3)$.
Let $\Gamma$ be a cyclic group generated by an element $\gamma\in\PGL(V)$.
\begin{enumerate}
  \item\label{ex:triangle:cc} Suppose $\gamma=\left(\begin{smallmatrix} a & 0 & 0\\ 0 & b & 0\\ 0 & 0 & c\end{smallmatrix}\right)$ where $a > b > c > 0$.
Then $\gamma$ has attracting fixed point $[e_1]$ and repelling fixed point $[e_3]$.
There is a $\Gamma$-invariant properly convex open neighborhood $\Omega$ of an open segment $([e_1], [e_3])$ connecting $[e_1]$ to $[e_3]$.
The full orbital limit set $\Lambdao_\Omega(\Gamma)$ is just $\{[e_1],[e_3]\}$ and its convex hull $\Ccore_{\Omega}(\Gamma) = ([e_1], [e_3])$ has compact quotient by~$\Gamma$ (a circle).
Thus $\Gamma$ is convex cocompact in $\PP(V)$.
  \item\label{ex:triangle:borderline} Suppose $\gamma=\left(\begin{smallmatrix} a & 0 & 0\\ 0 & b & 0\\ 0 & 0 & b\end{smallmatrix}\right)$ where $a > b > 0$.
Any $\Gamma$-invariant properly convex open set $\Omega$ is a triangle with tip $[e_1]$ and base an open segment $I$ of the line $\PP(\mathrm{span}(e_2,e_3))$. 
For any $x\in\Omega$, the $\Gamma$-orbit of~$x$ has two accumulation points, namely $[e_1]$ and the intersection of $I$ with the projective line through $[e_1]$ and~$x$. 
Therefore, the full orbital limit set $\Lambdao_\Omega(\Gamma)$ consists $\{[e_1]\} \cup I$, hence the convex hull $\Ccore_\Omega(\Gamma)$ of $\Lambdao_\Omega(\Gamma)$ is the whole of~$\Omega$, and $\Gamma$ does not act cocompactly on it.
Thus $\Gamma$ is not convex cocompact in $\PP(V)$.
However, $\Gamma$ is naively convex cocompact in $\PP(V)$ (Definition~\ref{def:cc-naive}): it acts cocompactly on the convex hull $\C$ in~$\Omega$ of $\{[e_1]\}$ and of any closed segment $I'\subset I$.
The quotient $\Gamma \backslash \C$ is a closed convex projective annulus (or a circle if $I'$ is reduced to a singleton).
  \item\label{ex:triangle:unipot} Suppose $\gamma=\left(\begin{smallmatrix} 2 & 0 & 0\\ 0 & 1 & t\\ 0 & 0 & 1\end{smallmatrix}\right)$ where $t>0$.
Then there exist nonempty $\Gamma$-invariant properly convex open subsets of $\PP(V)$, for instance $\Omega_s = \PP(\{ (v_1,v_2,1) \,|\, v_1>s\,2^{v_2/t}\})$ for any $s>0$.
However, for any such set~$\Omega$, we have $\Lambdao_\Omega(\Gamma)=\{[e_1],[e_2]\}$, and the convex hull of $\Lambdao_\Omega(\Gamma)$ in~$\overline{\Omega}$ is contained in $\partial\Omega$ (see \cite[Prop.\,2.13]{mar12}).
Hence $\Ccore_\Omega(\Gamma)$ is empty and $\Gamma$ is not convex cocompact in $\PP(V)$.
It is an easy exercise to check that $\Gamma$ is not even naively convex cocompact in $\PP(V)$.
  \item\label{ex:triangle:fail} Suppose $\gamma=\left(\begin{smallmatrix} a & 0 & 0\\ 0 & b\cos\theta & -b\sin\theta\\ 0 & b\sin\theta & b\cos\theta\end{smallmatrix}\right)$ where $a > b > 0$ and $0 < \theta \leq \pi$.
Then $\Gamma$ does not preserve any nonempty properly convex open subset of $\PP(V)$ (see \cite[Prop.\,2.4]{mar12}).
\end{enumerate}
\end{examples}

\begin{remarks} \label{rem:stable-precise}
\begin{enumerate}[label=(\alph*)]
  \item Convex cocompactness in $\PP(V)$ is not a closed condition in general.
Indeed, Example~\ref{ex:triangle}.\eqref{ex:triangle:cc}, which is convex cocompact, can limit to Example~\ref{ex:triangle}.\eqref{ex:triangle:borderline}, which is not.
  \item\label{item:weak-cond-not-open} Naive convex cocompactness (Definition~\ref{def:cc-naive}) is not an open condition (even if we require the cocompact convex subset $\C \subset \Omega$ to have nonempty interior, which we can always do by Lemma~\ref{lem:naive-cc-nonempty-int}).
  Indeed, Example~\ref{ex:triangle}.\eqref{ex:triangle:borderline}, which satisfies the condition, is a limit of both~\ref{ex:triangle}.\eqref{ex:triangle:unipot} and~\ref{ex:triangle}.\eqref{ex:triangle:fail}, which do not.
\end{enumerate}
\end{remarks}

Here is a slightly more complicated example of a discrete subgroup of $\PGL(V)$ which is naively convex cocompact but not convex cocompact in~$\PP(V)$.

\begin{example}\label{ex:cone-H2}
Let $\Gamma_1$ be a convex cocompact subgroup of $\SO(2,1) \subset \GL(\RR^3)$, let $\Gamma_2\simeq\ZZ$ be a discrete subgroup of $\GL(\RR^1)\simeq\RR^*$ acting on~$\RR^1$ by scaling, and let $\Gamma=\Gamma_1\times\Gamma_2\subset\GL(V)$, where $V:=\RR^3\oplus\RR^1$.
Any $\Gamma$-invariant properly convex open subset $\Omega$ of $\PP(\RR^3 \oplus \RR^1)$ is a cone with base some $\Gamma_1$-invariant properly convex open subset $\Omega_1$ of $\PP(\RR^3) \subset \PP(V)$ and tip $z := \PP(\RR^1) \subset \PP(V)$.
The full orbital limit set $\Lambdao_{\Omega}(\Gamma)$ contains both $\{z\}$ and the full base~$\Omega_1$, hence $\Ccore_{\Omega}(\Gamma)$ is equal to~$\Omega$.
Therefore $\Gamma$ acts convex cocompactly on $\Omega$ if and only if $\Gamma_1$ divides~$\Omega_1$, if and only if $\Gamma_1$ is cocompact (not just convex cocompact) in $\SO(2,1)$.
On the other hand, $\Gamma$ is always naively convex cocompact in $\PP(V)$: it acts cocompactly on the closed convex subcone $\C \subset \Omega$ with tip $z$ and base the convex hull in $\HH^2$ of the limit set of~$\Gamma_1$.
\end{example}

In further work \cite{dgk-bad-ex}, we shall describe examples of discrete subgroups of $\PGL(V)$ which are naively convex cocompact but not convex cocompact in $\PP(V)$, and which act \emph{irreducibly} on $\PP(V)$.
This includes free discrete subgroups of $\PGL(\RR^4)$ containing Example~\ref{ex:triangle}.\eqref{ex:triangle:borderline} as a free factor.

\subsection{Finite-index subgroups} \label{subsec:finite-index}

We observe that the notions of convex cocompactness and naive convex cocompactness in $\PP(V)$ behave well with respect to finite-index subgroups.

\begin{lemma} \label{lem:finite-index}
Let $\Gamma$ be a discrete subgroup of $\PGL(V)$ preserving a properly convex open subset $\Omega$ of $\PP(V)$, and let $\Gamma'$ be a finite-index subgroup of~$\Gamma$.
Then
\begin{itemize}
  \item $\Ccore_{\Omega}(\Gamma)=\Ccore_{\Omega}(\Gamma')$; in particular, $\Gamma'$ is convex cocompact in $\PP(V)$ if and only if $\Gamma$~is;
  \item if $\Gamma'$ acts cocompactly on some closed convex subset $\C'$ of~$\Omega$, then $\Gamma$ acts cocompactly on some closed convex subset $\C$ of~$\Omega$; in particular, $\Gamma'$ is naively convex cocompact in $\PP(V)$ if and only if $\Gamma$ is.
\end{itemize}
\end{lemma}

\begin{proof}
Write $\Gamma$ as the disjoint union of cosets $\Gamma'\gamma_1,\dots,\Gamma'\gamma_m$ where $\gamma_i\in\Gamma$.
Since any orbit $\Gamma\cdot x$, for $x\in\Omega$, is a union of $m$ orbits $\Gamma'\cdot(\gamma_i \cdot x)$, we have $\Lambdao_{\Omega}(\Gamma) = \Lambdao_{\Omega}(\Gamma')$, hence $\Ccore_{\Omega}(\Gamma)=\Ccore_{\Omega}(\Gamma')$.

Suppose $\Gamma'$ acts cocompactly on some closed convex subset $\C'$ of~$\Omega$, and let $\mathcal{D}' \subset \C'$ be a compact fundamental domain for this action.
There exists a closed uniform neighborhood $\C'_{\mathsf{unif}}$ of $\C'$ in $(\Omega,d_{\Omega})$ such that $\bigcup_{i=1}^m \gamma_i\cdot\mathcal{D}' \subset \C'_{\mathsf{unif}}$.
The group $\Gamma'$ acts cocompactly on $\C'_{\mathsf{unif}}$ (Lemma~\ref{lem:naive-cc-nonempty-int}).
We have $\Gamma\cdot\C' = \bigcup_{i=1}^m \Gamma'\gamma_i\cdot\mathcal{D}' \subset \C'_{\mathsf{unif}}$ since $\C'_{\mathsf{unif}}$ is $\Gamma'$-invariant.
Let $\C \subset \C'_{\mathsf{unif}}$ be the convex hull of $\Gamma\cdot\C'$ in~$\Omega$.
Then $\C$ is $\Gamma$-invariant and the action of $\Gamma'$ (hence of $\Gamma$) on~$\C$ is cocompact.
\end{proof}

\subsection{Conical limit points}

We use the following terminology.

\begin{definition} \label{def:conical}
Let $\Omega$ be a nonempty properly convex open subset of $\PP(V)$ and let $z\in\partial\Omega$.
\begin{itemize}
  \item A sequence $(y_m)_{m\in\NN}\in\Omega^{\NN}$ \emph{converges conically} to~$z$ if $y_m\to z$ and the supremum over $m\in\NN$ of the distances $d_{\Omega}(y_m, [y,z))$ is finite for some (hence any) $y\in\Omega$.
\end{itemize}
Suppose in addition that $\Omega$ is preserved by an infinite discrete subgroup $\Gamma$ of $\PGL(V)$.
\begin{itemize}
  \item The point $z$ is a \emph{conical limit point} of the $\Gamma$-orbit of some point $y\in\Omega$ if there exists a sequence $(\gamma_m)\in\Gamma^{\NN}$ such that $(\gamma_m\cdot y)_{m\in\NN}$ converges conically to~$z$.
  In this case we say that $z$ is a \emph{conical limit point of $\Gamma$ in $\partial\Omega$}.
  \item The \emph{conical limit set of $\Gamma$ in~$\partial\Omega$} is the set $\Lambdacon_{\Omega}(\Gamma)$ of conical limit points of $\Gamma$ in~$\partial\Omega$.
\end{itemize}
\end{definition}

Here $[y,z)$ denotes the projective ray of~$\Omega$ starting from~$y$ with endpoint~$z$.
By definition, the conical limit set $\Lambdacon_{\Omega}(\Gamma)$ is contained in the full orbital limit set $\Lambdao_{\Omega}(\Gamma)$.

The following observation will have several important consequences.

\begin{lemma} \label{lem:Lambda-con}
Let $\Gamma$ be an infinite discrete subgroup of $\PGL(V)$ and $\Omega$ a nonempty $\Gamma$-invariant properly convex open subset of $\PP(V)$.
For $z\in\partial\Omega$ and a sequence $(\gamma_m)\in\Gamma^{\NN}$ of pairwise distinct elements, if there exists $y\in\Omega$ such that $d_{\Omega}(\gamma_m\cdot y, [y,z))$ is bounded, then there exists $y'\in\Omega$ in the closure of $\bigcup_{\gamma\in\Gamma} \gamma\cdot [y,z)$ in~$\Omega$ such that a subsequence of $(\gamma_m\cdot y')_{m\in\NN}$ converges conically to~$z$; in particular, $z\in\Lambdacon_{\Omega}(\Gamma)$.
\end{lemma}

\begin{proof}
Suppose there exist a sequence $(\gamma_m)\in\Gamma^{\NN}$ of pairwise distinct elements, a point $y\in\Omega$, and a sequence $(y_m)_{m\in\NN}$ of points of $[y,z)$ such that $d_{\Omega}(\gamma_m\cdot y,y_m)$ is bounded.
Then $d_{\Omega}(y,\gamma_m^{-1}\cdot y_m)$ is bounded, and so, up to passing to a subsequence, we may assume that $\gamma_m^{-1}\cdot y_m\to y'$ for some $y'\in\Omega$.
We have $y_m\to z$ and $d_{\Omega}(y_m,\gamma_m\cdot y')\to 0$, hence $\gamma_m\cdot y'\to z$ by Corollary~\ref{cor:compare-Hilb-Eucl}.
Moreover, $d_{\Omega}(\gamma_m\cdot y', [y',z))$ is bounded.
Indeed, $d_{\Omega}(\gamma_m\cdot y',\gamma_m\cdot y) = d_{\Omega}(y',y)$ and $d_{\Omega}(\gamma_m\cdot y, [y,z))$ is bounded, hence $d_{\Omega}(\gamma_m\cdot y', [y,z))$ is bounded by the triangle inequality.
We conclude using the fact that the rays $[y,z)$ and $[y',z)$ remain at bounded distance for~$d_{\Omega}$.
\end{proof}

\begin{corollary} \label{cor:ideal-bound-naive-cc}
Let $\Gamma$ be an infinite discrete subgroup of $\PGL(V)$ and $\Omega$ a nonempty $\Gamma$-invariant properly convex open subset of $\PP(V)$.
Suppose $\Gamma$ acts cocompactly on some closed convex subset $\C$ of~$\Omega$.
Then
\begin{enumerate}
  \item\label{item:ideal-bound-naive-cc} the ideal boundary $\partiali\C$ is contained in $\Lambdacon_{\Omega}(\Gamma)$; more precisely, any point of $\partiali\C$ is a conical limit point of the $\Gamma$-orbit of some point of~$\C$;
  \item\label{item:Lambda-con-naive-cc} $\Lambdacon_{\Omega}(\Gamma) = \Lambdao_{\Omega}(\Gamma)$;
  \item\label{item:ideal-bound-cc} if $\C$ contains $\Ccore_{\Omega}(\Gamma)$, then $\partiali\C = \partiali\Ccore_{\Omega}(\Gamma) = \Lambdao_{\Omega}(\Gamma)$; in that case, the set $\Lambdao_{\Omega}(\Gamma)$ is closed in $\PP(V)$, the set $\Ccore_\Omega(\Gamma)$ is closed in~$\Omega$, and the action of $\Gamma$ on~$\Omega$ is convex cocompact (Definition~\ref{def:cc-general}).
\end{enumerate}
\end{corollary}

As a special case, Corollary~\ref{cor:ideal-bound-naive-cc}.\eqref{item:ideal-bound-cc} with $\C=\Omega$ shows that if $\Gamma$ divides (\ie acts cocompactly on) $\Omega$, then $\Lambdao_{\Omega}(\Gamma) = \partial\Omega$ and $\Ccore_{\Omega}(\Gamma) = \Omega$ (this also follows from \cite[Prop.\,3]{vey70}) and the action of $\Gamma$ on~$\Omega$ is convex cocompact in the sense of Definition~\ref{def:cc-general}.

\begin{proof}
\eqref{item:ideal-bound-naive-cc} Since $\C$ is convex and the action of $\Gamma$ on~$\C$ is cocompact, for $z\in\partiali\C$ and $y\in\C$, all points in the ray $[y,z)$ lie at uniformly bounded distance from $\Gamma\cdot y$.
We conclude using Lemma~\ref{lem:Lambda-con}.

\eqref{item:Lambda-con-naive-cc} Let $z\in\Lambdao_{\Omega}(\Gamma)$: it is a limit of points in $\Gamma\cdot y$ for some $y\in\Omega$.
Let $\C_{\mathsf{unif}}$ be a closed uniform neighborhood of~$\C$ in $(\Omega,d_{\Omega})$ containing~$y$.
The group $\Gamma$ still acts cocompactly on~$\C_{\mathsf{unif}}$, and $z\in\partiali\C_{\mathsf{unif}}$ (Lemma~\ref{lem:naive-cc-nonempty-int}), hence $z\in\Lambdacon_{\Omega}(\Gamma)$ by~\eqref{item:ideal-bound-naive-cc}.

\eqref{item:ideal-bound-cc} If $\C$ contains $\Ccore_{\Omega}(\Gamma)$, then $\partiali\C\supset\partiali\Ccore_{\Omega}(\Gamma)\supset\Lambdao_{\Omega}(\Gamma)$, and so $\partiali\C = \partiali\Ccore_\Omega(\Gamma) = \Lambdao_{\Omega}(\Gamma)$ by~\eqref{item:ideal-bound-naive-cc}.
In particular, $\Lambdao_{\Omega}(\Gamma)$ is closed in $\PP(V)$ by Remark~\ref{rem:closed-ideal-bound-C-Omega}, and so its convex hull $\Ccore_\Omega(\Gamma)$ is closed in $\Omega$, and compact modulo $\Gamma$ because $\C$ is.
\end{proof}

\subsection{Open faces of $\partial\Omega$}

We shall use the following terminology.

\begin{definition} \label{def:face}
For any properly convex open subset $\Omega$ of $\PP(V)$ and any $z\in\partial\Omega$, the \emph{open face} $F_{\scriptscriptstyle\partial\Omega}(z)$ of $\partial\Omega$ at~$z$ is the union of $\{z\}$ and of all open segments of $\partial\Omega$ containing~$z$.
\end{definition}

In other words, $F_{\scriptscriptstyle\partial\Omega}(z)$ is the largest convex subset of $\partial\Omega$ containing~$z$ which is relatively open, in the sense that it is open in the projective subspace $\PP(W)$ that is spans.
In particular, we can consider the Hilbert metric $d_{F_{\partial\Omega}(z)}$ on $F_{\scriptscriptstyle\partial\Omega}(z)$ seen as a properly convex open subset of $\PP(W)$ (if $F_{\scriptscriptstyle\partial\Omega}(z) = \{z\}$, take $d_{F_{\partial\Omega}(z)}:=0$).
The set $\partial\Omega$ is the disjoint union of its open faces.

Here is a consequence of Lemma~\ref{lem:Lambda-con}.

\begin{corollary} \label{cor:Lambda-con-faces}
Let $\Gamma$ be a discrete subgroup of $\PGL(V)$ and $\Omega$ a nonempty $\Gamma$-invariant properly convex open subset of $\PP(V)$.
Then the conical limit set $\Lambdacon_{\Omega}(\Gamma)$ is a union of open faces of $\partial\Omega$.
\end{corollary}

\begin{proof}
Suppose $z\in\Lambdacon_{\Omega}(\Gamma)$: there exist $y\in\Omega$ and $(\gamma_m)\in\Gamma^{\NN}$ such that $(\gamma_m\cdot y)_{m\in\NN}$ converges conically to~$z$.
For any $z'\in F_{\scriptscriptstyle\partial\Omega}(z)$, the rays $[y,z)$ and $[y,z')$ remain at bounded distance for~$d_{\Omega}$, hence $d_{\Omega}(\gamma_m\cdot y, [y,z'))$ is bounded.
By Lemma~\ref{lem:Lambda-con}, there exists $y'\in\Omega$ such that some subsequence of $(\gamma_m\cdot y')_{m\in\NN}$ converges conically to~$z'$, and so $z'\in\Lambdacon_{\Omega}(\Gamma)$.
\end{proof}

Corollaries \ref{cor:ideal-bound-naive-cc}.\eqref{item:Lambda-con-naive-cc} and~\ref{cor:Lambda-con-faces} immediately yield the following.

\begin{corollary} \label{cor:naive-cc-Lambdaorb-full}
Let $\Gamma$ be a discrete subgroup of $\PGL(V)$ and $\Omega$ a $\Gamma$-invariant properly convex open subset of $\PP(V)$.
If $\Gamma$ acts cocompactly on some nonempty closed convex subset $\C$ of~$\Omega$, then the full orbital limit set $\Lambdao_{\Omega}(\Gamma)$ is a union of open faces of $\partial\Omega$.
\end{corollary}

The following lemma will be used in the proofs of Lemma~\ref{lem:Ccore-min} and Proposition~\ref{prop:cc-Omega-subset-Omega'}.\eqref{item:cc-small->big} below.
We endow any open face $F$ of $\partial\Omega$ with its Hilbert metric $d_F$.

\begin{lemma} \label{lem:unif-neighb-face}
Let $\Omega$ be a nonempty properly convex open subset of $\PP(V)$ and let $R>0$.
\begin{enumerate}
  \item \label{item:distance-goes-down} Let $(x_m)_{m\in\NN}$ and $(x'_m)_{m\in\NN}$ be sequences of points of~$\Omega$ converging respectively to points $z$ and~$z'$ of $\partial\Omega$.
  If $d_{\Omega}(x_m,x'_m)\leq R$ for all $m\in\NN$, then $z$ and~$z'$ belong to the same open face $F$ of $\partial\Omega$ and $d_F(z,z') \leq R$.
  \item \label{item:R-nbhd-faces} Let $\C$ be a nonempty closed convex subset of~$\Omega$ and let $\C_R$ be the closed uniform $R$-neighborhood of $\C$ in $(\Omega,d_{\Omega})$.
  Then for any open face $F$ of $\partial\Omega$, the set $\partiali \C_R \cap F$ is equal to the closed uniform $R$-neighborhood of $\partiali \C \cap F$ in $(F,d_F)$. 
\end{enumerate}
\end{lemma}

\begin{proof}
\eqref{item:distance-goes-down} We may assume $z\neq z'$. 
For any $m\in\NN$, let $a_m, b_m \in \partial \Omega$ be such that $a_m, x_m, x'_m, b_m$ are aligned in this order.
Up to passing to a subsequence, we may assume that $(a_m)_{m\in\NN}, (b_m)_{m\in\NN}$ converge respectively to $a,b \in \partial\Omega$, with $a, z, z', b$ aligned in this order. 
By continuity of the cross-ratio, $\lim_m \cro{a_m}{x_m}{x'_m}{b_m} = \cro{a}{z}{z'}{b}$.
Using $d_{\Omega}(x_m,x'_m)\leq R$ and the definition \eqref{eqn:d-Omega} of the Hilbert metric, we deduce $a\neq z$ and $z'\neq b$ and $\lim_m d_{\Omega}(x_m, x'_m) = d_{(a,b)}(z,z')$, where $d_{(a,b)}$ is the Hilbert metric on the interval $(a,b)$.
The interval $(a,b)$ (hence also $z$ and~$z'$) is contained in some open face $F$ of $\partial\Omega$.
We have $d_F(z,z') \leq d_{(a,b)}(z,z')$ (Remark~\ref{rem:Hilb-metric-include}), with equality if and only if $a$ and~$b$ both lie in the boundary of~$F$.
Thus $d_F(z,z') \leq \lim_m d_{\Omega}(x_m, x'_m) \leq R$.

\eqref{item:R-nbhd-faces} By~\eqref{item:distance-goes-down}, the set $\partiali \C_R \cap F$ is contained in the closed uniform $R$-neighborhood of $\partiali \C \cap F$ in $(F,d_F)$.
Let us prove the reverse inclusion.
If $F$ is a singleton, there is nothing to prove.
If not, consider $z \in \partiali \C \cap F$ and  $z' \in F$ with $0<d_F(z,z') \leq R$.
In order to check that $z' \in \partiali\C_R$, it is enough to choose a point $x \in \C$ and check that the ray $[x,z')$ is contained in~$\C_R$.
Let $\Omega'$ (\resp $F'$) be the intersection of $\Omega$ (\resp $F$) with the two-dimensional projective subspace of $\PP(V)$ spanned by $x, z, z'$.
Observe that $F' = (a,b)$ is an interval and that $d_{F'}$ is equal to $d_F$ on~$F'$ (see Figure~\ref{fig:UnifNeighbFace}).
For any $y' \in [x,z')$, let $y \in [x,z) \subset \C$ be such that the line between $y, y'$ does not intersect~$F'$ (\eg choose $y$ to be the intersection with $(x,z)$ of the line through $y'$ parallel to the direction of $F'$ in some affine chart of $\PP(W)$ containing $\Omega'$). 
If $T \subset \Omega'$ denotes the open triangle spanned by $x, a, b$, then $d_{\Omega}(y,y') \leq d_T(y,y') = d_{F'}(z,z') \leq R$, hence $y' \in \C_R$.
\end{proof}

\begin{figure}
\labellist
\small\hair 2pt
\pinlabel {$x \in \C$} [u] at 50 91
\pinlabel {$y'$} [u] at 261 111
\pinlabel {$y\in \C$} [u] at 272 67
\pinlabel {$T\subset \Omega$} [u] at 320 45
\pinlabel {$F$} [u] at 406 50
\pinlabel {$z$} [u] at 390 72
\pinlabel {$z'$} [u] at 392 120
\pinlabel {$F'$} [u] at 372 138
\pinlabel {$a$} [u] at 381 188
\pinlabel {$b$} [u] at 381 13
\endlabellist
\includegraphics[width = 10.0cm]{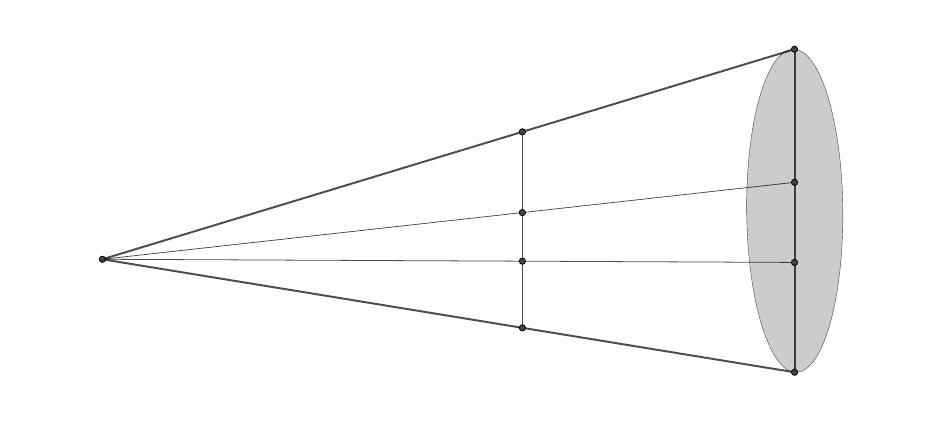}
\vspace{-0.3cm}
\caption{Illustration of the proof of Lemma~\ref{lem:unif-neighb-face}.\eqref{item:R-nbhd-faces}}
\label{fig:UnifNeighbFace}
\end{figure}

\begin{remark}
The way that Hilbert distances behave as points approach an open face $F$ of $\partial\Omega$ can be quite subtle.
Let $(x_m)_{m\in\NN}$ be a sequence of points in~$\Omega$ converging to some $z \in F$, and consider the sequence of closed $R$-balls $\overline{B}_{d_{\Omega}}(x_m,R)$.
By Lemma~\ref{lem:unif-neighb-face}.\eqref{item:distance-goes-down}, their limit in the Hausdorff topology (if it exists) is some closed subset of~$F$ contained in the closed $R$-ball $\overline{B}_{d_F}(z,R)$.
However, it is sometimes the case that the containment is strict: see Appendix~\ref{app:hilbert}.
In the case that $x_m \to z$ conically, it can be shown that the limit contains some ball around~$z$, but the radius may be smaller than~$R$: we give precise bounds in Lemma~\ref{lem:loss}.
In general, the limit may fail to contain a neighborhood of $z$ in~$F$ (see Example~\ref{ex:twocones}).
\end{remark}

\subsection{Minimality of the convex core for convex cocompact actions} \label{subsec:Ccore-min}

Recall that $\Ccore_\Omega(\Gamma)$ denotes the convex core of $\Omega$ for~$\Gamma$, \ie the convex hull of $\Lambdao_\Omega(\Gamma)$ in~$\Omega$.

\begin{lemma} \label{lem:Ccore-min}
Let $\Gamma$ be a discrete subgroup of $\PGL(V)$ acting convex cocompactly (Definition~\ref{def:cc-general}) on some nonempty properly convex open subset $\Omega$ of $\PP(V)$.
Then any nonempty $\Gamma$-invariant closed convex subset $\C$ of~$\Omega$ contains~$\Ccore_\Omega(\Gamma)$.
In particular, the nonempty $\Gamma$-invariant closed convex subsets of~$\Omega$ on which the action of $\Gamma$ is cocompact are exactly those nested between $\Ccore_\Omega(\Gamma)$ and some uniform neighborhood of $\Ccore_\Omega(\Gamma)$ in $(\Omega,d_{\Omega})$.
\end{lemma}

\begin{proof}
Consider a nonempty closed convex $\Gamma$-invariant subset $\C$ of~$\Omega$.

Let us prove that $\C$ contains $\Ccore_\Omega(\Gamma)$.
By definition of $\Ccore_\Omega(\Gamma)$, it is sufficient to prove that $\partiali \C$ contains $\Lambdao_{\Omega}(\Gamma)$.
For this, it is sufficient to prove that $\Lambdao_{\Omega}(\Gamma) \cap F \subset \partiali \C \cap F$ for any open face $F$ of $\partial\Omega$.
Since the action of $\Gamma$ on $\Ccore_{\Omega}(\Gamma)$ is cocompact, there exists $R > 0$ such that $\Ccore_\Omega(\Gamma)$ is contained in the closed uniform $R$-neighborhood $\C_R$ of $\C$ in $(\Omega,d_{\Omega})$.
We have $\Lambdao_{\Omega}(\Gamma) \subset \partiali \Ccore_{\Omega}(\Gamma) \subset \partiali \C_R$.
Let $F$ be an open face of $\partial \Omega$.
Then $ \Lambdao_\Omega(\Gamma) \cap F \subset \partiali \C_R \cap F$. 
By Corollary~\ref{cor:naive-cc-Lambdaorb-full}, the set $\Lambdao_\Omega(\Gamma) \cap F$ is either empty (in which case there is nothing to prove) or equal to~$F$.
Suppose $\Lambdao_\Omega(\Gamma) \cap F = F$ and hence $\partiali \C_R \cap F = F$.
By Lemma~\ref{lem:unif-neighb-face}.\eqref{item:R-nbhd-faces}, the set $\partiali \C_R \cap F$ is the closed uniform $R$-neighborhood of $\partiali \C$ in $(F,d_F)$.
Since the Hilbert metric $d_F$ is proper on $F$, it follows that $\partiali \C \cap F = F$.
This shows that $\partiali \C \supset \Lambdao_\Omega(\Gamma)$ and so $\C \supset \Ccore_\Omega(\Gamma)$.

Moreover, the action of $\Gamma$ on~$\C$ is cocompact if and only if $\C$ is contained in a uniform neighborhood of $\Ccore_\Omega(\Gamma)$ in $(\Omega,d_{\Omega})$.
\end{proof}

\begin{corollary}
Let $\Gamma$ be a discrete subgroup of $\PGL(V)$ dividing (\ie acting cocompactly on) some nonempty properly convex open subset $\Omega$ of $\PP(V)$.
Then any $\Gamma$-invariant properly convex open subset of $\PP(V)$ intersecting $\Omega$ nontrivially is equal to~$\Omega$.
\end{corollary}

\begin{proof}
Let $\Omega'$ be a $\Gamma$-invariant properly convex open subset of $\PP(V)$ intersecting~$\Omega$ nontrivially.
By Corollary~\ref{cor:ideal-bound-naive-cc}.\eqref{item:ideal-bound-cc}, we have $\partial\Omega = \Lambdao_{\Omega}(\Gamma)$ and $\Ccore_{\Omega}(\Gamma) = \Omega$.
By Lemma~\ref{lem:Ccore-min}, we have $\Omega\cap\Omega' = \Omega$, and so $\Omega \subset \Omega'$.
By Remark~\ref{rem:closed-ideal-bound-C-Omega} applied to $(\Omega',\Omega)$ instead of $(\Omega,\C)$, we have $\partial\Omega = \partial\Omega' \cap \overline{\Omega}$, hence $\Omega = \Omega'$.
\end{proof}

\subsection{On which convex sets $\Omega$ is the action convex cocompact?}

Here is a consequence of Corollary~\ref{cor:ideal-bound-naive-cc}.\eqref{item:ideal-bound-cc}.

\begin{proposition} \label{prop:cc-Omega-subset-Omega'}
Let $\Gamma$ be a discrete subgroup of $\PGL(V)$ preserving two nonempty properly convex open subsets $\Omega \subset \Omega'$ of $\PP(V)$.
\begin{enumerate}
  \item\label{item:cc-small->big} Suppose the action of $\Gamma$ on~$\Omega$ is convex cocompact.
  Then the action of $\Gamma$ on~$\Omega'$ is convex cocompact, $\Lambdao_{\Omega}(\Gamma) = \Lambdao_{\Omega'}(\Gamma)$, and $\Ccore_{\Omega}(\Gamma) = \Ccore_{\Omega'}(\Gamma)$.
  \item\label{item:cc-big->small} Suppose the action of $\Gamma$ on~$\Omega'$ is convex cocompact.
  Then the action of $\Gamma$ on~$\Omega$ is convex cocompact if and only if $\Ccore_{\Omega}(\Gamma) = \Ccore_{\Omega'}(\Gamma)$, if and only if $\Omega$ contains $\Ccore_{\Omega'}(\Gamma)$.
\end{enumerate}
\end{proposition}

\begin{proof}
\eqref{item:cc-small->big} We first check that the inclusion $\Lambdao_{\Omega}(\Gamma) \subset \Lambdao_{\Omega'}(\Gamma)$ is an equality.
For this it is sufficient to check that for any open face $F'$ of $\partial\Omega'$ meeting $\Lambdao_{\Omega'}(\Gamma)$, we have $F' \cap\nolinebreak \Lambdao_{\Omega}(\Gamma) = F'$.
The set $F' \cap \Lambdao_{\Omega}(\Gamma)$ is nonempty by Lemma~\ref{lem:unif-neighb-face}.\eqref{item:distance-goes-down}.
Since $\Gamma$ acts cocompactly on $\Ccore_{\Omega}(\Gamma)$, there exists $\varepsilon>0$ such that the uniform $\varepsilon$-neighborhood $\C_{\varepsilon}$ of $\Ccore_{\Omega}(\Gamma)$ in $(\Omega',d_{\Omega'})$ is contained in~$\Omega$.
The uniform $\varepsilon$-neighborhood of $F' \cap \Lambdao_{\Omega}(\Gamma)$ in $(F',d_{F'})$ is equal to $\partiali\C_{\varepsilon}$ by Lemma~\ref{lem:unif-neighb-face}.\eqref{item:R-nbhd-faces}, hence to $F' \cap \Lambdao_{\Omega}(\Gamma)$ itself by Corollary~\ref{cor:ideal-bound-naive-cc}.\eqref{item:ideal-bound-cc}, and so $F' \cap \Lambdao_{\Omega}(\Gamma) = F'$.
This shows that $\Lambdao_{\Omega}(\Gamma) = \Lambdao_{\Omega'}(\Gamma)$.
We deduce that $\Ccore_{\Omega'}(\Gamma)$ is the closure of $\Ccore_{\Omega}(\Gamma)$ in~$\Omega'$.
Since the action of $\Gamma$ on $\Ccore_{\Omega}(\Gamma)$ is cocompact by assumption, the set $\Ccore_{\Omega}(\Gamma)$ must already be closed in $\Omega'$ (Remark~\ref{rem:closed-ideal-bound-C-Omega}), hence $\Ccore_{\Omega'}(\Gamma) = \Ccore_{\Omega}(\Gamma)$ and the action of $\Gamma$ on $\Omega'$ is convex cocompact.

\eqref{item:cc-big->small} If the action of $\Gamma$ on~$\Omega$ is convex cocompact, then $\Ccore_{\Omega}(\Gamma) = \Ccore_{\Omega'}(\Gamma)$ by~\eqref{item:cc-small->big}, and so $\Omega$ contains $\Ccore_{\Omega'}(\Gamma)$.
Conversely, if $\Omega$ contains $\Ccore_{\Omega'}(\Gamma)$, then $\Ccore_{\Omega'}(\Gamma)$ is a closed convex subset of $\Omega$ containing $\Ccore_{\Omega}(\Gamma)$ on which $\Gamma$ acts cocompactly; by Corollary~\ref{cor:ideal-bound-naive-cc}.\eqref{item:ideal-bound-cc}, the action of $\Gamma$ on~$\Omega$ is convex cocompact.
\end{proof}

In the irreducible case, we obtain the following description of the convex sets on which the action is convex cocompact.
Using Fact~\ref{fact:benoist-irred-Gamma}.\eqref{item:benoist-irred-2}, it shows that if $\Gamma$ acts strongly irreducibly on $\PP(V)$ and is convex cocompact in $\PP(V)$, then the set $\Lambdao_\Omega(\Gamma)$ is the same for all properly convex open subsets $\Omega$ of $\PP(V)$ on which $\Gamma$ acts convex cocompactly.

\begin{corollary} \label{cor:cc-sets-irred}
Let $\Gamma$ be an infinite discrete subgroup of $\PGL(V)$ acting convex cocompactly on a nonempty properly convex open subset $\Omega$ of $\PP(V)$.
Suppose that $\Gamma$ contains a proximal element, so that the maximal convex open set $\Omega_{\max} \supset \Omega$ of Proposition~\ref{prop:max-inv-conv} is well defined, and suppose that $\Omega_{\max}$ is properly convex (this is always the case if $\Gamma$ acts irreducibly on $\PP(V)$, see Fact~\ref{fact:benoist-irred-Gamma}).
Then
\begin{enumerate}
  \item\label{item:Omega'-cc-irred} the properly convex open subsets $\Omega'$ of $\Omega_{\max}$ on which $\Gamma$ acts convex cocompactly are exactly those containing $\Ccore_{\Omega_{\max}}(\Gamma)$; they satisfy $\Lambdao_{\Omega'}(\Gamma) = \Lambdao_{\Omega_{\max}}(\Gamma)$ and $\Ccore_{\Omega'}(\Gamma) = \Ccore_{\Omega_{\max}}(\Gamma)$;
  \item\label{item:Lambda-orb-cc-irred} if $\Gamma$ acts irreducibly on $\PP(V)$, then
$$\Lambdao_{\Omega_{\max}}(\Gamma) = \partial \Omega_{\min} \cap \partial \Omega_{\max}$$
and
$$\Ccore_{\Omega_{\max}}(\Gamma) = \overline{\Omega_{\min}} \cap \Omega_{\max},$$
where $\Omega_{\min}\subset\Omega$ is the minimal nonempty $\Gamma$-invariant properly convex set given by Fact~\ref{fact:benoist-irred-Gamma}.
Thus, for any properly convex open subset $\Omega'$ of $\Omega_{\max}$ on which $\Gamma$ acts convex cocompactly, the convex core $\Ccore_{\Omega'}(\Gamma)$ is the convex hull in~$\Omega'$ of the proximal limit set $\Lambda_{\Gamma}$.
\end{enumerate}
\end{corollary}

\begin{proof}
\eqref{item:Omega'-cc-irred} This is an immediate consequence of Proposition~\ref{prop:cc-Omega-subset-Omega'}.

\eqref{item:Lambda-orb-cc-irred} Suppose $\Gamma$ acts irreducibly on $\PP(V)$.
By Remark~\ref{rem:closed-ideal-bound-C-Omega}, the set $\C := \Ccore_{\Omega_{\max}}(\Gamma)$ is closed in $\Omega_{\max}$.
The interior $\Int(\C)$ is nonempty since $\Gamma$ acts irreducibly on $\PP(V)$, and $\Int(\C)$ contains $\Omega_{\min}$.
In fact $\Int(\C) = \Omega_{\min}$: indeed, if $\Omega_{\min}$ were strictly smaller than $\Int(\C)$, then the closure of $\Omega_{\min}$ in $\C$ would be a nonempty strict closed subset of $\C$ on which $\Gamma$ acts properly discontinuously and cocompactly, contradicting Lemma~\ref{lem:Ccore-min}.
Thus $\C = \overline{\Omega_{\min}} \cap \Omega_{\max}$, which is the convex hull of the proximal limit set $\Lambda_{\Gamma}$ in~$\Omega_{\max}$.
Moreover, by Corollary~\ref{cor:ideal-bound-naive-cc}.\eqref{item:ideal-bound-cc} we have $\Lambdao_{\Omega_{\max}}(\Gamma) = \partiali\C = \overline{\Omega_{\min}} \cap \partial \Omega_{\max}$.
\end{proof}

For an arbitrary discrete subgroup $\Gamma$ of $\PGL(V)$ acting irreducibly on $\PP(V)$ and preserving a nonempty properly convex open subset $\Omega$ of $\PP(V)$, the convex hull $\Ccore_{\Omega}(\Gamma)$ of the full orbital limit set $\Lambdao_{\Omega}(\Gamma)$ may be larger than the convex hull $\overline{\Omega_{\min}} \cap \Omega_{\max}$ of the proximal limit set $\Lambda_{\Gamma}$ in~$\Omega$.
This happens for instance if $\Gamma$ is naively convex cocompact in $\PP(V)$ (Definition~\ref{def:cc-naive}) but not convex cocompact in $\PP(V)$.
Examples of such behavior will be given in the forthcoming paper \cite{dgk-bad-ex}.

\subsection{Convex cocompactness and conicality}

By Corollary~\ref{cor:ideal-bound-naive-cc}.\eqref{item:Lambda-con-naive-cc}, if $\Gamma$ acts convex cocompactly on~$\Omega$, then the full orbital limit set $\Lambdao_{\Omega}(\Gamma)$ consists entirely of conical limit points. 
We now investigate the converse; this is not needed anywhere in the paper.
We also refer the reader to recent work of Weisman~\cite{Weisman20}, in which conical convergence is studied further in relation to an expansion condition at faces of the full orbital limit set.

When $\Omega$ is strictly convex with boundary of class $C^1$, the property that $\Lambdao_{\Omega}(\Gamma)$ consist entirely of conical limit points implies that the action of $\Gamma$ on~$\Omega$ is convex cocompact by \cite[Cor.\,8.6]{cm14}.
In general, the following holds, as was pointed out to us by Pierre-Louis Blayac.

\begin{lemma}
Let $\Gamma$ be an infinite discrete subgroup of $\PGL(V)$ and $\Omega$ a nonempty $\Gamma$-invariant properly convex open subset of $\PP(V)$.
\begin{enumerate}
  \item\label{item:conical} If $\C$ is a nonempty $\Gamma$-invariant closed convex subset of~$\Omega$ whose ideal boundary $\partiali\C$ consists entirely of conical limit points of $\Gamma$ in $\partial\Omega$, then the action of $\Gamma$ on~$\C$ is cocompact.
  \item\label{item:conical-bis} In particular, if $\Ccore_\Omega(\Gamma)$ is closed in $\Omega$ and $\partiali \Ccore_\Omega(\Gamma)$ consists entirely of conical limit points of $\Gamma$ in $\partial\Omega$, then the action of $\Gamma$ on~$\Omega$ is convex cocompact.
\end{enumerate}
\end{lemma}

\begin{proof}
\eqref{item:conical} Fix a point $x_0\in\C$ and, by analogy with Dirichlet domains, let
$$\mathcal{D} := \{ x\in\C ~|~ d_{\Omega}(x,x_0) \leq d_{\Omega}(x,\gamma\cdot x_0) \quad\forall\gamma\in\Gamma\} .$$
Note that $\Gamma\cdot\mathcal{D} = \C$.
Therefore, it is sufficient to prove that $\mathcal{D}$ is compact.
Suppose by contradiction that this is not the case: there exists a sequence $(x_m)\in\mathcal{D}^{\NN}$ such that $d_{\Omega}(x_m,x_0)\to +\infty$.
Up to passing to a subsequence, we may assume that $(x_m)_{m\in\NN}$ converges to some point $z\in\partiali\C$.
By assumption $z$ is a conical limit point of $\Gamma$ in $\partial\Omega$: there exist $R>0$, and a sequence $(\gamma_m)\in\Gamma^{\NN}$ of pairwise distinct elements such that for any $m\in\NN$ we can find a point $y_m$ on the ray $[x_0,z)$ with $d_{\Omega}(y_m,\gamma_m\cdot x_0)\leq R$.
Then $y_m\to z$, and so we can find $m$ such that $d_{\Omega}(x_0,y_m)\geq R+2$.
Let us fix such an~$m$.
Since $x_k\to z$ as $k\to +\infty$, we can find a sequence $(x'_k)_{k\in\NN}$ such that $x'_k\in [x_0,x_k)$ for all~$k$ and $x'_k\to y_m$.
In particular, for large enough~$k$ we have $d_{\Omega}(x'_k,y_m)<1$.
Using the triangle inequality, we obtain
\begin{eqnarray*}
d_{\Omega}(x_k,\gamma_m\cdot x_0) & \leq & d_{\Omega}(x_k,x'_k) + d_{\Omega}(x'_k,y_m) + d_{\Omega}(y_m,\gamma_m\cdot x_0)\\
& = & d_{\Omega}(x_k,x_0) - d_{\Omega}(x_0,x'_k) + d_{\Omega}(x'_k,y_m) + d_{\Omega}(y_m,\gamma_m\cdot x_0)\\
& \leq & d_{\Omega}(x_k,x_0) - d_{\Omega}(x_0,y_m) + 2 d_{\Omega}(x'_k,y_m) + d_{\Omega}(y_m,\gamma_m\cdot x_0)\\
& < & d_{\Omega}(x_k,x_0) - (R+2) + 2 + R \,=\, d_{\Omega}(x_k,x_0).
\end{eqnarray*}
This contradicts the fact that $x_k\in\mathcal{D}$.

\eqref{item:conical-bis} Apply \eqref{item:conical} with $\C=\Ccore_\Omega(\Gamma)$.
\end{proof}

\section{Convex sets with bisaturated boundary and duality} \label{sec:bisat-dual}

In this section we interpret convex cocompactness in terms of convex sets with bisaturated boundary (Definition~\ref{def:boundaries}), and use this to prove that convex cocompactness is stable under duality (property~\ref{item:dual} of Theorem~\ref{thm:properties}).

More precisely, in Sections \ref{subsec:limit-set->bisat} and~\ref{subsec:bisat-boundary} we establish the implications \eqref{item:ccc-limit-set}~$\Rightarrow$~\eqref{item:ccc-bisat} and \eqref{item:ccc-bisat}~$\Rightarrow$~\eqref{item:ccc-limit-set} of Theorem~\ref{thm:main-general};
we prove (Corollaries \ref{coro:ccc-limit-set->bisat} and~\ref{cor:bisat-interior}) that when these conditions hold, sets $\Omega$ as in condition~\eqref{item:ccc-limit-set} and $\C_{\mathsf{bisat}}$ as in condition~\eqref{item:ccc-bisat} may be chosen so that the full orbital limit set $\Lambdao_\Omega(\Gamma)$ of $\Gamma$ in~$\Omega$ coincides with the ideal boundary $\partiali \C_{\mathsf{bisat}}$ of~$\C_{\mathsf{bisat}}$.
In Section~\ref{subsec:dual-def} we study a notion of duality for properly convex sets which are not necessarily open.
In Sections \ref{subsec:proper-cocomp-dual} and~\ref{subsec:proof-dual}, we prove Proposition~\ref{prop:dual-precise}, which is a more precise version of Theorem~\ref{thm:properties}.\ref{item:dual}.
Finally, in Section~\ref{subsec:which-bisat} we describe on which convex sets with bisaturated boundary the action of a given group is properly discontinuous and cocompact, when such sets exist.

\subsection{Bisaturated boundary for neighborhoods of the convex core} \label{subsec:limit-set->bisat}

Let us prove the implication \eqref{item:ccc-limit-set}~$\Rightarrow$~\eqref{item:ccc-bisat} of Theorem~\ref{thm:main-general}.
We start with the following observation.

\begin{lemma} \label{lem:1-neighb-bisat}
Let $\Gamma$ be a discrete subgroup of $\PGL(V)$ and $\Omega$ a nonempty $\Gamma$-invariant properly convex open subset of $\PP(V)$.
Let $\C_0\subset\C_1$ be two nonempty closed properly convex subsets of~$\Omega$ on which $\Gamma$ acts cocompactly.
Suppose that $\partiali\C_1=\partiali\C_0$ and that $\C_1$ contains a neighborhood of $\C_0$ in~$\Omega$. 
Then $\C_1$ has bisaturated boundary.
\end{lemma}

\begin{proof}
By cocompactness, there exists $\varepsilon>0$ such that any point of $\partialn\C_1$ is at $d_{\Omega}$-distance $\geq\varepsilon$ from $\C_0$.
Let $H$ be a supporting hyperplane of~$\C_1$ and suppose for contradiction that $H$ contains both a point $x \in \partialn \C_1$ and a point $z \in \partiali \C_1$.
Since $\partiali \C_1$ is closed in $\PP(V)$ by Remark~\ref{rem:closed-ideal-bound-C-Omega}, we may assume without loss of generality that the interval $[x,z)$ is contained in $\partialn \C_1$.
Consider a sequence $y_m \in [x,z)$ converging to $z$.
Since $\Gamma$ acts cocompactly on~$\C_1$, there is a sequence $(\gamma_m) \in \Gamma^{\NN}$ such that $\gamma_m\cdot y_m$ remains in some compact subset of~$\C_1$.
Up to taking a subsequence, we may assume that $\gamma_m\cdot x\to x_{\infty}$ and $\gamma_m\cdot y_m\to y_{\infty}$ and $\gamma_m\cdot z\to z_{\infty}$ for some $x_{\infty},y_{\infty},z_{\infty}\in\overline{\C_1}$.
We have $x_{\infty} \in \partiali \C_1$ since the action of $\Gamma$ on~$\C_1$ is properly discontinuous, $y_{\infty} \in \partialn \C_1$ since $\gamma_m\cdot y_m$ remains in some compact subset of $\C_1$, and $z_{\infty}\in\partiali\C_1$ since $\partiali\C_1$ is closed.
The point $y_{\infty} \in \partialn \C_1$ belongs to the convex hull of $\{x_{\infty}, z_{\infty}\} \subset \partiali \C_0$, hence $y_{\infty}\in\overline{\C_0}$.
This contradicts the fact that $y_{\infty} \in \partialn \C_1$ is at distance $\geq\varepsilon$ from~$\overline{\C_0}$.
\end{proof}

The implication \eqref{item:ccc-limit-set} $\Rightarrow$~\eqref{item:ccc-bisat} of Theorem~\ref{thm:main-general} is contained in the following consequence of Corollary~\ref{cor:ideal-bound-naive-cc} and Lemma~\ref{lem:1-neighb-bisat}.

\begin{corollary} \label{coro:ccc-limit-set->bisat}
Let $\Gamma$ be an infinite discrete subgroup of $\PGL(V)$ acting convex cocompactly (Definition~\ref{def:cc-general}) on some nonempty properly convex open subset $\Omega$ of $\PP(V)$.
Let $\C$ be a closed convex subset of~$\Omega$ on which $\Gamma$ acts cocompactly and which contains a neighborhood of $\Ccore_\Omega(\Gamma)$ (\eg a closed uniform neighborhood of $\Ccore_\Omega(\Gamma)$ in $(\Omega,d_{\Omega})$, see Lemma~\ref{lem:naive-cc-nonempty-int}).
Then $\C$ is a properly convex subset of $\PP(V)$ with bisaturated boundary on which $\Gamma$ acts properly discontinuously and cocompactly.
Moreover, $\partiali\C = \partiali\Ccore_\Omega(\Gamma) = \Lambdao_{\Omega}(\Gamma)$.
\end{corollary}

\begin{proof}
Since the convex set $\C$ is contained in~$\Omega$, it is properly convex and $\Gamma$ acts properly discontinuously and cocompactly on it.
By Corollary~\ref{cor:ideal-bound-naive-cc}.\eqref{item:ideal-bound-cc}, we have $\partiali\C=\Lambdao_{\Omega}(\Gamma)=\partiali\Ccore_{\Omega}(\Gamma)$.
By Lemma~\ref{lem:1-neighb-bisat} with $\C_0=\Ccore_{\Omega}(\Gamma)$ and $\C_1=\C$, the set $\C$ has bisaturated boundary.
\end{proof}

\subsection{Convex cocompact actions on the interior of convex sets with bisaturated boundary} \label{subsec:bisat-boundary}

In this section we prove the implication \eqref{item:ccc-bisat} $\Rightarrow$~\eqref{item:ccc-limit-set} of Theorem~\ref{thm:main-general}.
We first make the following general observations.

\begin{lemma} \label{lem:elementary-bisat}
Let $\C$ be a properly convex subset of $\PP(V)$ with bisaturated boundary.
Assume $\partiali \C$ is not empty. 
Then
\begin{enumerate}
  \item\label{item:nonempty-int} $\C$ has nonempty interior $\Int(\C)$;
  \item\label{item:closed-conv-hull} the convex hull of $\partiali\C$ in~$\C$ is contained in $\Int(\C)$ whenever $\partiali\C$ is closed in $\PP(V)$.
\end{enumerate}
\end{lemma}

\begin{proof}
\eqref{item:nonempty-int} 
If the interior of~$\C$ were empty, then $\C$ would be contained in a hyperplane.
Since $\C$ has bisaturated boundary, $\C$ would be equal either to its ideal boundary (hence empty) or to its nonideal boundary (hence closed).

\eqref{item:closed-conv-hull} Suppose $\partiali\C$ is closed in $\PP(V)$, hence compact.
Since $\C$ has bisaturated boundary, every supporting hyperplane of~$\C$ at a point $x \in \partialn \C$ misses $\partiali\C$, hence by compactness of $\partiali\C$ there is a hyperplane strictly separating $x$ from $\partiali\C$ (in an affine chart containing $\overline{\C}$).
It follows that the convex hull of $\partiali\C$ in~$\C$ does not meet $\partialn\C$, hence is contained in~$\Int(\C)$.
\end{proof}

The implication \eqref{item:ccc-bisat} $\Rightarrow$~\eqref{item:ccc-limit-set} of Theorem~\ref{thm:main-general} is contained in the following consequence of Lemma~\ref{lem:closed-ideal-boundary}, Corollary~\ref{cor:ideal-bound-naive-cc}, and Lemma~\ref{lem:elementary-bisat}.

\begin{corollary} \label{cor:bisat-interior}
Let $\Gamma$ be an infinite discrete subgroup of $\PGL(V)$ acting properly discontinuously and cocompactly on a nonempty properly convex subset $\C$ of $\PP(V)$ with bisaturated boundary.
Then $\Omega:=\Int(\C)$ is nonempty and $\Gamma$ acts convex cocompactly (Definition~\ref{def:cc-general}) on~$\Omega$.
Moreover, $\Lambdao_{\Omega}(\Gamma) = \partiali\C$, \ie $\C = \overline{\Omega} \smallsetminus \Lambdao_{\Omega}(\Gamma)$.
\end{corollary}

\begin{proof}
Since $\Gamma$ is infinite, $\C$ is not closed in $\PP(V)$, and so $\Omega=\Int(\C)$ is nonempty by Lemma~\ref{lem:elementary-bisat}.\eqref{item:nonempty-int}.
The set $\partiali\C$ is closed in $\PP(V)$ by Lemma~\ref{lem:closed-ideal-boundary}, and so the convex hull $\C_0$ of $\partiali\C$ in~$\C$ is closed in~$\C$ and contained in~$\Omega$ by Lemma~\ref{lem:elementary-bisat}.\eqref{item:closed-conv-hull}.
The action of $\Gamma$ on~$\C_0$ is still cocompact.
Since $\Gamma$ acts properly discontinuously on~$\C$, the set $\Lambdao_{\Omega}(\Gamma)$ is contained in $\partiali \C$, and $\Ccore_\Omega(\Gamma)$ is contained in~$\C_0$.
By Corollary~\ref{cor:ideal-bound-naive-cc}.\eqref{item:ideal-bound-cc}, the group $\Gamma$ acts convex cocompactly on~$\Omega$ and $\Lambdao_{\Omega}(\Gamma)=\partiali\C_0=\partiali\C$.
\end{proof}

\subsection{The dual of a properly convex set} \label{subsec:dual-def}

Given an open properly convex set $\Omega \subset \PP(V)$, there is a notion of dual convex set $\Omega^* \subset \PP(V^*)$ which is very useful in the study of divisible convex sets: see Section~\ref{subsec:prop-conv-proj}.
We generalize this notion here to properly convex sets with possibly nonempty nonideal boundary.

\begin{definition} \label{def:dual-C}
Let $\C$ be a properly convex subset of $\PP(V)$ with nonempty interior, but not necessarily open nor closed.
The \emph{dual} $\C^*\subset\PP(V^*)$ of~$\C$ is the set of elements of $\PP(V^*)$ which, viewed as projective hyperplanes in $\PP(V)$, do \emph{not} meet $\Int(\C)\cup\partiali\C$, \ie those hyperplanes meet $\overline{\C}$ in a (possibly empty) subset of $\partialn\C$.
\end{definition}

\begin{remark} \label{rem:explicit-dual}
The set $\C^*$ is convex in $\PP(V^*)$.
Indeed, $\overline{\C^*} = \overline{\Int(\C)^*}$ implies that $\overline{\C^*}$ is convex, and if $H,H'',H'\in \overline{\C^*} $ are three distinct points aligned in this order with $H,H'\in\nolinebreak\C^*$, and if we view them as projective hyperplanes in $\PP(V)$, then $H''\cap \overline{\C} \subset (H\cap H') \cap \overline{\C} \subset \partialn \C$, hence $H''\in \C^*$. 
By construction,
\begin{itemize}
  \item $\Int(\C^*)$ is the set of projective hyperplanes in $\PP(V)$ that miss~$\overline{\C}$, \ie $\Int(\C)$ and $\Int(\C^*)$ are dual in the usual sense for open convex sets, as in \eqref{eqn:dual-open-convex};
  \item $\partialn\C^*$ is the set of projective hyperplanes in $\PP(V)$ whose intersection with~$\overline{\C}$ is a nonempty subset of $\partialn\C$ (such a hyperplane misses $\partiali \C$ if $\C$ has bisaturated boundary);
  \item $\partiali\C^*$ is the set of supporting projective hyperplanes of $\C$ at points of $\partiali \C$ (such a hyperplane misses $\partialn \C$ if $\C$ has bisaturated boundary).
\end{itemize}
\end{remark}

\begin{lemma} \label{lem:dual}
Let $\C$ be a properly convex subset of $\PP(V)$, not necessarily open nor closed, but with bisaturated boundary and with nonempty interior.
Then
\begin{enumerate}
  \item \label{item:bisat} the dual~$\C^*$ has bisaturated boundary;
  \item \label{item:dual-dual} the bidual~$(\C^*)^*$ coincides with~$\C$ (after identifying $(V^*)^*$~with~$V$);
   \item \label{item:dual-PETs} the dual $\C^*$ has a PET (properly embedded triangle, Definition~\ref{def:PET}) if and only if $\C$ does.
\end{enumerate}
\end{lemma}

\begin{proof}
\eqref{item:bisat} By Remark~\ref{rem:explicit-dual}, since $\C$ has bisaturated boundary, $\partialn\C^*$ (\resp $\partiali\C^*$) is the set of projective hyperplanes in $\PP(V)$ whose intersection with~$\overline{\C}$ is a nonempty subset of $\partialn\C$ (\resp $\partiali\C$).
In particular, a point of $\partialn\C^*$ and a point of $\partiali\C^*$, seen as projective hyperplanes of $\PP(V)$, can only meet outside of~$\overline{\C}$.
This means exactly that a supporting hyperplane of~$\C^*$ in $\PP(V^*)$ cannot meet both $\partialn\C^*$ and $\partiali\C^*$.

\eqref{item:dual-dual} By definition, $(\C^*)^*$ is the set of hyperplanes of $\PP(V^*)$ missing $\mathrm{Int}(\C^*) \cup \partiali \C^*$.
Viewing hyperplanes of $\PP(V^*)$ as points of $\PP(V)$, by Remark~\ref{rem:explicit-dual} the set $(\C^*)^*$ consists of those points of $\PP(V)$ not belonging to any hyperplane that misses $\C$, or any supporting hyperplane at a point of $\partiali \C$.
Since $\C$ has bisaturated boundary, this is $\overline{\C} \smallsetminus \partiali \C$, namely~$\C$.

\eqref{item:dual-PETs} By \eqref{item:bisat} and~\eqref{item:dual-dual}, it is enough to prove one implication.
Suppose $\C$ has a PET contained in a two-dimensional projective plane $P$, \ie $\C \cap P = T$ is an open triangle.
Let $H_1, H_2, H_3$ be projective hyperplanes of $\PP(V)$ supporting~$\C$ and containing the edges $E_1, E_2, E_3 \subset \partiali \C$ of~$T$.
For $1\leq k\leq 3$, the supporting hyperplane $H_k$ intersects $\partiali\C$, hence lies in the ideal boundary $\partiali \C^*$ of the dual convex set.
Since $\C^*$ has bisaturated boundary by~\eqref{item:bisat}, the whole edge $[H_k, H_{k'}] \subset \overline{\C^*}$ is contained in $\partiali \C^*$ for $1\leq k<k'\leq 3$.
Hence $H_1, H_2, H_3$ span a $2$-plane $Q \subset \PP(V^*)$ whose intersection with $\C^*$ is a PET of~$\C^*$. 
\end{proof}

\subsection{Proper and cocompact actions on the dual} \label{subsec:proper-cocomp-dual}

The following is the key ingredient in Theorem~\ref{thm:properties}.\ref{item:dual}.

\begin{proposition} \label{prop:cc-duality}
Let $\Gamma$ be a discrete subgroup of $\PGL(V)$ and $\C$ a $\Gamma$-invariant convex subset of $\PP(V)$ with bisaturated boundary.
Suppose $\C$ has nonempty interior, so that the dual $\C^*$ is well defined.
Then the action of $\Gamma$ on~$\C$ is properly discontinuous and cocompact if and only if the action of $\Gamma$ on~$\C^*$ is.
\end{proposition}

Recall from Corollary~\ref{cor:bisat-interior} that if $\Gamma$ acts properly discontinuously and cocompactly on~$\C$, then $\C$ automatically has nonempty interior.

In order to prove Proposition~\ref{prop:cc-duality}, we assume $n=\dim(V)\geq 2 $ and first make some definitions.
Consider a properly convex open subset $\Omega$ in $\PP(V)$.
For any distinct points $x,y \in \PP(V) \smallsetminus \partial \Omega$, at least one of which is in $\Omega$, the line $L$ through $x$ and~$y$ intersects $\partial\Omega$ in two points $a,b$ and we set
\begin{align*}
\delta_\Omega(x,y) &:= \max \big\{\cro{a}{x}{y}{b}, \cro{b}{x}{y}{a} \big\}.
\end{align*}
If $x,y \in \Omega$, then $\delta(x,y) = \exp (2\,d_{\Omega}(x,y)) > 1$ where $d_{\Omega}$ is the Hilbert metric on~$\Omega$ (see Section~\ref{subsec:prop-conv-proj}).
However, if $x \in \Omega$ and $y \in \PP(V) \smallsetminus \overline \Omega$, then we have $-1 \leq \delta_\Omega(x,y) < 0$.
For any point $x \in \Omega$ and any projective hyperplane $H \in \Omega^*$ (\ie disjoint from $\overline{\Omega}$), we set
\begin{align*}
\delta_\Omega(x,H) &:= \max_{y\in H} \delta_\Omega(x,y).
\end{align*}
Then $\delta(x,H) \in [-1,0)$ is close to $0$ when $H$ is ``close'' to $\partial\Omega$ as seen from~$x$.

\begin{lemma}\label{lem:pseudo-distance}
Let $\Omega$ be a nonempty properly convex open subset of $\PP(V)=\PP(\RR^n)$.
For any $H \in \Omega^*$, there exists $x \in \Omega$ such that
$$\delta_\Omega(x,H) \leq \frac{-1}{n-1}.$$
\end{lemma}

This lemma is classical: see \eg \cite[Th.\,2.7]{dgk63}, which gives a proof using Helly's theorem.
The original result goes back to Radon~\cite{rad16}.
We include a proof for convenience.

\begin{proof}
Fix $H\in\Omega^*$ and consider an affine chart $\RR^{n-1}$ of $\PP(V)$ for which $H$ is at infinity, endowed with a Euclidean norm $\Vert\cdot\Vert$.
We take for $x$ the center of mass of $\Omega$ in this affine chart with respect to the Lebesgue measure.
It is enough to show that if $a,b\in\partial\Omega$ satisfy $x\in [a,b]$, then $\Vert x-a \Vert/\Vert b-a \Vert \leq (n-1)/n$.
Up to translation, we may assume $a=0\in \RR^{n-1}$.
Let $\varphi$ be a linear form on $\RR^{n-1}$ such that $\varphi(b)=1=\sup_{\Omega} \varphi$.
Let $h:=\varphi(x)$, so that $\Vert x-a \Vert/\Vert b-a\Vert=h$, and let $\Omega':=\RR_{>0}\cdot (\Omega\cap \varphi^{-1}(h)) \cap \varphi^{-1}((-\infty,1))$ (see Figure~\ref{fig:Strata}).
\begin{figure}[h]
\centering
\labellist
\small\hair 2pt
\pinlabel {$\Omega'$} [u] at 121 97
\pinlabel {$\Omega$} [u] at 173 93
\pinlabel {$a=0$} [u] at 230 24
\pinlabel {$x$} [u] at 222 64
\pinlabel {$b$} [u] at 232 104
\pinlabel {$\varphi^{-1}(h)$} [u] at 265 73
\pinlabel {$\varphi^{-1}(1)$} [u] at 340 110
\endlabellist
\includegraphics[scale=1.0]{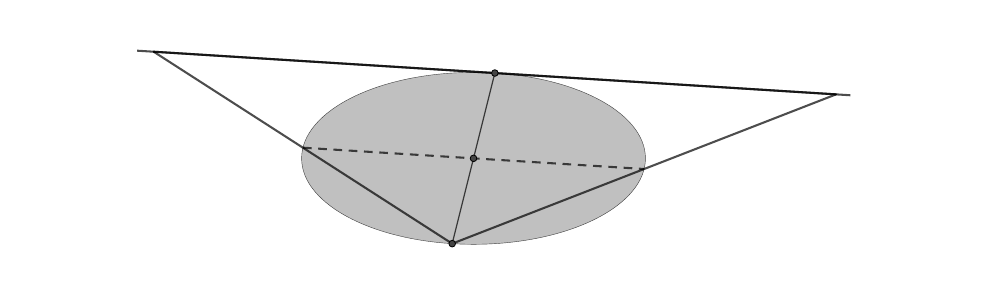}
\caption{Illustration for the proof of Lemma~\ref{lem:pseudo-distance}}
\label{fig:Strata}
\end{figure}
The average value $\mathbb{E}_\Omega(\varphi)$ of $\varphi$ on~$\Omega$, for the Lebesgue measure, is $\mathbb{E}_\Omega(\varphi)=\varphi(x)=h$, since $x$ is the center of mass of~$\Omega$.
Moreover, $\mathbb{E}_{\Omega'}(\varphi) \geq \mathbb{E}_{\Omega}(\varphi)$ since by convexity
\begin{eqnarray*}
\Omega' \cap \varphi^{-1}\big((-\infty, h]\big) & \subset & \Omega \cap \varphi^{-1}\big((-\infty, h]\big),\\
\Omega' \cap \varphi^{-1}\big([h,1)\big) & \supset & \Omega \cap \varphi^{-1}\big([h,1)\big).
\end{eqnarray*}
But $\mathbb{E}_{\Omega'}(\varphi)=(n-1)/n$ since $\Omega'$ is a truncated open cone in~$\RR^{n-1}$. Thus, $h\leq (n-1)/n$.
\end{proof}

\begin{proof}[Proof of Proposition~\ref{prop:cc-duality}]
By Lemma~\ref{lem:dual}, the set $\C^*$ has bisaturated boundary and $(\C^*)^* = \C$.
Thus it is enough to prove that if the action on $\C$ is properly discontinuous and cocompact, then so is the action on $\C^*$.

Let us begin with properness.
Recall from Lemma~\ref{lem:closed-ideal-boundary} that the set $\partiali \C$ is closed in $\PP(V)$.
Let $\C_0$ be the convex hull of $\partiali\C$ in~$\C$.
Note that any supporting hyperplane of~$\C_0$ contains a point of $\overline{\partiali \C}$: otherwise it would be separated from $\partiali\C$ by a hyperplane, contradicting the definition of~$\C_0$.
By Lemma~\ref{lem:elementary-bisat}.\eqref{item:closed-conv-hull}, we have $\C_0 \subset \Omega := \Int(\C)$.
Let $\C_1$ be the closed uniform $1$-neighborhood of $\C_0$ in $(\Omega,d_{\Omega})$.
It is properly convex \cite[(18.12)]{bus55} (see Lemma~\ref{lem:naive-cc-nonempty-int}), with nonempty interior, and $\partialn \C_1 \cap \partialn \C = \emptyset$. 
Taking the dual, we obtain that $\Int(\C_1)^*$ is a $\Gamma$-invariant properly convex open set containing~$\C^*$.
In particular, the action of $\Gamma$ on $\C^*\subset\Int(\C_1)^*$ is properly discontinuous (see Section~\ref{subsec:prop-conv-proj}).

Let  us show that the action of $\Gamma$ on $\C^*$ is cocompact.
Let $\Omega = \Int (\C)$ and $\Omega^* = \Int (\C^*)$. 
Let $\mathcal{D} \subset \C$ be a compact fundamental domain for the action of $\Gamma$ on~$\C$.
Consider the following subset of $\C^*$, where $n=\dim(V)$:
$$\mathcal{U}^* = \bigcup_{x\in\mathcal{D}\cap\Omega} \big\{ H \in \Omega^* ~|~ \delta_\Omega(x, H) \leq {\textstyle \frac{-1}{n-1}}\big\}.$$
It follows from Lemma~\ref{lem:pseudo-distance} that $\Gamma \cdot \mathcal{U}^* = \Omega^*$. 
We claim that $\overline{\mathcal{U}^*} \subset \C^*$.
To see this, suppose a sequence of elements $H_m \in \mathcal{U}^*$ converges to some $H \in \overline{\mathcal{U}^*}$; let us show that $H \in \C^*$.
If $H \in \Omega^*$ there is nothing to prove, so we may assume that $H \in \partial\Omega^* = \Fr(\C^*)$ is a supporting hyperplane of $\C$ at a point $y\in\Fr(\C)$.
For every~$m$, let $x_m \in \mathcal{D} \cap \Omega$ satisfy $\delta_\Omega(x_m, H_m) \leq \frac{-1}{n-1}$. 
Up to passing to a subsequence, we may assume $x_m \to x \in \mathcal{D} \subset \C$.
Let $y_m \in H_m$ such that $y_m \to y$.
For every~$m$, let $a_m,b_m\in\partial\Omega$ be such that $a_m, x_m, b_m, y_m$ are aligned in this order.
Then
$$\frac{-1}{n-1} \geq \delta_{\Omega}(x_m, y_m) \geq \cro{b_m}{x_m}{y_m}{a_m} \geq - \frac{\Vert b_m - y_m\Vert}{\Vert b_m - x_m\Vert},$$
where $\| \cdot \|$ is a fixed Euclidean norm on an affine chart containing $\overline{\Omega}$.
Since $\Vert b_m - y_m\Vert\to 0$, we deduce $\Vert b_m - x_m\Vert\to 0$, and so $x = y \in \C \cap H \subset \partialn \C$.
Since $\C$ has bisaturated boundary, we must have $H \in \C^*$, by definition of~$\C^*$. 
Therefore $\overline{\mathcal{U}^*} \subset \C^*$. 
Since $\overline{\mathcal{U}^*}$ is compact and the action on $\C^*$ is properly discontinuous, the fact that $\Gamma\cdot \mathcal{U}^*=\Omega^*$ yields $\Gamma\cdot \overline{\mathcal{U}^*}=\C^*$.
\end{proof}

\subsection{Proof of Theorem~\ref{thm:properties}.\ref{item:dual}} \label{subsec:proof-dual}

We establish the following more precise result, which implies Theorem~\ref{thm:properties}.\ref{item:dual}.

\begin{proposition} \label{prop:dual-precise}
Let $\Gamma$ be an infinite discrete subgroup of $\PGL(V)$ which is convex cocompact in $\PP(V)$.
Then there exists a nonempty properly convex open subset $\Omega$ of $\PP(V)$ such that $\Gamma$ acts convex cocompactly on both $\Omega$ and~$\Omega^*$.
Such sets $\Omega$ are exactly the interiors of the nonempty properly convex subsets of $\PP(V)$ with bisaturated boundary on which $\Gamma$ acts properly discontinuously and cocompactly.
\end{proposition}

\begin{proof}
Let $\C$ be a nonempty properly convex subset of $\PP(V)$ with bisaturated boundary on which $\Gamma$ acts properly discontinuously and cocompactly: such a $\C$ exists by Corollary~\ref{coro:ccc-limit-set->bisat}.
By Lemma~\ref{lem:dual}.\eqref{item:bisat}, the dual $\C^* \subset \PP(V^*)$ has bisaturated boundary, and $\Gamma$ acts properly discontinuously and cocompactly on~$\C^*$ by Proposition~\ref{prop:cc-duality}.
By construction (see Remark~\ref{rem:explicit-dual}), the dual $\Omega^*$ of $\Omega := \Int(\C)$ satisfies $\Omega^* = \Int(\C^*)$.
By Corollary~\ref{cor:bisat-interior}, the set $\Omega$ (\resp $\Omega^*$) is a nonempty properly convex open subset of $\PP(V)$ (\resp $\PP(V^*)$) on which $\Gamma$ acts convex cocompactly, and $\C = \overline{\Omega} \smallsetminus \Lambdao_{\Omega}(\Gamma)$ (\resp $\C^* = \overline{\Omega^*} \smallsetminus \Lambdao_{\Omega^*}(\Gamma)$).

Conversely, let $\Omega$ be a nonempty properly convex open subset of $\PP(V)$ such that the action of $\Gamma$ on both $\Omega$ and $\Omega^*$ is convex cocompact.
Let us check that $\C := \overline{\Omega} \smallsetminus \Lambdao_{\Omega}(\Gamma)$ is a properly convex set with bisaturated boundary on which $\Gamma$ acts properly discontinuously and cocompactly.
Let ${\C'}^*$ be a closed uniform neighborhood of $\Ccore_{\Omega^*}(\Gamma)$ in $(\Omega^*,d_{\Omega^*})$.
By Corollary~\ref{coro:ccc-limit-set->bisat}, the set ${\C'}^*$ is a properly convex subset of $\PP(V^*)$ with bisaturated boundary on which $\Gamma$ acts properly discontinuously and cocompactly.
By Lemma~\ref{lem:dual}.\eqref{item:bisat} and Proposition~\ref{prop:cc-duality}, the dual $\C' := ({\C'}^*)^*$ is a properly convex subset of $\PP(V)$ with bisaturated boundary on which $\Gamma$ acts properly discontinuously and cocompactly.
By Corollary~\ref{cor:bisat-interior} and Proposition~\ref{prop:cc-Omega-subset-Omega'}, the group $\Gamma$ acts convex cocompactly on $\Omega' := \Int(\C')$ and $\partiali\C' = \Lambdao_{\Omega'}(\Gamma) = \Lambdao_{\Omega}(\Gamma) = \partiali\C$.
In particular, $\C$ is contained and closed in~$\C'$, hence $\Gamma$ also acts properly discontinuously and cocompactly on~$\C$.
Since $\partialn\C^* \cap \partialn{\C'}^* = \emptyset$, Remark~\ref{rem:explicit-dual} implies that $\partialn\C \cap \partialn\C' = \emptyset$, and so $\C$ is contained $\Omega' = \Int(\C')$.
Moreover, $\Omega$ is a neighborhood in $\Omega'$ of $\Ccore_{\Omega'}(\Gamma) = \Ccore_{\Omega}(\Gamma)$. Since $\C$ contains~$\Omega$, Corollary~\ref{coro:ccc-limit-set->bisat} implies that $\C$ has bisaturated boundary.
\end{proof}

\subsection{On which convex sets with bisaturated boundary is the action properly discontinuous and cocompact?} \label{subsec:which-bisat}

Here is a consequence of Corollaries \ref{cor:cc-sets-irred}, \ref{coro:ccc-limit-set->bisat}, and~\ref{cor:bisat-interior}.

\begin{corollary} \label{cor:which-sets-bisat}
Let $\Gamma$ be an infinite discrete subgroup of $\PGL(V)$ acting convex cocompactly on a nonempty properly convex open subset $\Omega$ of $\PP(V)$.
Suppose that $\Gamma$ contains a proximal element, so that the maximal convex open set $\Omega_{\max} \supset \Omega$ of Proposition~\ref{prop:max-inv-conv} is well defined, and suppose that $\Omega_{\max}$ is properly convex (this is always the case if $\Gamma$ acts irreducibly on $\PP(V)$, see Fact~\ref{fact:benoist-irred-Gamma}).
Then the nonempty properly convex subsets of $\overline{\Omega_{\max}}$ with bisaturated boundary on which $\Gamma$ acts properly discontinuously and cocompactly are exactly the closed convex subsets of $\Omega_{\max}$ that are nested between two open uniform neighborhoods of $\Ccore_{\Omega_{\max}}(\Gamma)$ in $(\Omega_{\max},d_{\Omega_{\max}})$.
\end{corollary}

\begin{proof}
We first observe that, by Corollary~\ref{cor:cc-sets-irred}.\eqref{item:Omega'-cc-irred}, the action of $\Gamma$ on $\Omega_{\max}$ is convex cocompact.

If $\C$ is a closed convex subset of $\Omega_{\max}$ that is nested between two open uniform neighborhoods of $\Ccore_{\Omega_{\max}}(\Gamma)$ in $(\Omega_{\max},d_{\Omega_{\max}})$, then Corollary~\ref{coro:ccc-limit-set->bisat} yields that $\C$ is a properly convex set with bisaturated boundary on which $\Gamma$ acts properly discontinuously and cocompactly.

Conversely, let $\C$ be a nonempty properly convex subset of $\overline{\Omega_{\max}}$ with bisaturated boundary on which $\Gamma$ acts properly discontinuously and cocompactly.

First, we show that $\C$ is contained in $\Omega_{\max}$.
For this we observe that by Lemma~\ref{lem:dual}.\eqref{item:bisat} and Proposition~\ref{prop:cc-duality}, the dual~$\C^*$ is a properly convex subset of $\PP(V^*)$ with bisaturated boundary on which $\Gamma$ acts properly discontinuously and cocompactly.
In particular, $\C^*$ has nonempty interior by Lemma~\ref{lem:elementary-bisat}.\eqref{item:nonempty-int}, and so $\Lambda^*_{\Gamma} \subset \Lambdao_{\Int(\C^*)}(\Gamma) \subset \partiali\C^*$ by Lemma~\ref{lem:Lambda-prox-Lambda-orb}.
In other words, by Remark~\ref{rem:explicit-dual}, any element of $\Lambda^*_{\Gamma}$, seen as a hyperplane in $\PP(V)$, is a supporting hyperplane to $\Omega_{\max}$ at a point of $\partiali\C$, hence misses $\partialn \C$ since $\C$ has bisaturated boundary.
From this and from the definition of $\Omega_{\max}$ (see Proposition~\ref{prop:max-inv-conv}), we deduce that $\partial\Omega_{\max} \cap \partialn\C = \emptyset$, and so $\C$ is contained in $\Omega_{\max}$.

By Corollary~\ref{cor:bisat-interior}, the properly convex open set $\Omega':=\Int(\C)$ is nonempty and $\Gamma$ acts convex cocompactly on it.
By Corollary~\ref{cor:cc-sets-irred}.\eqref{item:Omega'-cc-irred}, the set $\Omega'$ contains $\Ccore_{\Omega_{\max}}$, and so $\C$ contains a neighborhood of $\Ccore_{\Omega_{\max}}$ in $\Omega_{\max}$.

By considering a compact fundamental domain for the action of $\Gamma$ on $\Ccore_{\Omega_{\max}}$, we see that $\Omega'$ (hence $\C$) actually contains an open uniform neighborhood of $\Ccore_{\Omega_{\max}}(\Gamma)$ in $(\Omega_{\max},d_{\Omega_{\max}})$, and that $\C$ is contained in an open uniform neighborhood of $\Ccore_{\Omega_{\max}}(\Gamma)$ in $(\Omega_{\max},d_{\Omega_{\max}})$.
\end{proof}

\section{Segments in the full orbital limit set} \label{sec:regularity-limit-set}

In this section we establish the equivalences \ref{item:ccc-hyp}~$\Leftrightarrow$~\ref{item:ccc-noseg-any}~$\Leftrightarrow$~\ref{item:ccc-noPETs-some}~$\Leftrightarrow$~\ref{item:ccc-ugly} in Theorem~\ref{thm:main-noPETs}.
For this it will be helpful to introduce two intermediate conditions, weaker than~\ref{item:ccc-noseg-any} but stronger than~\ref{item:ccc-noPETs-some}:
\begin{enumerate}[label=(\roman*)']
  \setcounter{enumi}{2}
  \item \label{item:ccc-noseg-some} $\Gamma$ is convex cocompact in $\PP(V)$ and for \emph{some} nonempty properly convex open set $\Omega$ on which $\Gamma$ acts convex cocompactly, $\Lambdao_{\Omega}(\Gamma)$ does not contain a nontrivial projective line segment;
  \item \label{item:ccc-noPETs-any} $\Gamma$ is convex cocompact in $\PP(V)$ and for \emph{any} nonempty properly convex open set $\Omega$ on which $\Gamma$ acts convex cocompactly, $\Ccore_\Omega(\Gamma)$ does not contain a PET.
\end{enumerate}
The implications \ref{item:ccc-noseg-any}~$\Rightarrow$~\ref{item:ccc-noseg-some}, \ref{item:ccc-noPETs-any}~$\Rightarrow$~\ref{item:ccc-noPETs-some}, \ref{item:ccc-noseg-any}~$\Rightarrow$~\ref{item:ccc-noPETs-any}, and \ref{item:ccc-noseg-some}~$\Rightarrow$~\ref{item:ccc-noPETs-some} are trivial.
The implication \ref{item:ccc-hyp}~$\Rightarrow$~\ref{item:ccc-noPETs-any} holds by Lemma~\ref{lem:noPETs} below.
In Section~\ref{subsec:segments-to-PETs} we simultaneously prove \ref{item:ccc-noPETs-any}~$\Rightarrow$~\ref{item:ccc-noseg-any} and \ref{item:ccc-noPETs-some}~$\Rightarrow$~\ref{item:ccc-noseg-some}.
In Section~\ref{subsec:no-seg-implies-hyp} we prove \ref{item:ccc-noseg-some}~$\Rightarrow$~\ref{item:ccc-hyp}.
In Section~\ref{subsec:ugly} we prove \ref{item:ccc-noseg-some}~$\Leftrightarrow$~\ref{item:ccc-ugly}.
We include the following diagram of implications for the reader's convenience:
\begin{equation*}
\xymatrix{  *+[F]{\begin{tabular}{c} \text{\tiny \ref{item:ccc-noseg-any}:} \\ \text{\tiny for any cc $\Omega$,}\\ \text{\tiny $\Lambdao_\Omega \not\supset$ segment}\end{tabular}} \ar@{=>}[rr]\ar@{=>}[dd] & & *+[F]{\begin{tabular}{c}\text{\tiny \ref{item:ccc-noseg-some}:}\\ \text{\tiny for some cc $\Omega$,}\\ \text{\tiny $\Lambdao_\Omega \not\supset$ segment}\end{tabular}} \ar@{==>}[dl]^{\S\,\ref{subsec:no-seg-implies-hyp}} \ar@{=>}[dd] \ar@{<==>}@/^24pt/[dr]^{\S\,\ref{subsec:ugly}}
&  \\
&   *+[F]{\begin{tabular}{c}\text{\tiny \ref{item:ccc-hyp}:}\\ \text{\tiny $\Gamma$ cc and}\\ \text{\tiny $\Gamma$ hyperbolic}\end{tabular}} \ar@{==>}[dl]^{\S\, \ref{subsec:veryshort}} & & *+[F]{\begin{tabular}{c}\text{\tiny \ref{item:ccc-ugly}:}\\ \text{\tiny $\exists \Omega \supset \C$}\\ \text{\tiny $\C$ cocompact, nonempty interior}\\ \text{\tiny $\partiali \C \not\supset$ segment } \end{tabular}} \\
  *+[F]{\begin{tabular}{c}\text{\tiny \ref{item:ccc-noPETs-any}:} \\ \text{\tiny for any cc $\Omega$,}\\ \text{\tiny $\Ccore_\Omega \not\supset$ PET}\end{tabular}} \ar@{==>}@/^16pt/[uu]^{\S\,\ref{subsec:segments-to-PETs}} \ar@{=>}[rr] & &  *+[F]{\begin{tabular}{c}\text{\tiny \ref{item:ccc-noPETs-some}:} \\ \text{\tiny for some cc $\Omega$,}\\ \text{\tiny$\Ccore_\Omega \not\supset$ PET}\end{tabular}}\ar@{==>}@/_16pt/[uu]_{\S\,\ref{subsec:segments-to-PETs}} }
\end{equation*}

\subsection{PETs obstruct hyperbolicity} \label{subsec:veryshort}

We start with an elementary remark.

\begin{lemma} \label{lem:noPETs}
Let $\Gamma$ be an infinite discrete subgroup of $\PGL(V)$ preserving a properly convex open subset $\Omega$ and acting cocompactly on a closed convex subset $\C$ of $\Omega$.
If $\C$ contains a PET (Definition~\ref{def:PET}), then $\Gamma$ is not word hyperbolic.
\end{lemma}

\begin{proof}
The classical \v{S}varc--Milnor lemma states that a finitely generated group is quasi-isometric to any proper, geodesic metric space on which it acts properly and cocompactly by isometries.
Hence $\Gamma$ is quasi-isometric to the metric space $(\C, d_{\Omega})$.
A PET in $\Omega$ is totally geodesic for the Hilbert metric $d_{\Omega}$ and is quasi-isometric to the Euclidean plane.
Thus, if $\C$ contains a PET, then the metric space $(\C, d_{\Omega})$ is not Gromov hyperbolic and therefore $\Gamma$ is not word hyperbolic.  
\end{proof}

\subsection{Segments in the full orbital limit set yield PETs in the convex core}\label{subsec:segments-to-PETs}

The implications \ref{item:ccc-noPETs-any}~$\Rightarrow$~\ref{item:ccc-noseg-any} and \ref{item:ccc-noPETs-some}~$\Rightarrow$~\ref{item:ccc-noseg-some} in Theorem~\ref{thm:main-noPETs} will be a consequence of the following lemma, which is similar to \cite[Prop.\,3.8.(a)]{ben06} but without the divisibility assumption nor the restriction to dimension~$3$; our proof is different, close to~\cite{ben04}.

\begin{lemma} \label{lem:segments-imply-PETs}
Let $\Gamma$ be an infinite discrete subgroup of $\PGL(V)$.
Let $\Omega$ be a nonempty $\Gamma$-invariant properly convex open subset of $\PP(V)$.
Suppose $\Gamma$ acts cocompactly on some nonempty closed convex subset $\C$ of~$\Omega$.
If $\partiali\C$ contains a nontrivial segment which is inextendable in $\partial\Omega$, then $\C$ contains a PET.
\end{lemma}

\begin{proof}
The ideal boundary $\partiali\C$ is closed in $\PP(V)$ by Remark~\ref{rem:closed-ideal-bound-C-Omega}.
Suppose $\partiali\C$ contains a nontrivial segment $[a,b]$ which is inextendable in $\partial \Omega$.
Let $c \in \C$ and consider a sequence of points $x_m\in\C$ lying inside the open triangle with vertices $a,b,c$ and converging to a point $x \in (a,b)$.

We claim that the $d_{\Omega}$-distance from $x_m$ to either projective interval $(a,c]$ or $(b,c]$ tends to infinity with~$m$.
Indeed, consider a sequence $(y_m)_m$ of points of $(a,c]$ converging to $y \in [a,c]$ and let us check that $d_{\Omega}(x_m,y_m) \to +\infty$ (the proof for $(b,c]$ is the same).
If $y \in (a,c]$, then $y\in\C$ and so $d_{\Omega}(x_m,y_m) \to +\infty$ by properness of the Hilbert metric.
Otherwise $y = a$.
In that case, for each~$m$, consider $x'_m, y'_m \in \partial \Omega$ such that $x'_m, x_m, y_m, z'_m$ are aligned in that order.
Up to taking a subsequence, we may assume $x'_m \to x'$ and $z'_m \to z'$ for some $x',z' \in \partial \Omega$, with $x', x, a, z'$ aligned in that order.
By inextendability of $[a,b]$ in $\partial \Omega$, we must have $z = a = z'$, hence $d_{\Omega}(x_m,z_m) \to +\infty$ in this case as well, proving the claim.

Since the action of $\Gamma$ on~$\C$ is cocompact, there is a sequence $(\gamma_m)\in\Gamma^{\NN}$ such that $\gamma_m\cdot x_m$ remains in a fixed compact subset of~$\C$.
Up to passing to a subsequence, we may assume that $(\gamma_m\cdot x_m)_m$ converges to some $x_{\infty} \in \C$, and $(\gamma_m\cdot a)_m$ and $(\gamma_m\cdot b)_m$ and $(\gamma_m\cdot c)_m$ converge respectively to some $a_{\infty}, b_{\infty}, c_{\infty} \in \partiali \C$, with $[a_{\infty},b_{\infty}]\subset\partiali\C$ since $\partiali\C$ is closed (Remark \ref{rem:closed-ideal-bound-C-Omega}).
The triangle with vertices $a_{\infty}, b_{\infty}, c_{\infty}$ is nondegenerate since it contains $x_{\infty} \in\nolinebreak\C$.
Further, $x_{\infty}$ is infinitely far (for the Hilbert metric on~$\Omega$) from the edges $[a_{\infty}, c_{\infty}]$ and $[b_{\infty}, c_{\infty}]$, and so these edges are fully contained in $\partiali\C$.
Thus the triangle with vertices $a_{\infty}, b_{\infty}, c_{\infty}$ is a PET of~$\C$.
\end{proof}

\begin{proof}[Simultaneous proof of \ref{item:ccc-noPETs-any}~$\Rightarrow$~\ref{item:ccc-noseg-any} and  \ref{item:ccc-noPETs-some}~$\Rightarrow$~\ref{item:ccc-noseg-some}]
We prove the contrapositive.
Suppose $\Gamma \subset \PGL(V)$ acts convex cocompactly on the properly convex open set $\Omega \subset \PP(V)$ and $\Lambdao_\Omega(\Gamma)$ contains a nontrivial segment.
By Corollary~\ref{cor:ideal-bound-naive-cc}.\eqref{item:ideal-bound-cc}, the set $\partiali\Ccore_\Omega(\Gamma)$ is equal to $\Lambdao_{\Omega}(\Gamma)$ and closed in $\PP(V)$, and it contains the open face of $\partial\Omega$ at any interior point of that segment; in particular, $\Lambdao_\Omega(\Gamma)$ contains a nontrivial segment which is inextendable in~$\partial \Omega$.
By Lemma~\ref{lem:segments-imply-PETs} with $\C=\Ccore_{\Omega}(\Gamma)$, the set $\Ccore_\Omega(\Gamma)$ contains a PET.
\end{proof}

\subsection{Word hyperbolicity of the group in the absence of segments}\label{subsec:no-seg-implies-hyp}

In this section we prove the implication \ref{item:ccc-noseg-some}~$\Rightarrow$~\ref{item:ccc-hyp} in Theorem~\ref{thm:main-noPETs}.
We proceed exactly as in \cite[\S\,4.3]{dgk-ccHpq}, with arguments inspired from \cite{ben04}.

Recall that any geodesic ray of $(\C,d_\Omega)$ has a well-defined endpoint in $\partiali\C$ (see \cite[Th.\,3]{fk05} or \cite[Lem.\,2.6.(1)]{dgk-ccHpq}).
It is sufficient to apply the following general result to $\C=\Ccore_{\Omega}$.

\begin{lemma} \label{lem:C-hyp}
Let $\Gamma$ be a discrete subgroup of $\PGL(V)$.
Let $\Omega$ be a nonempty $\Gamma$-invariant properly convex open subset of $\PP(V)$, with Hilbert metric $d_{\Omega}$.
Suppose $\Gamma$ acts cocompactly on some nonempty closed convex subset $\C$ of~$\Omega$ such that $\partiali\C$ does not contain any nontrivial projective line segment.
Then
\begin{enumerate}
  \item\label{item:geod-non-strict-conv} there exists $R>0$ such that any geodesic ray of $(\C,d_{\Omega})$ lies at Hausdorff distance $\leq R$ from the projective interval with the same endpoints;
  \item\label{item:C-hyp} the metric space $(\C,d_{\Omega})$ is Gromov hyperbolic with Gromov boundary $\Gamma$-equivariantly homeomorphic to $\partiali\C$;
  \item\label{item:Gamma-hyp} the group $\Gamma$ is word hyperbolic and any orbit map $\Gamma \to (\C,d_{\Omega})$ is a quasi-isometric embedding which extends to a $\Gamma$-equivariant homeomorphism $\xi : \partial_{\infty}\Gamma\to\partiali\C$;
  \item\label{item:boundary-map-indep-of-orbit} if $\C$ contains $\Ccore_{\Omega}(\Gamma)$ (hence $\Gamma$ acts convex cocompactly on~$\Omega$, see Corollary~\ref{cor:ideal-bound-naive-cc}.\eqref{item:ideal-bound-cc}), then any orbit map $\Gamma \to (\Omega, d_{\Omega})$ is a quasi-isometric embedding and extends to a $\Gamma$-equivariant homeomorphism $\xi: \partial_{\infty} \Gamma \to \Lambdao_\Omega(\Gamma) = \partiali \C$ which is independent of the orbit.
\end{enumerate}
\end{lemma}

\begin{proof}
\eqref{item:geod-non-strict-conv} Suppose by contradiction that for any $m\in\NN$ there is a geodesic ray $\mathcal{G}_m$ of $(\C,d_{\Omega})$ with endpoints $a_m\in\C$ and $b_m\in\partiali\C$ and a point $y_m\in\C$ on that geodesic which lies at distance $\geq m$ from the projective interval $[a_m,b_m)$. 
By cocompactness of the action of $\Gamma$ on~$\C$, for any $m\in\NN$ there exists $\gamma_m\in\Gamma$ such that $\gamma_m\cdot y_m$ belongs to a fixed compact set of~$\C$.
Up to taking a subsequence, $(\gamma_m\cdot y_m)_m$ converges to some $y_{\infty}\in\C$, and $(\gamma_m\cdot a_m)_m$ and $(\gamma_m\cdot\nolinebreak b_m)_m$ converge respectively to some $a_{\infty}\in\overline{\C}$ and $b_{\infty}\in\partiali\C$.
Since the distance from $y_m$ to $[a_m,b_m)$ goes to infinity, we have $[a_{\infty},b_{\infty}]\subset\partiali\C$, hence $a_{\infty}=b_{\infty}$ since $\partiali\C$ does not contain a segment.
Therefore, up to extracting, the geodesics~$\mathcal{G}_m$ converge to a biinfinite geodesic of $(\Omega,d_{\Omega})$ with both endpoints equal.
But such a geodesic does not exist (see \cite[Th.\,3]{fk05} or \cite[Lem.\,2.6.(2)]{dgk-ccHpq}): contradiction.

\eqref{item:C-hyp} Suppose by contradiction that triangles of $(\C,d_{\Omega})$ are not uniformly thin.
By~\eqref{item:geod-non-strict-conv}, triangles of $(\C,d_{\Omega})$ whose sides are projective line segments are not uniformly thin: namely, there exist $a_m,b_m,c_m\in\C$ and $y_m\in [a_m,b_m]$ such that
\begin{equation} \label{eqn:u_m-far}
d_{\Omega}(y_m, [a_m,c_m]\cup [c_m,b_m]) \underset{m\to +\infty}{\longrightarrow} +\infty.
\end{equation}
By cocompactness, for any $m$ there exists $\gamma_m\in\Gamma$ such that $\gamma_m\cdot y_m$ belongs to a fixed compact set of~$\C$.
Up to taking a subsequence, $(\gamma_m\cdot y_m)_m$ converges to some $y_{\infty}\in\C$, and $(\gamma_m\cdot a_m)_m$ and $(\gamma_m\cdot b_m)_m$ and $(\gamma_m\cdot\nolinebreak c_m)_m$ converge respectively to some $a_{\infty},b_{\infty},c_{\infty}\in\overline{\C}$.
By \eqref{eqn:u_m-far} we have $[a_{\infty},c_{\infty}]\cup [c_{\infty},b_{\infty}]\subset\partiali\C$, hence $a_{\infty}=b_{\infty}=c_{\infty}$ since $\partiali\C$ does not contain any nontrivial projective line segment.
This contradicts the fact that $y_{\infty}\in (a_{\infty},b_{\infty})$.
Therefore $(\C,d_{\Omega})$ is Gromov hyperbolic.

Fix a basepoint $y\in\C$.
The Gromov boundary of $(\C,d_{\Omega})$ is the set of equivalence classes of infinite geodesic rays in~$\C$ starting at~$y$, for the equivalence relation ``to remain at bounded distance for $d_{\Omega}$''.
Consider the $\Gamma$-equivariant continuous map from $\partiali\C$ to this Gromov boundary sending $z\in\partiali\C$ to the class of the geodesic ray from $y$ to~$z$.
This map is clearly surjective, since any infinite geodesic ray in $\C$ terminates at the ideal boundary~$\partiali \C$.
Moreover, it is injective, since the nonexistence of line segments in $\partiali \C$ means that no two points of $\partiali \C$ lie in a common face of $\partial \Omega$, hence the Hilbert distance between rays going out to two different points of $\partiali \C$ goes to infinity (see Lemma~\ref{lem:unif-neighb-face}.\eqref{item:distance-goes-down}).

\eqref{item:Gamma-hyp} The group $\Gamma$ acts properly discontinuously and cocompactly, by isometries, on the proper, geodesic metric space $(\C,d_{\Omega})$, which is Gromov hyperbolic with Gromov boundary~$\partiali\C$ by~\eqref{item:C-hyp} above.
We apply the \v{S}varc--Milnor lemma.

\eqref{item:boundary-map-indep-of-orbit} Suppose $\C$ contains $\Ccore_{\Omega}(\Gamma)$.
Consider a point $y \in \Omega$.
It lies in a closed uniform neighborhood $\C_y$ of $\C$ in $(\Omega,d_{\Omega})$, which is a properly convex subset of~$\Omega$ on which the action of $\Gamma$ is also cocompact (Lemma~\ref{lem:naive-cc-nonempty-int}).
By Corollary~\ref{cor:ideal-bound-naive-cc}.\eqref{item:ideal-bound-cc}, we have $\partiali \C_y = \partiali \C = \Lambdao_\Omega(\Gamma)$, and this set does not contain any nontrivial segment by assumption.
By~\eqref{item:Gamma-hyp}, the orbit map $\Gamma \to (\C_y,d_{\Omega})$ associated with~$y$ is a quasi-isometry which extends to a $\Gamma$-equivariant homeomorphism $\partial_{\infty}\Gamma\to\partiali\C_y = \partiali \C$.
This extension is independent of $y$ since $\partiali \C_y = \partiali \C$ does not contain any nontrivial segment (argue as in the proof of Lemma~\ref{lem:strict-convex-Lambda-orb}).
\end{proof}

\subsection{Absence of segments in $\partiali \C$ and in $\Lambdao_\Omega(\Gamma)$} \label{subsec:ugly}

In this section we prove the equivalence \ref{item:ccc-noseg-some}~$\Leftrightarrow$~\ref{item:ccc-ugly}.

\begin{proof}[Proof of \ref{item:ccc-noseg-some}~$\Rightarrow$~\ref{item:ccc-ugly}]
Suppose that $\Gamma$ acts convex cocompactly on some nonempty properly convex open subset $\Omega$ of $\PP(V)$ and that $\Lambdao_{\Omega}(\Gamma)$ does not contain any nontrivial projective line segment.
The group $\Gamma$ acts cocompactly on the closed uniform $1$-neighborhood $\C$ of $\Ccore_\Omega(\Gamma)$ in $(\Omega,d_{\Omega})$, which is properly convex with nonempty interior (Lemma~\ref{lem:naive-cc-nonempty-int}).
By Corollary~\ref{cor:ideal-bound-naive-cc}.\eqref{item:ideal-bound-cc}, we have $\partiali\C=\Lambdao_{\Omega}(\Gamma)$.
\end{proof}

The proof of \ref{item:ccc-ugly}~$\Rightarrow$~\ref{item:ccc-noseg-some} relies on the following lemma, which is similar to \cite[Lem.\,4.3]{dgk-ccHpq}.

\begin{lemma} \label{lem:no-segment-contains-Ccore}
Let $\Gamma$ be an infinite discrete subgroup of $\PGL(V)$ preserving a properly convex open subset $\Omega$ of $\PP(V)$ and acting cocompactly on some closed convex subset $\C$ of~$\Omega$ with nonempty interior.
If $\partiali\C$ does not contain any nontrivial projective segment, then $\Lambdao_\Omega(\Gamma)\subset\partiali\C$ (hence $\Ccore_{\Omega}(\Gamma) \subset \C$).
\end{lemma}

\begin{proof}
Suppose that $\partiali\C$ does not contain any nontrivial projective segment.
Let us show that any point $z_{\infty} = \lim_m \gamma_m\cdot z \in\Lambdao_\Omega(\Gamma)$, where $(\gamma_m)\in\Gamma^{\NN}$ and $z \in \Omega$, belongs to $\partiali\C$.
For this, consider two distinct points $x, y \in \Int(\C)$ and $b \in \partial\Omega$ with $x,y,z,b$ aligned in this order.
Up to passing to a subsequence, we may assume that for any $w \in \{ x,y,b\}$, the sequence $(\gamma_m\cdot w)_{m\in\NN}$ converges to some $w_{\infty} \in \partial\Omega$, with $x_{\infty}, y_{\infty}, z_{\infty}, b_{\infty}$ aligned in this order (not necessarily distinct).
We have $[x_{\infty},y_{\infty}]\subset\partiali\C$, hence $x_{\infty} = y_{\infty}$ since $\partiali\C$ does not contain any nontrivial segment.
Since $\cro{\gamma_m\cdot x}{\gamma_m\cdot y}{\gamma_m\cdot z}{\gamma_m\cdot b} = \cro{x}{y}{z}{b} \in (1,+\infty)$ for all~$m$, and all segments $[\gamma_m \cdot x, \gamma_m \cdot b]$ and $[\gamma_\infty \cdot x, \gamma_\infty \cdot b]$ live in an affine chart containing~$\overline{\Omega}$, we must have $x_{\infty} = y_{\infty} = z_{\infty}$.
In particular, $z_{\infty} \in \partiali\C$.
\end{proof}

\begin{proof}[Proof of \ref{item:ccc-ugly}~$\Rightarrow$~\ref{item:ccc-noseg-some}]
If $\Gamma$ preserves a properly convex open subset $\Omega$ of $\PP(V)$ and acts cocompactly on some closed convex subset $\C$ of~$\Omega$ with nonempty interior such that $\partiali\C$ does not contain any nontrivial projective line segment, then $\C$ contains $\Ccore_{\Omega}(\Gamma)$ by Lemma~\ref{lem:no-segment-contains-Ccore}, and so Corollary~\ref{cor:ideal-bound-naive-cc}.\eqref{item:ideal-bound-cc} implies that $\Gamma$ acts convex cocompactly on~$\Omega$ and that $\Lambdao_{\Omega}(\Gamma) = \partiali\C$ does not contain any nontrivial projective line segment.
\end{proof}

\section{Convex cocompactness and no segment implies $P_1$-Anosov} \label{sec:ccc-implies-Anosov}

In this section we continue with the proof of Theorem~\ref{thm:main-noPETs}.
We have already established the equivalences \ref{item:ccc-hyp}~$\Leftrightarrow$~\ref{item:ccc-noseg-any}~$\Leftrightarrow$~\ref{item:ccc-noPETs-some}~$\Leftrightarrow$~\ref{item:ccc-ugly} in Section~\ref{sec:regularity-limit-set}.
On the other hand, the implication \ref{item:ccc-CM}~$\Rightarrow$~\ref{item:ccc-noseg-any} is trivial.

We now prove the implication \ref{item:ccc-hyp}~$\Rightarrow$~\ref{item:P1Anosov} in Theorem~\ref{thm:main-noPETs}.
By the above, this yields the implication \ref{item:ccc-CM}~$\Rightarrow$~\ref{item:P1Anosov} in Theorem~\ref{thm:main-noPETs}, which is also the implication \eqref{item:strong-proj-cc}~$\Rightarrow$~\eqref{item:Ano-PGL} in Theorem~\ref{thm:Ano-PGL}.
We build on Lemma~\ref{lem:C-hyp}.\eqref{item:Gamma-hyp}.

\subsection{Compatible, transverse, dynamics-preserving boundary maps}

Let $\Gamma$ be an infinite discrete subgroup of $\PGL(V)$.
Suppose $\Gamma$ is word hyperbolic and convex cocompact in $\PP(V)$.
By Proposition~\ref{prop:dual-precise}, there is a nonempty $\Gamma$-invariant properly convex open subset $\Omega$ of $\PP(V)$ such that the actions of $\Gamma$ on~$\Omega$ and on its dual~$\Omega^*$ are both convex cocompact.
Our goal is to show that the natural inclusion $\Gamma\hookrightarrow\PGL(V)$ is $P_1$-Anosov.

By the implication \ref{item:ccc-hyp}~$\Rightarrow$~\ref{item:ccc-noseg-any} in Theorem~\ref{thm:main-noPETs} (which we have proved in Section~\ref{sec:regularity-limit-set}) and Theorem~\ref{thm:properties}.\ref{item:dual} (which we have proved in Section~\ref{sec:bisat-dual}), the full orbital limit sets $\Lambdao_\Omega(\Gamma)\subset\PP(V)$ and $\Lambdao_{\Omega^*}(\Gamma)\subset\PP(V^*)$ do not contain any nontrivial projective line segment.
Let $\Ccore_{\Omega}\subset\PP(V)$ (\resp $\Ccore_{\Omega^*}\subset\PP(V^*)$) be the convex hull of $\Lambdao_\Omega(\Gamma)$ in~$\Omega$ (\resp of $\Lambdao_{\Omega^*}(\Gamma)$ in~$\Omega^*$).
By Lemma~\ref{lem:C-hyp}.\eqref{item:Gamma-hyp}, any orbit map $\Gamma \to (\Ccore_{\Omega},d_{\Omega})$ (\resp $\Gamma \to (\Ccore_{\Omega^*},d_{\Omega^*})$) is a quasi-isometry which extends to a $\Gamma$-equivariant homeomorphism
$$\xi : \partial_{\infty}\Gamma \longrightarrow \Lambdao_{\Omega}(\Gamma)\subset\PP(V)
\quad \quad (\text{resp. } ~ \xi^*: \partial_{\infty} \Gamma \longrightarrow \Lambdao_{\Omega^*}(\Gamma) \subset \PP(V^*)).$$
We see $\PP(V^*)$ as the space of projective hyperplanes of $\PP(V)$.
 
\begin{lemma} \label{lem:compatible-maps}
The boundary maps $\xi$ and~$\xi^*$ are compatible, \ie for any $\eta\in\partial_{\infty}\Gamma$ we have $\xi(\eta) \in \xi^*(\eta)$; more precisely, $\xi^*(\eta)$ is a supporting hyperplane of $\Omega$ at $\xi(\eta)$.
\end{lemma}

\begin{proof}
Let $(\gamma_m)_{m\in\NN}$ be a quasi-geodesic ray in $\Gamma$ with limit $\eta \in \partial_{\infty} \Gamma$.
For any $x \in \Omega$ and any $H \in \Omega^*$, we have $\gamma_m\cdot x \to \xi(\eta)$ and $\gamma_m \cdot H \to \xi^*(\eta)$.
Lift $x$ to a vector $v \in V$ and lift $H$ to a linear form $\varphi \in V^*$.
Lift the sequence $\gamma_m$ to a sequence $\hat \gamma_m \in \SL^{\pm} (V)$.
By Lemma~\ref{lem:vector-growth}, the sequences $(\hat \gamma_m\cdot v)_{m\in\NN}$ and $(\hat \gamma_m\cdot\varphi)_{m\in\NN}$ go to infinity in $V$ and $V^*$ respectively.
Since $(\hat \gamma_m \cdot \varphi)(\hat \gamma_m \cdot v) = \varphi(v)$ is independent of~$m$, we obtain $\xi(\eta)\in\xi^*(\eta)$ by passing to the limit.
Since $\xi^*(\eta)$ belongs to $\partial\Omega^*$, it is a supporting hyperplane of~$\Omega$.
\end{proof}

\begin{lemma}
The maps $\xi$ and~$\xi^*$ are transverse, \ie for any $\eta\neq\eta'$ in $\partial_{\infty}\Gamma$ we have $\xi(\eta) \notin \xi^*(\eta')$.
\end{lemma}

\begin{proof}
Consider $\eta,\eta'\in\partial_{\infty}\Gamma$ such that $\xi(\eta) \in \xi^*(\eta')$.
Let us check that $\eta=\eta'$.
By Corollary~\ref{cor:ideal-bound-naive-cc}.\eqref{item:ideal-bound-cc}, we have $\Lambdao_{\Omega}(\Gamma)=\partiali\Ccore_{\Omega}(\Gamma)$, and so the projective line segment $[\xi(\eta),\xi(\eta')]$, contained in the supporting hyperplane $\xi^*(\eta')$, is contained in $\Lambdao_{\Omega}(\Gamma)$.
Since $\Lambdao_\Omega(\Gamma)$ contains no nontrivial projective line segment by condition~\ref{item:ccc-noseg-any} of Theorem~\ref{thm:main-noPETs}, we deduce $\eta=\eta'$.
\end{proof}

For any $\gamma\in\Gamma$ of infinite order, we denote by $\eta_{\gamma}^+$ (\resp $\eta_{\gamma}^-$) the attracting (\resp repelling) fixed point of $\gamma$ in $\partial_{\infty}\Gamma$.

\begin{lemma} \label{lem:dyn-preserv}
The maps $\xi$ and~$\xi^*$ are dynamics-preserving.
\end{lemma}

\begin{proof}
We only prove it for~$\xi$; the argument is similar for~$\xi^*$.
We fix a norm $\Vert\cdot\Vert_{\scriptscriptstyle V}$ on~$V$.
Let $\gamma\in\nolinebreak\Gamma$ be an element of infinite order; we lift it to an element $\hat{\gamma}\in\SL^{\pm}(V)$ that preserves a properly convex cone of $V\smallsetminus\{0\}$ lifting~$\Omega$.
Let $L^+$ be the line of $V$ corresponding to $\xi(\eta_{\gamma}^+)$ and $H^-$ the hyperplane of $V$ corresponding to $\xi^*(\eta_{\gamma}^-)$.
By transversality, we have $V = L^+ \oplus H^-$, and this decomposition is preserved by~$\hat{\gamma}$.
Let $[v] \in \Omega$ and write $v = \ell^+ + h^-$ with $\ell^+ \in L^+$ and $h^- \in H^-$.
Since $\Omega$ is open, we may choose $v$ so that $\ell^+ \neq 0$ and $h^-$ satisfies $\Vert{\hat{\gamma}}^m\cdot h^- \Vert_{\scriptscriptstyle V} \geq \delta t^m \Vert h^- \Vert_{\scriptscriptstyle V} > 0$ for all $m\in\NN$, where $\delta > 0$ and where $t > 0$ is the spectral radius of the restriction of $\hat{\gamma}$ to~$H^-$.
On the other hand, ${\hat{\gamma}}^m \cdot \ell^+ = s^m \ell^+$ where $s$ is the eigenvalue of $\hat{\gamma}$ on~$L^+$.
By Lemma~\ref{lem:C-hyp}.\eqref{item:boundary-map-indep-of-orbit}, we have ${\hat{\gamma}}^m\cdot [v]\to\xi(\eta_{\gamma}^+)$ as $m\to +\infty$, hence
$$\frac{\Vert {\hat{\gamma}}^m\cdot h^-\Vert_{\scriptscriptstyle V}}{\Vert {\hat{\gamma}}^m\cdot\ell^+\Vert_{\scriptscriptstyle V}} \underset{m\to +\infty}{\longrightarrow} 0. $$
Necessarily $s>t$, and so $\xi(\eta_{\gamma}^+)$ is an attracting fixed point for the action of $\gamma$ on $\PP(V)$.
\end{proof}

As an immediate consequence of Lemmas~\ref{lem:C-hyp}.\eqref{item:Gamma-hyp} and~\ref{lem:dyn-preserv}, we obtain the following.

\begin{corollary} \label{cor:filled-lim-set-prox}
For any infinite-order element $\gamma\in\Gamma$, the element $\rho(\gamma)\in\PGL(V)$ is proximal in $\PP(V)$, and the full orbital limit set $\Lambdao_{\Omega}(\Gamma)$ is equal to the proximal limit set $\Lambda_\Gamma$ (see Definition~\ref{def:prox-lim-set}).
\end{corollary}

\subsection{The natural inclusion $\Gamma\hookrightarrow\PGL(V)$ is $P_1$-Anosov}

\begin{lemma} \label{lem:P1-div}
We have $(\mu_1-\mu_2)(\gamma)\to +\infty$ as $\gamma\to\infty$ in~$\Gamma$.
\end{lemma}

\begin{proof}
Consider a sequence $(\gamma_m)\in\Gamma^{\NN}$ going to infinity in~$\Gamma$.
Up to extracting we can assume that there exists $\eta\in\partial_{\infty}\Gamma$ such that $\gamma_m\to\eta$.
Then $\gamma_m\cdot x\to\xi(\eta)$ for all $x\in\Omega$ by Lemma~\ref{lem:C-hyp}.\eqref{item:Gamma-hyp}.
For any~$m$ we can write $\gamma_m = k_m a_m k'_m \in K \exp(\mathfrak{a}^+) K$ where $a_m=\mathrm{diag}(a_{i,m})_{1\leq i\leq n}$ with $a_{i,m}\geq a_{i+1,m}$ (see Section~\ref{subsec:Cartan}).
Up to extracting, we may assume that $(k_m)_{m\in\NN}$ converges to some $k\in K$ and $(k'_m)_{m\in\NN}$ to some $k'\in K$.
Let $(e_1,\dots,e_n)$ be the standard basis of $V = \RR^n$, orthonormal for the inner product preserved by $\hat{K} = \OO(n)$.
Since $\Omega$ is open, we can find points $x,y$ of $\Omega$ lifting respectively to $v,w\in V$ with
$$k'\cdot v = \sum_{i=1}^n s_i e_i \quad\quad \text{ and } \quad\quad k'\cdot w = \sum_{i=1}^n t_i e_i$$
such that $s_1=t_1$ but $s_2\neq t_2$.
For any~$m$, let us write
$$k'_m\cdot v = \sum_{i=1}^n s_{i,m} e_i \quad\quad \text{ and } \quad\quad k'_m\cdot w = \sum_{i=1}^n t_{i,m} e_i$$
where $s_{i,m}\to s_i$ and $t_{i,m}\to t_i$.
Then
$$a_m k'_m\cdot x = \left[s_{1,m} e_1 + \sum_{i=2}^n \frac{a_{i,m}}{a_{1,m}} \, s_{i,m} e_i\right], \quad a_m k'_m\cdot y = \left[t_{1,m} e_1 + \sum_{i=2}^n \frac{a_{i,m}}{a_{1,m}} \, t_{i,m} e_i\right].$$
The sequences $(\gamma_m\cdot x)_{m\in\NN}$ and $(\gamma_m\cdot y)_{m\in\NN}$ both converge to $\xi(\eta)$.
Hence the sequences $(a_mk'_m\cdot\nolinebreak x)_{m\in\NN}$ and $(a_mk'_m\cdot y)_{m\in\NN}$ both converge to the same point $k^{-1}\cdot\xi(\eta)\in\PP(V)$.
Since $s_1=t_1$ and $s_2\neq t_2$, we must have $a_{2,m}/a_{1,m}\to 0$, \ie $(\mu_1-\mu_2)(\gamma_m)\to +\infty$.
\end{proof}

By Fact~\ref{fact:charact-Ano}, the natural inclusion $\Gamma\hookrightarrow\PGL(V)$ is $P_1$-Anosov.
This completes the proof of the implication \ref{item:ccc-hyp}~$\Rightarrow$~\ref{item:P1Anosov} of Theorem~\ref{thm:main-noPETs}.

\section{$P_1$-Anosov implies convex cocompactness}\label{sec:Anosov-implies-ccc-noPETs}

In this section we prove the implications \ref{item:P1Anosov}~$\Leftrightarrow$~\ref{item:P1Anosov-bis}~$\Rightarrow$~\ref{item:ccc-hyp} in Theorem~\ref{thm:main-noPETs}.

\begin{proof}[Proof of \ref{item:P1Anosov}~$\Rightarrow$~\ref{item:P1Anosov-bis}]
Let $\Gamma$ be an infinite discrete subgroup of $\PGL(V)$.
Suppose that $\Gamma$ is word hyperbolic, that the inclusion $\Gamma\hookrightarrow\PGL(V)$ is $P_1$-Anosov, and that $\Gamma$ preserves some nonempty properly convex open subset $\mathcal{O}$ of $\PP(V)$.
Then the proximal limit set $\Lambda_\Gamma^*$ of $\Gamma$ in $\PP(V^*)$ is nonempty and, by Proposition~\ref{prop:max-inv-conv}, it lifts to a cone $\widetilde{\Lambda}_{\Gamma}^*$ of $V^*\smallsetminus\{0\}$ such that the convex open set
$$\Omega_{\max} := \PP(\{ v\in V ~|~ \varphi(v)>0\quad \forall\varphi\in\widetilde{\Lambda}_{\Gamma}^*\})$$
contains~$\mathcal{O}$ and is $\Gamma$-invariant.
The set $\Omega_{\max}$ is a connected component of $\PP(V) \smallsetminus \bigcup_{H\in\Lambda_{\Gamma}^*} H$ which contains $\mathcal O$.
\end{proof}

The implications \ref{item:P1Anosov-bis} $\Rightarrow $ \ref{item:ccc-hyp} and  \ref{item:P1Anosov-bis} $\Rightarrow$ \ref{item:P1Anosov} are contained in the next proposition.

\begin{proposition} \label{prop:precise-Ano-implies-ccc}
Let $\Gamma$ be an infinite discrete subgroup of $\PGL(V)$ which is word hyperbolic, such that the natural inclusion $\Gamma\hookrightarrow\PGL(V)$ is $P_1$-Anosov, with boundary maps $\xi : \partial_{\infty}\Gamma\to\PP(V)$ and $\xi^* : \partial_{\infty}\Gamma\to\PP(V^*)$.
Suppose that the proximal limit set $\Lambda_{\Gamma}^*=\xi^*(\partial_{\infty}\Gamma)$ lifts to a cone $\widetilde{\Lambda}_{\Gamma}^*$ of $V^*\smallsetminus\{0\}$ such that the convex open set
$$\Omega_{\max} := \PP(\{ v\in V ~|~ \varphi(v)>0\quad \forall\varphi\in\widetilde{\Lambda}_{\Gamma}^*\})$$
is nonempty and $\Gamma$-invariant.
Then $\Gamma$ acts convex cocompactly (Definition~\ref{def:cc-general}) on some nonempty properly convex open set $\Omega\subset\Omega_{\max}$ (which can be taken to be $\Omega_{\max}$ if $\Omega_{\max}$ is properly convex).
Moreover, the full orbital limit set $\Lambdao_\Omega(\Gamma)$ for any such~$\Omega$ is equal to the proximal limit set $\Lambda_{\Gamma}=\xi(\partial_{\infty}\Gamma)$.
\end{proposition}

The rest of the section is devoted to the proof of Proposition~\ref{prop:precise-Ano-implies-ccc}.

\subsection{Convergence for Anosov representations}

We first make the following general observation.

\begin{lemma} \label{lem:accumulation}
Let $\Gamma$ be an infinite discrete subgroup of $\PGL(V)$ which is word hyperbolic, such that the natural inclusion $\Gamma\hookrightarrow\PGL(V)$ is $P_1$-Anosov, with boundary maps $\xi : \partial_{\infty}\Gamma\to\PP(V)$ and $\xi^* : \partial_{\infty}\Gamma\to\PP(V^*)$.
Let $(\gamma_m)_{m\in\NN}$ be a sequence of elements of~$\Gamma$ converging to some $\eta\in\partial_{\infty}\Gamma$, such that $(\gamma_m^{-1})_{m\in\NN}$ converges to some $\eta'\in\partial_{\infty}\Gamma$.
Then
\begin{enumerate}
  \item\label{item:conv-xi} for any $x\in\PP(V)$ with $x\notin\xi^*(\eta')$ we have $\gamma_m\cdot x\to\xi(\eta)$;
  \item\label{item:conv-xi*} for any $x^*\in\PP(V^*)$ with $\xi(\eta')\notin x^*$ we have $\gamma_m\cdot x^*\to\xi^*(\eta)$.
\end{enumerate}
Moreover, convergence is uniform for $x, x^*$ ranging over compact sets.
\end{lemma}

\begin{proof}
The two statements are dual to each other, so we only need to prove~\eqref{item:conv-xi}.
For any $m\in\NN$, we write $\gamma_m = k_m a_m k'_m \in K\exp(\mathfrak{a}^+)K$ (see Section~\ref{subsec:Cartan}).
Up to extracting, $(k_m)$ converges to some $k\in K$, and $(k'_m)$ converges to some $k'\in K$.
Note that
$$\gamma_m^{-1} = ({k'_m}^{-1} w_0) (w_0 a_m^{-1} w_0) (w_0 k_m^{-1}) \in KA^+K, $$
where $w_0\in\PGL(V)\simeq\PGL(\RR^n)$ is the image of the permutation matrix exchanging $e_i$ and~$e_{n+1-i}$ for all $1\leq i \leq n$.
Therefore, by Fact~\ref{fact:Im-xi-Cartan}, we have
$$\xi^*(\eta') = {k'}^{-1}w_0\cdot\PP(\mathrm{span}(e_1,\dots,e_{n-1})).$$
In particular, for any $x\in\PP(V)$ with $x\notin\xi^*(\eta')$ we have
$$k'\cdot x \notin w_0\cdot\PP(\mathrm{span}(e_1,\dots,e_{n-1})) = \PP(\mathrm{span}(e_2,\dots,e_n)). $$
On the other hand, writing $a_m = \mathrm{diag}(a_{m,i})_{1\leq i\leq n}$, we have $a_{m,1}/a_{m,2}=e^{(\mu_1-\mu_2)(\gamma_m)}\to +\infty$ by Fact~\ref{fact:charact-Ano}.
Therefore $a_mk'_m\cdot x\to [e_1]$, and so $\gamma_m\cdot x=k_ma_mk'_m\cdot x\to k\cdot [e_1]$.
By Fact~\ref{fact:Im-xi-Cartan}, we have $k\cdot [e_1]=\xi(\eta)$.
This proves~\eqref{item:conv-xi}. 
Uniformity follows from the fact that $k'\cdot x=[e_1 + \sum_{i=2}^n t_i e_i]$ with \emph{bounded} $t_i\in\RR$, when $x$ ranges over a compact set disjoint from the hyperplane $\xi^*(\eta')$.
\end{proof}

\begin{corollary} \label{cor:Ano-Lambda-orb-prox}
In the setting of Lemma~\ref{lem:accumulation}, for any nonempty $\Gamma$-invariant properly convex open subset $\Omega$ of $\PP(V)$, the full orbital limit set $\Lambdao_{\Omega}(\Gamma)$ is equal to the proximal limit set $\Lambda_{\Gamma}=\xi(\partial_{\infty}\Gamma)$.
\end{corollary}

\begin{proof}
By Proposition~\ref{prop:max-inv-conv}, we have $\Lambda_{\Gamma}^*=\xi^*(\partial_{\infty}\Gamma)\subset\partial\Omega^*$, hence $x\notin\xi^*(\eta')$ for all $x\in\Omega$ and $\eta'\in\partial_{\infty}\Gamma$.
We then apply Lemma~\ref{lem:accumulation}.\eqref{item:conv-xi} to get both $\Lambda_{\Gamma} \subset \Lambdao_{\Omega}(\Gamma)$ and $\Lambdao_{\Omega}(\Gamma) \subset \Lambda_{\Gamma}$.
\end{proof}

\subsection{Proof of Proposition~\ref{prop:precise-Ano-implies-ccc}} \label{subsec:Ano-constr-Omega'}

Suppose $\Gamma \subset \PGL(V)$, $\Lambda_\Gamma$, $\Lambda_\Gamma^*$, and $\Omega_{\max}$ are as in the statement of Proposition~\ref{prop:precise-Ano-implies-ccc}.
Let $\C_0$ be the convex hull of $\Lambda_{\Gamma}$ in~$\Omega_{\max}$.
This is clearly defined if $\Omega_{\max}$ is properly convex, and otherwise it is defined as follows.

The preimage of $\Omega_{\max}$ in $V\smallsetminus\{0\}$ has two connected components $\widetilde{\Omega}_{\max}$ and $-\widetilde{\Omega}_{\max}$.
The intersection of their closures in~$V$ is a linear subspace $W$ of~$V$ (which is reduced to $\{0\}$ if $\Omega_{\max}$ is properly convex, in which case $\PP(W) = \emptyset$).
For any complementary subspace $W'$ of $W$ in~$V$, we can write $\widetilde{\Omega}_{\max}$ as a product of the form $W \times \widetilde{\Omega}'$ for some properly convex open cone $\widetilde{\Omega}' \subset W'$.
Since $\widetilde{\Omega}'$ is properly convex, we can choose a supporting hyperplane of $\widetilde{\Omega}_{\max}$ in~$V$ whose intersection with $\partial\widetilde{\Omega}_{\max}$ is reduced to~$W$.
Then $\Omega_{\max}$ is contained, and convex, in the corresponding affine chart of $\PP(V)$, and escapes to infinity precisely in the directions of $\PP(W)$. 
In $\PP(V)$, every projective hyperplane missing $\Omega_{\max}$, and in particular every $x^*\in \Lambda_{\Gamma}^*$, contains $\PP(W)$.
Since $\Lambda_\Gamma$ and $\Lambda_\Gamma^*$ are transverse, this implies that $\Lambda_\Gamma$ is disjoint from $\PP(W)$, hence (by compactness) it is bounded in the chart. 
By definition, $\C_0$ is the convex hull of $\Lambda_\Gamma$ in the chart: it is independent of the choice of chart, as it can be lifted to a properly convex cone contained in the convex cone $\widetilde{\Omega}_{\max}$.

\begin{lemma} \label{lem:bound-C0-Ano}
We have $\partiali\C_0 = \Lambda_{\Gamma}$.
\end{lemma}

\begin{proof}
By construction, $\partiali\C_0 = \overline{\C_0}\cap\partial\Omega_{\max} \supset \Lambda_{\Gamma}$.
Suppose a point $z \in \overline{\C_0}$ lies in $\partial\Omega_{\max}$.
Then $z$ lies in a hyperplane $\xi^*(\eta)$ for some $\eta \in \partial_{\infty} \Gamma$.
Any minimal subset $S$ of $\Lambda_{\Gamma}$ for which $z$ lies in the convex hull of $S$ must also be contained in $\xi^*(\eta)$.
By the transversality of $\xi$ and~$\xi^*$, we must have $S = \{ \xi(\eta)\}$, hence $z = \xi(\eta) \in \Lambda_{\Gamma}$.
\end{proof}

\begin{lemma} \label{lem:Omega'}
There exists a $\Gamma$-invariant properly convex open subset $\Omega \subset \Omega_{\max}$ containing~$\C_0$.
\end{lemma}

\begin{proof}
If $\Omega_{\max}$ is properly convex, then we take $\Omega = \Omega_{\max}$.
Suppose that $\Omega_{\max}$ fails to be properly convex (see Example~\ref{exa:aff-hp}).
Choose projective hyperplanes $H_1, \ldots, H_n$ bounding an open simplex $\Delta$ containing~$\overline{\C_0}$.
Define $\mathcal{V}:=\Omega_{\max}\cap \Delta$ and
$$\Omega := \bigcap_{\gamma \in \Gamma} \gamma \cdot \mathcal{V}.$$
Then $\Omega$ contains $\C_0$ and is properly convex.
We must show $\Omega$ is open.
Suppose not and let $z \in \partialn \Omega$.
Then there is a sequence $(\gamma_m) \in \Gamma^{\NN}$ and $i \in\{1, \ldots, n\}$ such that $\gamma_m\cdot H_i$ converges to a hyperplane $H_{\infty}$ containing~$z$.
By Lemma~\ref{lem:accumulation}.\eqref{item:conv-xi*}, the hyperplane $H_{\infty}$ lies in $\Lambda_{\Gamma}^*$, contradicting the fact that the hyperplanes of $\Lambda_{\Gamma}^*$ are disjoint from~$\Omega_{\max}$. 
\end{proof}

\begin{lemma}\label{lem:limit-sets-equal}
For any $\Gamma$-invariant properly convex open subset $\Omega$ of $\PP(V)$ containing~$\C_0$, we have $\Ccore_\Omega(\Gamma) = \C_0$.
\end{lemma}

\begin{proof}
By Corollary~\ref{cor:Ano-Lambda-orb-prox}, we have $\Lambdao_\Omega(\Gamma) = \Lambda_\Gamma$, hence $\Ccore_\Omega(\Gamma)$ is the convex hull of $\Lambda_\Gamma$ in~$\Omega$, \ie $\Ccore_\Omega(\Gamma)=\C_0\cap\Omega$.
But $\C_0$ is contained in~$\Omega$ by construction (see Lemma~\ref{lem:Omega'}), hence $\Ccore_\Omega(\Gamma) = \C_0$.
\end{proof}

In order to conclude the proof of Proposition~\ref{prop:precise-Ano-implies-ccc}, it only remains to check the following.

\begin{lemma} \label{lem:cocompact}
The action of $\Gamma$ on $\C_0$ is cocompact.
\end{lemma}

\begin{proof}
We endow $\PP(V)$ with the spherical metric $d_\SS(\cdot, \cdot)$ induced by some Euclidean norm on $V$.
By \cite[Th.\,1.7]{klp14} (see also \cite[Rem.\,5.15]{ggkw17}), the action of $\Gamma$ on $\PP(V)$ at any point $z\in\Lambda_{\Gamma}$ is expanding: there exist an element $\gamma\in\Gamma$, a neighborhood $\mathcal{U}$ of $z$ in $\PP(V)$, and a constant $c>1$ such that $\gamma$ is $c$-expanding on $\mathcal{U}$ for $d_{\SS}$.
We now use a version of the argument of \cite[Prop.\,2.5]{klp14}, inspired by Sullivan's dynamical characterization \cite{sul85} of convex cocompactness in the real hyperbolic space.
(The argument in \cite{klp14} is a little more technical because it deals with bundles, whereas we work directly in $\PP(V)$.)

Suppose by contradiction that the action of $\Gamma$ on $\C_0$ is \emph{not} cocompact, and let $(\varepsilon_m)_{m\in\NN}$ be a sequence of positive reals going to~$0$.
For any~$m$, the set $K_m := \{ x\in\C_0 \,|\, d_{\SS}(x,\Lambda_{\Gamma}) \geq\nolinebreak \varepsilon_m\}$ is compact (Lemma~\ref{lem:bound-C0-Ano}), hence there exists a $\Gamma$-orbit contained in $\C_0\smallsetminus K_m$. 
By proper discontinuity of the action on $\C_0$, the supremum of $d_{\SS}(\cdot,\Lambda_{\Gamma})$ on this orbit is achieved at some point $x_m\in\C_0$, and by construction $0 < d_{\mathbb{S}}(x_m,\Lambda_{\Gamma}) \leq \varepsilon_m$.
Then, for all $\gamma\in\Gamma$,
$$d_{\SS}(\gamma\cdot x_m,\Lambda_{\Gamma}) \leq d_{\SS}(x_m,\Lambda_{\Gamma}). $$
Up to extracting, we may assume that \((x_m)_{m\in\NN}\) converges to some $z\in\Lambda_{\Gamma}$.
Consider an element $\gamma\in\Gamma$, a neighborhood $\mathcal{U}$ of $z$ in $\PP(V)$, and a constant $c>1$ such that $\gamma$ is $c$-expanding on~$\mathcal{U}$.
For any $m\in\NN$, there exists $z_m\in\Lambda_{\Gamma}$ such that $d_{\SS}(\gamma\cdot x_m,\Lambda_{\Gamma}) = d_{\SS}(\gamma\cdot x_m,\gamma\cdot z_m)$.
For large enough~$m$ we have $x_m,z_m\in\mathcal{U}$, and so
$$d_{\SS}(\gamma\cdot x_m,\Lambda_{\Gamma}) \geq c\, d_{\SS}( x_m, z_m) \geq c\, d_{\SS}(x_m,\Lambda_{\Gamma}) \geq c\, d_{\SS}(\gamma\cdot x_m,\Lambda_{\Gamma})>0.$$
This is impossible since $c>1$.
\end{proof}

This shows that the word hyperbolic group $\Gamma$ acts convex cocompactly on any $\Gamma$-invariant properly convex open subset $\Omega$ of $\PP(V)$ containing~$\C_0$, hence condition~\ref{item:ccc-hyp} of Theorem~\ref{thm:main-noPETs} is satisfied.

\subsection{The case of groups with connected boundary}

Here is an immediate consequence of Propositions \ref{prop:max-inv-conv} and~\ref{prop:precise-Ano-implies-ccc} and Lemma~\ref{lem:Omega-max-connected}.

\begin{corollary} \label{cor:Ano->cc-connected}
Let $\Gamma$ be an infinite discrete subgroup of $\PGL(V)$ which is word hyperbolic with \emph{connected} boundary $\partial_{\infty}\Gamma$, which preserves a properly convex open subset of $\PP(V)$, and such that the natural inclusion $\Gamma\hookrightarrow\PGL(V)$ is $P_1$-Anosov.
Then the set
$$\Omega_{\max} := \PP(V) \smallsetminus \bigcup_{z^*\in\Lambda_{\Gamma}^*} z^*$$
is a nonempty $\Gamma$-invariant convex open subset of $\PP(V)$, containing all other such sets.
If $\Omega_{\max}$ is properly convex (\eg if the action of $\Gamma$ on $\PP(V)$ is irreducible), then $\Gamma$ acts convex cocompactly on $\Omega_{\max}$.
\end{corollary}

\section{Smoothing out the nonideal boundary} \label{sec:smooth}

We now make the connection with strong projective convex cocompactness and prove the remaining implications of Theorems \ref{thm:main-noPETs} and~\ref{thm:main-general}.

Concerning Theorem~\ref{thm:main-noPETs}, the equivalences \ref{item:ccc-hyp}~$\Leftrightarrow$~\ref{item:ccc-noseg-any}~$\Leftrightarrow$~\ref{item:ccc-noPETs-some}~$\Leftrightarrow$~\ref{item:ccc-ugly} have been proved in Section~\ref{sec:regularity-limit-set}, the implication \ref{item:ccc-hyp}~$\Rightarrow$~\ref{item:P1Anosov} in Section~\ref{sec:ccc-implies-Anosov}, and the implications~\ref{item:P1Anosov}~$\Leftrightarrow$~\ref{item:P1Anosov-bis}~$\Rightarrow$~\ref{item:ccc-hyp}  in Section~\ref{sec:Anosov-implies-ccc-noPETs}.
The implication \ref{item:ccc-CM}~$\Rightarrow$~\ref{item:ccc-noseg-any} is trivial.
We shall prove \ref{item:ccc-noseg-any}~$\Rightarrow$~\ref{item:ccc-CM} in Section~\ref{subsec:ccc-implies-strict-C1}, which will complete the proof of Theorem~\ref{thm:main-noPETs}.
It will also complete the proof of Theorem~\ref{thm:Ano-PGL}, since the implication \eqref{item:Ano-PGL}~$\Rightarrow$~\eqref{item:strong-proj-cc} of Theorem~\ref{thm:Ano-PGL} is the implication \ref{item:P1Anosov}~$\Rightarrow$~\ref{item:ccc-CM} of Theorem~\ref{thm:main-noPETs}.

Concerning Theorem~\ref{thm:main-general}, the equivalence \eqref{item:ccc-limit-set}~$\Leftrightarrow$~\eqref{item:ccc-bisat} has been proved in Section~\ref{sec:bisat-dual}.
The implication \eqref{item:ccc-strict-C1}~$\Rightarrow$~\eqref{item:ccc-strict} is immediate.
We shall prove the implication \eqref{item:ccc-limit-set}~$\Rightarrow$~\eqref{item:ccc-strict-C1} in Section~\ref{subsec:ccc-implies-strict-C1}.
The implication \eqref{item:ccc-strict}~$\Rightarrow$~\eqref{item:ccc-bisat} is an immediate consequence of Lemma~\ref{lem:closed-ideal-boundary} and of the following lemma, and completes the proof of Theorem~\ref{thm:main-general}. 
We refer to Definition~\ref{def:boundaries} for the various notions of boundary regularity used throughout this section.

\begin{lemma} \label{lem:strict-conv-implies-bisat}
Let $\C_{\mathsf{strict}}$ be a nonempty convex subset of $\PP(V)$ with strictly convex nonideal boundary $\partialn\C_{\mathsf{strict}}$ and whose ideal boundary $\partiali\C_{\mathsf{strict}}$ is closed in $\PP(V)$.
Then $\C_{\mathsf{strict}}$ has bisaturated boundary.
\end{lemma}

\begin{proof}
Let $H$ be a supporting hyperplane to $\C_{\mathsf{strict}}$.
Then $H \cap \Fr(\C_{\mathsf{strict}}) = H \cap \overline{\C_{\mathsf{strict}}}$ is a closed, convex subset of the frontier $\Fr(\C_{\mathsf{strict}}) = \partialn\C_{\mathsf{strict}} \cup \partiali\C_{\mathsf{strict}}$.
Since $\partiali \C_{\mathsf{strict}}$ is closed, so is $H \cap \partiali \C_{\mathsf{strict}}$, hence $H \cap \partialn \C_{\mathsf{strict}}$ is open in $H \cap \Fr(\C_{\mathsf{strict}})$.
However, $\partialn\C_{\mathsf{strict}}$ is strictly convex.
Therefore, either $H \cap \partialn\C_{\mathsf{strict}}$ is empty, or $H \cap \partialn\C_{\mathsf{strict}} = H \cap \Fr(\C_{\mathsf{strict}})$ is a single point.
In particular, $H \cap \Fr(\C_{\mathsf{strict}})$ is entirely contained in $\partiali\C_{\mathsf{strict}}$ or entirely contained in $\partialn\C_{\mathsf{strict}}$.
This shows that $\C_{\mathsf{strict}}$ has bisaturated boundary.
\end{proof}

\subsection{Smoothing out the nonideal boundary}

Here is the main result of Section~\ref{sec:smooth}.

\begin{lemma}\label{lem:strict-C1-domain}
Let $\Gamma$ be an infinite discrete subgroup of $\PGL(V)$ and $\Omega$ a nonempty $\Gamma$-invariant properly convex open subset of $\PP(V)$.
Suppose $\Gamma$ acts convex cocompactly on~$\Omega$.
Fix a uniform neighborhood $\C_{\mathsf{unif}}$ of $\Ccore_{\Omega}(\Gamma)$ in $(\Omega,d_{\Omega})$.
Then the convex core $\Ccore_{\Omega}(\Gamma)$ admits a $\Gamma$-invariant, properly convex, closed neighborhood $\C\subset\C_{\mathsf{unif}}$ in~$\Omega$ which has $C^1$, strictly convex nonideal boundary.
\end{lemma}

Constructing a neighborhood~$\C$ as in the lemma clearly involves arbitrary choices; here is one of many possible constructions, taken from \cite[Lem.\,6.4]{dgk-ccHpq}.
Cooper--Long--Tillmann~\cite[Prop.\,8.3]{clt18} give a different construction yielding, in the case that $\Gamma$ is torsion-free, a convex set $\C$ as in the lemma whose nonideal boundary has the slightly stronger property that it is locally the graph of a smooth function with positive definite Hessian.

\begin{proof}[Proof of Lemma~\ref{lem:strict-C1-domain}]
In this proof, we fix a finite-index subgroup $\Gamma_0$ of~$\Gamma$ which is torsion-free; such a subgroup exists by the Selberg lemma \cite[Lem.\,8]{sel60}.

We proceed in three steps.
Firstly, we construct a $\Gamma$-invariant closed neighborhood $\C_{\scriptscriptstyle\clubsuit}$ of $\Ccore_{\Omega}(\Gamma)$ in~$\Omega$ which is contained in~$\C_{\mathsf{unif}}$ and whose nonideal boundary is $C^1$ but not necessarily strictly convex.
Secondly, we construct a small deformation $\C_{\scriptscriptstyle\diamondsuit}\subset\C_{\mathsf{unif}}$ of~$\C_{\scriptscriptstyle\clubsuit}$ which has $C^1$ and strictly convex nonideal boundary, but which is only $\Gamma_0$-invariant, not necessarily $\Gamma$-invariant.
Finally, we use an averaging procedure over translates $\gamma\cdot\C_{\scriptscriptstyle\diamondsuit}$ of~$\C_{\scriptscriptstyle\diamondsuit}$, for $\gamma\Gamma_0$ ranging over the $\Gamma_0$-cosets of~$\Gamma$, to construct a $\Gamma$-invariant closed neighborhood $\C\subset\C_{\mathsf{unif}}$ of~$\Ccore_{\Omega}(\Gamma)$ in~$\Omega$ which has $C^1$ and strictly convex nonideal boundary.

\smallskip
\noindent
$\bullet$ \textbf{Construction of~$\C_{\scriptscriptstyle\clubsuit}$:}
Consider a compact fundamental domain $\mathcal{D}$ for the action of~$\Gamma$ on~$\Ccore_{\Omega}(\Gamma)$.
The convex hull of $\mathcal{D}$ in~$\Omega$ is still contained in~$\Ccore_{\Omega}(\Gamma)$.
Let $\mathcal{D}'\subset\C_{\mathsf{unif}}$ be a closed neighborhood of this convex hull in~$\Omega$ which has frontier of class $C^1$, and let $\C_{\scriptscriptstyle\clubsuit}\subset\C_{\mathsf{unif}}$ be the closure in~$\Omega$ of the convex hull of $\Gamma\cdot\mathcal{D}'$.
By Corollary~\ref{cor:ideal-bound-naive-cc}.\eqref{item:ideal-bound-cc} we have $\partiali\C_{\scriptscriptstyle\clubsuit}=\Lambdao_{\Omega}(\Gamma)$, and by Lemma~\ref{lem:1-neighb-bisat} the set $\C_{\scriptscriptstyle\clubsuit}$ has bisaturated boundary.

Let us check that $\C_{\scriptscriptstyle\clubsuit}$ has $C^1$ nonideal boundary.
We first observe that any supporting hyperplane $\Pi_x$ of $\C_{\scriptscriptstyle\clubsuit}$ at a point $x\in\partialn\C_{\scriptscriptstyle\clubsuit}$ stays away from $\partiali\C_{\scriptscriptstyle\clubsuit}$ since $\C_{\scriptscriptstyle\clubsuit}$ has bisaturated boundary.
On the other hand, since the action of $\Gamma$ on~$\C_{\scriptscriptstyle\clubsuit}$ is properly discontinuous, for any neighborhood $\mathcal{N}$ of $\partiali\C_{\scriptscriptstyle\clubsuit}$ in $\PP(V)$ and any infinite sequence of distinct elements $\gamma_j\in\Gamma$, the translates $\gamma_j \cdot \mathcal{D}'$ are eventually all contained in~$\mathcal{N}$.
Therefore, in a neighborhood of~$x$, the hypersurface $\Fr(\C_{\scriptscriptstyle\clubsuit})$ coincides with the convex hull of a \emph{finite} union of translates $\bigcup_{i=1}^m \gamma_i\cdot\mathcal{D}'$, and so it is locally~$C^1$: indeed that convex hull is dual to $\bigcap_{i=1}^m (\gamma_i\cdot\mathcal{D}')^*$, which has \emph{strictly} convex frontier because $(\mathcal{D}')^*$ does (a convex set has~$C^1$ frontier if and only if its dual has strictly convex frontier).

\smallskip
\noindent
$\bullet$ \textbf{Construction of~$\C_{\scriptscriptstyle\diamondsuit}$:}
For any $x\in\partialn\C_{\scriptscriptstyle\clubsuit}$, let $F_x$ be the open face of $\Fr(\C_{\scriptscriptstyle\clubsuit})$ at~$x$, namely the intersection of $\overline{\C_{\scriptscriptstyle\clubsuit}}$ with the unique supporting hyperplane $\Pi_x$ at~$x$.
Since $\C_{\scriptscriptstyle\clubsuit}$ has bisaturated boundary, $F_x$ is a compact convex subset of $\partialn\C_{\scriptscriptstyle\clubsuit}$.

We claim that $F_x$ is disjoint from $\gamma\cdot F_x=F_{\gamma\cdot x}$ for all $\gamma\in\Gamma_0\smallsetminus\{\mathrm{Id}\}$.
Indeed, if there existed $y\in F_x\cap F_{\gamma\cdot x}$, then by uniqueness the supporting hyperplanes would satisfy $\Pi_x=\Pi_y=\Pi_{\gamma\cdot x}$, hence $F_x = F_y = F_{\gamma\cdot x}=\gamma\cdot F_x$.
This would imply $F_x = \gamma^m\cdot F_x$ for all $m\in\NN$, hence $\gamma^m\cdot x\in F_x$.
Using the fact that the action of $\Gamma_0$ on $\C_{\scriptscriptstyle\clubsuit}$ is properly discontinuous and taking a limit, we see that $F_x$ would contain a point of~$\partiali\C_{\scriptscriptstyle\clubsuit}$, which we have seen is not true.
Therefore $F_x$ is disjoint from $\gamma\cdot F_x$ for all $\gamma\in\Gamma_0\smallsetminus\{\mathrm{Id}\}$.

For any $x\in\partialn\C_{\scriptscriptstyle\clubsuit}$, the subset of $\PP(V^*)$ consisting of those projective hyperplanes near the supporting hyperplane $\Pi_x$ that separate $F_x$ from $\partiali \C_{\scriptscriptstyle\clubsuit}$ is open and nonempty, hence $(n-1)$-dimensional where $\dim(V)=n$.
Choose $n-1$ such hyperplanes $\Pi_x^1,\dots, \Pi_x^{n-1}$ in generic position, with $\Pi_x^i$ cutting off a compact region $\mathcal{Q}_x^i\supset F_x$ from~$\C_{\scriptscriptstyle\clubsuit}$.
One may imagine each $\Pi_x^i$ is obtained by pushing $\Pi_x$ normally into $\C_{\scriptscriptstyle\clubsuit}$ and then tilting slightly in one of $n-1$ independent directions.
The intersection $\bigcap_{i=1}^{n-1} \Pi_x^i \subset \PP(V)$ is reduced to a singleton.
By taking each hyperplane $\Pi_x^i$ very close to $\Pi_x$, we may assume that the union $\mathcal{Q}_x:=\bigcup_{i=1}^{n-1} \mathcal{Q}_x^i$ is disjoint from all its $\gamma$-translates for $\gamma \in \Gamma_0 \smallsetminus \{\mathrm{Id}\}$.
In addition, we ensure that $F_x$ has a neighborhood $\mathcal{Q}'_x$ contained in $\bigcap_{i=1}^{n-1} \mathcal{Q}_x^i$.

Since the action of $\Gamma_0$ on $\partialn\C_{\scriptscriptstyle\clubsuit}$ is cocompact, there exist finitely many points $x_1,\dots,x_m\in\partialn\C_{\scriptscriptstyle\clubsuit}$ such that $\partialn\C_{\scriptscriptstyle\clubsuit}\subset \Gamma_0\cdot (\mathcal{Q}'_{x_1}\cup\dots\cup \mathcal{Q}'_{x_m})$.
 
We now explain, for any $x\in\partialn\C_{\scriptscriptstyle\clubsuit}$, how to deform $\C_{\scriptscriptstyle\clubsuit}$ into a new, smaller properly convex $\Gamma_0$-invariant closed neighborhood of~$\Ccore_{\Omega}(\Gamma)$ in~$\Omega$ with $C^1$ nonideal boundary, in a way that destroys all segments in $\mathcal{Q}'_x$.
Repeating for $x=x_1, \dots, x_m$, this will produce a properly convex $\Gamma_0$-invariant closed neighborhood $\C_{\scriptscriptstyle\diamondsuit}$ of~$\Ccore_{\Omega}(\Gamma)$ in~$\Omega$ with $C^1$ and strictly convex nonideal boundary.

Choose an affine chart containing~$\Omega$, and an auxiliary Euclidean metric $g$ on this chart.
We may assume that for every $1\leq i \leq n-1$, the $g$-orthogonal projection $\pi_x^i$ onto $\Pi_x^i$ defines a homeomorphism $\mathcal{Q}_x^i \cap \partialn \C_{\scriptscriptstyle\clubsuit} \to \C_{\scriptscriptstyle\clubsuit} \cap \Pi_x^i$. Define a map $\varphi_x^i: \mathcal{Q}_x^i \rightarrow \mathcal{Q}_x^i$ by the property that $\varphi_x^i$ preserves each fiber $(\pi_x^i)^{-1}(y) \cap \mathcal{Q}_x^i$ (a segment), taking the point at distance $t$ from $y$ to the point at distance~$\tanh(t)$. 

We claim that $\varphi_x^i(\mathcal Q_x^i) \cup (\C_{\scriptscriptstyle\clubsuit} \smallsetminus \mathcal Q_x^i)$ is convex with boundary of class~$C^1$ and that, further, any segment in $\varphi_x^i(\partialn \C_{\scriptscriptstyle\clubsuit} \cap \mathcal{Q}_x^i)$ is the image of a segment of $\partialn \C_{\scriptscriptstyle\clubsuit} \cap \mathcal{Q}_x^i$ which is parallel to~$\Pi_x^i$.
To see this, decompose the affine chart orthogonally as $\Pi_x^i \oplus \RR$, so that $\mathcal{Q}_x^i \cap \partialn \C_{\scriptscriptstyle\clubsuit}$ is the graph of a concave function $f_x^i: \C_{\scriptscriptstyle\clubsuit} \cap \Pi_x^i \to \RR$ of class~$C^1$, vanishing on the boundary $\partialn  \C_{\scriptscriptstyle\clubsuit}\cap \Pi_x^i$.
 Replacing $\mathcal{Q}_x^i$ with $\varphi_x^i(\mathcal Q_x^i)$ amounts to replacing $f_x^i$ with $\tanh \circ f_x^i$. The claim follows from the fact that $\tanh$ is strictly concave, monotonic, smooth, and tangent to the identity at~$0$ (any other function with these properties may be used in place of $\tanh$).

Extending $\varphi_x^i$ by the identity on $\mathcal{Q}_x\smallsetminus \mathcal{Q}_x^i$ and repeating with varying $i$, we find that the composition $\varphi_x:=\varphi_x^1\circ\dots\circ\varphi_x^{n-1}$, defined on $\mathcal{Q}_x$, takes $\mathcal{Q}'_x\cap \partialn \C_{\scriptscriptstyle\clubsuit}$ to a strictly convex hypersurface. Indeed, a segment of $\mathcal{Q}'_x\cap \partialn \C_{\scriptscriptstyle\clubsuit}$ cannot be parallel to all of $\Pi_x^1, \ldots, \Pi_x^{n-1}$ since these hyperplanes were chosen to intersect in a singleton.
We can extend $\varphi_x$ in a $\Gamma_0$-equivariant fashion to $\Gamma_0\cdot\mathcal{Q}_x$, and extend it further by the identity on the rest of~$\C_{\scriptscriptstyle\clubsuit}$: the set $\varphi_x(\C_{\scriptscriptstyle\clubsuit})$ is still a $\Gamma_0$-invariant closed neighborhood of $\Ccore_{\Omega}(\Gamma)$ in~$\Omega$, contained in~$\C_{\mathsf{unif}}$, with $C^1$ nonideal boundary.

Repeating with finitely many points $x_1, \dots, x_m$ as above, we obtain a $\Gamma_0$-invariant properly convex closed neighborhood $\C_{\scriptscriptstyle\diamondsuit}\subset\C_{\mathsf{unif}}$ of~$\Ccore_{\Omega}(\Gamma)$ in~$\Omega$ with $C^1$ and strictly convex boundary.

\smallskip
\noindent
$\bullet$ \textbf{Construction of~$\C$:}
Consider the finitely many $\Gamma_0$-cosets $\gamma_1\Gamma_0, \dots, \gamma_k\Gamma_0$ of~$\Gamma$ and the corresponding translates $\C_{\scriptscriptstyle\diamondsuit}^i:=\gamma_i\cdot\C_{\scriptscriptstyle\diamondsuit}$; we denote by $\Omega^i$ the interior of~$\C_{\scriptscriptstyle\diamondsuit}^i$.
Let $\C'$ be a $\Gamma$-invariant properly convex closed neighborhood of~$\Ccore_{\Omega}(\Gamma)$ in $\Omega$ which has $C^1$ (but not necessarily strictly convex) nonideal boundary and is contained in all $\C_{\scriptscriptstyle\diamondsuit}^i$, $1\leq i\leq k$.
(Such a neighborhood $\C'$ can be constructed for instance by the same method as $\C_{\scriptscriptstyle\clubsuit}$ above.)
Since $\C_{\scriptscriptstyle\diamondsuit}^i$ has strictly convex nonideal boundary, uniform neighborhoods of $\C'$ in $(\Omega^i,d_{\Omega^i})$ have strictly convex nonideal boundary \cite[(18.12)]{bus55}.
Therefore, by cocompactness, if $h : [0,1]\to [0,1]$ is a convex function with sufficiently fast growth (\eg $h(t)=t^{\alpha}$ for large enough $\alpha>0$), then the $\Gamma_0$-invariant function $H_i := h\circ d_{\Omega^i}(\cdot, \C')$ is convex on the convex region $H_i^{-1}([0,1])$, and in fact smooth and strictly convex near every point outside~$\C'$.
The function $H:=\sum_{i=1}^k H_i$ is $\Gamma$-invariant and its sublevel set $\C:=H^{-1}([0,1])$ is a $\Gamma$-invariant closed neighborhood of $\Ccore_{\Omega}(\Gamma)$ in~$\Omega$ which has $C^1$, strictly convex nonideal boundary.
Moreover, $\C\subset \C_{\scriptscriptstyle\diamondsuit}\subset \C_{\scriptscriptstyle\clubsuit}\subset \C_{\mathsf{unif}}$ by construction.
\end{proof}

\subsection{Proof of the remaining implications of Theorems \ref{thm:main-noPETs} and~\ref{thm:main-general}} \label{subsec:ccc-implies-strict-C1}

\begin{proof}[Proof of \eqref{item:ccc-limit-set}~$\Rightarrow$~\eqref{item:ccc-strict-C1} in Theorem~\ref{thm:main-general}]
Suppose $\Gamma$ is convex cocompact in $\PP(V)$, \ie it preserves a properly convex open set $\Omega \subset \PP(V)$ and acts cocompactly on the convex core $\Ccore_{\Omega}(\Gamma)$.
Let $\C_{\mathsf{unif}}$ be the closed uniform $1$-neighborhood of $\Ccore_{\Omega}(\Gamma)$ in $(\Omega,d_{\Omega})$.
By Lemma~\ref{lem:naive-cc-nonempty-int}, the set $\C_{\mathsf{unif}}$ is properly convex and the action of $\Gamma$ on~$\C_{\mathsf{unif}}$ is properly discontinuous and cocompact.
By Lemma~\ref{lem:strict-C1-domain}, the set $\Ccore_{\Omega}(\Gamma)$ admits a $\Gamma$-invariant, properly convex, closed neighborhood $\C_{\mathsf{smooth}}\subset\C_{\mathsf{unif}}$ in~$\Omega$ which has $C^1$, strictly convex nonideal boundary.
The action of $\Gamma$ on $\C_{\mathsf{smooth}}$ is still properly discontinuous and cocompact, and $\partiali\C_{\mathsf{smooth}} = \Lambdao_{\Omega}(\Gamma)$ by Corollary~\ref{cor:ideal-bound-naive-cc}.\eqref{item:ideal-bound-cc}.
This proves the implication \eqref{item:ccc-limit-set}~$\Rightarrow$~\eqref{item:ccc-strict-C1} in Theorem~\ref{thm:main-general}.
\end{proof}

\begin{proof}[Proof of \ref{item:ccc-noseg-any}~$\Rightarrow$~\ref{item:ccc-CM} in Theorem~\ref{thm:main-noPETs}]
Suppose  that $\Gamma$ acts convex cocompactly on a properly convex open set $\Omega$ and that the full orbital limit set $\Lambdao_\Omega(\Gamma)$ contains no nontrivial segment.
As in the proof of \eqref{item:ccc-limit-set}~$\Rightarrow$~\eqref{item:ccc-strict-C1} in Theorem~\ref{thm:main-general} just above, $\Gamma$ acts properly discontinuously and cocompactly on some nonempty closed properly convex subset $\C_{\mathsf{smooth}}$ of~$\Omega$ which has strictly convex and $C^1$ nonideal boundary and whose interior $\Omega_{\mathsf{smooth}} := \Int(\C_{\mathsf{smooth}})$ contains $\Ccore_{\Omega}(\Gamma)$.
The set $\C_{\mathsf{smooth}}$ has bisaturated boundary (Lemma~\ref{lem:strict-conv-implies-bisat}), hence the action of $\Gamma$ on $\Omega_{\mathsf{smooth}}$ is convex cocompact by Corollary~\ref{cor:bisat-interior}.

By Corollary~\ref{cor:ideal-bound-naive-cc}.\eqref{item:ideal-bound-cc}, the ideal boundary $\partiali\C_{\mathsf{smooth}}$ is equal to $\Lambdao_{\Omega}(\Gamma)$.
Since $\Lambdao_{\Omega}(\Gamma)$ contains no nontrivial segment by assumption, and since it is a union of faces of $\partial\Omega$ (Corollaries \ref{cor:ideal-bound-naive-cc}.\eqref{item:Lambda-con-naive-cc} and~\ref{cor:Lambda-con-faces}), we deduce that any $z \in \partiali\C_{\mathsf{smooth}}$ is an extreme point of~$\partial \Omega$.
Thus the full boundary of $\Omega_{\mathsf{smooth}}$ is strictly convex.

Consider the dual convex set $\C_{\mathsf{smooth}}^* \subset \PP(V^*)$ (Definition~\ref{def:dual-C}); it is properly convex.
By Lemma~\ref{lem:dual}, the set $\C_{\mathsf{smooth}}^*$ has bisaturated boundary and does not contain any PET (since $\C_{\mathsf{smooth}}$ itself does not contain any PET).
By Proposition~\ref{prop:cc-duality}, the action of $\Gamma$ on $\C_{\mathsf{smooth}}^*$ is properly discontinuous and cocompact.
It follows from Lemma~\ref{lem:segments-imply-PETs} that $\partiali \C_{\mathsf{smooth}}^*$ contains no nontrivial segment, hence each point of $\partiali \C_{\mathsf{smooth}}^*$ is an extreme point. 
Hence there is exactly one hyperplane supporting $\C_{\mathsf{smooth}}$ at any given point of $\partiali\C_{\mathsf{smooth}}$.
This is also true at any given point of $\partialn\C_{\mathsf{smooth}}$ by assumption.
Thus the boundary of $\Omega_{\mathsf{smooth}}$ is of class~$C^1$.
\end{proof}

\section{Properties of convex cocompact groups} \label{sec:other-properties}

In this section we prove Theorem~\ref{thm:properties}.
Property~\ref{item:dual} has already been established in Section~\ref{subsec:proof-dual}; we now establish the other properties.

\subsection{\ref{item:cc-QI}: Quasi-isometric embedding}

We first prove the following very general result, using the notation $\mu_1-\mu_n$ from \eqref{eqn:mu-i-j}.
The fact that a statement of this flavor should exist was suggested to us by Yves Benoist and Pierre-Louis Blayac.

\begin{proposition} \label{prop:d-Omega-mu}
For any properly convex open subset $\Omega$ of $\PP(V) = \PP(\RR^n)$ and any $x\in\Omega$, there exists $\kappa_{\Omega,x}\geq 0$ such that for any $g\in\mathrm{Aut}(\Omega)$,
\begin{equation} \label{eqn:hilball}
\bigg| d_{\Omega}(x,g\cdot x) - \frac{1}{2} (\mu_1-\mu_n)(g) \bigg| \leq \kappa_{\Omega,x}.
\end{equation}
Moreover, $\kappa_{\Omega,x}$ can be taken to be uniform as $(\Omega,x)$ varies in a compact subset of the set of pointed properly convex open subsets of $\PP(V)$, endowed with the pointed Hausdorff topology.
\end{proposition}

Note that we can take $\kappa_{\Omega,x} = 0$ in the following examples, where $(e_1,\dots,e_n)$ is the standard basis of $V=\RR^n$:
\begin{itemize}
  \item $\Omega = \HH^{n-1} = \{ [v] \in \PP(\RR^n) \,|\, v_1^2 + \dots + v_{n-1}^2 - v_n^2 < 0\}$ (see Example~\ref{ex:H-p}) and $x = [e_n]$;
  \item $\Omega = \PP(\RR^+\text{-span}(e_1,\dots,e_n))$ and $x = [e_1 + \dots + e_n]$.
\end{itemize}

Here is an easy consequence of the inequality $(\mu_1-\mu_n)(g)\geq 2(d_\Omega(x,g\cdot x) -\kappa_{\Omega,x})$ of Proposition~\ref{prop:d-Omega-mu}.

\begin{corollary} \label{cor:QI-embed}
Let $\Gamma$ be a discrete subgroup of $G:=\PGL(V)$.
\begin{enumerate}
  \item\label{item:naive-cc-QI} If $\Gamma$ is naively convex cocompact in $\PP(V)$ (Definition~\ref{def:cc-naive}), then it is finitely generated and the natural inclusion $\Gamma\hookrightarrow G$ is a quasi-isometric embedding.
  \item\label{item:naive-cc-QI-subgroup} In particular, for any subgroup $\Gamma'$ of~$\Gamma$, if $\Gamma'$ is naively convex cocompact in $\PP(V)$, then it is finitely generated and the natural inclusion $\Gamma'\hookrightarrow\Gamma$ is a quasi-isometric embedding.
\end{enumerate}
\end{corollary}

Here the finitely generated groups $\Gamma$ and~$\Gamma'$ are endowed with the word metric with respect to some fixed finite generating subset.
The group $G$ is endowed with any $G$-invariant Riemannian metric~$d_G$.

\begin{proof}[Proof of Corollary~\ref{cor:QI-embed} assuming Proposition~\ref{prop:d-Omega-mu}]
Let $K=\PO(n)$ be the maximal compact subgroup of $G=\PGL(V)$ from Section~\ref{subsec:Cartan}.
Let $p:=\mathrm{Id} K\in G/K$.
There is a $G$-invariant Finsler metric $d$ on $G/K$ such that $d(p,g\cdot p) = (\mu_1-\mu_n)(g)$ for all $g\in G$.

\eqref{item:naive-cc-QI} Suppose $\Gamma$ is naively convex cocompact in $\PP(V)$: it preserves a properly convex open subset $\Omega$ of $\PP(V)$ and acts cocompactly on some nonempty $\Gamma$-invariant closed convex subset $\C$ of~$\Omega$.
By the \v{S}varc--Milnor lemma, $\Gamma$ is finitely generated and any orbital map $\Gamma\to (\C,d_{\Omega})$ is a quasi-isometry.
Proposition~\ref{prop:d-Omega-mu} then implies that any orbital map $\Gamma\to (G/K,d)$ is a quasi-isometric embedding.
Since $K$ is compact, the natural inclusion $\Gamma\hookrightarrow G$ is a quasi-isometric embedding.

\eqref{item:naive-cc-QI-subgroup} Let $\mathrm{length}_{\Gamma} : \Gamma\to\NN$ and $\mathrm{length}_{\Gamma'} : \Gamma'\to\NN$ be the word length functions of $\Gamma$ and~$\Gamma'$ for our fixed finite generating subsets.
By the triangle inequality, there exists $\kappa>0$ such that $d_G(\gamma,\mathrm{Id}) \leq \kappa\,\mathrm{length}_{\Gamma}(\gamma)$ for all $\gamma\in\Gamma$.
By~\eqref{item:naive-cc-QI}, if $\Gamma'$ is naively convex cocompact in $\PP(V)$, then there exist $\kappa_1,\kappa_2>0$ such that $d_G(\gamma',\mathrm{Id}) \geq \kappa_1\,\mathrm{length}_{\Gamma'}(\gamma') - \kappa_2$ for all $\gamma'\in\Gamma$, hence $\mathrm{length}_{\Gamma}({\gamma'}) \geq \kappa_1\kappa^{-1}\,\mathrm{length}_{\Gamma'}(\gamma') - \kappa_2\kappa^{-1}$ for all $\gamma'\in{\Gamma'}$ and $\Gamma'\hookrightarrow\Gamma$ is a quasi-isometric embedding.
\end{proof}

\begin{proof}[Proof of Proposition~\ref{prop:d-Omega-mu}]
We endow $\mathbb{P}(V)$ with the spherical metric $d_\mathbb{S}$, for which any projective line is geodesic, of length~$\pi$.
By the spherical sine law, two lines $\ell, \ell' \subset \mathbb{P}(V)$ intersecting in a point $b$ at an angle $\alpha \in [0,\pi/2]$ get no further than $\alpha$ apart, and for all $y\in \ell$,
\begin{equation} \label{eqn:spherical}
\frac{\sin d_\mathbb{S}(y,\ell')}{\sin \alpha} = \frac{\sin d_\mathbb{S}(y,b)}{\sin \frac{\pi}{2}} = \sin d_\mathbb{S}(y,b).
\end{equation}
Any tangent vector $v\in T\PP(V)$ can be written as $v = \frac{\mathrm{d}}{\mathrm{d}t}\big|_{t=0} \, [u+tw]$ where $u,w\in\RR^n$ are orthogonal for the standard Euclidean norm $\Vert\cdot\Vert_{\mathrm{Euc}}$ on~$\RR^n$ preserved by $\hat{K} = \OO(n)$, and $u\neq 0$; the norm $\Vert v\Vert_{\SS}$ of~$v$ for the spherical metric is then $\Vert w\Vert_{\mathrm{Euc}}/\Vert u\Vert_{\mathrm{Euc}}$.

We claim that any element $g\in\PGL(V)$, seen as a homeomorphism of $\PP(V) = \PP(\RR^n)$, is $e^{(\mu_1-\mu_n)(g)}$-bi-Lipschitz for~$d_\mathbb{S}$.
Indeed, since $(\mu_1-\mu_n)(g^{-1}) = (\mu_1-\mu_n)(g)$, it is enough to prove the Lipschitz direction. 
Since $d_\mathbb{S}$ is a geodesic metric, it suffices to check that $\Vert Dg (v) \Vert_\SS \leq e^{(\mu_1-\mu_n)(g)} \Vert v\Vert_{\SS}$ for any tangent vector $v\in T\PP(V)$.
We can lift $g$ to an element $\hat{g}$ of $\GL(V)$, with singular values $e^{\mu_1(\hat{g})} \geq \dots \geq e^{\mu_n(\hat{g})}$ satisfying $\mu_1(\hat{g})- \mu_n(\hat{g}) = (\mu_1-\mu_n)(g)$. 
Writing $v = \frac{\mathrm{d}}{\mathrm{d}t}\big|_{t=0} \, [u+tw]$ where $u,w\in\RR^n$ are orthogonal for $\Vert\cdot\Vert_{\mathrm{Euc}}$ and $u\neq 0$, we have $Dg(v) = \frac{\mathrm{d}}{\mathrm{d}t}\big|_{t=0} \, [\hat{g}\cdot u+t\,w']$ where $w'$ is the orthogonal projection of $\hat{g}\cdot w$ to the orthogonal of~$\hat{g}\cdot u$.
The bounds $\Vert \hat{g}\cdot u \Vert_{\mathrm{Euc}} \geq e^{\mu_n(\hat{g})} \Vert u\Vert_{\mathrm{Euc}}$ and $\Vert w'\Vert_{\mathrm{Euc}} \leq \Vert \hat{g}\cdot w \Vert_{\mathrm{Euc}} \leq e^{\mu_1(\hat{g})} \Vert w\Vert_{\mathrm{Euc}}$ yield
$$\Vert Dg(v)\Vert_\SS 
= \frac{\Vert w'\Vert_{\mathrm{Euc}}}{\Vert \hat{g}\cdot u\Vert_{\mathrm{Euc}}} 
\leq e^{\mu_1(\hat{g})-\mu_n(\hat{g})} \frac{\Vert w \Vert_{\mathrm{Euc}}}{\Vert u \Vert_{\mathrm{Euc}}}
= e^{(\mu_1-\mu_n)(g)} \Vert v \Vert_{\SS},$$
proving the claim.

For any $x\in\Omega$ and any $\varepsilon>0$, we denote by $\mathbb{B}_{\mathbb{S}}(x,\varepsilon)$ (\resp $\mathbb{B}_{\Omega}(x,\varepsilon)$) the open ball centered at $x\in\Omega$ with radius $\varepsilon>0$ in $(\PP(V),d_{\mathbb{S}})$ (\resp in $(\Omega,d_{\Omega})$).
Since $\Omega$ is properly convex, we can choose $0< \varepsilon < \pi/4$, depending only on $(\Omega, x)$, such that
\begin{enumerate}[label=(\roman*)]
  \item\label{xi} $\mathbb{B}_\mathbb{S}(x,\varepsilon) \subset \mathbb{B}_\Omega(x,1)$;
  \item\label{xii} $\mathbb{B}_\mathbb{S}(x,2\varepsilon)\subset \Omega$;
  \item\label{xiii} every segment of $\Omega$ has $d_\mathbb{S}$-length at most $\pi-2\varepsilon$.
\end{enumerate}

Fix $g \in \mathrm{Aut}(\Omega)$.
We claim that there exists $x' \in \overline{\mathbb{B}_\mathbb{S}(x,\varepsilon)}$ such that
\begin{equation} \label{eqn:squish}
e^{-(\mu_1-\mu_n)(g)} \sin \varepsilon \leq \sin d_\mathbb{S}(g\cdot x', \partial \Omega) \leq e^{-(\mu_1-\mu_n)(g)} (\sin \varepsilon)^{-3}.
\end{equation}
Indeed, as above, lift $g$ to $\hat{g}\in\GL(V)$ with singular values $e^{\mu_1(\hat{g})} \geq \dots \geq e^{\mu_n(\hat{g})}$ satisfying $\mu_1(\hat{g})- \mu_n(\hat{g}) = (\mu_1-\mu_n)(g)$.
We can write $\hat{g} = k a k'$ where $k,k' \in \hat{K} = \OO(n)$ and $a = \mathrm{diag}(e^{\mu_1(\hat{g})}, \dots, e^{\mu_n(\hat{g})}) \in \exp(\hat{\mathfrak{a}}^+)$ (see Section~\ref{subsec:Cartan}).
Let $(e_1, \dots, e_n)$ be the standard basis of $V = \RR^n$, orthonormal for $\Vert\cdot\Vert_{\mathrm{Euc}}$.
Set $v_i := {k'}^{-1}\cdot e_i$ and consider the projective hyperplane $H := \PP(\mathrm{span}\{v_2,\dots, v_n\})$.
There exists $x'\in \overline{\mathbb{B}_\mathbb{S}(x, \varepsilon)}$ such that $d_\mathbb{S}(x', H)\geq \nolinebreak \varepsilon$.
Note that any segment realizing $d_\mathbb{S}(g\cdot x', \partial \Omega)$ gets mapped by $g^{-1}$ to a segment of length $\geq \varepsilon$, by~\ref{xii}. 
Since $g^{-1}$ is an $e^{(\mu_1-\mu_n)(g)}$-Lipschitz homeomorphism of $(\PP(V), d_\mathbb{S})$, concavity of $\sin$ yields the lower bound in \eqref{eqn:squish}.
We now check the upper bound.
For this, consider $\delta<\varepsilon$ such that $\sin \delta = (\sin\varepsilon)^2$.
By~\eqref{eqn:spherical}, since $d_\SS(x', H)\geq \varepsilon$,  the line through $[v_n]$ and $x'$ forms an angle $\geq \varepsilon$ with $H$.
Again by~\eqref{eqn:spherical}, this line therefore contains a segment of length $\geq \pi-2\varepsilon$ that stays outside the $\delta$-neighborhood $\mathcal{N}$ of~$H$. 
By~\ref{xiii}, that segment must exit $\Omega$ at some $z \in \partial \Omega \smallsetminus \mathcal{N}$.
We have $d_\mathbb{S}(g\cdot x', \partial \Omega) \leq d_\mathbb{S}(g\cdot x', g\cdot z)$, hence it is enough to bound $d_\mathbb{S}(g\cdot x', g\cdot z)$ from above.
Since $z\in \PP(V)\smallsetminus \mathcal{N}$, a unit lift $\widehat{z}$ of~$z$ satisfies $|\langle \widehat{z}, v_1 \rangle | \geq \sin \delta = (\sin \varepsilon)^2$, \ie $|\langle k' \cdot \widehat{z}, e_1 \rangle | \geq (\sin \varepsilon)^2$, so that 
\begin{equation} \label{eqn:intermed}
\Vert \hat{g} \cdot \widehat z \Vert_{\mathrm{Euc}} = \Vert ak' \cdot \widehat z \Vert_{\mathrm{Euc}} \geq e^{\mu_1(\hat{g})} (\sin \varepsilon)^2.
\end{equation}
Since $x', [v_n],z \in \PP(V)$ are collinear, $x'$ has a unique lift of the form $\widehat{x}' = v_n+t \widehat z \in V$, and $|t|\geq \sin \varepsilon$ since $d_\mathbb{S}(x', [v_n]) \geq d_\mathbb{S}(x', H)\geq \varepsilon$.
Then $\Vert \hat{g} \cdot \widehat z - \hat{g} \cdot (\frac{1}{t}\widehat{x}') \Vert_{\mathrm{Euc}} = e^{\mu_n(\hat{g})}/|t| \leq e^{\mu_n(\hat{g})}/\sin \varepsilon$.
Using~\eqref{eqn:intermed}, we deduce
\begin{equation} \label{eqn:squish-bis}
\sin d_\mathbb{S}(g\cdot x', \partial \Omega) \leq \sin d_\mathbb{S}(g\cdot x', g\cdot z) \leq e^{\mu_n(\hat{g})-\mu_1(\hat{g})} (\sin \varepsilon)^{-3},
\end{equation}
which completes the proof of~\eqref{eqn:squish}. 
 
The set $\mathcal{K}_\varepsilon := \bigcup_{u\in \partial \Omega} \big( [x,u) \smallsetminus \mathbb{B}_\mathbb{S}(u, \varepsilon) \big)$ is compact in~$\Omega$. 
Since $\mathrm{Aut}(\Omega)$ acts properly on~$\Omega$, the set
$$F_\varepsilon := \{h\in\mathrm{Aut}(\Omega)~|~ \mathcal{K}_\varepsilon \cap h \cdot \overline{\mathbb{B}_\mathbb{S}(x,\varepsilon)} \neq \emptyset \} $$
is compact in $\PGL(V)$, hence $\kappa':=\max_{h\in F_\varepsilon}  \left | d_{\Omega}(x,h\cdot x) - \frac{1}{2} (\mu_1-\mu_n)(h) \right |$ is finite. 

Suppose that our given $g\in \mathrm{Aut}(\Omega)$ lies outside $F_\varepsilon$. 
Let $\ell$ be the projective line through $x$ and $g\cdot x'$, crossing $\partial \Omega$ at points $a$ and~$b$ with $a,x,g\cdot x',b$ aligned in this order (see Figure~\ref{fig:d-Omega-mu}).
\begin{figure}[h!]
\centering
\labellist
\small\hair 2pt
\pinlabel $a$ [u] at 177 146
\pinlabel $x$ [u] at 370 136
\pinlabel $g\cdot x'$ [u] at 630 115
\pinlabel $b$ [u] at 718 112
\pinlabel $b'$ [u] at 785 150
\pinlabel $g\cdot z$ [u] at 727 217
\pinlabel $\ell$ [u] at 90 150
\pinlabel $\ell'$ [u] at 820 50
\pinlabel $\Omega$ [u] at 430 60
\endlabellist
\includegraphics[width=10cm]{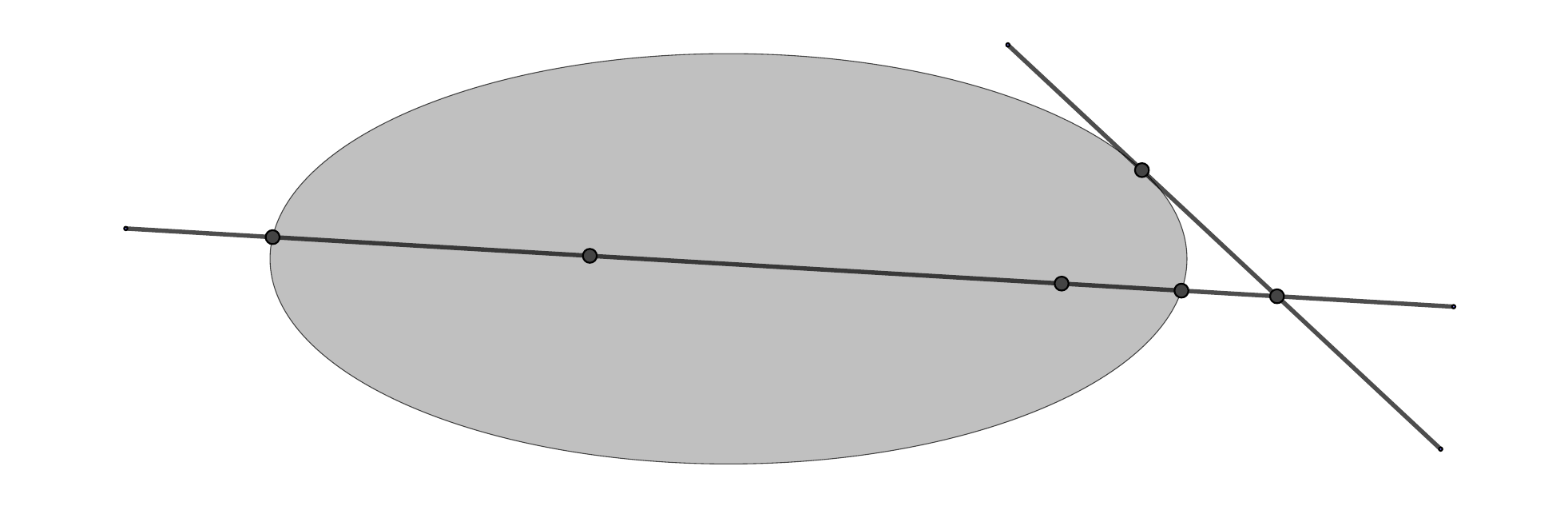}
\vspace{-0.3cm}
\caption{Illustration of the proof of Proposition~\ref{prop:d-Omega-mu}}
\label{fig:d-Omega-mu}
\end{figure}
Since $g \notin F_\varepsilon$ and $x'\in \overline{\mathbb{B}_\mathbb{S}(x,\varepsilon)}$, we have $\varepsilon' := d_\mathbb{S}(g\cdot x',b) \leq \varepsilon$.
Let $y \mapsto \theta_y \in \RR/ \pi \mathbb Z$ be the $d_\mathbb{S}$-arclength parametrization of the projective line $\ell$ such that $\theta_a = 0$ and $\theta_y \in [0, \pi - 2 \varepsilon]$ for every $y \in \overline{\Omega} \cap \ell$.
Then $y \mapsto \tan\theta_y$ is a projective identification between $\ell$ and $\PP^1(\RR) = \RR\cup\{\infty\}$, and so
\begin{align*}
\cro{a}{x}{g \cdot x'}{b} & = \cro{0}{\tan \theta_x}{\tan \theta_{g\cdot x'}}{\tan \theta_b} = \frac{\tan \theta_{g\cdot x'} \, (\tan \theta_b - \tan \theta_x)}{\tan \theta_x \, (\tan \theta_b - \tan \theta_{g\cdot x'})}\\
& = \frac{\sin \theta_{g\cdot x'} \, \sin (\theta_b-\theta_x)}{\sin \theta_x \, \sin (\theta_b-\theta_{g\cdot x'})} = \frac{\sin \theta_{g\cdot x'} \, \sin (\theta_b-\theta_x)}{\sin \theta_x \,\sin \varepsilon'}.
\end{align*}
Each of $\theta_x$, $\theta_b-\theta_x$, and $\theta_{g\cdot x'}$ lies in $[\varepsilon, \pi-2\varepsilon]$ by \ref{xi} and~\ref{xiii}, hence their sines lie in $[\sin \varepsilon, 1]$, and so
\begin{equation} \label{eqn:a-x-g.x'-b}
\cro{a}{x}{g \cdot x'}{b} \in (\sin\varepsilon')^{-1} \, [(\sin\varepsilon)^2, (\sin\varepsilon)^{-1}].
\end{equation}
On the other hand, we claim that
\begin{equation} \label{eqn:sin-eps'}
\sin \varepsilon' \in e^{(\mu_n-\mu_1)(g)} \, [\sin \varepsilon, (\sin \varepsilon)^{-4}].
\end{equation}
Indeed, the lower bound comes from~\eqref{eqn:squish}. 
For the upper bound, consider the point $z \in \partial \Omega$ introduced above, satisfying \eqref{eqn:squish-bis}.
There is a line $\ell'$ tangent to $\partial \Omega$ at $g\cdot z$ and intersecting $\ell$ in some point $b'\in \PP(V)\smallsetminus \Omega$ at some angle $\alpha$; the lines $\ell$ and~$\ell'$ get at least $\varepsilon$ apart for $d_{\mathbb{S}}$ by~\ref{xi}, hence $\alpha \geq \varepsilon$ by \eqref{eqn:spherical}.
On the other hand, $\ell \cap \mathbb{B}_\mathbb{S}(g\cdot x', \varepsilon') \subset \Omega$, using $\varepsilon'\leq \varepsilon$ and~\ref{xii}, hence $\varepsilon' = d_{\mathbb{S}} (g\cdot x',b) \leq d_{\mathbb{S}}(g\cdot x',b')$.
Using \eqref{eqn:spherical} again, we obtain
$$\sin \varepsilon' \leq \sin d_{\mathbb{S}} (g\cdot x',b')  = \frac{\sin d_\mathbb{S} (g\cdot x', \ell')}{\sin \alpha} \leq \frac{\sin d_\mathbb{S} (g\cdot x', g\cdot z)}{\sin \varepsilon}.$$
Together with \eqref{eqn:squish-bis}, this implies \eqref{eqn:sin-eps'}.

Combining \eqref{eqn:a-x-g.x'-b} and \eqref{eqn:sin-eps'}, we find that $e^{2 d_{\Omega}(x,g\cdot x')} = \cro{a}{x}{g\cdot x'}{b}$ lies in $e^{(\mu_1-\mu_n)(g)} [(\sin \varepsilon)^6, (\sin \varepsilon)^{-2}]$.
Since $d_{\Omega}(g\cdot x,g\cdot x')\leq 1$  by~\ref{xi}, the triangle inequality yields 
$\big | d_{\Omega}(x,g\cdot x) - \frac{1}{2} (\mu_1-\mu_n)(g) \big | \leq 3 \, |\mathrm{log} \sin \varepsilon|+1 =:\kappa''$.
Thus \eqref{eqn:hilball} holds with $\kappa_{\Omega, x}=\max\{\kappa', \kappa''\}$.

Finally, let us address (local) uniformity in terms of $(\Omega, x)$.
The number $\varepsilon>0$ subject to \ref{xi}--\ref{xii}--\ref{xiii} can be chosen uniformly.
The compact set $\mathcal{K}_\varepsilon \subset \Omega$ in the above proof then stays a uniform $d_{\mathbb{S}}$-distance away from $\partial \Omega$. 
It follows that $d_\Omega(x, g\cdot x)$ is uniformly bounded for all $g\in F_\varepsilon$.
In fact, the property $g\in F_\varepsilon$ imposes some uniform lower bound on $d_\mathbb{S}(g\cdot x', \partial \Omega)$ for all $x'\in \overline{\mathbb{B}_\mathbb{S}(x,\varepsilon)}$.
The upper bound of~\eqref{eqn:squish} then gives a uniform bound on $(\mu_1-\mu_n)(F_\varepsilon)$, \ie $F_\varepsilon$ stays within a uniform compact subset of $\PGL(V)$.
Therefore $\kappa'$ is also (locally) uniform, and so is $\kappa_{\Omega, x}$.
\end{proof}

\subsection{\ref{item:cc-no-unipotent}: No unipotent elements} \label{subsec:cc-no-unip}

For any element $g\in\PGL(V)$ preserving a properly convex open subset $\Omega$ of~$\PP(V)$, we define the \emph{translation length of $g$ in~$\Omega$} to be
$$\mathrm{length}_{\Omega}(g) := \inf_{x\in\Omega} d_{\Omega}(x,g\cdot x) \geq 0 ;$$
we say that $g$ \emph{achieves its translation length in~$\Omega$} if this infimum is a minimum.
It is easy to check (see \eg \cite[Prop.\,2.1]{clt15}) that $\mathrm{length}_{\Omega}(g) = (\lambda_1-\lambda_n)(g)/2$ (see \eqref{eqn:lambda-i-j}).

Property~\ref{item:cc-no-unipotent} of Theorem~\ref{thm:properties} is contained in the following more general statement.

\begin{proposition} \label{prop:cc-no-unipotent}
Let $\Gamma$ be a discrete subgroup of $\PGL(V)$ which is naively convex cocompact in $\PP(V)$ (Definition~\ref{def:cc-naive}), \ie $\Gamma$ preserves a properly convex open subset $\Omega$ of $\PP(V)$ and acts cocompactly on some nonempty closed convex subset $\C$ of~$\Omega$.
Then any element $\gamma\in\Gamma$ achieves its translation length in~$\Omega$.
In particular (see \cite[Prop.\,2.13]{clt15}), any infinite-order element of~$\Gamma$ lifts to an element of $\SL^{\pm}(V)$ with an eigenvalue of modulus $>1$, and so $\Gamma$ does not contain any unipotent element.
\end{proposition}

\begin{proof}
By Lemma~\ref{lem:naive-cc-nonempty-int}, up to replacing $\C$ by some closed uniform neighborhood in $(\Omega,d_{\Omega})$, we may assume that $\C$ has nonempty interior.
Let $\gamma\in\Gamma$.
We have
$$\mathrm{length}_{\Omega}(\gamma) \leq \inf_{x\in\C} d_{\Omega}(x,\gamma\cdot x) \leq \inf_{x\in\Int(\C)} d_{\Omega}(x,\gamma\cdot x) \leq \mathrm{length}_{\Int(\C)}(\gamma),$$
where the last inequality follows from the definition \eqref{eqn:d-Omega} of the Hilbert metric.
In fact all these inequalities are equalities since $\mathrm{length}_{\Omega}(\gamma)$ and $\mathrm{length}_{\Int(\C)}(\gamma)$ are both equal to $(\lambda_1-\lambda_n)(\gamma)$.
Thus we only need to check that the infimum $R := \inf_{x\in\C} d_{\Omega}(x,\gamma\cdot x)$ is a minimum.

Consider $(x_m)\in\C^{\NN}$ such that $d_{\Omega}(x_m,\gamma\cdot x_m)\to R$.
For any~$m$, there exists $\gamma_m\in\Gamma$ such that $\gamma_m\cdot x_m$ belongs to some fixed compact fundamental domain for the action of $\Gamma$ on~$\C$.
Up to passing to a subsequence, $(\gamma_m\cdot x_m)_{m\in\NN}$ converges to some $x\in\C$, and $d_{\Omega}(x_m,\gamma\cdot x_m)\leq R+1$ for all~$m$, which means that $\gamma_m\gamma\gamma_m^{-1}\in\Gamma$ sends $\gamma_m\cdot x_m\in\C$ at distance $\leq R+1$ from itself.
Since the action of $\Gamma$ on~$\C$ is properly discontinuous, we deduce that the discrete sequence $(\gamma_m\gamma\gamma_m^{-1})_{m\in\NN}$ is bounded; up to passing to a subsequence, we may therefore assume that it is constant, equal to $\gamma_{\infty}\gamma\gamma_{\infty}^{-1}$ for some $\gamma_{\infty}\in\Gamma$.
We then have
\begin{align*}
R & = \lim_m d_{\Omega}(x_m, \gamma\cdot x_m) \hspace{-2.3cm} && = \lim_m d_{\Omega}(\gamma_m\cdot x_m, (\gamma_m\gamma\gamma_m^{-1})\cdot (\gamma_m \cdot x_m))\\
& = d_{\Omega}(x, \gamma_{\infty}\gamma\gamma_{\infty}^{-1}\cdot x) \hspace{-2.3cm} && = d_{\Omega}(y, \gamma\cdot y),
\end{align*}
where $y:=\gamma_{\infty}^{-1}\cdot x$.
\end{proof}

\subsection{\ref{item:stable}: Stability}

Property~\ref{item:stable} of Theorem~\ref{thm:properties} follows from the equivalence \eqref{item:ccc-limit-set}~$\Leftrightarrow$~\eqref{item:ccc-strict-C1} of Theorem~\ref{thm:main-general} and from the work of Cooper--Long--Tillmann (namely \cite[Th.\,0.1]{clt18} with empty collection~$\mathcal{V}$ of generalized cusps).
Indeed, the equivalence \eqref{item:ccc-limit-set}~$\Leftrightarrow$~\eqref{item:ccc-strict-C1} of Theorem~\ref{thm:main-general} states that $\Gamma$ is convex cocompact in $\PP(V)$ if and only if it acts properly discontinuously and cocompactly on a properly convex subset $\C$ of $\PP(V)$ with strictly convex nonideal boundary, and this condition is stable under small perturbations by \cite{clt18}.

Let us give a brief sketch of why this is true.
The ideas are taken from~\cite{clt18} and go back to Koszul~\cite{kos68}, who proved this in the case that $\partialn \C = \emptyset$.
We assume $\Gamma$ is torsion-free for simplicity.
The quotient $M = \Gamma \backslash \C$ is a compact convex projective manifold with boundary, whose boundary $\partial M = \Gamma \backslash \partialn \C$ is locally strictly convex.
Let $\rho: \Gamma \to \PGL(V)$ be a small deformation of the inclusion map.
By the Ehresmann--Thurston principle (see~\cite{thu80}), if $\rho$ is a sufficiently small deformation, then it is still the holonomy representation of a projective structure on~$M$.
Further, since $\partial M$ is compact, one easily arranges that $\partial M$ remains locally strictly convex in the new projective structure.
To see that this new projective structure is properly convex uses the following observation:
A real projective structure on a manifold $M$ with locally strictly convex boundary is properly convex if and only if the tautological line bundle $\xi M \to M$, naturally an affine manifold, admits a strictly convex section.
We have such a strictly convex section before deformation, and since $M$ is compact, this section can be controlled if the deformation~$\rho$ is small enough.
This gives the desired stability result.

\subsection{\ref{item:include}: Inclusion into a larger space} \label{subsec:cc-include}

The following proposition implies property~\ref{item:include} of Theorem~\ref{thm:properties}, and also provides a converse to it.

\begin{proposition} \label{prop:include-cc-equiv}
Let $V'=\RR^{n'}$ and let $i : \SL^{\pm}(V) \hookrightarrow \SL^{\pm}(V \oplus V')$ be the natural inclusion, whose image acts trivially on the second factor.
Let $\hat{\Gamma}$ be a discrete subgroup of $\SL^{\pm}(V)$.
Then $\hat{\Gamma}$ is convex cocompact in $\PP(V)$ if and only if $i(\hat{\Gamma})$ is convex cocompact in $\PP(V\oplus V')$.
\end{proposition}

The proof of Proposition~\ref{prop:include-cc-equiv} builds on the following consequence of Lemma~\ref{lem:vector-growth} and Corollary~\ref{cor:compare-Hilb-Eucl}.

\begin{lemma} \label{lem:limit-hyperplanes}
Let $V'=\RR^{n'}$, and let $i : \SL^{\pm}(V) \hookrightarrow \SL^{\pm}(V \oplus V')$  and $j : V \hookrightarrow V\oplus V'$, and $j^* : V^* \hookrightarrow (V\oplus V')^* \simeq V^* \oplus (V')^*$ be the natural inclusions.
Let $\hat{\Gamma}$ be a discrete subgroup of $\SL^{\pm}(V)$ preserving a nonempty properly convex open subset $\Omega$ of $\PP(V)$.
\begin{enumerate}
  \item\label{item:include-accum} Let $\mathcal{K}$ be a compact subset of $\PP(V\oplus V')$ that does not meet any projective hyperplane $z^*\in j^*(\overline{\Omega^*})$.
  Then any accumulation point in $\PP(V\oplus V')$ of the $i(\hat\Gamma)$-orbit of~$\mathcal{K}$ is contained in $j(\Lambdao_{\Omega}(\Gamma))$.
  \item\label{item:include-accum-dual} Let $\mathcal{K}^*$ be a compact subset of $\PP((V\oplus V')^*)$ whose elements correspond to projective hyperplanes disjoint from $j(\overline{\Omega})$ in $\PP(V\oplus V')$.
  Then any accumulation point in $\PP((V\oplus V')^*)$ of the $i(\hat\Gamma)$-orbit of~$\mathcal{K}^*$ is contained in $j^*(\Lambdao_{\Omega^*}(\Gamma))$.
\end{enumerate}
\end{lemma}

\begin{proof}
The two statements are dual to each other, so we only need to prove~\eqref{item:include-accum}.
Let $\pi : \PP(V\oplus V') \smallsetminus \PP(V') \to\PP(V)$ be the map induced by the projection onto the first factor of $V\oplus V'$, so that $\pi\circ j$ is the identity of $\PP(V)$.
Note that $V'$ is contained in every projective hyperplane of $\PP(V\oplus V')$ corresponding to an element of $j^*(\overline{\Omega}^*)$, hence $\pi$ is defined on~$\mathcal{K}$.
Consider a sequence $(x_m)_{m\in\NN}$ of points of~$\mathcal{K}$ and a sequence $(\gamma_m)_{m\in\NN}$ of pairwise disjoint elements of~$\hat\Gamma$ such that $(i(\gamma_m)\cdot x_m)_{m\in\NN}$ converges in $\PP(V\oplus V')$.
By construction, for any~$m$, the point $x_m$ does not belong to any hyperplane $z^*\in j^*(\overline{\Omega^*})$, hence $\pi(x_m)$ does not belong to any hyperplane in $\overline{\Omega^*}$, \ie $\pi(x_m)\in\Omega$.
Since $\hat\Gamma$ acts properly discontinuously on~$\Omega$, the point $z := \lim_m \gamma_m\cdot\pi(x_m)$ is contained in $\partial \Omega$.
Up to passing to a subsequence, we may assume that $(\pi(x_m))_{m\in\NN}\subset\pi(\mathcal{K})$ converges to some $y \in \Omega$.
Then $d_{\Omega}(\gamma_m\cdot\pi(x_m),\gamma_m\cdot y) = d_{\Omega}(\pi(x_m),y) \to 0$, and so $\gamma_m\cdot y\to z$ by Corollary~\ref{cor:compare-Hilb-Eucl}.
In particular, $z \in \Lambdao_{\Omega}(\Gamma)$.
Now lift $x_m$ to a vector $v_m + v'_m \in V \oplus V'$, with $(v_m)_{m\in\NN}\subset V\oplus\{0\}$ and $(v'_m)_{m\in\NN}\subset \{0\}\oplus V'$ bounded.
The image of $v_m$ in $\PP(V)$ is $\pi(x_m)$.
By Lemma~\ref{lem:vector-growth}, the sequence $(i(\gamma_m)\cdot v_m)_{m\in\NN}$ tends to infinity in~$V$.
On the other hand, $i(\gamma_m)\cdot v'_m = v'_m$ remains bounded in~$V'$.
Therefore,
$$\lim_m i(\gamma_m)\cdot x_m = \lim_m i(\gamma_m)\cdot j\circ\pi(x_m) = \lim_m j(\gamma_m\cdot\pi(x_m)) \in j(\Lambdao_{\Omega}(\Gamma)). \hfill\qedhere$$
\end{proof}

\begin{lemma} \label{lem:cc-include-preserve-O}
In the setting of Lemma~\ref{lem:limit-hyperplanes}, if $\Gamma$ acts convex cocompactly on $\Omega \subset \PP(V)$, then the set $j(\Omega)$ is contained in a $\Gamma$-invariant properly convex open subset $\mathcal{O}$ of $\PP(V\oplus V')$.
\end{lemma}

\begin{proof}
We argue similarly to the proof of Lemma~\ref{lem:Omega'}.
The set $j(\Omega)$ is contained in the $\Gamma$-invariant open subset
$$\mathcal{O}_{\max} := \PP(V\oplus V') \smallsetminus \bigcup_{z^*\in j^*(\overline{\Omega^*})} z^*$$
of $\PP(V\oplus V')$, which is convex but not properly convex.
This set $\mathcal{O}_{\max}$ is the union of all projective lines of $\PP(V\oplus V')$ intersecting both $j(\Omega)$ and $\PP(\{0\}\times V')$.
Choose projective hyperplanes $H_1, \ldots, H_N$ of $\PP(V\oplus V')$ bounding an open simplex $\Delta$ containing $j(\overline{\Omega})$.
Let $\mathcal{U} := \Delta \cap \mathcal{O}_{\max}$.
Define
$$\mathcal{O} := \bigcap_{\gamma \in \hat\Gamma} i(\gamma)\cdot\mathcal{U}.$$
Then $j(\Omega)\subset\mathcal{O} \subset\mathcal{O}_{\max}$ and $\mathcal{O}$ is properly convex.
We claim that $\mathcal{O}$ is open.
Indeed, suppose by contradiction that there exists a point $x \in \partialn\mathcal{O}$.
Then there exist $(\gamma_m) \in {\hat\Gamma}^{\NN}$ and $i \in\{1, \ldots, N\}$ such that $\gamma_m\cdot H_i$ converges to a hyperplane $H_{\infty}$ containing~$x$.
This is impossible since, by Lemma~\ref{lem:limit-hyperplanes}.\eqref{item:include-accum-dual}, the hyperplane $H_{\infty} \in j^*(\Lambdao_{\Omega}(\Gamma))$ supports the open set~$\mathcal{O}_{\max}$ which contains~$x$.
\end{proof}

\begin{proof}[Proof of Proposition~\ref{prop:include-cc-equiv}]
Suppose that $\hat{\Gamma}$ acts convex cocompactly on some nonempty properly convex open subset $\Omega$ of $\PP(V)$.
The set $j(\Omega)$ is contained in some $\Gamma$-invariant properly convex open subset $\mathcal{O}$ of $\PP(V\oplus V')$ by Lemma~\ref{lem:cc-include-preserve-O}, and $\Lambdao_{\mathcal{O}}(i(\hat \Gamma)) = j(\Lambdao_{\Omega}(\Gamma))$ by Lemma~\ref{lem:limit-hyperplanes}.\eqref{item:include-accum}.
Since the action of $\hat\Gamma$ on $\Ccore_\Omega(\Gamma)$ is cocompact, so is the action of $i(\hat\Gamma)$ on $\Ccore_{\mathcal{O}}(i(\hat \Gamma)) = j(\Ccore_\Omega(\Gamma))$.
Therefore $i(\hat \Gamma)$ acts convex cocompactly on~$\mathcal{O}$, and so $i(\hat \Gamma)$ is convex cocompact in $\PP(V \oplus V')$.

Conversely, suppose that $i(\hat{\Gamma})$ acts convex cocompactly on some nonempty properly convex open subset $\Omega'$ of $\PP(V\oplus V')$.
For any $\gamma \in \hat{\Gamma}$, applying $i(\gamma)$ to the direct sum $V\oplus V'$ only changes the $V$ factor.
Therefore, for any sequence $(\gamma_m)_{m\in\NN}$ of pairwise distinct elements of~$\hat{\Gamma}$, and any $v \in V$ and $v' \in V'$ such that $[v + v'] \in \Omega'$, Lemma~\ref{lem:vector-growth} implies that the limit of $[i(\gamma_m)(v + v')]$, if it exists, must lie in $\PP(V)$.
This shows that $\Lambdao_{\Omega'}(i(\hat{\Gamma})) \subset \PP(V)$, hence $\Ccore_{\Omega'}(i(\hat{\Gamma})) \subset \PP(V)$.
It follows that $\Omega := \Omega' \cap \PP(V)$ is a nonempty, $\hat{\Gamma}$-invariant, open, properly convex subset of $\PP(V)$ on which the action of $\hat{\Gamma}$ is convex cocompact.
\end{proof}

\subsection{\ref{item:cc-ss}: Semisimplification} \label{subsec:cc-ss}

We start with a general observation.

\begin{lemma} \label{lem:semidirect-prod}
Let $L$ and~$U$ be two subgroups of $G=\PGL(V)$ such that
\begin{itemize}
  \item $L\cap U=\{\mathrm{Id}\}$ and $L$ normalizes~$U$, so that the subgroup of $\PGL(V)$ generated by $L$ and~$U$ is a semidirect product $L\ltimes U$;
  \item there is a sequence $(g_m)_{m\in\NN}$ in the centralizer of $L$ in $\PGL(V)$ such that $g_m u g_m^{-1} \to \mathrm{Id}$ for all $u\in U$.
\end{itemize}
Let $\pi : L\ltimes U\to L$ be the natural projection.
For any discrete group $\Gamma$ and any representation $\rho : \Gamma\to L\ltimes U$, if $\pi\circ\rho : \Gamma\to L$ is injective and if its image is convex cocompact in $\PP(V)$, then the same holds for~$\rho$.
\end{lemma}

\begin{proof}
If $\pi\circ\rho$ is injective, then so is~$\rho$.
The group $\pi\circ\rho(\Gamma)$ is a limit of $\PGL(V)$-conjugates of $\rho(\Gamma)$ and convex cocompactness is open (property~\ref{item:stable} above) and invariant under conjugation.
Therefore, if $\pi\circ\rho(\Gamma)$ is convex cocompact in $\PP(V)$, then so is $\rho(\Gamma)$.
\end{proof}

Let $\Gamma$ be a discrete group and $\rho : \Gamma\to\PGL(V)$ a representation.
The Zariski closure $H$ of $\rho(\Gamma)$ in $\PGL(V)$ admits a Levi decomposition $H=L\ltimes R_u(H)$ where $R_u(H)$ is the unipotent radical of~$H$ and $L$ is a reductive subgroup called a \emph{Levi factor}.
This decomposition is unique up to conjugation by $R_u(H)$.
We shall use the following terminology.

\begin{definition} \label{def:reduc-part}
The \emph{semisimplification} of~$\rho$ is the composition $\rho^{ss}$ of~$\rho$ with the projection onto~$L$.
\end{definition}

The semisimplification $\rho^{ss} : \Gamma\to\PGL(V)=G$ does not depend, up to conjugation by $R_u(H)$, on the choice of the Levi factor~$L$.
There is a sequence $(g_m)_{m\in\NN}$ in the centralizer of $L$ in~$H$ such that $g_m u g_m^{-1} \to \mathrm{Id}$ for all $u\in R_u(H)$, and so $\rho^{ss}$ is a limit of $H$-conjugates of~$\rho$.
If $\rho(\Gamma)$ is discrete in $\PGL(V)$, then so is $\rho^{ss}(\Gamma)$ (see \cite[Th.\,8.24]{rag72}).
The group $\rho^{ss}(\Gamma)$, like the Levi factor containing it, acts projectively on~$V$ in a semisimple way: namely, any invariant linear subspace of~$V$ admits an invariant complementary subspace.

If the semisimplification $\rho^{ss} : \Gamma\to\PGL(V)$ is injective and its image $\rho^{ss}(\Gamma)$ is convex cocompact in $\PP(V)$, then the same holds for~$\rho$ by Lemma~\ref{lem:semidirect-prod}.
This proves~\ref{item:cc-ss}, the last remaining property of Theorem~\ref{thm:properties}.

In the rest of this section, we study further operations that preserve convex cocompactness, strengthening properties~\ref{item:include} and~\ref{item:cc-ss} of Theorem~\ref{thm:properties}.

\subsection{Convex cocompactness for cocycle deformations} \label{subsec:cocycles}

Let $V = V_1 \oplus V_2$, let $\Gamma$ be a discrete group, and let $\rho: \Gamma \to \SL^\pm(V_1)$ be a representation.
A \emph{$\rho$-cocycle} is a map $u: \Gamma \to \Hom(V_2, V_1)$ satisfying the cocycle condition
$$u(\gamma_1\gamma_2) = u(\gamma_1) + \rho(\gamma_1)\circ u(\gamma_2)$$
for all $\gamma_1,\gamma_2\in\Gamma$.
Any $\rho$-cocycle $u$ 
defines a representation $\rho^u: \Gamma \to \SL^\pm(V) = \SL^{\pm}(V_1\oplus\nolinebreak V_2)$ by the formula:
$$\rho^u(\gamma)(v_1 \oplus v_2) := (\rho(\gamma)\cdot v_1 + u(\gamma)(v_2)) \oplus v_2 \ \ \ \ \ \forall v_1 \in V_1, v_2 \in V_2$$
or in matrix notation: 
\begin{align*}\label{eqn:rhou-matrix}
\rho^u(\gamma) = \begin{pmatrix} \rho(\gamma) & u(\gamma) \\ 0 & \mathrm{Id}_{V_2} \end{pmatrix}.
\end{align*}
In the case that $u = 0$ is the zero cocycle, the corresponding representation $\rho^0$ is simply the composition of $\rho$ with the inclusion $V_1 \hookrightarrow V$.
The representation $\rho^u$ is conjugate to $\rho^0$ if and only if $u : \Gamma\to\Hom(V_2,V_1)$ is a $\rho$-\emph{coboundary}: there exists $\Phi\in\Hom(V_2,V_1)$ such that $u(\gamma) = \Phi - \rho(\gamma)\circ\Phi$ for all $\gamma\in\Gamma$.
If $u$ is not a $\rho$-coboundary, then the representation $\rho^u$ is not semisimple (\ie not conjugate to its semisimplification from Definition~\ref{def:reduc-part}), even if $\rho$ is.
Cocycles which are not coboundaries always exist \eg if $\Gamma$ is a free group.
See also Example~\ref{exa:aff-hp} below. 

\begin{proposition} \label{prop:cc-block}
Let $V = V_1 \oplus V_2$, let $\Gamma$ be a discrete group, let $\rho: \Gamma \to \SL^\pm(V_1)$ be a representation, let $u: \Gamma \to \Hom(V_2, V_1)$ be a $\rho$-cocycle, and let $\rho^u: \Gamma \to \SL^\pm(V)$ be the associated representation.
Then the following are equivalent:
\begin{enumerate}
  \item\label{item:cc-block-Gamma} $\rho$ is injective and $\rho(\Gamma)$ is convex cocompact in $\PP(V_1)$,
  \item\label{item:cc-block-rhou} $\rho^u$ is injective and $\rho^u(\Gamma)$ is convex cocompact in $\PP(V)$.
\end{enumerate}
\end{proposition}

\begin{proof}
We first check the implication \eqref{item:cc-block-Gamma}~$\Rightarrow$~\eqref{item:cc-block-rhou}.
Suppose $\rho$ is injective and $\rho(\Gamma)$ is convex cocompact in $\PP(V_1)$.
Then $\rho^u$ is clearly injective.
The composition of $\rho$ with the inclusion of $V_1$ into $V = V_1 \oplus V_2$,
$$\gamma \longmapsto \begin{pmatrix} \rho(\gamma) & 0 \\ 0 & \mathrm{Id}_{V_2} \end{pmatrix} \in \SL^{\pm}(V)$$
is injective and by property~\ref{item:include} of Theorem~\ref{thm:properties} (proved in Section~\ref{subsec:cc-include} above), its image is convex cocompact in $\PP(V)$. 
It then follows that $\rho^u(\Gamma)$ is convex cocompact in $\PP(V)$ by applying Lemma~\ref{lem:semidirect-prod} with $L$ the block diagonal subgroup, $U$ the block upper triangular subgroup, $(g_m)_{m\in\NN}$ a sequence of block scalar matrices which contract $V_1$ relative to~$V_2$.

We now check the implication \eqref{item:cc-block-rhou}~$\Rightarrow$~\eqref{item:cc-block-Gamma}.
Suppose $\rho^u$ is injective and $\rho^u(\Gamma)$ is convex cocompact in $\PP(V)$.
By property~\ref{item:cc-no-unipotent} of Theorem~\ref{thm:properties} (proved in Section~\ref{subsec:cc-no-unip} above), $\rho^u(\Gamma)$ does not contain any unipotent element, hence $\rho$ is also injective.
Let $\Omega \subset \PP(V)$ be a nonempty properly convex open set on which $\rho^u(\Gamma)$ acts convex cocompactly.
For any $\gamma \in \Gamma$, applying $\rho^u(\gamma)$ to the direct sum $V = V_1 \oplus V_2$ only changes the $V_1$ factor.
Therefore, for any sequence $(\gamma_m)_{m\in\NN}$ of pairwise distinct elements of~$\Gamma$, and any $v_1 \in V_1$, $v_2 \in V_2$ such that $[v_1 \oplus v_2] \in \Omega$, Lemma~\ref{lem:vector-growth} implies that the limit of $[\rho(\gamma_m)(v_1 \oplus v_2)]$, if it exists, must lie in $\PP(V_1)$.
This shows that $\Lambdao_\Omega(\rho^u(\Gamma)) \subset \PP(V_1)$, hence $\Ccore_\Omega(\rho^u(\Gamma)) \subset \PP(V_1)$.
It follows that $\Omega_1 := \Omega \cap \PP(V_1)$ is a nonempty, $\rho^u(\Gamma)$-invariant, open, properly convex subset of $\PP(V_1)$ on which the restricted action, namely $\rho$, of $\Gamma$ is convex cocompact.
\end{proof}

\begin{example}\label{exa:aff-hp}
We briefly examine the special case where $V_1 = \RR^3$, where $V_2 = \RR^1$, where $\Gamma$ is isomorphic to the fundamental group of a closed orientable surface $S_g$ of genus $g \geq 2$, and where $\rho: \Gamma \to \SL(V_1)$ is discrete, injective, with image in a special orthogonal group $\SO(2,1)$ of signature $(2,1)$.  
Then $\rho$ represents a point of the Teichm\"uller space $\mathcal T(S_g)$ and the space of cohomology classes of cocycles $u : \Gamma\to\Hom(V_2,V_1)\simeq V$ identifies with the tangent space to $\mathcal T(S_g)$ at $\Gamma$ (see \eg \cite[\S\,2.3]{dgk16}) and has dimension $6g-6$. 

(1) By the implication \eqref{item:cc-block-Gamma}~$\Rightarrow$~\eqref{item:cc-block-rhou} of Proposition~\ref{prop:cc-block}, for any $\rho$-cocycle $u$, the associated representation $\rho^u$ is injective and $\rho^u(\Gamma)$ preserves and acts convex cocompactly on a nonempty properly convex open subset $\Omega$ of $\PP(\RR^3 \oplus \RR^1)$.
The $\rho^u(\Gamma)$-invariant convex open set $\Omega_{\max} \supset \Omega$ given by Proposition~\ref{prop:max-inv-conv}.\eqref{item:Omega-max} turns out to be properly convex as soon as $u$ is not a coboundary.
It can be described as follows.
The group $\rho^u(\Gamma)$ preserves an affine chart $U \simeq \RR^3$ of $\PP(\RR^3 \oplus \RR^1)$ and a flat Lorentzian metric on~$U$.
The action on $U$ is not properly discontinuous.
However, Mess~\cite{mes90} described two maximal globally hyperbolic domains of discontinuity, called \emph{domains of dependence}, one oriented to the future and one oriented to the past.
The domain $\Omega_{\max}$ is the union of the two domains of dependence plus a copy of the hyperbolic plane $\HH^2$ in $\PP(\RR^3\oplus 0)$.
We have that $\Lambdao_{\Omega_{\max}}(\rho^u(\Gamma)) \subset \partial \HH^2 \subset \PP(\RR^3 \oplus 0)$ and $\Ccore_{\Omega_{\max}}(\rho^u(\Gamma)) \subset \HH^2$.
Note that $\Omega_{\max}$ depends on~$u$ but~$\Ccore_{\Omega_{\max}}(\rho^u(\Gamma))$ does not.

(2) By Theorem~\ref{thm:properties}.\ref{item:dual}, the group $\rho^u(\Gamma)$ also acts convex cocompactly on a nonempty properly convex open subset $\Omega^*$ of the dual projective space $\PP((\RR^3 \oplus \RR)^*)$.
In this case the $\rho^u(\Gamma)$-invariant convex open set $(\Omega^*)_{\max} \supset \Omega^*$, given by Proposition~\ref{prop:max-inv-conv}.\eqref{item:Omega-max} applied to $\Omega^*$, is \emph{not} properly convex: it is the suspension $\mathrm{HP}^3$ of the dual copy $(\HH^2)^* \subset \PP((\RR^3)^*)$ of the hyperbolic plane with the point $\PP((\RR^1)^*)$.
The convex set $\mathrm{HP}^3 \subset \PP((\RR^3 \oplus \RR)^*)$ is the projective model for \emph{half-pipe geometry}, a transitional geometry lying between hyperbolic geometry and anti-de Sitter geometry (see~\cite{dan13}).
Note that while $(\Omega^*)_{\max} = \mathrm{HP}^3$ does not depend on~$u$, the convex core $\Ccore_{\Omega^*}(\rho^u(\Gamma))$ does depend on~$u$.
The full orbital limit set $\Lambdao_{\Omega^*}(\rho^u(\Gamma)) = \overline{\Omega^*} \cap \partial \mathrm{HP}^3$ is \emph{not} contained in a hyperplane if $u$ is not a coboundary.
Indeed, $\Ccore_{\Omega^*}(\rho^u(\Gamma))$ may be thought of as a rescaled limit of the collapsing convex cores for quasi-Fuchsian subgroups of $\PO(3,1)$ (or of $\PO(2,2)$) which converge to the Fuchsian group~$\rho(\Gamma)$.
This situation was described in some detail by Kerckhoff in various lectures during the period 2010-2012 about the work \cite{dk-future} (still in preparation); this included a lecture at the 2010 workshop on Geometry, topology, and dynamics of character varieties at the National University of Singapore.
\end{example}

\subsection{Convex cocompactness on quotient spaces} \label{subsec:cc-quotient}

Here is a consequence of Proposition~\ref{prop:cc-block} and Property~\ref{item:dual} of Theorem~\ref{thm:properties}.

\begin{proposition} \label{prop:cc-quotient}
Let $\Gamma$ be a discrete subgroup of\ $\SL^{\pm}(V)$ acting trivially on some linear subspace $V_0$ of~$V$.
Then the following are equivalent:
\begin{enumerate}
  \item\label{item:cc-quotient-V} $\Gamma$ is convex cocompact in $\PP(V)$,
  \item\label{item:cc-quotient-V/V_0} the induced representation $\Gamma\to\SL^{\pm}(V/V_0)$ is injective and its image is convex cocompact in $\PP(V/V_0)$.
\end{enumerate}
\end{proposition}

\begin{proof}
Fix a complement $W$ to $V_0$ in $V$ so that $V = W \oplus V_0$.
In a basis $\mathscr B$ respecting this decomposition, each element $\gamma \in \Gamma$ has block matrix form
$$[\gamma]_{V_0 \oplus W} = \begin{bmatrix} A_\gamma & 0 \\ B_
\gamma & \mathrm{Id}_{V_0} \end{bmatrix}.$$
Note that the induced representation $\Gamma \to \SL^\pm(V/V_0)$ is equivalent to the representation $\gamma \mapsto A_\gamma$ on $W_0$. 
Consider the dual action of $\Gamma$ on $\PP(V^*)$.
There is a dual splitting $V^* = W^* \oplus V_0^*$, where $W^*$ and $V_0^*$ are the subspaces of~$V^*$ consisting of the linear functionals that vanish on $V_0$ and $W$, respectively.
The basis $\mathscr B^*$ dual to the basis $\mathscr B$ above respects this splitting. In this basis, the action of $\Gamma$ on~$V^*$ has block matrix form
$$[\gamma ]_{W^* \oplus V_0^*} = \begin{bmatrix} ^t\!A_{\gamma} & ^t\!B_{\gamma}\\ 0 & \mathrm{Id}_{V_0^*} \end{bmatrix}.$$ Hence the dual action on $V^*$ is of the form $\rho^u$ where $\rho: \Gamma \to \SL^\pm(V^*)$ is the dual action and $u: \Gamma \to \Hom(V_0^*, W^*)$ is the $\rho$-cocycle such that $u(\gamma)$ has matrix $^t\!B_{\gamma}$. 

By Property~\ref{item:dual} of Theorem~\ref{thm:properties}, the group $\Gamma$ is convex cocompact in $\PP(V)$ if and only if it is convex cocompact in $\PP(V^*)$.
By Proposition~\ref{prop:cc-block}, this is equivalent to the representation $\gamma \mapsto {}^t\!A_{\gamma}$ being injective with convex cocompact image in $\PP(W^*)$.
By Property~\ref{item:dual} again, this is equivalent to the representation $\gamma \mapsto A_\gamma$ being injective with convex cocompact image in $\PP(W)$.
Since the representation $\gamma \mapsto A_\gamma$ on $W$ is equivalent to the induced representation $\Gamma \to \SL^\pm(V/V_0)$, the proof is complete.
\end{proof}

\section{Convex cocompactness in $\HH^{p,q-1}$} \label{sec:Hpq-cc}

Fix $p,q\in\NN^*$.
Recall from Section~\ref{subsec:intro-cc-POpq} that the projective space $\PP(\RR^{p+q})$ is the disjoint union of
$$\HH^{p,q-1} = \big\{[v]\in\PP(\RR^{p,q})\,|\,\langle v,v\rangle_{p,q}<0\big\},$$
of $\SS^{p-1,q} = \{[v]\in\PP(\RR^{p,q})\,|\,\langle v,v\rangle_{p,q}>0\}$, and of
$$\partial\HH^{p,q-1} = \partial\SS^{p-1,q} = \big\{[v]\in\PP(\RR^{p,q})\,|\,\langle v,v\rangle_{p,q}=0\big\}.$$
For instance, Figure~\ref{fig:H123} shows
$$\PP(\RR^4) = \HH^{3,0} \sqcup \big(\partial\HH^{3,0} = \partial\SS^{2,1}\big) \sqcup \SS^{2,1}$$
and
$$\PP(\RR^4) = \HH^{2,1} \sqcup \big(\partial\HH^{2,1} = \partial\SS^{1,2}\big) \sqcup \SS^{1,2}.$$
\begin{figure}[h!]
\centering
\labellist
\small\hair 2pt
\pinlabel $\SS^{1,2}$ [u] at 160 61
\pinlabel $\SS^{2,1}$ [u] at -20 61
\pinlabel $\HH^{3,0}$ [u] at 50 35
\pinlabel $\HH^{2,1}$ [u] at 207 35
\pinlabel $\ell$ [u] at 50 87
\pinlabel $\ell_2$ [u] at 207 52
\pinlabel $\ell_0$ [u] at 219 90 
\pinlabel $\ell_1$ [u] at 197 90
\endlabellist
\includegraphics[scale=1]{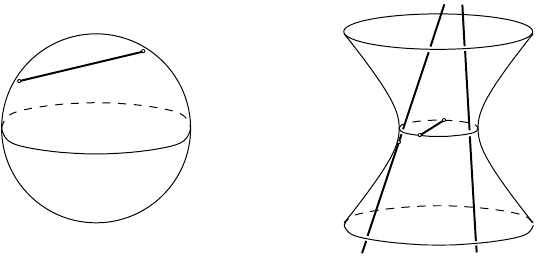}
\caption{Left: $\HH^{3,0}$ with a geodesic line $\ell$ (necessarily spacelike), and $\SS^{2,1}$. Right: $\HH^{2,1}$ with three geodesic lines $\ell_2$ (spacelike), $\ell_1$ (lightlike), and $\ell_0$ (timelike), and $\SS^{1,2}$.}
\label{fig:H123}
\end{figure}
As explained in \eg \cite[\S\,2.4]{wolf} or \cite[\S\,2.1]{dgk-ccHpq}, the space $\HH^{p,q-1}$ has a natural $\PO(p,q)$-invariant pseudo-Riemannian structure of constant negative curvature, for which the geodesic lines are the intersections of $\HH^{p,q-1}$ with projective lines in $\PP(\RR^{p+q})$.
Such a line is called \emph{spacelike} (\resp \emph{lightlike}, \resp \emph{timelike}) if it meets $\partial_{\scriptscriptstyle\PP}\HH^{p,q-1}$ in two (\resp one, \resp zero) points: see Figure~\ref{fig:H123}.

\begin{remark} \label{rem:prox-lim-set-POpq}
An element $g\in\PO(p,q)$ is proximal in $\PP(\RR^{p+q})$ (Definition~\ref{def:prox}) if and only if it admits a unique attracting fixed point $\xi_g^+$ in $\partial\HH^{p,q-1}$, in which case we shall say that $g$ is proximal in $\partial\HH^{p,q-1}$.
In particular, for a discrete subgroup $\Gamma$ of $\PO(p,q)\subset\PGL(\RR^{p+q})$, the proximal limit set $\Lambda_{\Gamma}$ of $\Gamma$ in $\PP(\RR^{p+q})$ (Definition~\ref{def:prox-lim-set}) is contained in $\partial\HH^{p,q-1}$, and called the proximal limit set of $\Gamma$ in $\partial\HH^{p,q-1}$.
\end{remark}

In this section we prove Theorems \ref{thm:main-POpq-reducible} and~\ref{thm:Hpq-general}.

For Theorem~\ref{thm:main-POpq-reducible}, the equivalences \ref{item:ccc-POpq-pos}~$\Leftrightarrow$~\ref{item:ccc-CM-Hpq-pos}~$\Leftrightarrow$~\ref{item:Hqp-ccc}~$\Leftrightarrow$~\ref{item:ccc-CM-Hpq-in-Hqp}~$\Leftrightarrow$~\ref{item:Anosov-pos} follow from the equivalences \ref{item:ccc-POpq-neg}~$\Leftrightarrow$~\ref{item:ccc-CM-Hpq-neg}~$\Leftrightarrow$~\ref{item:Hpq-ccc}~$\Leftrightarrow$~\ref{item:ccc-CM-Hpq-in-Hpq}~$\Leftrightarrow$~\ref{item:Anosov-neg} by replacing the symmetric bilinear form $\langle\cdot,\cdot\rangle_{p,q}$ by $-\langle\cdot,\cdot\rangle_{p,q}\simeq\langle\cdot,\cdot\rangle_{q,p}$: see \cite[Rem.\,4.1]{dgk-ccHpq}.
The implications \eqref{item:ccc-CM-Hpq}~$\Rightarrow$~\eqref{item:ccc-POpq}, \ref{item:ccc-CM-Hpq-pos}~$\Rightarrow$~\ref{item:ccc-POpq-pos}, and \ref{item:ccc-CM-Hpq-neg}~$\Rightarrow$~\ref{item:ccc-POpq-neg} hold by definition.
The implication \ref{item:ccc-CM-Hpq-neg}~$\Rightarrow$~\ref{item:Anosov-neg} is contained in the forward implication of Theorem~\ref{thm:Ano-PGL}, which has been established in Section~\ref{sec:ccc-implies-Anosov}.
Here we shall prove the implications \eqref{item:ccc-POpq}~$\Rightarrow$~\eqref{item:ccc-CM-Hpq}~$\Rightarrow$~\eqref{item:ccc-Hpq-Hqp} and \ref{item:Anosov-neg}~$\Rightarrow$~\ref{item:ccc-CM-Hpq-in-Hpq}~$\Rightarrow$~\ref{item:Hpq-ccc}~$\Rightarrow$~\ref{item:ccc-CM-Hpq-neg}.
We note that \eqref{item:ccc-Hpq-Hqp}~$\Rightarrow$~\eqref{item:ccc-CM-Hpq} is contained in \ref{item:Hpq-ccc}~$\Rightarrow$~\ref{item:ccc-CM-Hpq-neg} and \ref{item:Hqp-ccc}~$\Rightarrow$~\ref{item:ccc-CM-Hpq-pos}, and that \eqref{item:ccc-POpq}~$\Rightarrow$~\eqref{item:ccc-CM-Hpq} implies \ref{item:ccc-POpq-pos}~$\Rightarrow$~\ref{item:ccc-CM-Hpq-pos} and \ref{item:ccc-POpq-neg}~$\Rightarrow$~\ref{item:ccc-CM-Hpq-neg}.

\subsection{Proof of the implication \eqref{item:ccc-POpq}~$\Rightarrow$~\eqref{item:ccc-CM-Hpq} in Theorem~\ref{thm:main-POpq-reducible}}
\label{subsec:ccimpliesscc}

We start with the following general lemma.

\begin{lemma} \label{lem:extreme-pt-supp-hyp-PO}
Let $\Gamma$ be a discrete subgroup of $\PO(p,q)$ preserving a nonempty properly convex open subset $\Omega$ of $\PP(\RR^{p+q})$.
If $z\in\Lambdao_{\Omega}$ is an extreme point of $\partial\Omega$, then $z^{\perp}$ is a supporting hyperplane to $\Omega$ at~$z$.
In particular (Corollary~\ref{cor:naive-cc-Lambdaorb-full}), if the action of $\Gamma$ on~$\Omega$ is convex cocompact and if $z\in\Lambdao_{\Omega}$ is an extreme point of $\overline{\Ccore_{\Gamma}(\Omega)}$, then $z^{\perp}$ is a supporting hyperplane to $\Omega$ at~$z$.
\end{lemma}

\begin{proof}[Proof of Lemma~\ref{lem:extreme-pt-supp-hyp-PO}]
Since $z\in\Lambdao_{\Omega}$, there exist $x\in\Omega$ and $(\gamma_m)\in\Gamma^{\NN}$ such that $\gamma_m\cdot x\to z$.
If $z$ is an extreme point of $\partial\Omega$, then we actually have $\gamma_m\cdot x\to z$ for all $x\in\Omega$, by Lemma~\ref{lem:unif-neighb-face}.\eqref{item:distance-goes-down}.

Up to replacing the symmetric bilinear form $\langle\cdot,\cdot\rangle_{p,q}$ by its negative, we may assume that $p\geq q$, so that $\PO(p,q)$ has real rank~$q$.
Choose a basis $(e_1,\dots,e_{p+q})$ of $\RR^{p+q}$ in which the quadratic form $\langle\cdot,\cdot\rangle_{p,q}$ takes the form $2 v_1 v_{p+q} + \cdots + 2 v_q v_{p+1} + v_{q+1}^2 + \cdots + v_p^2$.
Similarly to Section~\ref{subsec:Cartan}, the Cartan decomposition 
$\PO(p,q) = \underline{K} \exp(\underline{\mathfrak{a}}^+) \underline{K}$ holds, where 
$\underline{K} \simeq \mathrm{P}(\OO(p)\times\OO(q))$ is the stabilizer of 
$\mathrm{span}(e_1+e_{p+q},\dots,e_q+e_{p+1},e_{q+1},\dots,e_p) \simeq \RR^{p,0}$ and 
$\mathrm{span}(e_1-e_{p+q},\dots,e_p-e_{q+1}) \simeq \RR^{0,q}$
and where
$$\underline{\mathfrak{a}}^+ := \left\{ 
\begin{array}{ll}
\{ \mathrm{diag}(t_1,\dots,t_q,0,\dots,0,-t_q,\dots,-t_1) ~|~ t_1\geq\dots\geq t_q\geq 0\} & \text{if }p>q;\\
\{ \mathrm{diag}(t_1,\dots,t_q,0,\dots,0,-t_q,\dots,-t_1) ~|~ t_1\geq\dots\geq t_{q-1}\geq |t_q|\} & \text{if }p=q.
\end{array}
\right.$$
This means that any $g\in\PO(p,q)$ may be written $g = k\exp(a)k'$ for some $k,k'\in\underline{K}$ and a unique $a=\mathrm{diag}(t_1,\dots,t_{p+q})\in\underline{\mathfrak{a}}^+$ with $t_i = -t_{p+q-i}$ for all~$i$; the real number $t_i$ is the logarithm $\mu_i(g)$ of the $i$-th largest \emph{singular value} of (any representative in $\OO(p,q)$ of)~$g$.

For any~$m$ we can write $\gamma_m = k_m a_m k'_m \in \underline{K} \exp(\underline{\mathfrak{a}}^+) \underline{K}$.
Up to extracting, we may assume that $(k_m)_{m\in\NN}$ converges to some $k\in\underline{K}$ and $(k'_m)_{m\in\NN}$ to some $k'\in\underline{K}$.
Arguing exactly as in the proof of Lemma~\ref{lem:P1-div}, the fact that $\gamma_m\cdot x\to z$ for all $x\in\Omega$ implies that $z = k\cdot [e_1]$ and that $(\mu_1-\mu_2)(\gamma_m)\to +\infty$.
Therefore $(\mu_{p+q-1}-\mu_{p+q})(\gamma_m) = (\mu_1-\mu_2)(\gamma_m)\to +\infty$.
As in the proof of Lemma~\ref{lem:accumulation}.\eqref{item:conv-xi*}, we then obtain that $\gamma_m\cdot x^*\to k\cdot\PP(\mathrm{span}(e_1,\dots,e_{p+q-1}) = z^{\perp}$ for all projective hyperplanes $x^*$ of $\PP(\RR^{p+q})$ that do not contain ${k'}^{-1}\cdot [e_{p+q}]$.
In particular, $\gamma_m\cdot x^*\to z^{\perp}$ for all $x^*$ in a dense open subset of~$\Omega^*$, hence $z^{\perp} \in \Lambdao_{\Omega^*}(\Gamma) \subset \partial\Omega^*$.
This implies that $z^{\perp}$ is a supporting hyperplane to $\Omega$ at~$z$.
\end{proof}

\begin{proof}[Proof of the implication \eqref{item:ccc-POpq}~$\Rightarrow$~\eqref{item:ccc-CM-Hpq} in Theorem~\ref{thm:main-POpq-reducible}]
Suppose $\Gamma$ acts convex cocompactly on some nonempty properly convex open subset $\Omega$ of $\PP(\RR^{p+q})$.
Let us prove that $\Gamma$ is strongly convex cocompact in $\PP(\RR^{p+q})$.

It is sufficient to check that $\Ccore_{\Omega}(\Gamma)$ is contained in $\HH^{p,q-1}$ or in $\PP(\RR^{p+q}) \smallsetminus \overline{\HH^{p,q-1}} \simeq \HH^{q,p-1}$.
Indeed, suppose this is the case.
By Corollaries \ref{cor:Lambdao-Hpq} and~\ref{cor:ideal-bound-naive-cc}.\eqref{item:ideal-bound-cc}, we have $\partiali\Ccore_{\Omega}(\Gamma) = \Lambdao_{\Omega}(\Gamma) \subset \partial\HH^{p,q-1}$.
Therefore, if $\Ccore_{\Omega}(\Gamma)$ contained a PET, with vertices $[a], [b], [c]$, then the edges of the PET would lie in $\partial\HH^{p,q-1}$ and we would have $\langle a,b\rangle_{p,q} = \langle b,c\rangle_{p,q} = \langle c,a\rangle_{p,q} = 0$.
Therefore the projective plane $\PP(\mathrm{span}(a,b,c))$ would be entirely contained in $\partial\HH^{p,q-1}$: contradiction since the PET is contained in this projective plane and in $\Ccore_{\Omega}(\Gamma)$.
This shows that if $\Ccore_{\Omega}(\Gamma)$ is contained in $\HH^{p,q-1}$ or in $\PP(\RR^{p+q}) \smallsetminus \overline{\HH^{p,q-1}} \simeq \HH^{q,p-1}$, then $\Ccore_{\Omega}(\Gamma)$ does not contain any PET, and so $\Gamma$ is strongly convex cocompact in $\PP(\RR^{p+q})$ by Theorem~\ref{thm:main-noPETs}.

Let us now show that $\Ccore_{\Omega}(\Gamma)$ is contained in $\HH^{p,q-1}$ or in $\PP(\RR^{p+q}) \smallsetminus \overline{\HH^{p,q-1}} \simeq \HH^{q,p-1}$.

Up to replacing $\Gamma$ by an index-two subgroup, we may assume that it lifts to a subgroup of $\OO(p,q)$, still denoted by~$\Gamma$.
Then the $\Gamma$-invariant properly convex open subset $\Omega$ of $\PP(\RR^{p+q})$ lifts to a $\Gamma$-invariant properly convex open cone $\widetilde{\Omega}$ of $\RR^{p+q}$, and $\Lambdao_{\Omega}(\Gamma)$ lifts to a $\Gamma$-invariant closed cone $\widetilde{\Lambda}^{\mathsf{orb}}_{\Omega}(\Gamma)$ of $\RR^{p+q}$ contained in the boundary of~$\widetilde{\Omega}$.

Let $z\in\Lambdao_{\Omega}(\Gamma)$ be an extreme point of $\overline{\Ccore_{\Gamma}(\Omega)}$.
Let $\Lambda$ be the closure of the $\Gamma$-orbit of~$z$ in $\partial\Omega$, and let $\widetilde{\Lambda}$ be the preimage of $\Lambda$ in $\widetilde{\Lambda}^{\mathsf{orb}}_{\Omega}(\Gamma)$.

Firstly, we claim that the set
$$S := \{ \langle\tilde{y},\tilde{y}'\rangle_{p,q} ~|~ \tilde{y}\in\widetilde{\Lambda},\ \tilde{y}'\in\widetilde{\Lambda}^{\mathsf{orb}}_{\Omega}(\Gamma)\}$$
is contained in $\RR_{\geq 0}$ or in $\RR_{\leq 0}$, and cannot be reduced to $\{0\}$.
Indeed, we know from Lemma~\ref{lem:extreme-pt-supp-hyp-PO} that $z^{\perp}$ is a supporting hyperplane to $\Omega$ at~$z$.
Therefore the set $\{ \langle\tilde{z},\tilde{y}'\rangle_{p,q} \,|\linebreak \tilde{y}'\in\nolinebreak\widetilde{\Lambda}^{\mathsf{orb}}_{\Omega}(\Gamma)\}$ (and more generally the set $\{ \langle\tilde{z},\tilde{y}'\rangle_{p,q} \,|\, \tilde{y}'\in\overline{\widetilde{\Omega}}\}$) is contained in $\RR_{\geq 0}$ or in~$\RR_{\leq 0}$.
This set is equal to $\{ \langle\tilde{y},\tilde{y}'\rangle_{p,q} \,|\, \tilde{y}'\in\widetilde{\Lambda}^{\mathsf{orb}}_{\Omega}(\Gamma)\}$ for any $\tilde{y}\in\Gamma\cdot\tilde{z}$, hence for any $\tilde{y}\in\widetilde{\Lambda}$.
Therefore it is equal to~$S$.
This set is not $\{0\}$, for otherwise $\Lambdao_{\Omega}(\Gamma)$, hence $\Ccore_{\Omega}(\Gamma)$, would be contained in~$z^{\perp}$, contradicting the fact that $z^{\perp}$ is a supporting hyperplane to $\Omega$ at~$z$.

Since $S\neq\{0\}$, we can find $z'\in\Lambdao_{\Omega}(\Gamma)$ which is an extreme point of $\overline{\Ccore_{\Gamma}(\Omega)}$ and which lifts to $\tilde{z}'\in\widetilde{\Lambda}^{\mathsf{orb}}_{\Omega}(\Gamma)$ with $\langle\tilde{z},\tilde{z}'\rangle_{p,q} \neq 0$.
Let $\Lambda'$ be the closure of the $\Gamma$-orbit of~$z'$ in $\partial\Omega$, and let $\C$ be the convex hull of $\Lambda\cup\Lambda'$ inside~$\Omega$: it is a $\Gamma$-invariant closed convex subset of~$\Omega$.

We claim that $\C$ is nonempty.
Indeed, otherwise the convex hull of $\Lambda\cup\Lambda'$ in $\overline{\Omega}$ would be contained in $\partial\Omega$, hence in $\partiali \Ccore_\Omega(\Gamma) = \Lambdao_{\Omega}(\Gamma)$, hence in $\partial\HH^{p,q-1}$ by Corollary~\ref{cor:Lambdao-Hpq}.
This would contradict the fact that $\langle\tilde{z},\tilde{z}'\rangle_{p,q} \neq 0$.
Therefore $\C$ is nonempty.

We claim that $\C$ does not meet $\partial\HH^{p,q-1}$.
Indeed, consider $x \in (\overline{\C} \cap \partial\HH^{p,q-1}) \smallsetminus (\Lambda\cup\Lambda')$.
Let $\widetilde{\Lambda}'$ be the preimage of $\Lambda'$ in $\widetilde{\Lambda}^{\mathsf{orb}}_{\Omega}$.
Since $x \in \overline{\C} \smallsetminus (\Lambda\cup\Lambda')$, we can lift it to $\tilde{x} = \sum_{i=1}^m t_i \tilde{y}_i + \sum_{i=1}^{m'} t'_i \tilde{y}'_i$, where $1\leq m,m'\leq p+q$ and where $t_i,t'_i>0$ and $\tilde{y}_i\in\widetilde{\Lambda}$ and $\tilde{y}'_i\in\widetilde{\Lambda}'$ for all~$i$.
Since $x \in \partial\HH^{p,q-1}$ and $\Lambda,\Lambda' \subset \partial\HH^{p,q-1}$, we have
$$0 = \langle\tilde{x},\tilde{x}\rangle_{p,q} = \sum_{1\leq i\leq m,\ 1\leq j\leq m'} t_i t'_j \, \langle\tilde{y}_i,\tilde{y}'_j\rangle_{p,q}.$$
Moreover, all the $\langle\tilde{y}_i,\tilde{y}'_j\rangle_{p,q}$ belong to $S$, which is contained in $\RR_{\geq 0}$ or in $\RR_{\leq 0}$.
Therefore we must have $\langle\tilde{y}_i,\tilde{y}'_j\rangle_{p,q} = 0$ for all $i,j$.
In particular, $x \in y_1^{\perp}$, where $y_1\in\Lambda$ is the projection of~$\tilde{y}_1\in\widetilde{\Lambda}$.
But we have seen that $y_1^{\perp}$ is a supporting hyperplane to $\Omega$ at~$y_1$ (by Lemma~\ref{lem:extreme-pt-supp-hyp-PO}).
Therefore $x$ cannot be contained in $\C \subset \Omega$.
This shows that $\C$ does not meet $\partial\HH^{p,q-1}$.

We deduce that $\C$ is contained in $\HH^{p,q-1}$ or in $\PP(\RR^{p+q}) \smallsetminus \overline{\HH^{p,q-1}} \simeq \HH^{q,p-1}$.

By minimality of $\Ccore_{\Omega}(\Gamma)$ (Lemma~\ref{lem:Ccore-min}), we have $\C = \Ccore_{\Omega}(\Gamma)$.
Thus $\Ccore_{\Omega}(\Gamma)$ is contained in $\HH^{p,q-1}$ or in $\PP(\RR^{p+q}) \smallsetminus \overline{\HH^{p,q-1}} \simeq \HH^{q,p-1}$, which completes the proof.
\end{proof}

\subsection{Proof of the implication \ref{item:Anosov-neg}~$\Rightarrow$~\ref{item:ccc-CM-Hpq-in-Hpq} in Theorem~\ref{thm:main-POpq-reducible}} \label{subsec:proof-POpq-6->4}

Suppose that $\Gamma$ is word hyperbolic, that the natural inclusion $\Gamma\hookrightarrow\PO(p,q)$ is $P_1^{p,q}$-Anosov, and that the proximal limit set $\Lambda_{\Gamma}\subset\partial\HH^{p,q-1}$ is negative.

By definition of negativity, the set $\Lambda_{\Gamma}$ lifts to a cone $\widetilde{\Lambda}_{\Gamma}$ of $\RR^{p,q}\smallsetminus\{ 0\}$ on which all inner products $\langle\cdot,\cdot\rangle_{p,q}$ of noncollinear points are negative. 
The set
$$\Omega_{\max} := \PP\big(\big\{ v\in\RR^{p,q} ~|~ \langle v,v'\rangle_{p,q}<0\quad \forall v'\in\widetilde{\Lambda}_{\Gamma}\big\}\big)$$
is a connected component of $\PP(\RR^{p,q}) \smallsetminus \bigcup_{z\in\Lambda_{\Gamma}} z^{\perp}$ which is open and convex.
It is nonempty because it contains $[v_1+v_2]$ for any noncollinear $v_1,v_2\in\widetilde{\Lambda}_{\Gamma}$.
Note that if $\widehat{\Lambda}_\Gamma$ is another cone of $\RR^{p,q}\smallsetminus\{0\}$ lifting $\Lambda_{\Gamma}$ on which all inner products $\langle\cdot,\cdot\rangle_{p,q}$ of noncollinear points are negative, then either $\widehat{\Lambda}_\Gamma = \widetilde{\Lambda}_\Gamma$ or $\widehat{\Lambda}_\Gamma = -\widetilde{\Lambda}_\Gamma$ (using negativity).
Thus $\Omega_{\max}$ is well defined independently of the lift $\widetilde{\Lambda}_\Gamma$, and so $\Omega_{\max}$ is $\Gamma$-invariant because $\Lambda_\Gamma$ is.

By Proposition~\ref{prop:precise-Ano-implies-ccc}, the group $\Gamma$ acts convex cocompactly on some non\-empty, properly convex open subset $\Omega\subset\Omega_{\max}$, and $\Lambdao_{\Omega}(\Gamma)=\Lambda_{\Gamma}$.
In particular, the convex hull $\Ccore_{\Omega}(\Gamma)$ of $\Lambda_{\Gamma}$ in~$\Omega$ has compact quotient by~$\Gamma$.

As in \cite[Lem.\,3.6.(1)]{dgk-ccHpq}, this convex hull $\Ccore_{\Omega}(\Gamma)$ is contained in $\HH^{p,q-1}$: indeed, any point of $\Ccore_{\Omega}(\Gamma)$ lifts to a vector of the form $v = \sum_{i=1}^k t_i v_i$ with $k\geq 2$, where $v_1, \ldots, v_k \in \widetilde{\Lambda}_\Gamma$ are distinct and $t_1, \ldots, t_k >0$; for all $i\neq j$ we have $\langle v_i,v_i\rangle_{p,q}=0$ and $\langle v_i,v_j\rangle_{p,q}<0$ by negativity of~$\Lambda_\Gamma$, hence $\langle v,v\rangle_{p,q}<0$.

Since $\Ccore_{\Omega}(\Gamma)$ has compact quotient by~$\Gamma$, it is easy to check that for small enough $r > 0$ the open uniform $r$-neighborhood $\mathcal{U}$ of $\Ccore_{\Omega}(\Gamma)$ in $(\Omega,d_{\Omega})$ is contained in $\HH^{p,q-1}$ and properly convex (see \cite[Lem.\,6.3]{dgk-ccHpq}).
The full orbital limit set $\Lambdao_{\mathcal{U}}(\Gamma)$ is a nonempty closed $\Gamma$-invariant subset of $\Lambdao_{\Omega}(\Gamma)$, hence its convex hull $\Ccore_{\mathcal{U}}(\Gamma)$ in~$\mathcal{U}$ is a nonempty closed $\Gamma$-invariant subset of $\Ccore_{\Omega}(\Gamma)$.
(In fact $\Ccore_{\mathcal{U}}(\Gamma)=\Ccore_{\Omega}(\Gamma)$ by Lemma~\ref{lem:Ccore-min}.)
The nonempty set $\Ccore_{\mathcal{U}}(\Gamma)$ has compact quotient by~$\Gamma$.
Thus $\Gamma$ acts convex cocompactly on $\mathcal{U}\subset\HH^{p,q-1}$, which completes the proof.

\subsection{Proof of the implication \ref{item:ccc-CM-Hpq-in-Hpq}~$\Rightarrow$~\ref{item:Hpq-ccc} in Theorem~\ref{thm:main-POpq-reducible}} \label{subsec:proof-POpq-4->5}

Suppose that $\Gamma$ acts convex cocompactly on some nonempty properly convex open subset $\Omega$ of $\HH^{p,q-1}$.
By the implication \eqref{item:ccc-POpq}~$\Rightarrow$~\eqref{item:ccc-CM-Hpq} of Theorem~\ref{thm:main-POpq-reducible} (proved in Section~\ref{subsec:ccimpliesscc} above), $\Gamma$ is strongly convex cocompact, and so $\Lambda_\Gamma$ is transverse by the forward implication of Theorem~\ref{thm:Ano-PGL} (proved in Section~\ref{sec:ccc-implies-Anosov} above).

The set $\Lambda_{\Gamma}$ is contained in $\Lambdao_\Omega(\Gamma)$.
By Lemma~\ref{lem:Ccore-min} and Corollary~\ref{cor:ideal-bound-naive-cc}.\eqref{item:ideal-bound-cc}, the convex hull of $\Lambda_{\Gamma}$ in~$\Omega$ is the convex hull $\Ccore_\Omega(\Gamma)$ of $\Lambdao_{\Omega}(\Gamma)$ in~$\Omega$, and the ideal boundary $\partiali \Ccore_\Omega(\Gamma)$ is equal to $\Lambdao_{\Omega}(\Gamma)$.
By Corollary~\ref{cor:Lambdao-Hpq}, this set is contained in~$\partial \HH^{p,q-1}$.
It follows that $\Lambda_\Gamma = \Lambdao_{\Omega}(\Gamma)$: indeed, otherwise $\Lambdao_{\Omega}(\Gamma)$ would contain a nontrivial segment between two points of~$\Lambda_{\Gamma}$, this segment would be contained in $\partial \HH^{p,q-1}$, and this would contradict that $\Lambda_\Gamma$ is transverse.
Thus $\Ccore_\Omega(\Gamma)$ is a closed properly convex subset of $\HH^{p,q-1}$ on which $\Gamma$ acts properly discontinuously and cocompactly, and whose ideal boundary does not contain any nontrivial projective line segment.

It could be the case that $\Ccore_\Omega(\Gamma)$ has empty interior.
However, for any $r > 0$ the closed uniform $r$-neighborhood $\C_r$ of $\Ccore_{\Omega}(\Gamma)$ in $(\Omega,d_{\Omega})$ has nonempty interior, and is still properly convex with compact quotient by~$\Gamma$.
By Corollary~\ref{cor:ideal-bound-naive-cc}.\eqref{item:ideal-bound-cc}, we have $\partiali \C_r = \Lambdao_{\Omega}(\Gamma)$, hence $\partiali \C_r$ does not contain any nontrivial projective line segment.
This shows that $\Gamma$ is $\HH^{p,q-1}$-convex cocompact, which completes the proof.

\subsection{Proof of the implication \eqref{item:ccc-CM-Hpq}~$\Rightarrow$~\eqref{item:ccc-Hpq-Hqp} in Theorem~\ref{thm:main-POpq-reducible}}

The proof relies on the following proposition, which is stated in \cite[Prop.\,3.7]{dgk-ccHpq} for $\Gamma$ acting irreducibly on $\PP(\RR^{p+q})$.
The proof for general~$\Gamma$ is literally the same; we recall it for the reader's convenience.

\begin{proposition} \label{prop:limit-set-POpq-min}
For $p,q\in\NN^*$, let $\Gamma$ be a discrete subgroup of $\PO(p,q)$ preserving a nonempty properly convex open subset $\Omega$ of $\PP(\RR^{p,q})$.
Let $\Lambda_{\Gamma} \subset \partial\HH^{p,q-1}$ be the proximal limit set of~$\Gamma$ (Definition~\ref{def:prox-lim-set} and Remark~\ref{rem:prox-lim-set-POpq}).
If $\Lambda_{\Gamma}$ contains at least two points and is transverse, and if the action of $\Gamma$ on~$\Lambda_{\Gamma}$ is minimal (\ie every orbit is dense), then $\Lambda_{\Gamma}$ is negative or positive.
\end{proposition}

Recall from Remark~\ref{rem:irred-prox-lim-set} that if the action of $\Gamma$ on $\PP(V)$ is irreducible and $\Lambda_{\Gamma}\neq\emptyset$, then the action of $\Gamma$ on $\Lambda_{\Gamma}$ is always minimal.

\begin{proof}
Suppose that $\Lambda_{\Gamma}$ contains at least two points and is transverse, and that the action of $\Gamma$ on~$\Lambda_{\Gamma}$ is minimal.
By Proposition~\ref{prop:max-inv-conv}.\eqref{item:Lambda-prox-lift-Omega}, the sets $\Omega$ and $\Lambda_{\Gamma}$ lift to cones $\widetilde{\Omega}$ and $\widetilde{\Lambda}_\Gamma$ of $V\smallsetminus\{0\}$ with $\widetilde{\Omega}$ properly convex containing $\widetilde{\Lambda}_{\Gamma}$ in its boundary, and $\Omega^*$ and $\Lambda_{\Gamma}^*$ lift to cones $\widetilde{\Omega}^*$ and $\widetilde{\Lambda}_\Gamma^*$ of $V^*\smallsetminus\{0\}$ with $\widetilde{\Omega}^*$ properly convex containing $\widetilde{\Lambda}_{\Gamma}^*$ in its boundary, such that $\varphi(v)\geq 0$ for all $v\in\widetilde{\Lambda}_{\Gamma}$ and $\varphi\in\widetilde{\Lambda}_{\Gamma}^*$.
By Remark~\ref{rem:lift-Gamma}, the group $\Gamma$ lifts to a discrete subgroup $\hat\Gamma$ of $\OO(p,q)$ preserving~$\widetilde{\Omega}$ (hence also $\widetilde{\Lambda}_\Gamma$, $\widetilde{\Omega}^*$, and $\widetilde{\Lambda}_\Gamma^*$).
Note that the map $\psi : v\mapsto\langle v,\cdot\rangle_{p,q}$ from $\RR^{p,q}$ to $(\RR^{p,q})^*$ induces a homeomorphism $\Lambda_{\Gamma}\simeq\Lambda_{\Gamma}^*$.
For any $v\in\widetilde{\Lambda}_{\Gamma}$ we have $\psi(v)\in\widetilde{\Lambda}_{\Gamma}^*\cup -\widetilde{\Lambda}_{\Gamma}^*$.
Let $F^+$ (\resp $F^-$) be the subcone of $\widetilde{\Lambda}_{\Gamma}$ consisting of those vectors $v$ such that $\psi(v)\in\widetilde{\Lambda}_{\Gamma}^*$ (\resp $\psi(v)\in -\widetilde{\Lambda}_{\Gamma}^*$).
By construction, we have $v\in F^+$ if and only if $\langle v,v'\rangle_{p,q}\geq 0$ for all $v'\in\widetilde{\Lambda}_{\Gamma}$; in particular, $F^+$ is closed in~$\widetilde{\Lambda}_{\Gamma}$ and $\hat\Gamma$-invariant.
Similarly, $F^-$ is closed and $\hat\Gamma$-invariant.
The sets $F^+$ and~$F^-$ are disjoint since $\Lambda_{\Gamma}$ contains at least two points and is transverse.
Thus $F^+$ and~$F^-$ are disjoint, $\hat\Gamma$-invariant, closed subcones of~$\widetilde{\Lambda}_{\Gamma}$, whose projections to $\PP(\RR^{p,q})$ are disjoint, $\Gamma$-invariant, closed subsets of~$\Lambda_{\Gamma}$.
Since the action of $\Gamma$ on~$\Lambda_{\Gamma}$ is minimal, $\Lambda_{\Gamma}$ is the smallest nonempty $\Gamma$-invariant closed subset of $\PP(\RR^{p,q})$, and so $\{F^+, F^-\}=\{\widetilde{\Lambda}_{\Gamma}, \emptyset\}$.
If $\widetilde{\Lambda}_{\Gamma}=F^+$ then $\Lambda_{\Gamma}$ is nonnegative, hence positive by transversality. Similarly, if $\widetilde{\Lambda}_{\Gamma}=F^-$ then $\Lambda_{\Gamma}$ is negative.
\end{proof}

\begin{proof}[Proof of the implication \eqref{item:ccc-CM-Hpq}~$\Rightarrow$~\eqref{item:ccc-Hpq-Hqp} in Theorem~\ref{thm:main-POpq-reducible}]
Suppose $\Gamma$ is strongly convex cocompact in $\PP(\RR^{p+q})$.
By the implication \ref{item:ccc-CM}~$\Rightarrow$~\ref{item:P1Anosov} of Theorem~\ref{thm:main-noPETs}, the group $\Gamma$ is word hyperbolic and the natural inclusion $\Gamma\hookrightarrow\PO(p,q)$ is $P_1^{p,q}$-Anosov.
If $\#\Lambda_{\Gamma}> 2$, then the action of $\Gamma$ on $\partial_{\infty}{\Gamma}$, hence on $\Lambda_{\Gamma}$, is minimal, and so Proposition~\ref{prop:limit-set-POpq-min} implies that the set $\Lambda_{\Gamma}$ is negative or positive.
This last conclusion also holds, vacuously, if $\#\Lambda_{\Gamma}\leq 2$.
Thus the implications \ref{item:Anosov-neg}~$\Rightarrow$~\ref{item:Hpq-ccc} and \ref{item:Anosov-pos}~$\Rightarrow$~\ref{item:Hqp-ccc} of Theorem~\ref{thm:main-POpq-reducible} (proved in Sections \ref{subsec:proof-POpq-6->4} and~\ref{subsec:proof-POpq-4->5} just above) show that $\Gamma$ is $\HH^{p,q-1}$-convex cocompact or $\HH^{q,p-1}$-convex cocompact.
\end{proof}

\subsection{Proof of the implication \ref{item:Hpq-ccc}~$\Rightarrow$~\ref{item:ccc-CM-Hpq-neg} in Theorem~\ref{thm:main-POpq-reducible}} \label{subsec:Hpq-bis}

Suppose $\Gamma \subset \PO(p,q)$ is $\HH^{p,q-1}$-convex cocompact, \ie it acts properly discontinuously and cocompactly on a closed convex subset $\C$ of $\HH^{p,q-1}$ such that $\C$ has nonempty interior and $\partiali \C$ does not contain any nontrivial projective line segment.
We shall first show that $\Gamma$ satisfies condition~\ref{item:ccc-ugly} of Theorem~\ref{thm:main-noPETs}: namely, $\Gamma$ preserves a nonempty properly convex open subset $\Omega$ of $\PP(\RR^{p,q})$ and acts cocompactly on a closed convex subset $\C'$ of~$\Omega$ with nonempty interior, such that $\partiali \C'$ does not contain any nontrivial segment. 

One naive idea would be to take $\Omega=\Int(\C)$ and $\C'$ to be the convex hull $\C_0$ of $\partiali\C$ in~$\C$ (or some small thickening), but $\C_0$ might not be contained in $\Int(\C)$ (\eg if $\C=\C_0$).
So we must find a larger open set~$\Omega$.
Lemma~\ref{lem:no-degenerate-faces} below implies that the $\Gamma$-invariant convex open subset $\Omega_{\max} \supset \Int(\C)$ of Proposition~\ref{prop:max-inv-conv} contains~$\C$.
If $\Omega_{\max}$ is properly convex (\eg if the action of $\Gamma$ on $\PP(V)$ is irreducible), then we may take $\Omega=\Omega_{\max}$ and $\C'=\C$.
However, $\Omega_{\max}$ might not be properly convex; we shall show (Lemma~\ref{lem:nonzero}) that the $\Gamma$-invariant properly convex open set $\Omega = \Int(\C)^*$ (realized in the same projective space via the quadratic form) contains $\C_0$ (though possibly not $\C$) and we shall take $\C'$ to be the intersection of $\C$ with a neighborhood of~$\C_0$ in~$\Omega$.

The following key observation is similar to \cite[Lem.\,4.2]{dgk-ccHpq}.

\begin{lemma} \label{lem:no-degenerate-faces}
For $p,q\in\NN^*$, let $\Gamma$ be a discrete subgroup of $\PO(p,q)$ acting properly discontinuously and cocompactly on a nonempty properly convex closed subset $\C$ of~$\HH^{p,q-1}$.
Then $\C$ does not meet any hyperplane $z^{\perp}$ with $z\in\partiali\C$.
In other words, any point of~$\C$ sees any point of $\partiali\C$ in a spacelike direction.
\end{lemma}

\begin{proof}
Suppose by contradiction that $\C$ meets $z^{\perp}$ for some $z\in\partiali\C$.
Then $z^{\perp}$ contains a ray $[y,z) \subset \partialn\C$.
Let $(x_m)_{m\in\NN}$ be a sequence of points of $[y,z)$ converging to~$z$ (see Figure~\ref{fig:CinOhm}).
Since $\Gamma$ acts cocompactly on~$\C$, for any $m$ there exists $\gamma_m\in\Gamma$ such that $\gamma_m\cdot x_m$ belongs to a fixed compact subset of $\partialn\C$.
Up to taking a subsequence, the sequences $(\gamma_m\cdot x_m)_m$ and $(\gamma_m\cdot y)_m$ and $(\gamma_m\cdot z)_m$ converge respectively to some points $x_{\infty}, y_{\infty}, z_{\infty}$ in $\PP(\RR^{p,q})$.
We have $x_{\infty}\in\partialn\C$ and $y_{\infty}\in\partiali\C$ (because the action of $\Gamma$ on~$\C$ is properly discontinuous) and $z_{\infty}\in\partiali\C$ (because $\partiali\C=\Fr(\C)\cap\partial\HH^{p,q-1}$ is closed in $\PP(\RR^{p,q})$).
The segment $[y_{\infty},z_{\infty}]$ is contained in $z_{\infty}^{\perp}$, hence its intersection with $\HH^{p,q-1}$ is contained in a lightlike geodesic and can meet $\partial\HH^{p,q-1}$ only at $z_{\infty}$.
Therefore $y_{\infty}=z_{\infty}$ and the closure of $\C$ in $\PP(\RR^{p,q})$ contains a full projective line, contradicting the proper convexity of~$\C$.
\end{proof}
 
\begin{figure}[h]
\centering
\labellist
\small\hair 2pt
\pinlabel {$z$} [u] at 37 49
\pinlabel {$x_m$} [u] at 48 54
\pinlabel {$y$} [u] at 44 71
\pinlabel {$\gamma_m \! \cdot\! z$} [u] at 74 31
\pinlabel {$\gamma_m \!\cdot\! x_m$} [u] at 72 53
\pinlabel {$\gamma_m \!\cdot\! y$} [u] at 85 79
\pinlabel {$\C$} [u] at 100 65
\pinlabel {$\partiali\C$} [u] at 106.5 33
\pinlabel {$\partial\HH^{p,q-1}$} [u] at 10 86
\endlabellist
\includegraphics[scale=1.5]{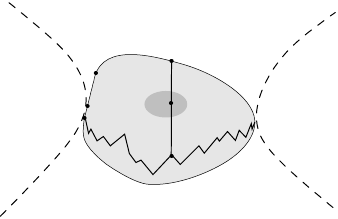}
\caption{Illustration for the proof of Lemma~\ref{lem:no-degenerate-faces}}
\label{fig:CinOhm}
\end{figure}

\begin{lemma} \label{lem:nonzero}
In the setting of Lemma~\ref{lem:no-degenerate-faces}, suppose $\C$ has nonempty interior.
Then the convex hull $\C_0$ of $\partiali\C$ in~$\C$ is contained in the $\Gamma$-invariant properly convex open set
$$\Omega := \big\{ x\in\PP(\RR^{p,q}) ~|~ x^\perp \cap \overline{\C} = \emptyset\big\}.$$
\end{lemma}

Note that $\Omega$ is the dual of $\Int(\C)$ realized in $\RR^{p,q}$ (rather than $(\RR^{p,q})^*$) via the symmetric bilinear form $\langle\cdot,\cdot\rangle_{p,q}$.

\begin{proof}
The properly convex set $\C \subset \HH^{p,q-1}$ lifts to a properly convex cone $\widetilde{\C}$ of $\RR^{p,q}\smallsetminus\{0\}$ such that $\langle v,v\rangle_{p,q}<0$ for all $v\in\widetilde{\C}$.
We denote by $\widetilde{\C_0} \subset \widetilde{\C}$ the preimage of~$\C_0$.
The ideal boundary $\partiali\C$ lifts to the intersection $\widetilde{\partiali\C} \subset \RR^{p,q} \smallsetminus \{0\}$ of the closure $\overline{\widetilde{\C}}$ of $\widetilde{\C}$ with the null cone of $\langle\cdot,\cdot\rangle_{p,q}$ minus~$\{0\}$.
For any $v \in \overline{\widetilde{\C}}$ and $v' \in \widetilde{\partiali\C}$, we have $\langle v,v' \rangle_{p,q} \leq 0$: indeed, this is easily seen by considering $tv+v' \in \RR^{p,q}$, which for small $t>0$ must belong to $\overline{\widetilde{\C}}$ hence have nonpositive norm; see also \cite[Lem.\,3.6.(1)]{dgk-ccHpq}.
By Lemma~\ref{lem:no-degenerate-faces} we have in fact $\langle v, v'\rangle_{p,q} < 0$ for all $v\in\widetilde{\C}$ and $v'\in\widetilde{\partiali\C}$.
In particular, this holds for $v\in\widetilde{\C_0}$.
Now consider $v\in\widetilde{\C_0}$ and $v' \in \widetilde{\C}$.
Since $\C_0$ is the convex hull of $\partiali\C$ in~$\C$, we may write $v = \sum_{i=1}^k t_i v_i$ where $v_1, \ldots, v_k \in \widetilde{\partiali\C}$ and $t_1, \ldots, t_k >0$.
By Lemma~\ref{lem:no-degenerate-faces} we have $\langle v_i,v' \rangle_{p,q} < 0$ for all~$i$, hence $\langle v, v'\rangle_{p,q} < 0$.
This proves that $\C_0$ is contained in~$\Omega$.
The open set $\Omega$ is properly convex because $\C$ has nonempty interior.
\end{proof}

\begin{corollary} \label{cor:Hpq-cc-implies-v}
In the setting of Lemma~\ref{lem:nonzero}, the group $\Gamma$ acts cocompactly on some closed properly convex subset $\C'$ of~$\Omega$ with nonempty interior which is contained in $\C\subset\HH^{p,q-1}$, with $\partiali\C' = \partiali\C$.
\end{corollary}

\begin{proof}
Since the action of $\Gamma$ on $\C_0$ is cocompact, it is easy to check that for any small enough $r > 0$ the closed uniform $r$-neighborhood $\C_r$ of $\C_0$ in $(\Omega, d_{\Omega})$ is contained in $\HH^{p,q-1}$ (see \cite[Lem.\,6.3]{dgk-ccHpq}).
The set $\C' := \C_r \cap \C$ is then a closed properly convex subset of~$\Omega$ with nonempty interior, and $\Gamma$ acts properly discontinuously and cocompactly on~$\C'$.
Since $\C'$ is also a closed subset of $\HH^{p,q-1}$, we have $\partiali \C' = \overline{\C'} \cap \partial\HH^{p,q-1} = \partiali \C_r \cap \partiali \C = \partiali \C$.
\end{proof}

If $\partiali\C$ does not contain any nontrivial segment, then neither does $\partiali \C'$ for $\C'$ as in Lemma~\ref{cor:Hpq-cc-implies-v}.
This proves that if $\Gamma$ is $\HH^{p,q-1}$-convex cocompact (\ie satisfies \ref{item:Hpq-ccc} of Theorem~\ref{thm:main-POpq-reducible}, \ie \eqref{item:ccc-Hpq-Hqp} up to switching $p$ and~$q$), then it satisfies condition~\ref{item:ccc-ugly} of Theorem~\ref{thm:main-noPETs}, as announced.

Now, Theorem~\ref{thm:main-noPETs} states that \ref{item:ccc-ugly} is equivalent to strong convex cocompactness: this yields the implication \eqref{item:ccc-Hpq-Hqp}~$\Rightarrow$~\eqref{item:ccc-CM-Hpq} of Theorem~\ref{thm:main-POpq-reducible}.
Condition~\ref{item:P1Anosov} of Theorem~\ref{thm:main-noPETs}, equivalent to~\ref{item:ccc-ugly}, also says that the group $\Gamma$ is word hyperbolic and that the natural inclusion $\Gamma \hookrightarrow \PO(p,q) $ is $P_1^{p,q}$-Anosov.
Since in addition the proximal limit set $\Lambda_{\Gamma}$ is the ideal boundary $\partiali \C'$, we find that $\Lambda_{\Gamma}$ is negative. 

This completes the proof of the implication \ref{item:Hpq-ccc}~$\Rightarrow$~\ref{item:ccc-CM-Hpq-neg} in Theorem~\ref{thm:main-POpq-reducible}.

\subsection{Proof of Theorem~\ref{thm:Hpq-general}}

The implications \eqref{item:Hpq-strict-C1}~$\Rightarrow$~\eqref{item:Hpq-strict}~$\Rightarrow$~\eqref{item:Hpq-bisat} of Theorem~\ref{thm:Hpq-general} hold trivially.
We now prove \eqref{item:Hpq-cc}~$\Leftrightarrow$~\eqref{item:Hpq-bisat}~$\Rightarrow$~\eqref{item:Hpq-strict-C1}.
We start with the following observation.

\begin{lemma} \label{lem:bisat-Hpq}
For $p,q\in\NN^*$, let $\Gamma$ be an infinite discrete subgroup of $\PO(p,q)$ acting properly discontinuously and cocompactly on a closed convex subset  $\C$ of $\HH^{p,q-1}$ with nonempty interior.
Then $\C$ has bisaturated boundary if and only if $\partialn\C$ does not contain any infinite geodesic line of~$\HH^{p,q-1}$.
\end{lemma}

\begin{proof}
Suppose $\partialn\C$ contains an infinite geodesic line of~$\HH^{p,q-1}$.
Since $\C$ is properly convex and closed in~$\HH^{p,q-1}$, this line must meet $\partial\HH^{p,q-1}$ in a point of $\partiali\C$, and so $\C$ does not have bisaturated boundary.

Conversely, suppose $\C$ does not have bisaturated boundary.
Since $\partiali \C=\Fr(\C)\cap\partial\HH^{p,q-1}$ is closed in $\PP(\RR^{p,q})$, there exists a ray $[y,z) \subset \partialn \C$ terminating at a point $z \in \partiali \C$.
Let $(a_m)_{m\in\NN}$ be a sequence of points of $[y,z)$ converging to~$z$ (see Figure~\ref{fig:CinOhm}).
Since $\Gamma$ acts cocompactly on~$\C$, for any $m$ there exists $\gamma_m\in\Gamma$ such that $\gamma_m\cdot a_m$ belongs to a fixed compact subset of~$\partialn\C$.
Up to taking a subsequence, the sequences $(\gamma_m\cdot a_m)_m$ and $(\gamma_m\cdot y)_m$ and $(\gamma_m\cdot z)_m$ converge respectively to some points $a_{\infty}, y_{\infty}, z_{\infty}$ in $\PP(V)$.
We have $a_{\infty}\in\partialn\C$ and $y_{\infty}\in\partiali\C$ (because the action of $\Gamma$ on~$\C$ is properly discontinuous) and $z_{\infty}\in\partiali\C$ (because $\partiali\C$ is closed).
Moreover, $a_{\infty}\in (y_{\infty},z_{\infty})$.
Thus $(y_{\infty},z_{\infty})$ is an infinite geodesic line of~$\HH^{p,q-1}$ contained in $\partialn \C$.
\end{proof}

Suppose condition \eqref{item:Hpq-cc} of Theorem~\ref{thm:Hpq-general} holds, \ie $\Gamma<\PO(p,q)$ is $\HH^{p,q-1}$-convex cocompact.
By Corollary~\ref{cor:Hpq-cc-implies-v}, the group $\Gamma$ preserves a properly convex open subset $\Omega$ of $\PP(\RR^{p,q})$ and acts cocompactly on a closed convex subset $\C'$ of~$\Omega$ with nonempty interior which is contained in $\HH^{p,q-1}$, and whose ideal boundary $\partiali \C'$ does not contain any nontrivial segment.
As in the proof of Corollary~\ref{cor:Hpq-cc-implies-v}, for small enough $r > 0$ the closed uniform $r$-neighborhood $\C'_r$ of $\C'$ in $(\Omega, d_{\Omega})$ is contained in $\HH^{p,q-1}$.
By Lemma~\ref{lem:naive-cc-nonempty-int} and Corollary~\ref{cor:ideal-bound-naive-cc}.\eqref{item:ideal-bound-cc}, we have $\partiali \C'_r = \partiali \C' = \Lambdao_\Omega(\Gamma)$, hence $\C'_r$ has bisaturated boundary by Lemma~\ref{lem:1-neighb-bisat}.
In particular, $\partialn \C'_r$ does not contain any infinite geodesic line of~$\HH^{p,q-1}$ by Lemma~\ref{lem:bisat-Hpq}.
Thus condition \eqref{item:Hpq-bisat} of Theorem~\ref{thm:Hpq-general} holds, with $\C_{\mathsf{bisat}} = \C'_r$.

Conversely, suppose condition~\eqref{item:Hpq-bisat} of Theorem~\ref{thm:Hpq-general} holds, \ie $\Gamma$ acts properly discontinuously and cocompactly on a closed convex subset $\C_{\mathsf{bisat}}$ of $\HH^{p,q-1}$ such that $\partialn \C_{\mathsf{bisat}}$ does not contain any infinite geodesic line of~$\HH^{p,q-1}$.
By Lemma~\ref{lem:bisat-Hpq}, the set $\C_{\mathsf{bisat}}$ has bisaturated boundary.
By Corollary~\ref{cor:bisat-interior}, the group $\Gamma$ acts convex cocompactly on $\Omega := \Int(\C_{\mathsf{bisat}})$ and $\Lambdao_{\Omega}(\Gamma) = \partiali \C_{\mathsf{bisat}}$.
By Lemma~\ref{lem:strict-C1-domain}, the group $\Gamma$ acts cocompactly on a closed convex subset $\C_{\mathsf{smooth}} \supset \Ccore_{\Omega}(\Gamma)$ of~$\Omega$ whose nonideal boundary is strictly convex and~$C^1$, and $\partiali \C_{\mathsf{smooth}} = \Lambdao_\Omega(\Gamma) \subset \partial \HH^{p,q-1}$ (see Corollary~\ref{cor:ideal-bound-naive-cc}.\eqref{item:ideal-bound-cc}).
In particular, $\C_{\mathsf{smooth}}$ is closed in $\HH^{p,q-1}$ and condition~\eqref{item:Hpq-strict-C1} of Theorem~\ref{thm:Hpq-general} holds.
Since $\C_{\mathsf{bisat}}$ has bisaturated boundary, any inextendable segment in $\partiali \C_{\mathsf{bisat}}$ is inextendable in $\partial \Omega = \partiali \C_{\mathsf{bisat}} \cup \partialn \C_{\mathsf{bisat}}$.
By Lemma~\ref{lem:segments-imply-PETs}, if $\partiali \C_{\mathsf{smooth}} = \partiali \C_{\mathsf{bisat}}$ contained a nontrivial segment, then $\C_{\mathsf{smooth}}$ would contain a PET, but that is not possible since closed subsets of $\HH^{p,q-1}$ do not contain PETs; see Remark~\ref{rem:Hpq-no-PET} below.
Therefore $\partiali \C_{\mathsf{smooth}} = \partiali \C_{\mathsf{bisat}}$ does not contain any nontrivial segment, and so $\Gamma$ is $\HH^{p,q-1}$-convex cocompact, \ie condition~\eqref{item:Hpq-cc} of Theorem~\ref{thm:Hpq-general} holds.
This concludes the proof of Theorem~\ref{thm:Hpq-general}.

In the proof we have used the following elementary observation.

\begin{remark} \label{rem:Hpq-no-PET}
For $p,q\in\NN^*$ with $p+q\geq 3$, a closed subset of $\HH^{p,q-1}$ cannot contain a PET.
Indeed, any triangle of $\PP(\RR^{p,q})$ whose edges lie in $\partial \HH^{p,q-1}$ must have interior in $\partial \HH^{p,q-1}$, because in this case the symmetric bilinear form is zero on the projective span.
\end{remark}

\subsection{$\HH^{p,q-1}$-convex cocompact groups whose boundary is a $(p-1)$-sphere}

In what follows we use the term \emph{standard $(p-1)$-sphere in $\partial \HH^{p,q-1}$} to mean the intersection with $\partial \HH^{p,q-1}$ of the projectivization of a $(p+1)$-dimensional subspace of $\RR^{p+q}$ of signature $(p,1)$; by an \emph{embedding} we mean a map which is a homeomorphism onto its image.

\begin{lemma} \label{lem:Hpq-cc-sphere}
For $p,q\geq 1$, let $\Gamma$ be an infinite discrete subgroup of $\PO(p,q)$ which is $\HH^{p,q-1}$-convex cocompact, acting properly discontinuously and cocompactly on some closed convex subset $\C$ of $\HH^{p,q-1}$ with nonempty interior whose ideal boundary $\partiali \C \subset \partial \HH^{p,q-1}$ does not contain any nontrivial projective line segment.
Let $\partial_{\infty}\Gamma$ be the Gromov boundary of~$\Gamma$, and let $\mathcal{S}$ be any standard $(p-1)$-sphere in $\partial \HH^{p,q-1}$.
Then
\begin{enumerate}
  \item\label{item:Hpq-sphere-1} the $\Gamma$-equivariant boundary map $\xi : \partial_{\infty}\Gamma \to \partial \HH^{p,q-1}$, which is an embedding with image~$\partiali\C$, is homotopic to an embedding of $\partial_{\infty}\Gamma$ whose image is contained in the $(p-1)$-sphere~$\mathcal{S}$.
\end{enumerate}
Suppose $\partial_{\infty} \Gamma$ is homeomorphic to a $(p-1)$-dimensional sphere and $p\geq q$.
Then
\begin{enumerate}[start = 2]
  \item\label{item:Hpq-sphere-2} the unique maximal $\Gamma$-invariant convex open subset $\Omega_{\max}$ of $\PP(\RR^{p,q})$ containing~$\C$ (see Proposition~\ref{prop:max-inv-conv}) is contained in $\HH^{p,q-1}$;
  \item\label{item:Hpq-sphere-3} any supporting hyperplane of~$\C$ at a point of $\partialn \C$ is the projectivization of a linear hyperplane of $\RR^{p,q}$ of signature $(p,q-1)$.
\end{enumerate}
\end{lemma}

In particular, if $q = 2$ and $\partial_{\infty} \Gamma$ is homeomorphic to a $(p-1)$-dimensional sphere with $p\geq 2$, then any hyperplane tangent to $\partialn \C$ is spacelike.

Lemma~\ref{lem:Hpq-cc-sphere}.\eqref{item:Hpq-sphere-1} has the following consequences.

\begin{corollary}\label{cor:easy}
Let $p,q\geq 1$.
\begin{enumerate}
  \item\label{item:vcd-leq-p} Let $\Gamma$ be an infinite discrete subgroup of $\PO(p,q)$ which is $\HH^{p,q-1}$-convex cocompact.
  Then $\Gamma$ is Gromov hyperbolic and has virtual cohomological dimension $\leq p$.
  Moreover, the virtual cohomological dimension of $\Gamma$ is~$p$ if and only if the Gromov boundary $\partial_{\infty}\Gamma$ is homeomorphic to a $(p-1)$-dimensional sphere.
  \item\label{item:vcd-under-Anosov} Let $\Gamma$ be a word hyperbolic group with connected Gromov boundary $\partial_{\infty}\Gamma$ and virtual cohomological dimension $>q$.
  If there is a $P_1^{p,q}$-Anosov representation $\rho : \Gamma\to\PO(p,q)$, then $p>q$ and the group $\rho(\Gamma)$ is $\HH^{p,q-1}$-convex cocompact.
\end{enumerate}
\end{corollary}

\begin{proof}[Proof of Corollary~\ref{cor:easy} assuming Lemma~\ref{lem:Hpq-cc-sphere}]
\eqref{item:vcd-leq-p} By Theorem~\ref{thm:main-POpq-reducible}, the group $\Gamma$ is Gromov hyperbolic.
By Lemma~\ref{lem:Hpq-cc-sphere}.\eqref{item:Hpq-sphere-1}, the Gromov boundary $\partial_{\infty}\Gamma$ embeds into $\mathbb{S}^{p-1}$; in particular, it has Lebesgue covering dimension $\leq p-1$.
This Lebesgue covering dimension is equal to the virtual cohomological dimension of~$\Gamma$ minus one by \cite{bm91}.
If the Lebesgue covering dimension of $\partial_{\infty}\Gamma$ is $p-1$, then the small inductive dimension (or Menger--Urysohn dimension) of $\partial_{\infty}\Gamma$ is also $p-1$ by Urysohn's theorem (see \eg \cite[Th.\,1.7.7]{eng78}), and so an embedding of $\partial_{\infty}\Gamma$ in $\mathbb S^{p-1}$ must have nonempty interior (see \eg \cite[Th.\,1.8.10]{eng78}).
Therefore $\partial_{\infty}\Gamma$ is homeomorphic to $\mathbb{S}^{p-1}$ by \cite[Th.\,4.4]{bk02}.

\eqref{item:vcd-under-Anosov} By Corollary~\ref{cor:pqqp-intro}, since $\partial_{\infty} \Gamma$ is connected, the group $\rho(\Gamma)$ is $\HH^{p,q-1}$-convex cocompact or $\HH^{q,p-1}$-convex cocompact.
We conclude using~\eqref{item:vcd-leq-p}.
\end{proof}

\begin{proof}[Proof of Lemma~\ref{lem:Hpq-cc-sphere}]
We work in an affine chart $\mathbb A$ containing $\overline{\C}$, and choose coordinates $(v_1, \ldots, v_{p+q})$ on $\RR^{p,q}$ so that the quadratic form $\langle\cdot,\cdot\rangle_{p,q}$ takes the usual form $v_1^2 + \cdots + v_p^2 - v_{p+1}^2 - \cdots - v_{p+q}^2$, the affine chart $\mathbb A$ is defined by $v_{p+q} \neq 0$, and our chosen standard $(p-1)$-sphere in $\partial \HH^{p,q-1}$ is
$$\mathcal{S} = \big\{ [(v_1, \ldots,  v_p, 0, \ldots, 0, v_{p+q})] \in \PP(\RR^{p,q}) ~|~ v_1^2 + \cdots + v_p^2 = v_{p+q}^2\big\}.$$

For any $0\leq t\leq 1$, consider the map $f_t : \mathbb A \to \mathbb A$ sending $[(v_1, \ldots, v_{p+q})]$~to
$$\left[\left(v_1, \ldots, v_p, \sqrt{1-t}\,v_{p+1}, \ldots, \sqrt{1-t}\,v_{p+q-1}, \sqrt{1 + t\,\alpha_{(v_{p+1},\dots,v_{p+q})}} \, v_{p+q}\right)\right],$$
where $\alpha_{(v_{p+1},\dots,v_{p+q})} = (v_{p+1}^2 + \cdots + v_{p+q-1}^2)/v_{p+q}^2$.
Then $(f_t)_{0\leq t\leq 1}$ restricts to a homotopy of maps $\mathbb{A} \cap \partial \HH^{p,q-1} \to \mathbb{A} \cap \partial \HH^{p,q-1}$ from the identity map to the map $f_1|_{\mathbb{A} \cap \partial \HH^{p,q-1}}$, which has image $\mathcal S \subset \PP(\RR^{p,1})$.
Thus $(f_t\circ\xi)_{0\leq t\leq 1}$ defines a homotopy between $\xi : \partial_{\infty}\Gamma \to \partial \HH^{p,q-1}$ and the continuous map $f_1\circ\xi : \partial_{\infty}\Gamma \to \partial \HH^{p,q-1}$ whose image lies in $\mathcal{S}$.
By the Cauchy--Schwarz inequality for Euclidean inner products, two points of $\mathbb{A} \cap \partial \HH^{p,q-1}$ which have the same image under~$f_1$ are connected by a line segment in~$\mathbb{A}$ whose interior lies outside of $\overline{\HH^{p,q-1}}$.
Since $\C$ is contained in~$\HH^{p,q-1}$, we deduce that the restriction of $f_1$ to $\partiali\C$ is injective.
Since $\partial_{\infty}\Gamma$ is compact, we obtain that $f_1\circ\xi : \partial_{\infty}\Gamma \to \mathcal{S}$ is an embedding.
This proves~\eqref{item:Hpq-sphere-1}.

Henceforth we assume that $\partial_{\infty} \Gamma$ is homeomorphic to a $(p-1)$-dimensional sphere and that $p\geq q$.
Let $\Lambda_{\Gamma} = \xi(\partial_{\infty}\Gamma) \subset \partial \HH^{p,q-1}$ be the proximal limit set of $\Gamma$ in $\PP(\RR^{p+q})$.
The map $f_1 : \Lambda_\Gamma\to\mathcal{S}$ is then an embedding of a compact $(p-1)$-manifold into a connected $(p-1)$-manifold, and hence by the Invariance of Domain Theorem of Brouwer, it is a homeomorphism.

Let us prove \eqref{item:Hpq-sphere-2}.
By definition of $\Omega_{\max}$ (see Proposition~\ref{prop:max-inv-conv}), it is sufficient to prove that any point of $\partial \HH^{p,q-1}$ is contained in $z^{\perp}$ for some $z\in\Lambda_{\Gamma}$.
Therefore it is sufficient to prove that $\Lambda_{\Gamma}$ intersects $\PP(W)$ for every maximal totally isotropic subspace $W$ of $\RR^{p,q}$.
Let $W$ be such a subspace.
We work in the coordinates from above.
Up to changing the first $p$ coordinates by applying an element of $\OO(p)$, we may assume that in the splitting
$$\RR^{p,q} = \RR^{p-q,0} \oplus \RR^{q, 0} \oplus \RR^{0,q}$$
defined by these coordinates, the space $W$ is $\{(0,v',-v') \,|\, v' \in \RR^q\}$.
Consider the map $\varphi : \mathbb{S}^{p-1} \to \mathbb{S}^{q-1}$ sending a unit vector $(v,v') \in \RR^{p,0} = \RR^{p-q,0} \oplus \RR^{q,0}$ to
$$\varphi(v,v') = \pi_q \circ (f_1|_{\Lambda_\Gamma})^{-1}(v,v',a),$$
where $a = (0,\ldots,0,1) \in \RR^{0,q}$ and $\pi_q: \RR^{p,q} \to \RR^{0,q}$ is the projection onto the last $q$ coordinates.
Then $\varphi$ is homotopic to a constant map (namely the map sending all points to $a$) via the homotopy $\varphi_t = \pi_q \circ f_t \circ (f_1|_{\Lambda_\Gamma})^{-1}$.
Now consider the restriction $\psi: \mathbb S^{q-1} \to \mathbb S^{q-1}$ of $\varphi$ to the unit sphere in $\RR^{q,0}$, which is also homotopic to a constant map via the restriction of the homotopy $\varphi_t$.
If we had $\Lambda_\Gamma \cap \PP(W) = \emptyset$, then we would have $\psi(v') = \varphi(0,v') \neq -v'$ for all $v' \in \mathbb{S}^{q-1} \subset \RR^q$, and
$$\psi_t(v') = \frac{ (1-t)\,\psi(v') + tv'}{\|(1-t)\,\psi(v') + tv'\|}$$
would define a homotopy from $\psi$ to the identity map on $\mathbb{S}^{q-1}$, showing that a constant map is homotopic to the identity map: contradiction since $\mathbb{S}^{q-1}$ is not contractible. 
This completes the proof of \eqref{item:Hpq-sphere-2}: namely, $\Omega_{\max} \subset \HH^{p,q-1}$.
We note that $\partial \Omega_{\max} \cap \partial \HH^{p,q-1} = \Lambda_{\Gamma}$, since we saw that $f_1$ maps $\partial \Omega_{\max} \cap \partial \HH^{p,q-1}$ injectively to~$\mathcal{S}$.

Finally, we prove~\eqref{item:Hpq-sphere-3}.
The dual convex $\Int(\C)^*$ to $\Int(\C)$ naturally identifies, via $\langle\cdot,\cdot\rangle_{p,q}$, with a properly convex subset of $\PP(\RR^{p+q})$ which must be contained in $\Omega_{\max} \subset \HH^{p,q-1}$.
By Lemma~\ref{lem:no-degenerate-faces}, we have $\C \cap z^\perp = \emptyset$ for any $z \in \partiali \C = \Lambda_\Gamma$, hence $\C \subset \Omega_{\max}$. 
Hence, a projective hyperplane $y^\perp$ supporting $\C$ at some point of $\partialn \C$ is dual to a point $y \in \overline{\Omega_{\max}} \smallsetminus \Lambda_\Gamma$, which is contained in $\HH^{p,q-1}$ by the previous paragraph.
Hence $\langle y, y \rangle_{p,q} < 0$ and so $y^\perp$ is the projectivization of a linear hyperplane of signature $(p,q-1)$.
\end{proof}

\subsection{$\HH^{p,1}$-convex cocompactness and global hyperbolicity}

Recall that a Lorentzian manifold $M$ is said to be \emph{globally hyperbolic} if it is causal (\ie contains no timelike loop) and for any two points $x,x'\in M$, the intersection $J^+(x)\cap J^-(x')$ is compact (possibly empty).
Here we denote by $J^+(x)$ (\resp $J^-(x)$) the set of points of~$M$ which are seen from $x$ by a future-pointing (\resp past-pointing) timelike or lightlike geodesic.
Equivalently \cite{cg69}, the Lorentzian manifold~$M$ admits a \emph{Cauchy hypersurface}, \ie an achronal subset which intersects every inextendible timelike curve in exactly one point.

We make the following observation; it extends \cite[\S\,4.2]{bm12}, which focused on the case that $\Gamma \subset \PO(p,2)$ is isomorphic to a uniform lattice of $\PO(p,1)$.
The case considered here is a bit more general, see \cite{lm19}.
We argue in an elementary way, without using the notion of CT-regularity.

\begin{proposition}\label{prop:AdS-GH}
For $p\geq 2$, let $\Gamma$ be a torsion-free infinite discrete subgroup of $\PO(p,2)$ which is $\HH^{p,1}$-convex cocompact and whose Gromov boundary $\partial_{\infty}\Gamma$ is homeomorphic to a $(p-1)$-dimensional sphere.
For any nonempty properly convex open subset $\Omega$ of $\HH^{p,1}$ on which $\Gamma$ acts convex cocompactly, the quotient $M = \Gamma\backslash\Omega$ is a globally hyperbolic Lorentzian manifold.
\end{proposition}

\begin{proof}
Consider two points $x,x'\in M = \Gamma\backslash\Omega$.
There exists $R>0$ such that any lifts $y,y'\in\Omega$ of~$x, x'$ belong to the uniform $R$-neighborhood $\C$ of $\Ccore_{\Omega}(\Gamma)$ in~$\Omega$ for the Hilbert metric $d_{\Omega}$.
Let $J^+(y)$ (\resp $J^-(y)$) be the set of points of~$\Omega$ which are seen from $y$ by a future-pointing (\resp past-pointing) timelike or lightlike geodesic.
By Lemma~\ref{lem:Hpq-cc-sphere}.\eqref{item:Hpq-sphere-3}, all supporting hyperplanes of $\C$ at points of $\partialn \C$ are spacelike.
We may decompose $\partialn \C$ into two disjoint open subsets, namely the subset $\partialn^+ \C$ of points for which the outward pointing normal to a supporting plane is future pointing, and the subset $\partialn^-\C$ of points for which it is past pointing.
Indeed, $\partialn^+ \C$ and $\partialn^- \C$ are the two path connected components of the complement $\partialn \C$ of the embedded $(p-1)$-sphere $\Lambda_\Gamma$ in the $p$-sphere $\Fr(\C)$. The set $\Omega \smallsetminus \C$ similarly has two components, a component to the future of $\partialn^+ \C$ and a component to the past of $\partialn^- \C$. 
Any point of $J^+(y) \cap (\Omega \smallsetminus \C)$ lies in the future component of $\Omega \smallsetminus \C$ and similarly, any point of $J^-(y') \cap (\Omega \smallsetminus \C)$ lies in the past component of $\Omega \smallsetminus \C$. 
By proper discontinuity of the $\Gamma$-action on $\C$, in order to check that $J^+(x)\cap J^-(x')$ is compact in~$M$, it is therefore enough to check that $J^+(y)\cap\C$ and $J^-(y')\cap\C$ are compact in~$\Omega$.
This follows from the fact that the ideal boundary of~$\C$ is $\Lambdao_{\Omega}(\Gamma)$ (Corollary~\ref{cor:ideal-bound-naive-cc}.\eqref{item:ideal-bound-cc}) and that any point of~$\Omega$ sees any point of $\Lambdao_{\Omega}(\Gamma)$ in a spacelike direction (Lemma~\ref{lem:no-degenerate-faces}).
\end{proof}

\subsection{Examples of $\HH^{p,q-1}$-convex cocompact groups} \label{subsec:ex-Hpq-cc}

\subsubsection{Quasi-Fuchsian $\HH^{p,q-1}$-convex cocompact groups} \label{subsubsec:ex-Hpq-QF}

Let $H$ be a real semisimple Lie group of real rank~$1$ and $\tau : H\to\PO(p,q)$ a representation whose image contains an element which is proximal in $\partial\HH^{p,q-1}$ (see Remark~\ref{rem:prox-lim-set-POpq}).
By \cite[Prop.\,7.1]{dgk-ccHpq}, for any word hyperbolic group $\Gamma$ and any (classical) convex cocompact representation $\sigma_0 : \Gamma\to H$,
\begin{enumerate}
  \item\label{item:qF1} the composition $\rho_0 := \tau \circ \sigma_0 : \Gamma \rightarrow \PO(p,q)$ is $P_1^{p,q}$-Anosov and the proximal limit set $\Lambda_{\rho_0(\Gamma)}\subset\partial\HH^{p,q-1}$ is negative or positive;
  \item\label{item:qF2} the connected component $\mathcal{T}_{\rho_0}$ of~$\rho_0$ in the space of $P_1^{p,q}$-Anosov representations from $\Gamma$ to $\PO(p,q)$ is a neighborhood of $\rho_0$ in $\Hom(\Gamma,\PO(p,q))$ consisting entirely of representations $\rho$ with negative proximal limit set $\Lambda_{\rho(\Gamma)}\subset\partial\HH^{p,q-1}$ or entirely of representations $\rho$ with positive proximal limit set $\Lambda_{\rho(\Gamma)}\subset\partial\HH^{p,q-1}$.
\end{enumerate}
Here is an immediate consequence of this result and of Theorem~\ref{thm:main-POpq-reducible}.

\begin{corollary}\label{cor:pqqp}
In this setting, either $\rho(\Gamma)$ is $\HH^{p,q-1}$-convex cocompact for all $\rho\in\mathcal{T}_{\rho_0}$, or $\rho(\Gamma)$ is $\HH^{q,p-1}$-convex cocompact (after identifying $\PO(p,q)$ with $\PO(q,p)$) for all $\rho\in\mathcal{T}_{\rho_0}$.
\end{corollary}

This improves \cite[Cor.\,7.2]{dgk-ccHpq}, which assumed $\rho$ to be irreducible.

Here are two examples.
We refer to \cite[\S\,7]{dgk-ccHpq} for more details.

\begin{example} \label{ex:Hpq-quasi-Fuchsian}
Let $\Gamma$ be the fundamental group of a convex cocompact (\eg closed) hyperbolic manifold $M$ of dimension $m\geq 2$, with holonomy $\sigma_0 : \Gamma\to\PO(m,1)=\mathrm{Isom}(\HH^m)$.
The representation $\sigma_0$ lifts to a representation $\widetilde{\sigma}_0 : \Gamma\to H:=\OO(m,1)$.
For $p,q\in\NN^*$ with $p\geq m, q \geq 1$, let $\tau : \OO(m,1)\to\PO(p,q)$ be induced by the natural embedding $\RR^{m,1}\hookrightarrow\RR^{p,q}$.
Then $\rho_0 := \tau\circ\widetilde{\sigma}_0 : \Gamma \to \PO(p,q)$ is $P_1^{p,q}$-Anosov, and one checks that the set $\Lambda_{\rho_0(\Gamma)}\subset\partial\HH^{p,q-1}$ is negative.
Let $\mathcal{T}_{\rho_0}$ be the connected component of~$\rho_0$ in the space of $P_1^{p,q}$-Anosov representations from $\Gamma$ to $\PO(p,q)$.
By Corollary~\ref{cor:pqqp}, for any $\rho\in\mathcal{T}_{\rho_0}$, the group $\rho(\Gamma)$ is $\HH^{p,q-1}$-convex cocompact.
\end{example}

By \cite{mes90,bar15}, when $p=m$ and $q=2$ and when the hyperbolic $m$-manifold $M$ is closed, the space $\mathcal{T}_{\rho_0}$ of Example~\ref{ex:Hpq-quasi-Fuchsian} is a full connected component of $\Hom(\Gamma,\PO(p,2))$, consisting of so-called \emph{AdS quasi-Fuchsian} representations; that $\rho(\Gamma)$ is $\HH^{p,1}$-convex cocompact in that case follows from \cite{mes90,bm12}.

\begin{example} \label{ex:gen-Hitchin}
For $n\geq 2$, let
\begin{equation} \label{eqn:tau-n}
\tau_n : \SL(\RR^2) \longrightarrow \SL(\RR^n)
\end{equation}
be the irreducible $n$-dimensional linear representation of $\SL(\RR^2)$ obtained from the action of $\SL(\RR^2)$ on the $(n-1)^{\mathrm{st}}$ symmetric power $\mathrm{Sym}^{n-1}(\RR^2) \simeq \RR^n$.
The image of~$\tau_n$ preserves the nondegenerate bilinear form $B_n := -\omega^{\otimes (n-1)}$ induced from the area form $\omega$ of~$\RR^2$.
This form is symmetric if $n$ is odd, and antisymmetric (\ie symplectic) if $n$ is even.
Suppose $n = 2m+1$ is odd.
The symmetric bilinear form $B_n$ has signature
\begin{equation} \label{eqn:kn-ln}
(k_n, \ell_n) := \left\{ \begin{array}{ll} (m+1,m) & \text{ if $m$ is odd,}\\
(m,m+1) & \text{ if $m$ is even.} \end{array} \right.
\end{equation}
If we identify the orthogonal group $\OO(B_n)$ (containing the image of~$\tau_n$) with $\OO(k_n,\ell_n)$, then there is a unique $\tau_n$-equivariant embedding $\iota_n : \partial_{\infty}\HH^2\hookrightarrow\partial_{\scriptscriptstyle\PP}\HH^{k_n,\ell_n-1}$, and an easy computation shows that its image $\Lambda_n := \iota(\partial_{\infty} \HH^2)$ is negative.
For $p \geq k_n$ and $q \geq \ell_n$, the representation $\tau_n : \SL_2(\RR)\to\OO(B_n)\simeq\OO(k_n,\ell_n)$ and the natural embedding $\RR^{k_n,\ell_n}\hookrightarrow\RR^{p,q}$ induce a representation $\tau : H = \SL_2(\RR) \to \PO(p,q)$ whose image contains an element which is proximal in $\partial\HH^{p,q-1}$, and a $\tau$-equivariant embedding $\iota : \partial_{\infty}\HH^2\hookrightarrow\partial_{\scriptscriptstyle\PP}\HH^{k_n,\ell_n-1}\hookrightarrow\partial_{\scriptscriptstyle\PP}\HH^{p,q-1}$.
The set $\Lambda:=\iota(\partial_{\infty}\HH^2)\subset\partial_{\scriptscriptstyle\PP}\HH^{p,q-1}$ is negative by construction.
Let $\Gamma$ be the fundamental group of a convex cocompact orientable hyperbolic surface, with holonomy $\sigma_0 : \Gamma\to\PSL_2(\RR)$.
The representation $\sigma_0$ lifts to a representation $\widetilde{\sigma}_0 : \Gamma\to H:=\SL_2(\RR)$.
Let $\rho_0 := \tau\circ\widetilde{\sigma}_0 : \Gamma \to G:=\PO(p,q)$.
The proximal limit set $\Lambda_{\rho_0(\Gamma)} = \iota(\Lambda_{\sigma_0(\Gamma)}) \subset \Lambda$ is negative.
Thus Corollary~\ref{cor:pqqp} implies that $\rho_0(\Gamma)$ is $\HH^{p,q-1}$-convex cocompact and so is $\rho(\Gamma)$ where $\rho$ is any representation in the connected component $\mathcal{T}_{\rho_0}$ of $\rho$ in the space of $P_1^{p,q}$-Anosov representations from $\Gamma$ to $\PO(p,q)$. 
\end{example}

It follows from \cite{lab06,fg06} (see \eg \cite[\S\,6.1]{biw14}) that when $(p,q) = (k_n, \ell_n)$ or\linebreak $(m+1,m+1)$ and when $\Gamma$ is a closed surface group, the space $\mathcal{T}_{\rho_0}$ of Example~\ref{ex:gen-Hitchin} is a full connected component of $\Hom(\Gamma,\PO(p,q))$, consisting of so-called \emph{Hitchin representations}.
Example~\ref{ex:gen-Hitchin} thus states the following.

\begin{corollary}
Let $\Gamma$ be the fundamental group of a closed orientable hyperbolic surface and let $m\geq 1$.

For any Hitchin representation $\rho : \Gamma\to\PO(m+1,m)$, the group $\rho(\Gamma)$ is $\HH^{m+1,m-1}$-convex cocompact if $m$ is odd, and $\HH^{m,m}$-convex cocompact if $m$ is even.

For any Hitchin representation $\rho : \Gamma\to\PO(m+1,m+1)$, the group $\rho(\Gamma)$ is $\HH^{m+1,m}$-convex cocompact.
\end{corollary}

By \cite{biw10,biw}, when $p = m+1 = 2$ and $\Gamma$ is a closed surface group, the space $\mathcal{T}_{\rho_0}$ of Example~\ref{ex:gen-Hitchin} is a full connected component of $\Hom(\Gamma,\PO(2,q))$, consisting of so-called \emph{maximal representations}.
Example~\ref{ex:gen-Hitchin} thus states the following.

\begin{corollary} \label{cor:max-rep-PO-p-2}
Let $\Gamma$ be the fundamental group of a closed orientable hyperbolic surface and let $q\geq 1$.
Any connected component of $\Hom(\Gamma,\PO(2,q))$ consisting of maximal representations and containing a Fuchsian representation $\rho_0 : \Gamma \to \PO(2,1)_0 \hookrightarrow \PO(2,q)$ consists entirely of $\HH^{2,q-1}$-convex cocompact representations.
\end{corollary}

\subsubsection{Groups with connected boundary} \label{subsubsec:connected-boundary}

We now briefly discuss a class of examples that does not necessarily come from the deformation of ``Fuchsian'' representations as above.

Suppose the word hyperbolic group $\Gamma$ has connected boundary $\partial_{\infty}\Gamma$ (for instance $\Gamma$ is the fundamental group of a closed negatively-curved Riemannian manifold).
By \cite[Prop.\,1.10 \& Prop.\,3.5]{dgk-ccHpq}, any connected component in the space of $P_1^{p,q}$-Anosov representations from $\Gamma$ to $\PO(p,q)$ consists entirely of representations $\rho$ with negative proximal limit set $\Lambda_{\rho(\Gamma)}\subset\partial\HH^{p,q-1}$ or entirely of representations with $\rho$ with positive proximal limit set $\Lambda_{\rho(\Gamma)}\subset\partial\HH^{p,q-1}$.
Theorem~\ref{thm:main-POpq-reducible} then implies that for any connected component $\mathcal{T}$ in the space of $P_1$-Anosov representations of $\Gamma$ with values in $\PO(p,q)\subset\PGL(\RR^{p+q})$, either $\rho(\Gamma)$ is $\HH^{p,q-1}$-convex cocompact for all $\rho\in\mathcal{T}$, or $\rho(\Gamma)$ is $\HH^{q,p-1}$-convex cocompact for all $\rho\in\mathcal{T}$, as in Corollary~\ref{cor:pqqp-intro}.

This applies for instance to the case that $\Gamma$ is the fundamental group of a closed hyperbolic surface and $\mathcal{T}$ is a connected component of $\Hom(\Gamma,\PO(2,q))$ consisting of maximal representations \cite{biw10,biw}. By~\cite{ctt19}, the proximal limit set is negative for all representations in any such~$\mathcal{T}$, hence these are $\HH^{2,q-1}$ convex cocompact. 
By \cite{gw10}, for $q=3$ there exist such connected components $\mathcal{T}$ that consist entirely of Zariski-dense representations, hence that do not come from the deformation of ``Fuchsian'' representations as in Corollary~\ref{cor:max-rep-PO-p-2}.

\section{Examples of groups which are convex cocompact in $\PP(V)$} \label{sec:examples}

In Section~\ref{subsec:ex-Hpq-cc} we constructed examples of discrete subgroups of $\PO(p,q) \subset \PGL(\RR^{p+q})$ which are $\HH^{p,q-1}$-convex cocompact; these groups are strongly convex cocompact in $\PP(\RR^{p+q})$ by Theorem~\ref{thm:main-POpq-reducible}.
We now discuss several constructions of discrete subgroups of $\PGL(V)$ which are convex cocompact in $\PP(V)$ but which do not necessarily preserve a nonzero quadratic form on~$V$.

Based on Theorem~\ref{thm:main-noPETs}.\ref{item:stable}--\ref{item:include}, another fruitful source of groups that are convex cocompact in $\PP(V)$ is the continuous deformation of ``Fuchsian'' groups, coming from an algebraic embedding.
These ``Fuchsian'' groups can be for instance:
\begin{enumerate}
  \item\label{item:QF-rk1} convex cocompact subgroups (in the classical sense) of appropriate rank-one Lie subgroups $H$ of $\PGL(V)$, as in Section~\ref{subsubsec:ex-Hpq-QF};
  \item\label{item:QF-div} discrete subgroups of $\PGL(V)$ dividing a properly convex open subset of some projective subspace $\PP(V_0)$ of $\PP(V)$ and acting trivially on a complementary subspace.
\end{enumerate}
The groups of \eqref{item:QF-rk1} are always word hyperbolic; in Section~\ref{subsec:QF-rk1}, we explain how to choose $H$ so that they are strongly convex cocompact in $\PP(V)$, and we prove Proposition~\ref{prop:Hitchin}.
The groups of \eqref{item:QF-div} are not necessarily word hyperbolic; they are convex cocompact in $\PP(V)$ by Theorem~\ref{thm:main-noPETs}.\ref{item:include}, but not necessarily strongly convex cocompact; we discuss them in Section~\ref{subsubsec:divisible-deform}.

We also mention other constructions of convex cocompact groups in Sections \ref{subsubsec:benoistPET} and~\ref{subsec:other-ex-nonhyp-cc}, which do not involve any deformation.

\subsection{``Quasi-Fuchsian'' strongly convex cocompact groups} \label{subsec:QF-rk1}

We start by discussing a similar construction to Section~\ref{subsubsec:ex-Hpq-QF}, but for discrete subgroups of $\PGL(V)$ which do not necessarily preserve a nonzero quadratic form on~$V$.

Let $\Gamma$ be a word hyperbolic group, $H$ a real semisimple Lie group of real rank one, and $\tau : H\to\PGL(V)$ a representation whose image contains a proximal element.
Any (classical) convex cocompact representation $\sigma_0 :\nolinebreak \Gamma\to\nolinebreak H$ is Anosov, hence the composition $\tau\circ\sigma_0 : \Gamma\to\PGL(V)$ is $P_1$-Anosov (see \cite[Prop.\,3.1]{lab06} and \cite[Prop.\,4.7]{gw12}).
In particular, by Theorem~\ref{thm:Ano-PGL}, the group $\tau\circ\sigma_0(\Gamma)$ is strongly convex cocompact in $\PP(V)$ as soon as $\PP(V) \smallsetminus \bigcup_{\eta \in \partial_{\infty} \Gamma} \xi^*(\eta)$ admits a $\tau\circ\sigma_0(\Gamma)$-invariant connected component, where $\xi^*: \partial_{\infty} \Gamma \to \PP(V^*)$ denotes the Anosov boundary map in dual projective space of $\tau \circ \sigma_0$.

Suppose this is the case.
By Theorem~\ref{thm:properties}.\ref{item:stable} (see Remark~\ref{rem:properties-strong-cc}), the group $\rho(\Gamma)$ remains strongly convex cocompact in $\PP(V)$ for any $\rho\in\Hom(\Gamma,\PGL(V))$ close enough to $\tau\circ\sigma_0$.
Sometimes $\rho(\Gamma)$ also remains strongly convex cocompact for some $\rho$ which are continuous deformations of $\tau\circ\sigma_0$ quite far away from $\tau\circ\sigma_0$; we now discuss this in view of proving Proposition~\ref{prop:Hitchin}.

\subsubsection{Connected open sets of strongly convex cocompact representations}

We prove the following.

\begin{proposition} \label{prop:conn-comp-Ano}
Let $\Gamma$ be a word hyperbolic group and $\mathcal{A}$ a connected open subset of $\Hom(\Gamma,\PGL(V))$ consisting entirely of $P_1$-Anosov representations. If $\rho(\Gamma)$ is convex cocompact in $\PP(V)$ for some $\rho\in\mathcal{A}$, then $\rho(\Gamma)$ is convex cocompact in $\PP(V)$ for all $\rho\in\mathcal{A}$.
\end{proposition}

\begin{proof}
First, observe that the finite normal subgroups $\mathrm{Ker}(\rho) \subset \Gamma$ are constant over $\rho \in \mathcal A$, since representations of a finite group are rigid up to conjugation.
Hence by passing to the quotient group $\Gamma/\mathrm{Ker}(\rho)$, we may assume all representations $\rho \in \mathcal A$ are faithful.

By Theorem~\ref{thm:properties}.\ref{item:stable}, the property of being convex cocompact in $\PP(V)$ is open in $\Hom(\Gamma,\PGL(V))$.
Thus the subset of $\rho \in \mathcal A$ for which $\rho(\Gamma)$ is convex cocompact is open. 
Let us show that it is also closed.
Consider a sequence of representations $\rho_m \in \mathcal A$ converging to $\rho \in \mathcal A$.
Assume that $\rho_m(\Gamma)$ is convex cocompact for all~$m$, and let us show that $\rho$ is also convex cocompact.

Let $\xi_m : \partial_{\infty} \Gamma \to \PP(V)$ and $\xi_m^*: \partial_{\infty} \Gamma \to \PP(V^*)$ be the boundary maps for the Anosov representation~$\rho_m$, and $\xi : \partial_{\infty} \Gamma \to \PP(V)$ and $\xi^*: \partial_{\infty} \Gamma \to \PP(V^*)$ those for~$\rho$.
By~\cite[Th.\,5.13]{gw12}, the maps $\xi_m$ (\resp $\xi_m^*$) converge uniformly to $\xi$ (\resp $\xi^*$). 

By Theorem~\ref{thm:Ano-PGL}, for any~$m$, the set $\PP(V) \smallsetminus \bigcup_{\eta \in \partial_{\infty} \Gamma} \xi_m^*(\eta)$ admits a $\rho_m(\Gamma)$-invariant connected component $\Omega_m$.
After passing to a subsequence, the compact subsets $\overline{\Omega_m}$ converge to a nonempty compact subset $\mathcal K$ of $\PP(V)$ which is invariant under $\rho(\Gamma)$.
Note that for each~$m$, any open segment $(a,b)$ in $\overline{\Omega_m}$ is either contained in or disjoint from each supporting hyperplane $\xi_m^*(\eta)$ for $\eta \in \partial_{\infty} \Gamma$.
This property passes to the limit: 

\begin{enumerate}
\item[($\star$)] \label{item:property-star} For each open segment $(a,b)$ in $\mathcal K$ and each hyperplane $\xi^*(\eta)$ for $\eta \in \partial_{\infty} \Gamma$, $(a,b)$ is either contained in or disjoint from $\xi^*(\eta)$.
\end{enumerate}

Suppose first that there is a point $x \in \mathcal K$ which is not contained in any hyperplane $\xi^*(\eta)$ for $\eta \in \partial_{\infty} \Gamma$.
By compactness of $\partial_{\infty} \Gamma$, there is an open subset $U \ni x$ which does not intersect $\xi^*(\eta)$ for all $\eta \in \partial_{\infty} \Gamma$.
It follows that a slightly smaller open set $U' \ni x$ does not intersect any $\xi_m^*(\eta)$ for $\eta \in \partial_{\infty} \Gamma$ and $m$ sufficiently large, and hence that $U'$ is contained in $\Omega_m$  for all $m$ sufficiently large, and hence that $U'$ is contained in $\mathcal K$.
Thus the interior of~$\mathcal K$ is nonempty.
This interior is $\rho(\Gamma)$-invariant and, by property~($\star$), it is contained in $\PP(V) \smallsetminus \bigcup_{\eta \in \partial_{\infty} \Gamma} \xi^*(\eta)$.
This shows that $\PP(V) \smallsetminus \bigcup_{\eta \in \partial_{\infty} \Gamma} \xi^*(\eta)$ admits a $\rho(\Gamma)$-invariant connected component.
It follows from the implication \ref{item:P1Anosov-bis}~$\Rightarrow$\ref{item:ccc-hyp} of Theorem~\ref{thm:main-noPETs} that $\rho(\Gamma)$ is convex cocompact, as desired.
 
Suppose then that any point in $\mathcal K$ is contained in $\xi^*(\eta)$ for some $\eta \in \partial_{\infty} \Gamma$.
By considering such a point in the relative interior of $\mathcal K$ and using property~($\star$) above, we see that all of $\mathcal K$ is contained in $\xi^*(\eta)$.
Since $\xi_m(\partial_{\infty} \Gamma) \subset \overline{\Omega_m}$ for all~$m$ and the $\xi_m$ converge uniformly to~$\xi$, it follows that $\xi(\partial_{\infty} \Gamma) \subset \mathcal K$, hence $\xi(\partial_{\infty} \Gamma) \subset \xi^*(\eta)$.
This contradicts the transversality of the boundary maps $\xi$ and~$\xi^*$ (property \ref{item:ano-trans} in Definition~\ref{def:P1-Ano}).
\end{proof}

\subsubsection{Hitchin representations}\label{sec:Hitchin}

We now prove Proposition~\ref{prop:Hitchin} by applying Proposition~\ref{prop:conn-comp-Ano} in the following specific context.

Let $\Gamma$ be the fundamental group of a closed orientable hyperbolic surface~$S$.
For $n\geq 2$, let $\tau_n : \SL(\RR^2) \to \SL(\RR^n)$ be the irreducible $n$-dimensional linear representation from \eqref{eqn:tau-n}.
We still denote by $\tau_n$ the representation $\PSL(\RR^2) \to \PSL(\RR^n)$ obtained by modding out by $\{\pm I\}$.
A representation $\rho\in\Hom(\Gamma,\PSL(\RR^n))$ is said to be \emph{Fuchsian} if it is of the form $\rho = \tau_n\circ \rho_0$ where $\rho_0: \Gamma \to \PSL(2,\RR)$ is discrete and faithful.
By definition, a \emph{Hitchin representation} is a continuous deformation of a Fuchsian representation; the \emph{Hitchin component} $\mathrm{Hit}_n(S)$ is the space of Hitchin representations $\rho\in\Hom(\Gamma, \PSL(\RR^n))$ modulo conjugation by $\PGL(\RR^n)$.
Hitchin~\cite{hit92} used Higgs bundles techniques to parametrize $\mathrm{Hit}_n(S)$, showing in particular that it is homeomorphic to a ball of dimension $(n^2 -1)(2g-2)$.
Labourie~\cite{lab06} proved that any Hitchin representation is $P_1$-Anosov (in fact it has the stronger property of being Anosov \emph{with respect to a minimal parabolic subgroup of $\PSL(\RR^n)$}).

In order to prove Proposition~\ref{prop:Hitchin}, we first consider the case of Fuchsian representations.

\begin{lemma}\label{lem:odd-even}
Let $\rho: \Gamma \to \PSL(\RR^2) \rightarrow \PSL(\RR^n)$ be Fuchsian.
\begin{enumerate}
  \item\label{item:odd-Fuchsian} If $n$ is odd, then $\rho(\Gamma)$ is strongly convex cocompact in $\PP(\RR^n)$.
  \item\label{item:even-Fuchsian} If $n$ is even, then the boundary map of the $P_1$-Anosov representation~$\rho$ defines a nontrivial loop in $\PP(\RR^n)$ and $\rho(\Gamma)$ does not preserve any nonempty properly convex open subset of $\PP(\RR^n)$.
\end{enumerate}
\end{lemma}

For even~$n$, the fact that a Fuchsian representation~$\rho$ cannot preserve a properly convex open subset of $\PP(\RR^n)$ also follows from \cite[Th.\,1.5]{ben00}, since in this case $\rho$ takes values in the projective symplectic group $\mathrm{PSp}(n/2,\RR)$.

\begin{proof}
\eqref{item:odd-Fuchsian} Suppose $n = 2k+1$ is odd.
Then $\rho$ takes values in the projective orthogonal group $\PSO(k, k+1)\simeq\SO(k,k+1)$, and so $\rho(\Gamma)$ is strongly convex cocompact in $\PP(\RR^n)$ by \cite[Prop.\,1.17 \& 1.19]{dgk-ccHpq}.

\eqref{item:even-Fuchsian} Suppose $n = 2k$ is even.
Then $\rho$ takes values in the symplectic group $\Sp(n,\RR)$.
It is well known that, in natural coordinates identifying $\partial_{\infty} \Gamma \simeq \SS^1$ with $\PP(\RR^2)$, the boundary map $\xi: \partial_{\infty} \Gamma \to \PP(\RR^n)$ of the $P_1$-Anosov representation~$\rho$ is the \emph{Veronese curve}
$$\PP(\RR^2) \ni [x,y] \longmapsto \left [ \left ( x^{n-1}, x^{n-2}y, \ldots, xy^{n-2}, y^{n-1} \right ) \right ] \in \PP(\RR^n).$$
This map is homotopic to the map $[x,y] \mapsto [(x^{n-1}, 0, \ldots, 0, y^{n-1})]$ which, since $n-1$ is odd, is a homeomorphism from $\PP(\RR^2)$ to the projective plane spanned by the first and last coordinate vectors, hence is nontrivial in $\pi_1(\PP(\RR^n))$.
Hence the image of the boundary map $\xi$ crosses every hyperplane in $\PP(\RR^n)$.
It follows that $\rho(\Gamma)$ does not preserve any properly convex open subset of $\PP(\RR^n)$, because the image of~$\xi$ must lie in the boundary of any $\rho(\Gamma)$-invariant such set.
\end{proof}

\begin{proof}[Proof of Proposition~\ref{prop:Hitchin}]
\eqref{item:Hit-odd} Suppose $n$ is odd.
By \cite{lab06,fg06}, all Hitchin representations $\rho : \Gamma\to\PSL(\RR^n)$ are $P_1$-Anosov.
Moreover, $\rho(\Gamma)$ is convex cocompact in $\PP(V)$ as soon as $\rho$ is Fuchsian, by Lemma~\ref{lem:odd-even}.\eqref{item:odd-Fuchsian}.
Applying Proposition~\ref{prop:conn-comp-Ano}, we obtain that for any Hitchin representation $\rho : \Gamma\to\PSL(\RR^n)$ the group $\rho(\Gamma)$ is convex cocompact in $\PP(V)$, hence strongly convex cocompact since $\Gamma$ is word hyperbolic (Theorem~\ref{thm:main-noPETs}).

\eqref{item:Hit-even} Suppose $n$ is even.
By Lemma~\ref{lem:odd-even}, for any Fuchsian $\rho : \Gamma\to\PSL(\RR^n)$, the image of the boundary map of the $P_1$-Anosov representation~$\rho$ defines a nontrivial loop in $\PP(\RR^n)$.
Since the assignment of a boundary map to an Anosov representation is continuous \cite[Th.\,5.13]{gw12}, we obtain that for any Hitchin $\rho : \Gamma\to\PSL(\RR^n)$, the image of the boundary map of the $P_1$-Anosov representation~$\rho$ defines a nontrivial loop in $\PP(\RR^n)$.
In particular, $\rho(\Gamma)$ does not preserve any nonempty properly convex open subset of $\PP(\RR^n)$, by the same reasoning as in the proof of Lemma~\ref{lem:odd-even}.
More generally, $\rho$ does not preserve any nonempty properly convex subset of $\PP(\RR^n)$, since any such set would have nonempty interior by irreducibility of the Hitchin representation~$\rho$ (see \cite[Lem.\,10.1]{lab06}).
\end{proof}

\subsection{Convex cocompact groups coming from divisible convex sets} \label{subsec:QF-div}

Recall that any discrete subgroup $\Gamma$ of $\PGL(V)$ dividing a properly convex open subset $\Omega \subset \PP(V)$ is convex cocompact in $\PP(V)$ (Example~\ref{ex:div-implies-cc}).
The group $\Gamma$ is strongly convex cocompact if and only if it is word hyperbolic (Theorem~\ref{thm:main-noPETs}), and this is equivalent to the fact that $\Omega$ is strictly convex (by \cite{ben04} or Theorem~\ref{thm:main-noPETs} again).

In this section, we produce convex cocompact groups which are not strongly convex cocompact and do not divide convex domains.
We do this via two constructions: one by deformation (Section~\ref{subsubsec:divisible-deform}) and one by restriction to subgroups (Section~\ref{subsubsec:benoistPET}).
Both constructions start with a group dividing a properly convex but not strictly convex open set $\Omega \subset \PP(V)$; we now recall how such groups can be obtained.

\subsubsection{Groups $\Gamma$ dividing a properly convex but not strictly convex open set} \label{subsubsec:ex-div}

Examples have been obtained in several ways:
\begin{enumerate}[label=(\roman*)] 
  \item \label{item:ex-ccc1} letting a cocompact lattice of $\SL(\RR^m)$ act on the Riemannian symmetric space of $\SL(\RR^m)$, realized as a properly convex open subset $\Omega_m$ of $\PP(\RR^{m(m+1)/2})$ as follows: identify $\RR^{m(m+1)/2}$ with the space of symmetric $(m\times m)$ real matrices, and let $\Omega_m$ be the projectivization of the positive definite symmetric matrices (the corresponding examples are called \emph{symmetric}; there are also complex, quaternionic, and octonionic variants: see \cite[\S\,2.4]{ben08});
  \item \label{item:ex-ccc2} letting certain Coxeter groups act by reflections on $\PP(\RR^n)$, for $4\leq n \leq 7$, with maximal abelian subgroups isomorphic to $\ZZ^{n-2}$: see~\cite[Prop.\,4.2]{ben06};
  \item \label{item:ex-ccc3} deforming the holonomies of certain hyperbolic 3-manifolds in $\PGL(\RR^4)$ so that the cusp groups become diagonalizable (the resulting groups are convex cocompact: see \eg Lemma~\ref{lem:PETfriendly} below), and doubling across peripheral tori \cite{bdl18} (this idea is already implicit in the examples of \cite[\S\,4.3]{ben06});
  \item \label{item:ex-ccc4} letting other Coxeter groups, with maximal abelian subgroups isomorphic to~$\ZZ^{n-3}$, divide a properly convex open subset of $\PP(\RR^n)$, for $5\leq n\leq 7$,  using a Dehn filling construction~\cite{clm20}.
\end{enumerate}

In \ref{item:ex-ccc2} for $n=4$, as well as in \ref{item:ex-ccc3}, there are subgroups (virtually) isomorphic to $\ZZ^2$ that stabilize PETs (all pairwise disjoint) in $\Omega \subset \PP(\RR^4)$: Benoist~\cite{ben06} proved that this is in fact \emph{always} the case for nonhyperbolic divisible convex sets in $\PP(\RR^4)$; the group then splits as a graph of groups, with the PET stabilizers as edge groups, and the PETs project to the JSJ decomposition of $\Gamma \backslash \Omega$ in the sense of Thurston's geometrization (all pieces are hyperbolic).

\subsubsection{Convex cocompact deformations of groups dividing a convex set} \label{subsubsec:divisible-deform}

Let $\Gamma$ be a discrete subgroup of $\PGL(V)$ dividing a nonempty properly convex open subset $\Omega$ of $\PP(V)$.
Let $\hat \Gamma$ be the lift of $\Gamma$ to $\SL^{\pm}(V)$ that preserves a properly convex cone of $V$ lifting~$\Omega$ (see Remark~\ref{rem:lift-Gamma}).
By Theorem~\ref{thm:properties}.\ref{item:stable}--\ref{item:include}, any small deformation of the inclusion of $\hat \Gamma$ into a larger projective linear group $\PGL(V \oplus V')$ is convex cocompact in $\PP(V \oplus V')$.
The issue is to find nontrivial such deformations: for instance, there are none in case~\ref{item:ex-ccc1} above for $m\geq 3$, by Margulis superrigidity.
 
The situation is more favorable in cases \ref{item:ex-ccc2} and~\ref{item:ex-ccc3} above, for instance when $n=\nolinebreak 4$: the group splits along the PET stabilizers, hence the inclusion of $\hat \Gamma$ into $\PGL(\RR^4 \oplus \RR^{n'})$ may be deformed by a Johnson--Millson bending~\cite{jm87}.
(Note that such deformations exist already when $n'=0$, \eg so-called \emph{bulging deformations}.)
Since the PET stabilizers in $\hat \Gamma \subset \SL^\pm(\RR^4)$ have $1$ as an eigenvalue, the bending matrices may mix the summands of $\RR^4 \oplus \RR^{n'}$, so that the convex core is no longer contained in a copy of $\PP(\RR^4)$, and the group may even act irreducibly on $\PP(\RR^{4+n'})$, see the forthcoming paper~\cite{dgk-bad-ex}. Small such deformations give examples of discrete subgroups which are convex cocompact in $\PP(\RR^4 \oplus \RR^{n'})$ without being word hyperbolic, and which do not divide a properly convex~set.

\subsubsection{Convex cocompact actions coming from divisible convex sets by taking subgroups} \label{subsubsec:benoistPET}

Nonhyperbolic groups dividing a properly convex open set $\Omega \subset \PP(\RR^4)$ admit nonhyperbolic subgroups that are still convex cocompact in $\PP(\RR^4)$, but that do not divide any properly convex open subset of $\PP(\RR^4)$.  

Indeed, let $\Gamma$ be a torsion-free, discrete subgroup of $\PGL(\RR^4)$ dividing a nonempty properly convex open subset $\Omega$ of $\PP(\RR^4)$ containing PETs, as above. 
By~\cite{ben06}, the PETs descend to a finite collection $\mathcal{T}$ of pairwise disjoint planar tori and Klein bottles in the closed manifold $N := \Gamma \backslash \Omega$.
The universal cover of one connected component $M$ of $N\smallsetminus \mathcal{T}$ identifies with a convex subset $\C_M$ of~$\Omega$ whose interior is a connected component of the complement in $\Omega$ of the union of all PETs, and whose nonideal boundary $\partialn\C_M$ is a disjoint union of PETs.
The fundamental group of~$M$ identifies with the subgroup $\Gamma_M$ of $\Gamma$ that preserves~$\C_M$, and it acts properly discontinuously and cocompactly on~$\C_M$.

\begin{lemma}\label{lem:PETfriendly}
The action of $\Gamma_M$ on~$\Omega$ is convex cocompact.
\end{lemma}

\begin{proof}
By Corollary~\ref{cor:ideal-bound-naive-cc}.\eqref{item:ideal-bound-cc}, it is enough to check that $\Ccore_\Omega(\Gamma_M) \subset \C_M$.
Since each component of $\partialn \C_M$ is planar, $\C_M$ is the convex hull of its ideal boundary $\partiali \C_M$, and so it is enough to show that $\Lambdao_\Omega(\Gamma_M) \subset \partiali \C_M$.
Suppose by contradiction that this is not the case: namely, there exists $x \in \Omega \smallsetminus \C_M$ and a sequence $(\gamma_m)$ in $\Gamma_M$ such that $(\gamma_m\cdot x)$ converges to some $x_{\infty} \in \Lambdao_\Omega(\Gamma_M) \smallsetminus \partiali \C_M$.
Let $y$ be the point of $\partialn\C_M$ which is closest to~$x$ for the Hilbert metric $d_{\Omega}$; it is contained in a PET of $\partialn\C_M$.
Let $a,b\in\partial\Omega$ be such that $a,x,y,b$ are aligned in that order.
Up to taking a subsequence, we may assume that $(\gamma_m\cdot a)$, $(\gamma_m\cdot y)$, $(\gamma_m\cdot b)$ converge respectively to some $a_{\infty}, y_{\infty}, b_{\infty} \in \partial\Omega$, with $y_{\infty} \in \partiali \C_M$ and $a_{\infty}, x_{\infty}, y_{\infty}, b_{\infty}$ aligned in that order.
Since $\cro{\gamma_m\cdot a}{\gamma_m\cdot x}{\gamma_m\cdot y}{\gamma_m\cdot b} = \cro{a}{x}{y}{b} \in (1,+\infty)$ for all~$m$, and $x_{\infty} \neq y_{\infty}$, and all segments $[\gamma_m \cdot a, \gamma_m \cdot b]$ and $[a_{\infty}, b_{\infty}]$ lie in an affine chart containing~$\overline{\Omega}$, the points $a_{\infty}, x_{\infty}, y_{\infty}, b_{\infty}$ are pairwise distinct and contained in a segment of $\partial\Omega$. 
However, any segment on $\partial \Omega$ lies on the boundary of some PET~\cite{ben06}.
Thus we have found a PET whose closure intersects the closure of $\C_M$ but is not contained in it; its closure must cross the closure of a second PET on $\partialn \C_M$, contradicting the fact~\cite{ben06} that PETs have disjoint closures.
\end{proof}

\subsection{Convex cocompact groups as free products} \label{subsec:other-ex-nonhyp-cc}

Being convex cocompact in $\PP(V)$ is a much more flexible property than dividing a properly convex open subset of $\PP(V)$, and there is a rich world of examples, which we shall explore in forthcoming work \cite{dgk-bad-ex}.
In particular, we shall prove the following.

\begin{proposition}[{\cite{dgk-bad-ex}}]
Let $\Gamma_1$ and~$\Gamma_2$ be infinite discrete subgroups of $\PGL(V)$ which are convex cocompact in $\PP(V)$ but do not divide any nonempty properly convex open subset of $\PP(V)$.
Then there exists $g\in\PGL(V)$ such that the group generated by $\Gamma_1$ and $g\Gamma_2g^{-1}$ is isomorphic to the free product $\Gamma_1\ast\Gamma_2$ and is convex cocompact in $\PP(V)$. 
\end{proposition}

This yields many examples of non word hyperbolic convex cocompact groups. 
For instance, one could take $\Gamma_1=\Gamma_2 \subset \PGL(V')$ equal to one of the superrigid symmetric examples \ref{item:ex-ccc1} of Section~\ref{subsubsec:ex-div}, embedded into $\PGL(V)$ for some $V=V'\oplus \RR^n$.

\appendix

\section{Some open questions} \label{app:open-q}

Here we list some open questions about discrete subgroups $\Gamma$ of $\PGL(V)$ that are convex cocompact in $\PP(V)$.
The case that $\Gamma$ is word hyperbolic boils down to Anosov representations by Theorem~\ref{thm:Ano-PGL} and is reasonably understood; on the other hand, the case that $\Gamma$ is \emph{not} word hyperbolic corresponds to a new class of discrete groups whose study is still in its infancy.
Most of the following questions are interesting even in the case that $\Gamma$ divides a (nonstrictly convex) properly convex open subset of $\PP(V)$.

\emph{We fix a discrete subgroup $\Gamma$ of $\PGL(V)$ acting convex cocompactly on a nonempty properly convex open subset $\Omega$ of $\PP(V)$}.

\begin{question} \label{qu:cc-subgroup}
Which finitely generated subgroups $\Gamma'$ of~$\Gamma$ are convex cocompact in $\PP(V)$?
\end{question}

Such subgroups include all finite-index subgroups of~$\Gamma$ (see Lemma~\ref{lem:finite-index}).
In general, they are quasi-isometrically embedded in~$\Gamma$, by Corollary~\ref{cor:QI-embed}.
Conversely, when $\Gamma$ is word hyperbolic, any quasi-isometrically embedded (or equivalently quasi-convex) subgroup $\Gamma'$ of~$\Gamma$ is convex cocompact in $\PP(V)$: indeed, the boundary maps for the $P_1$-Anosov representation $\Gamma\hookrightarrow\PGL(V)$ induce boundary maps for $\Gamma'\hookrightarrow\PGL(V)$ that make it $P_1$-Anosov, hence $\Gamma'$ is convex cocompact in $\PP(V)$ by Theorem~\ref{thm:Ano-PGL}.

When $\Gamma$ is not word hyperbolic, Question~\ref{qu:cc-subgroup} becomes more subtle.
For instance, let $\Gamma\simeq\ZZ^2$ be the subgroup of $\PGL(\RR^3)$ consisting of all diagonal matrices whose entries are powers of some fixed $t>0$; it is convex cocompact in $\PP(\RR^3)$ (see Example~\ref{ex:Zn}).
Any cyclic subgroup $\Gamma' = \langle \gamma' \rangle$ of~$\Gamma$ is quasi-isometrically embedded in~$\Gamma$.
However, $\Gamma'$ is convex cocompact in $\PP(V)$ if and only if $\gamma'$ has distinct eigenvalues (see Examples~\ref{ex:triangle}.\eqref{ex:triangle:cc}--\eqref{ex:triangle:borderline}).

\begin{question}
Assume $\Omega$ is indecomposable.
Under what conditions is $\Gamma$ relatively hyperbolic, relative to a family of virtually abelian subgroups?
\end{question}

This is always the case if $\dim(V)\leq 3$.
If $\dim(V)=4$ and $\Gamma$ divides $\Omega$, then this is also seen to be true from work of Benoist~\cite{ben06}: in this case there are finitely many conjugacy classes of PET stabilizers in $\Gamma$, which are virtually $\ZZ^2$, and $\Gamma$ is relatively hyperbolic with respect to the PET stabilizers (using \cite[Th.\,0.1]{dah03}).
However, when $\dim(V)=m(m-1)/2$ for some $m\geq 3$, we can take for $\Omega\subset\PP(V)$ the projective realization of the Riemannian symmetric space of $\SL_m(\RR)$ (see \cite[\S\,2.4]{ben08}) and for $\Gamma$ a uniform lattice of $\SL_m(\RR)$: this $\Gamma$ is \emph{not} relatively hyperbolic with respect to any subgroups, see \cite{bdm09}.

\begin{question}\label{que:best}
If $\Gamma$ is not word hyperbolic, must there be a properly embedded maximal $k$-simplex invariant under some subgroup isomorphic to $\ZZ^k$ for some $k\geq 2$?
\end{question}

This question is a specialization, to the class of convex cocompact subgroups of $\PGL(V)$, of the following more general question (see \cite[Q\,1.1]{bes-questions}): if $\mathcal{G}$ is a finitely generated group admitting a finite $K(\mathcal{G},1)$, must it be word hyperbolic as soon as it does not contain any Baumslag--Solitar group $\mathrm{BS}(m,n)$?
Note that if $|n| = |m|$, then $\mathrm{BS}(m,n)$ contains $\ZZ^2$ (and indeed $\mathrm{BS}(1,1) = \ZZ^2$), while if $|n| \neq |m|$, then any linear embedding of $\mathrm{BS}(m,n)$ contains unipotent elements.
Hence, by Theorem~\ref{thm:properties}.\ref{item:cc-no-unipotent}, our convex cocompact group $\Gamma\subset\PGL(V)$ contains a Baumslag--Solitar group if and only if it contains~$\ZZ^2$.

In light of the equivalence ~\ref{item:ccc-hyp} $\Leftrightarrow$~\ref{item:P1Anosov} of Theorem~\ref{thm:main-noPETs}, we ask the following:

\begin{question}
When $\Gamma$ is not word hyperbolic, is there a dynamical description, similar to $P_1$-Anosov, that characterizes the action of $\Gamma$ on $\PP(V)$, for example in terms of $\Lambdao_\Omega(\Gamma)$ and divergence of Cartan projections?
\end{question}

While this question is vague, a good answer could lead to the definition of new classes of nonhyperbolic discrete subgroups in other higher-rank reductive Lie groups for which there is not necessarily a good notion of convexity.

A variant of the following question was asked by Olivier Guichard.

\begin{question}
In $\Hom(\Gamma,\PGL(V))$, does the interior of the set of naively convex cocompact representations consist of convex cocompact representations?
\end{question}

Here we say that a representation is naively convex cocompact (\resp convex cocompact) if it is faithful and if its image is naively convex cocompact in $\PP(V)$ (\resp convex cocompact in $\PP(V)$) in the sense of Definition~\ref{def:cc-naive} (\resp Definition~\ref{def:cc-general}).

\section{Limits of Hilbert balls} \label{app:hilbert}

For $n\geq 3$, let $\Omega \subset \PP(\RR^n)$ be a properly convex open set.
As in Definition~\ref{def:face}, for any $z \in \partial\Omega$, the \emph{open face} $F$ of $\partial\Omega$ at~$z$ is the union of $\{z\}$ and of all open segments of $\partial\Omega$ containing~$z$.
It is the largest convex subset of $\partial\Omega$ containing~$z$ which is relatively open, in the sense that it is open in the projective subspace $\PP(W)$ that is spans.
In particular, we can consider the Hilbert metric $d_F$ on~$F$ seen as a properly convex open subset of $\PP(W)$.

Let $R>0$.
By Lemma~\ref{lem:unif-neighb-face}.\eqref{item:distance-goes-down}, any Hausdorff limit of $R$-balls of $(\Omega,d_{\Omega})$ whose centers converge to~$z$, is contained in the $R$-ball of $(F,d_F)$ centered at~$z$.
In this appendix, we investigate to what extent the limit may be smaller than an $R$-ball.

The following example shows that the limit may be a point even when the face $F$ is not a point.
Thus \cite[Prop.\,7.7]{mar12bis} is not correct as stated.

\begin{example} \label{ex:twocones}
Consider $\HH^3$ as a properly convex open subset of $\PP(\RR^4)$ as in Example~\ref{ex:H-p}.
Let $\gamma \in \mathrm{Isom}(\HH^3) \subset \PGL(\RR^4)$ be a unipotent element.
Then $\gamma$ fixes pointwise a certain projective line $\ell$ of $\PP(\RR^4)$ tangent to $\partial \HH^3$ at a point~$z$.
Let $(a,b)$ be an open segment of $\ell$ containing~$z$, and let $\Omega$ be the interior of the convex hull of $\HH^3 \cup (a,b)$.
By construction, $\gamma$ preserves~$\Omega$, and $F := (a,b)$ is a face of $\partial\Omega$.
For any compact subset $B$ of $\HH^3$ (or indeed of $\PP(\RR^4) \smallsetminus \ell$), we have $\gamma^n\cdot B \to \{z\} \subset F$ as $n\to +\infty$.
In particular, we can take $B$ to be a closed ball of $(\Omega, d_\Omega)$; the sets $\gamma^n\cdot B$ are then all balls of the same radius in~$\Omega$ which limit to a point in~$F$.
\end{example}

The following lemma states that in the case that the centers of the $R$-balls converge conically to~$z$ (Definition~\ref{def:conical}), the Hausdorff limit does contain a nontrivial ball of $(F,d_F)$ centered at~$z$, possibly of smaller radius.
For $R>0$ and $x\in\Omega$ (\resp $z\in F\subset\partial\Omega$), we denote by $\overline{\mathbb{B}}_{\Omega}(x,R)$ (\resp $\overline{\mathbb{B}}_F(z,R)$) the closed ball of radius~$R$ centered at $x$ (\resp $z$) in $(\Omega,d_{\Omega})$ (\resp $(F,d_F)$).

\begin{lemma} \label{lem:loss}
Let $\Omega$ be a properly convex open subset of $\PP(V)$ and $(x_m)_{m\in\NN}$ a sequence of points of~$\Omega$ converging to some $z\in\partial\Omega$.
Suppose there exist a ray $[y,z) \subset \Omega$ and a constant $D>0$ such that $d_\Omega(x_m,[y,z)) \leq D$ for all $m\in\NN$.
Let $F$ be the open face of $z$ in $\partial\Omega$.
Then for any $R\geq 0$, any Hausdorff limit (\ie accumulation point for the Hausdorff topology) of the balls $\overline{\mathbb{B}}_{{\Omega}}(x_m, R)$ is a subset of~$F$ containing the ball $\overline{\mathbb{B}}_{F}(z, f_D(R))$ where
$$f_D(R) := \frac{1}{2} \log \left ( 1+ \frac{e^{2R}-1}{e^{2D}}\right ) \geq R \, e^{-2D}.$$
\end{lemma}

Note that the function $f_D: \RR_+ \rightarrow \RR_+$ is convex for each $D\geq 0$, with $f_0=\mathrm{Id}_{\RR^+}$.
We have $f_D(R) \sim R e^{-2D}$ as $R \rightarrow 0$, and $f_D(R)=R-D+o(1)$ as $R\rightarrow +\infty$. 

When $\Omega$ has dimension~$2$, no loss occurs: Lemma~\ref{lem:loss} holds with $f_D$ replaced by $\mathrm{Id}_{\RR^+}$ (in the proof below, $(a_\infty, b_\infty)=(a,b)$ automatically).

\begin{proof}
Up to passing to a subsequence, we may assume that the balls $\overline{\mathbb{B}}_{{\Omega}}(x_m, R)$ admit a Hausdorff limit.
This limit is a closed convex subset of $\PP(V)$ which is contained in $F\subset\partial\Omega$ by Lemma~\ref{lem:unif-neighb-face}.\eqref{item:distance-goes-down}.
It is sufficient to prove that for any maximal open segment $(a,b) \subset F$ containing~$z$, the limit of the $\overline{\mathbb{B}}_{{\Omega}}(x_m, R)$ contains the point of $(z,b)$ at $d_F$-distance $f_D(R)$ from~$z$.
We now fix such a segment $(a,b)$.

For each $m\in\NN$, choose $y_m \in [y,z)$ such that $d_\Omega(x_m, y_m) \leq D$. 
Consider an arbitrary number $h>R$ (later we will take $h\rightarrow \infty$), and let $z' \in (z,b)$ be such that $d_F(z, z')=h$.
For each $m\in\NN$, choose $y'_m \in [y, z')$ such that the line through $y_m$ and $y'_m$ intersects $[y,a)$ and $[y,b)$.
Since $\Omega$ contains the open triangle $T:=(y,a,b)$, we have $d_{\Omega}(y_m, y'_m) \leq d_T(y_m, y'_m) = \nolinebreak h$ (Remark~\ref{rem:Hilb-metric-include}).
By the triangle inequality, $d_\Omega (x_m, y'_m) \leq d_\Omega (x_m, y_m) + d_\Omega (y_m, y'_m) \leq D+h$. 

Let $a_m, b_m \in \partial \Omega$ be such that $a_m, x_m, y'_m, b_m$ are aligned in this order.
In dimension~$\geq 3$, it is not necessarily the case that $a_m \to a$ or $b_m \to b$. 
However, up to passing to a subsequence we may assume that $a_m \to a_{\infty}$ and $b_m \to b_{\infty}$ for some $a_{\infty},b_{\infty} \in \partial\Omega$ with $(a_\infty, b_\infty) \subset (a,b)$.
As in the proof of Lemma~\ref{lem:unif-neighb-face}.\eqref{item:distance-goes-down}, we have
\begin{equation} \label{eqn:3lines}
d_{(a_\infty, b_\infty)}(z, z') \leq \limsup_{m\to +\infty} d_\Omega (x_m, y'_m) \leq D + h.
\end{equation}
For any~$m$, the point $w_m \in (x_m, b_m)$ with $\cro{a_m}{x_m}{w_m}{b_m} = e^{2R}$ belongs to $\overline{\mathbb{B}}_{{\Omega}}(x_m, R)$.
Therefore the limit of the $\overline{\mathbb{B}}_{{\Omega}}(x_m, R)$ contains the point $w\in (a,b)$ such that 
\begin{equation} \label{eqn:wis}
\cro{a_\infty}{z}{w}{b_\infty}=e^{2R}.
\end{equation}
The assumption $h>R$ implies $w\in (z,z')$: indeed,
\begin{equation} \label{eqn:3nains}
0 < d_{(a_\infty, b_\infty)}(z,w) = R < h = d_{(a,b)}(z,z') \leq d_{(a_\infty, b_\infty)}(z,z')
\end{equation}
since $(a_\infty, b_\infty) \subset (a,b)$.
It is sufficient to prove that for any $\varepsilon>0$, if $h$ has been chosen large enough, then $d_F(z,w) \geq f_D(R) - \varepsilon$.

We map $\PP(\mathrm{span} \{a,b\})$ to the standard projective line by identifying $a,z,z',b$ with $0,1,e^{2h}, \infty \in \PP^1(\RR)$, so that
\begin{equation} \label{eqn:7nains}
0 = a \:\:\leq\:\: a_\infty \:\:<\:\: z = 1 \:\:<\:\: w \:\:<\:\: e^{2h} = z' \:\:<\:\: b_\infty \:\:\leq\:\: b =\infty
\end{equation}
and we aim to show $w\geq e^{2 (f_D(R)-\varepsilon)}$.
The number $\Delta:= \cro{a_\infty}{z}{z'}{b_\infty}$ satisfies
\begin{equation} \label{eqn:Deltais}
e^{2h} \underset{\text{\eqref{eqn:3nains}}}{\leq} \Delta 
\underset{\text{\eqref{eqn:3lines}}}{\leq} e^{2(h+D)} \quad \text{and} \quad
 \Delta \underset{\text{\eqref{eqn:7nains}}}{=} \cro{a_\infty}{1}{e^{2h}}{b_\infty}.
\end{equation} 
Let us express $a_\infty$ and~$w$ in terms of $(\Delta$, $b_\infty)\in [e^{2h}, e^{2(h+D)}]\times (e^{2h},+\infty]$ and of the fixed parameters $h$, $R$:
\begin{align}
a_\infty & =  \frac{(b_\infty - e^{2h})\Delta - e^{2h}(b_\infty-1)}{(b_\infty - e^{2h})\Delta - (b_\infty-1)}  \quad \quad  \text{by the equality of~\eqref{eqn:Deltais};} \label{eqn:b7} 
\end{align}
\begin{align}
w& = \frac{b_\infty e^{2R}(1 - a_\infty) + a_\infty (b_\infty-1)}{e^{2R}(1 - a_\infty) + (b_\infty-1)}
\quad \quad \text{using $z=1$ and~\eqref{eqn:wis}} \notag 
\\ &= 1 + \frac{(e^{2R}-1)(e^{2h}-1)}{\Delta-1} + \frac{(e^{2R}-1)(e^{2h}-1)^2 (\Delta-e^{2R})(\Delta-1)^{-1}}{(b_\infty-e^{2h})(\Delta-1) + (e^{2R}-1)(e^{2h}-1)} \label{eqn:dragon}
\end{align}
by substituting~\eqref{eqn:b7} for $a_\infty$ and a routine computation.
In~\eqref{eqn:dragon}, the variable $b_\infty$ appears only once, and every bracket is positive, using~\eqref{eqn:7nains}--\eqref{eqn:Deltais} and $h>R$.
Hence,
$$\min_{\Delta\in [e^{2h}, e^{2(h+D)}]} \: \min_{b_\infty \in (e^{2h},+\infty]} \, w = 1 + \frac{(e^{2R}-1)(e^{2h}-1)}{e^{2(h+D)}-1} \underset{h \rightarrow +\infty}{\longrightarrow} 1+\frac{e^{2R}-1}{e^{2D}} = e^{2 f_D(R)}. \qedhere$$
\end{proof}

The following example shows that the function $f_D$ of Lemma~\ref{lem:loss} is best possible.

\begin{example} \label{ex:two-p-cones} 
Fix $R,D>0$. 
Let $a:=(-1,0,0)$, $b:=(1,0,0)$, and for each $0<\varepsilon < 1 < p$, consider the properly convex open set
$$\Omega_{\varepsilon, p} : = \mathrm{Int} \left ( \mathrm{Conv} \left \{ a,b, (0,t,t^p)_{t\geq 0}, (0,t,|\varepsilon t|^p)_{t\leq 0} \right \}\right ) \subset \RR^3 \subset \PP^3(\RR). $$
The open segment $F := (a,b)$ is the open face of~$\Omega_{\varepsilon,p}$ at $z:=(0,0,0) \in \partial\Omega_{\varepsilon,p}$.
The map $\gamma_p: (u,v,w) \mapsto (u, v/2, w/2^p)$ preserves $\Omega_{\varepsilon,p}$, and preserves the axis $L := \{(0,0,t)_{t>0}\}$ with endpoint~$z$.
Let $x_0:=(0,1-e^{-2D},1) \in \Omega_{\varepsilon,p}$ and $x_m := \gamma_p^m \cdot x_0 \in \Omega_{\varepsilon,p}$.
As $(p,\varepsilon) \rightarrow (+\infty,0)$, the domain $\Omega_{\varepsilon,p}$ converges for the Hausdorff topology to $\Omega_\infty:=\Pi \times \RR_{>0}$ where
$$\Pi = \{(u,v) \in \RR^2 ~|~ -1 < u<1,~ v<1-|u| \}$$
(projectively, $\Omega_\infty$ is equivalent to a square-based pyramid).
Therefore $d_{\Omega_{\varepsilon,p}}(x_0, L)$ converges to $d_{\Omega_\infty}(x_0, L) = D$ as $(p,\varepsilon) \rightarrow (+\infty,0)$.
Moreover, the balls $\overline{\mathbb{B}}_{\Omega_{\varepsilon,p}} (x_0, R)$ converge, for the Hausdorff topology in $\PP^3(\RR)$, to $\overline{\mathbb{B}}_{\Omega_{\infty}} (x_0, R)$.
But the latter projects on the first axis to $[-s,s]$, where $s := (1+e^{2D-R}/\sinh(R))^{-1}$: this can be seen by computing balls in $(\Pi, d_\Pi)$ (Figure~\ref{fig:quadball}), and using that the projection $\pi_\Pi: \Omega_\infty \rightarrow \Pi$ is $1$-Lipschitz. 
Therefore, for large enough $p$ and small enough~$\varepsilon$, the Hausdorff limit of $(\gamma_p^m\cdot\overline{\mathbb{B}}_{\Omega_{\varepsilon,p}}(x_0, R))_{m\geq 0}$ becomes arbitrarily Hausdorff-close to $[-s,s]\times \{0\}^2$.
Since $\frac{1}{2} \log \cro{-1}{0}{s}{1}=f_D(R)$, it follows that: \emph{for any $R'>f_D(R)$, there exist $p>1$ (large), $\varepsilon>0$ (small) and a sequence $x_m =\gamma_p^m\cdot x_0 \rightarrow z $ in $\Omega_{\varepsilon, p}$ such that the $x_m$ lie within $D$ from the axis $L$ but the Hausdorff limit of the $\overline{\mathbb{B}}_{\Omega_{\varepsilon,p}}(x_m,R)$ does not contain $\overline{\mathbb{B}}_F(z, R')$}.

Note that we could also consider $\langle \gamma_p\rangle$-orbits in the domains $\Omega_p=\lim_{\varepsilon \rightarrow 0}\Omega_{\varepsilon, p}$ (for increasingly large $p$) and obtain similar estimates: but the $\Omega_{\varepsilon, p}$ have the clean feature that their closures intersect the supporting plane $\RR^2\times \{0\}$ precisely along $\overline{F} = [a,b]$.
\end{example}

\begin{figure}
\centering
\labellist
\small\hair 2pt
\pinlabel {$(0,1)$} [u] at 410 477 
\pinlabel {$(1,0)$} [u] at 719 195  
\pinlabel {$(-1,0)$} [u] at -3 195  
\pinlabel {$(0,0)$} [u] at 329 175 
\pinlabel {$\Pi$} [u] at 57 55 
\pinlabel {${(0 , 1\text{-}e^{\text{-}2D})}$} [u] at 300 350 
\pinlabel {$x$} [u] at 375 317 
\pinlabel {${\color{blue} \mathbb{B}_\Pi(x,R)}$} [u] at 237 230 
\pinlabel {$(s,0)$} [u] at 447 175 
\pinlabel {$(\tanh(R),0)$} [u] at 564 173 
\endlabellist
\includegraphics[width = 9cm]{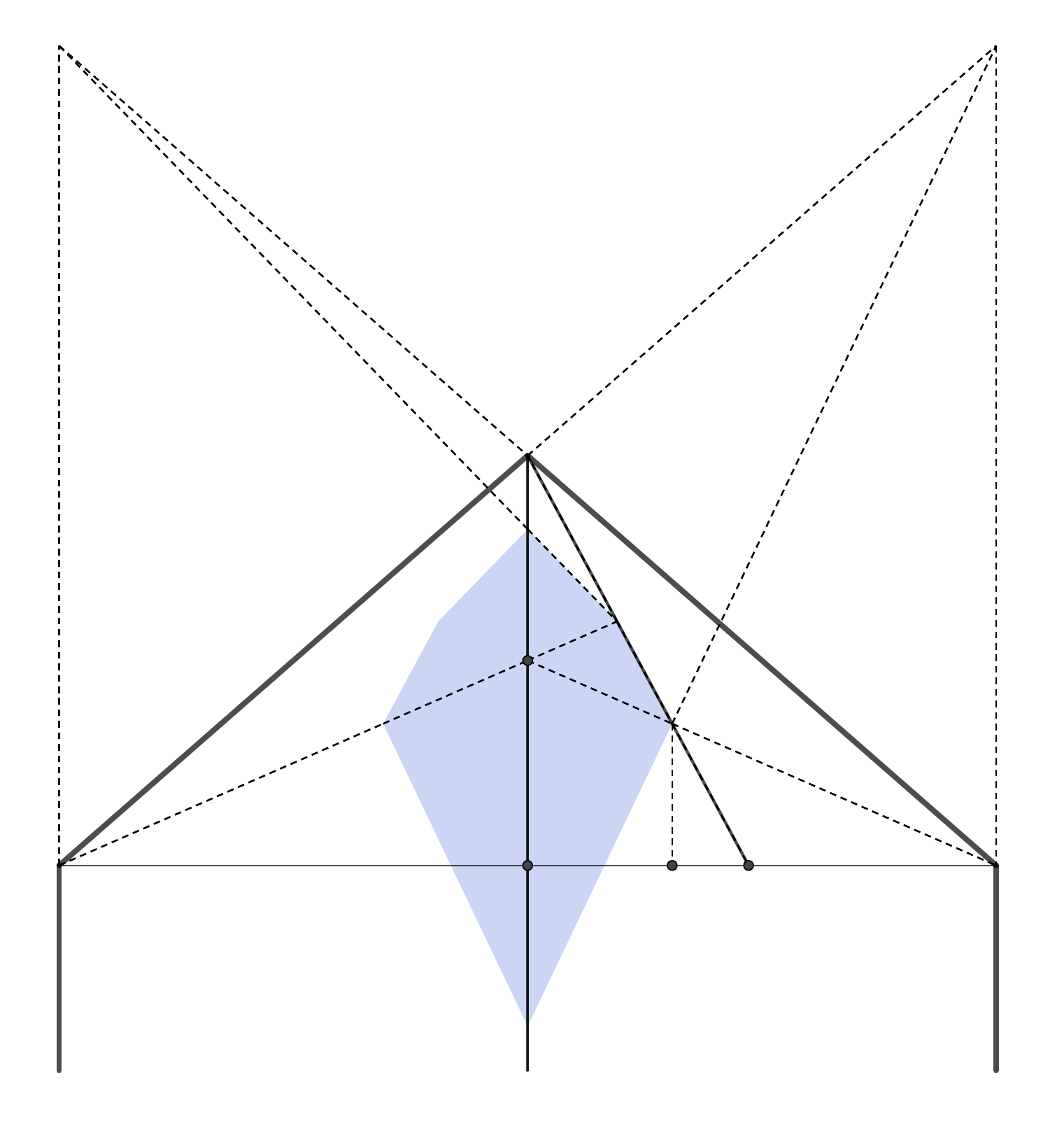}
\caption{A Hilbert ball of radius $R$ in the open set $\Pi \subset \RR^2$ of Example~\ref{ex:two-p-cones}, centered at $x:=\pi_\Pi(x_0)$. The horizontal segment is $[\pi_\Pi(a), \pi_\Pi(b)]$. The number $s$ can be computed as a function of $(R,D)$.}
\label{fig:quadball}
\end{figure}



\begin{thebibliography}{DGKLM}

\bibitem[BDL]{bdl18}
\textsc{S. Ballas, J. Danciger, G.-S. Lee}, \textit{Convex projective structures on non-hyperbolic three-manifolds}, Geom. Topol.~22 (2018), p.~1593--1646.

\bibitem[Ba]{bar15}
\textsc{T. Barbot}, \textit{Deformations of Fuchsian AdS representations are quasi-Fuchsian}, J. Differential Geom.~101 (2015), p.~1--46.
  
\bibitem[BM]{bm12}
\textsc{T. Barbot, Q. M\'erigot}, \textit{Anosov AdS representations are quasi-Fuchsian}, Groups Geom. Dyn.~6 (2012), p.~441--483.

\bibitem[BDM]{bdm09}
\textsc{J. Behrstock, C. Drutu, L. Mosher}, \textit{Thick metric spaces, relative hyperbolicity, and quasi-isometric rigidity}, Math. Ann.~344 (2009), p.~543--595.

\bibitem[BK]{bk02}
\textsc{N. Benakli, I. Kapovich}, \textit{Boundaries of hyperbolic groups}, in \textit{Combinatorial and geometric group theory (New York, 2000/Hoboken, NJ, 2001)}, p.~39--93, Contemporary Mathematics, vol.~296, American Mathematical Society, Providence, RI, 2002.

\bibitem[B1]{ben97}
\textsc{Y. Benoist}, \textit{Propri\'et\'es asymptotiques des groupes lin\'eaires}, Geom. Funct. Anal.~7 (1997), p.~1--47.

\bibitem[B2]{ben00}
\textsc{Y. Benoist}, \textit{Automorphismes des c\^ones convexes}, Invent. Math.~141 (2000), p.~149--193.

\bibitem[B3]{ben04}
\textsc{Y. Benoist}, \textit{Convexes divisibles I}, in \textit{Algebraic groups and arithmetic}, Tata Inst. Fund. Res. Stud. Math.~17 (2004), p.~339--374.

\bibitem[B4]{ben03}
\textsc{Y. Benoist}, \textit{Convexes divisibles II}, Duke Math. J.~120 (2003), p.~97--120.

\bibitem[B5]{ben05}
\textsc{Y. Benoist}, \textit{Convexes divisibles III}, Ann. Sci. \'Ec. Norm. Sup.~38 (2005), p.~793--832.

\bibitem[B6]{ben06}
\textsc{Y. Benoist}, \textit{Convexes divisibles IV}, Invent. Math.~164 (2006), p.~249--278.

\bibitem[B7]{ben08}
\textsc{Y. Benoist}, \textit{A survey on divisible convex sets}, Adv. Lect. Math.~6 (2008), p. 1--18.

\bibitem[Be]{bes-questions}
\textsc{M. Bestvina}, \textit{Questions in Geometric Group Theory}, see \url{https://www.math.utah.edu/~bestvina/eprints/questions-updated.pdf}.

\bibitem[BeM]{bm91}
\textsc{M. Bestvina, G. Mess}, \textit{The boundary of negatively curved groups}, J. Amer. Math. Soc.~4 (1991), p.~469--481.

\bibitem[BPS]{bps}
\textsc{J. Bochi, R. Potrie, A. Sambarino}, \textit{Anosov representations and dominated splittings}, J. Eur. Math. Soc.~21 (2019), p.~3343--3414.

\bibitem[Bo]{bourbaki}
\textsc{N. Bourbaki}, \textit{Elements of Mathematics. General Topology}, Parts I--IV, Hermann, Paris, 1966.

\bibitem[BCLS]{bcls15}
\textsc{M. Bridgeman, R. D. Canary, F. Labourie, A. Sambarino}, \textit{The pressure metric for Anosov representations}, Geom. Funct. Anal.~25 (2015), p.~1089--1179.

\bibitem[BIW1]{biw10}
\textsc{M. Burger, A. Iozzi, A. Wienhard}, \textit{Surface group representations with maximal Toledo invariant}, Ann. Math.~172 (2010), p.~517--566.

\bibitem[BIW2]{biw14}
\textsc{M. Burger, A. Iozzi, A. Wienhard}, \textit{Higher Teichm\"uller spaces: from $\SL(2,\RR)$ to other Lie groups}, Handbook of Teichm\"uller theory~IV, p.~539--618, IRMA Lect.\ Math.\ Theor.\ Phys.~19, 2014.

\bibitem[BIW3]{biw}
\textsc{M. Burger, A. Iozzi, A. Wienhard}, \textit{Maximal representations and Anosov structures}, in preparation.

\bibitem[Bu]{bus55}
\textsc{H. Busemann}, \textit{The geometry of geodesics}, Academic Press Inc., New York, 1955.

\bibitem[CGo]{cg05}
\textsc{S. Choi, W. M. Goldman}, \textit{The deformation space of convex $\RR\PP_2$-structures on $2$-orbifolds}, Amer. J. Math.~127 (2005) p.~1019--1102.

\bibitem[CLM]{clm20}
\textsc{S. Choi, G.-S. Lee, L. Marquis}, \textit{Convex projective generalized Dehn filling}, Ann. Sci. \'Ec. Norm. Sup.~53 (2020), p.~217--266.

\bibitem[CGe]{cg69}
\textsc{Y. Choquet-Bruhat, R. Geroch}, \textit{Global aspects of the Cauchy problem in general relativity}, Commun. Math. Phys.~14 (1969), p.~329--335.

\bibitem[CTT]{ctt19}
\textsc{B. Collier, N. Tholozan, J. Toulisse}, \textit{The geometry of maximal representations of surface groups into $\SO(2,n)$}, Duke Math. J.~168 (2019), p.~2873--2949.

\bibitem[CLT1]{clt15}
\textsc{D. Cooper, D. Long, S. Tillmann}, \textit{On projective manifolds and cusps}, Adv. Math.~277 (2015), p.~181--251.

\bibitem[CLT2]{clt18}
\textsc{D. Cooper, D. Long, S. Tillmann}, \textit{Deforming convex projective manifolds}, Geom. Topol.~22 (2018), p.~1349--1404.

\bibitem[CM]{cm14}
\textsc{M. Crampon, L. Marquis}, \textit{Finitude g\'eom\'etrique en g\'eom\'etrie de Hilbert}, with an appendix by C. Vernicos, Ann. Inst. Fourier~64 (2014), p.~2299--2377.

\bibitem[Dah]{dah03}
\textsc{F. Dahmani}, \textit{Combination of convergence groups}, Geom. Topol.~7 (2003), p.~933--963.

\bibitem[Dan]{dan13}
\textsc{J. Danciger}, \textit{A geometric transition from hyperbolic to anti de Sitter geometry}, Geom. Topol.~17 (2013), p.~3077--3134.

\bibitem[DGK1]{dgk16}
\textsc{J. Danciger, F. Gu\'eritaud, F. Kassel}, \textit{Geometry and topology of complete Lorentz spacetimes of constant curvature}, Ann. Sci. \'Ec. Norm. Sup\'er.~49 (2016), p.~1--56.

\bibitem[DGK2]{dgk-ccHpq}
\textsc{J. Danciger, F. Gu\'eritaud, F. Kassel}, \textit{Convex cocompactness in pseudo-Riemannian hyperbolic spaces}, Geom. Dedicata~192 (2018), p.~87--126, special issue \textit{Geometries: A Celebration of Bill Goldman's 60th Birthday}.

\bibitem[DGK3]{dgk-bad-ex}
\textsc{J. Danciger, F. Gu\'eritaud, F. Kassel}, \textit{Examples and non-examples of convex cocompact groups in projective space}, in preparation.

\bibitem[DGKLM]{dgklm}
\textsc{J. Danciger, F. Gu\'eritaud, F. Kassel, G.-S. Lee, L. Marquis}, \textit{Convex cocompactness for Coxeter groups}, preprint, arXiv:2102.02757.

\bibitem[DK]{dk-future}
\textsc{J. Danciger, S. Kerckhoff}, \textit{The transition from quasifuchsian manifolds to {AdS} globally hyperbolic spacetimes}, in preparation.

\bibitem[DaGrKl]{dgk63}
\textsc{L. Danzer, B. Gr\"unbaum, V. Klee}, \textit{Helly's theorem and its relatives}, in \emph{Convexity (Proceedings of the Symposia on Pure Mathematics, vol.~7)}, p.~101--180, American Mathematical Society, Providence, RI, 1963.

\bibitem[E]{eng78}
R. Engelking, \textit{Dimension theory}, North-Holland Mathematical Library, vol.~19, North-Holland Publishing Company, Amsterdam, 1978.

\bibitem[FG]{fg06}
\textsc{V. V. Fock,  A. B. Goncharov}, \textit{Moduli spaces of local systems and higher Teichm\"uller theory}, Publ. Math. Inst. Hautes \'Etudes Sci.~103 (2006), p.~1--211.

\bibitem[FK]{fk05}
\textsc{T. Foertsch, A. Karlsson}, \textit{Hilbert metrics and Minkowski norms}, J. Geom.~83 (2005), p.~22--31.

\bibitem[Go]{gol90}
\textsc{W. M. Goldman}, \textit{Convex real projective structures on surfaces}, J. Differential Geom.~31 (1990), p.~791--845.

\bibitem[GGKW]{ggkw17}
\textsc{F. Gu\'eritaud, O. Guichard, F. Kassel, A. Wienhard}, \textit{Anosov representations and proper actions}, Geom. Topol.~21 (2017), p.~485--584.

\bibitem[GW1]{gw08}
\textsc{O. Guichard, A. Wienhard}, \textit{Convex foliated projective structures and the Hitchin component for $\PSL_4(\RR)$}, Duke Math. J.~144 (2008), p.~381--445.

\bibitem[GW2]{gw10}
\textsc{O. Guichard, A. Wienhard}, \textit{Topological invariants of Anosov representations}, J. Topol.~3 (2010), p.~578--642.

\bibitem[GW3]{gw12}
\textsc{O. Guichard, A. Wienhard}, \textit{Anosov representations : Domains of discontinuity and applications}, Invent. Math.~190 (2012), p.~357--438.

\bibitem[Gu]{gui90}
\textsc{Y. Guivarc'h}, \textit{Produits de matrices al\'eatoires et applications aux propri\'et\'es g\'eom\'etriques des sous-groupes du groupe lin\'eaire}, Ergodic Theory Dynam. Systems~10 (1990), p.~483--512.

\bibitem[H]{hit92}
\textsc{N. J. Hitchin}, \textit{Lie groups and Teichm\"uller space}, Topology~31 (1992), p.~339--365.

\bibitem[JM]{jm87}
\textsc{D. Johnson, J. J. Millson}, \textit{Deformation spaces associated to compact hyperbolic manifolds}, in \textit{Discrete groups in geometry and analysis}, p.~48--106, Prog. Math.~67, Birkh\"auser, Boston, MA, 1987.

\bibitem[Ka]{kap-prop-disc}
\textsc{M. Kapovich}, \textit{A note on properly discontinuous actions}, see \url{https://www.math.ucdavis.edu/~kapovich/EPR/prop-disc.pdf}.

\bibitem[KLPa]{klp14}
\textsc{M. Kapovich, B. Leeb, J. Porti}, \textit{Morse actions of discrete groups on symmetric spaces}, preprint, arXiv:1403.7671.
  
\bibitem[KLPb]{klp18}
\textsc{M. Kapovich, B. Leeb, J. Porti}, \textit{Dynamics on flag manifolds: domains of proper discontinuity and cocompactness}, Geom. Topol.~22 (2018), p.~157--234.

\bibitem[KLPc]{klp-survey}
\textsc{M. Kapovich, B. Leeb, J. Porti}, \textit{Some recent results on Anosov representations}, Transform. Groups~21 (2016), p.~1105--1121.

\bibitem[KL]{kl06}
\textsc{B. Kleiner, B. Leeb}, \textit{Rigidity of invariant convex sets in symmetric spaces}, Invent. Math.~163 (2006), p.~657--676.

\bibitem[Ko]{kos68}
\textsc{J.-L. Koszul}, \textit{D\'eformation des connexions localement plates}, Ann. Inst. Fourier~18 (1968), p~103--114.

\bibitem[L]{lab06}
\textsc{F. Labourie}, \textit{Anosov flows, surface groups and curves in projective space}, Invent. Math.~165 (2006), p.~51--114.

\bibitem[LM]{lm19}
\textsc{G.-S. Lee, L. Marquis}, \textit{Anti-de Sitter strictly GHC-regular groups which are not lattices}, Trans. Amer. Math. Soc.~372 (2019), p.~153--186.

\bibitem[Ma1]{mar12}
\textsc{L. Marquis}, \textit{Surface projective convexe de volume fini}, Ann. Inst. Fourier~62 (2012), p.~325--392.

\bibitem[Ma2]{mar12bis}
\textsc{L. Marquis}, \textit{Exemples de vari\'et\'es projectives strictement convexes de volume fini en dimension quelconque}, Enseign. Math.~58 (2012), p.~3--47.

\bibitem[Me]{mes90}
\textsc{G. Mess}, \textit{Lorentz spacetimes of constant curvature} (1990), Geom. Dedicata~126 (2007), p.~3--45.

\bibitem[Q]{qui05}
\textsc{J.-F. Quint}, \textit{Groupes convexes cocompacts en rang sup\'erieur}, Geom. Dedicata~113 (2005), p.~1--19.

\bibitem[Rad]{rad16}
\textsc{J. Radon}, \textit{\"Uber eine Erweiterung des Begriffs der konvexen Funktionen, mit einer Anwendung auf die Theorie der konvexen K\"orper}, S.-B. Akad. Wiss. Wien~125 (1916), p.~241--258.

\bibitem[Rag]{rag72}
\textsc{M. S. Raghunathan}, \textit{Discrete subgroups of Lie groups}, Springer, New York, 1972.

\bibitem[Se]{sel60}
\textsc{A. Selberg}, \textit{On discontinuous groups in higher-dimensional symmetric spaces} (1960), in ``Collected papers'', vol.~1, p.~475--492, Springer-Verlag, Berlin, 1989.

\bibitem[Su]{sul85}
\textsc{D. Sullivan}, \textit{Quasiconformal homeomorphisms and dynamics II: Structural stability implies hyperbolicity for Kleinian groups}, Acta Math.~155 (1985), p.~243--260.

\bibitem[T]{thu80}
\textsc{W. P. Thurston}, \textit{The geometry and topology of three-manifolds}, lecture notes, 1980.

\bibitem[Ve]{vey70}
\textsc{J. Vey}, \textit{Sur les automorphismes affines des ouverts convexes saillants}, Ann. Sc. Norm. Super.
Pisa~24 (1970), p.~641--665.

\bibitem[We]{Weisman20}
\textsc{T. Weisman}, \textit{Dynamical properties of convex cocompact actions in projective space}, preprint, arXiv:2009:10994, 2020.

\bibitem[Wo]{wolf}
\textsc{J. A. Wolf}, \textit{Spaces of constant curvature}, sixth edition, AMS Chelsea Publishing, American Mathematical Society, Providence, RI, 2011.

\bibitem[Z]{zim}
\textsc{A. Zimmer}, \textit{Projective Anosov representations, convex cocompact actions, and rigidity}, J. Differential Geom.~119, p.~513-586, 2021.

\end{thebibliography}
\end{document}